\documentclass[oneside,english]{amsart}

\usepackage[T1]{fontenc}
\usepackage[latin9]{inputenc}
\usepackage{verbatim}
\usepackage{prettyref}
\usepackage{textcomp}
\usepackage{amstext}
\usepackage{amsthm}
\usepackage{amssymb}

\makeatletter
\numberwithin{equation}{section}
\numberwithin{figure}{section}
\theoremstyle{plain}
\newtheorem{thm}{\protect\theoremname}[section]
  \theoremstyle{plain}
  \newtheorem{prop}[thm]{\protect\propositionname}
  \theoremstyle{plain}
  \newtheorem{cor}[thm]{\protect\corollaryname}
  \theoremstyle{definition}
  \newtheorem{defn}[thm]{\protect\definitionname}
  \theoremstyle{plain}
  \newtheorem{lem}[thm]{\protect\lemmaname}
  \theoremstyle{definition}
  \newtheorem{example}[thm]{\protect\examplename}
  \theoremstyle{remark}
  \newtheorem{rem}[thm]{\protect\remarkname}
  \theoremstyle{remark}
  \newtheorem{claim}[thm]{\protect\claimname}

\makeatother

\usepackage{babel}
  \providecommand{\claimname}{Claim}
  \providecommand{\corollaryname}{Corollary}
  \providecommand{\definitionname}{Definition}
  \providecommand{\examplename}{Example}
  \providecommand{\lemmaname}{Lemma}
  \providecommand{\propositionname}{Proposition}
  \providecommand{\remarkname}{Remark}
\providecommand{\theoremname}{Theorem}

\begin{document}

\title{Witt Differential Operators}

\author{Christopher Dodd}
\begin{abstract}
For a smooth scheme $X$ over a perfect field $k$ of positive characteristic,
we define (for each $m\in\mathbb{Z}$) a sheaf of rings $\mathcal{\widehat{D}}_{W(X)}^{(m)}$
of differential operators (of level $m$) over the Witt vectors of
$X$. If $\mathfrak{X}$ is a lift of $X$ to a smooth formal scheme
over $W(k)$, then for $m\geq0$ modules over $\mathcal{\widehat{D}}_{W(X)}^{(m)}$
are closely related to modules over Berthelot's ring $\widehat{\mathcal{D}}_{\mathfrak{X}}^{(m)}$
of differential operators of level $m$ on $\mathfrak{X}$. Our construction
therefore gives an description of suitable categories of modules over
these algebras, which depends only on the special fibre $X$. There
is an embedding of the category of crystals on $X$ (over $W_{r}(k)$)
into modules over $\mathcal{\widehat{D}}_{W(X)}^{(0)}/p^{r}$; and
so we obtain an alternate description of this category as well. For
a map $\varphi:X\to Y$ we develop the formalism of pullback and pushforward
of $\mathcal{\widehat{D}}_{W(X)}^{(m)}$-modules and show all of the
expected properties. When working mod $p^{r}$, this includes compatibility
with the corresponding formalism for crystals, assuming $\varphi$
is smooth. In this case we also show that there is a ``relative de
Rham Witt resolution'' (analogous to the usual relative de Rham resolution
in $\mathcal{D}$-module theory) and therefore that the pushforward
of (a quite general subcategory of) modules over $\mathcal{\widehat{D}}_{W(X)}^{(0)}/p^{r}$
can be computed via the reduction mod $p^{r}$ of Langer-Zink's relative
de Rham Witt complex. Finally we explain a generalization of Bloch's
theorem relating integrable de Rham-Witt connections to crystals.
\end{abstract}

\maketitle
\tableofcontents{}

\section{Introduction}

Let $X$ be a smooth algebraic variety over $\mathbb{C}$. Then one
may attach to $X$ its (algebraic) de Rham cohomology, $\mathbb{H}_{\text{dR}}^{\cdot}(X)$,
which is defined as the hypercohomology of the algebraic de Rham complex
of $X$. It is a fundamental theorem of Grothendieck that this cohomology
is isomorphic to the singular cohomology of the analytic space attached
to $X$, $X^{\text{an}}$. Guided by the philosophy that any cohomology
theory should come with a good theory of local coefficients, Grothendieck
defined a site, the infinitesimal site of $X$, which has the property
that the de Rham cohomology of $X$ is isomorphic to the cohomology
of the structure sheaf of $X$ in the infinitesimal site. He further
showed that the data of a sheaf on the infinitesimal site is equivalent
to the data of an $\mathcal{O}_{X}$-module $\mathcal{M}$ (in the
usual Zariski topology of $X$), equipped with a flat connection.
Let us recall that this is a morphism
\[
\nabla:\mathcal{M}\to\mathcal{M}\otimes\Omega_{X}^{1}
\]
which satisfies the Leibniz rule $\nabla(f\cdot m)=f\cdot\nabla(m)+m\otimes df$,
as well as the following: we can extend $\nabla$ to a morphism 
\[
\nabla:\mathcal{M}\otimes\Omega_{X}^{i}\to\mathcal{M}\otimes\Omega_{X}^{i+1}
\]
via the rule $\nabla(m\otimes\phi)=\nabla(m)\cdot\phi+m\otimes d\phi$,
where the multiplication in the first term denotes the map $(\mathcal{M}\otimes\Omega_{X}^{1})\otimes\Omega_{X}^{i}\to\mathcal{M}\otimes\Omega_{X}^{i+1}$
coming from multiplication in the de Rham complex. Then we demand
that $\nabla\circ\nabla=0$ (this is what it means for a connection
to be flat). 

Therefore, to any such $(\mathcal{M},\nabla)$ we can associated a
complex of sheaves\\
$(\mathcal{M}\otimes\Omega_{X}^{\cdot},\nabla)$, whose hypercohomology
$\mathbb{H}_{\text{dR}}^{\cdot}(\mathcal{M})$ is the de Rham cohomology
of $\mathcal{M}$. If $\tilde{\mathcal{M}}$ is the sheaf in the infinitesimal
site associated to $(\mathcal{M},\nabla)$, then there is an isomorphism
\[
\mathbb{H}_{\text{dR}}^{\cdot}(\mathcal{M})\tilde{\to}\mathbb{H}_{\text{inf}}^{\cdot}(\tilde{\mathcal{M}})
\]
where the cohomology group on the right is the cohomology of $\tilde{\mathcal{M}}$
in the infinitesimal site (c.f. \cite{key-57}). 

At around the same time, the theory of $\mathcal{O}_{X}$-modules
with flat connection, in a rather different guise, was beginning to
be developed, by both M. Sato and his collaborators in Kyoto and by
J. Bernstein and his collaborators in Moscow. This different guise
goes under the name of $\mathcal{D}$-modules\footnote{To be completely precise, Sato and his school worked with complex
analytic spaces, while Bernstein worked with algebraic varieties.
The two theories are similar enough, and had enough influence on each
other, that we will mention them both in this introduction}, where $\mathcal{D}$ stands for the sheaf of (finite order) differential
operators on $X$. In fact, a sheaf of modules over $\mathcal{D}$
is equivalent to an $\mathcal{O}_{X}$-module with flat connection.
Thus $\mathcal{D}$-modules (i.e. sheaves of modules over $\mathcal{D}$)
provide another (equivalent) answer to the question of how to find
a good theory of local coefficients for the de Rham cohomology. Furthermore,
the fact that $\mathcal{D}$ is itself a locally noetherian sheaf
of algebras with a nice presentation allows one to develop the theory
in new directions. For instance, there is a well behaved notion of
coherence for $\mathcal{D}$-modules, and to a coherent $\mathcal{D}$-module
one may attach an invariant called the singular support, which is
a closed subvariety of $T^{*}X$, the cotangent bundle on $X$.

Furthermore, if $\varphi:X\to Y$ is a morphism of smooth algebraic
varieties over $\mathbb{C}$, then there is formalism of push-forward
and pull-back. As there is, in general, no morphism of sheaves of
rings between $\varphi^{-1}(\mathcal{D}_{Y})$ and $\mathcal{D}_{X}$,
this formalism is a bit non-trivial. Let us take a moment to review
it. First, the pullback: as a $\mathcal{D}$-module is equivalent
to a module with flat connection; if one has a flat connection on
some $\mathcal{O}_{Y}$-module $\mathcal{M}$, then the composed map
\[
\varphi^{*}\mathcal{M}\xrightarrow{\varphi^{*}\nabla}\varphi^{*}\mathcal{M}\otimes\varphi^{*}\Omega_{Y}^{1}\to\varphi^{*}\mathcal{M}\otimes\Omega_{X}^{1}
\]
is a flat connection on $\varphi^{*}\mathcal{M}$. Applying this to
$\mathcal{D}_{Y}$ itself, one endows $\varphi^{*}\mathcal{D}_{Y}$
with the structure of a left $\mathcal{D}_{X}$-module. As it is also
a right $\varphi^{-1}(\mathcal{D}_{Y})$-module via the right action
of $\mathcal{D}_{Y}$ on itself, we see that $\varphi^{*}\mathcal{D}_{Y}$
acquires the structure of a $(\mathcal{D}_{X},\varphi^{-1}(\mathcal{D}_{Y}))$
bimodule; this object is denoted $\mathcal{D}_{X\to Y}$. Thus we
may define a functor on derived categories 
\[
L\varphi^{*}:D(\mathcal{D}_{Y}-\text{mod})\to D(\mathcal{D}_{X}-\text{mod})
\]
via 
\[
L\varphi^{*}\mathcal{M}^{\cdot}:=\mathcal{D}_{X\to Y}\otimes_{\varphi^{-1}(\mathcal{D}_{Y})}^{L}\varphi^{-1}(\mathcal{M}^{\cdot})
\]
If $\mathcal{M}$ is an object in $\mathcal{D}_{Y}-\text{mod}$, then
$\mathcal{H}^{0}(L\varphi^{*}\mathcal{M})=\varphi^{*}\mathcal{M}$,
equipped with the flat connection defined above. 

Now let us consider the push-forward. The existence of the bimodule
$\mathcal{D}_{X\to Y}$ allows us to construct a push-forward on \emph{right
}$\mathcal{D}$-modules as follows: if $\mathcal{N}^{\cdot}$ is a
complex of right $\mathcal{D}$-modules on $X$, then 
\[
\mathcal{N}^{\cdot}\otimes_{\mathcal{D}_{X}}^{L}\mathcal{D}_{X\to Y}
\]
is a complex of \emph{right $\varphi^{-1}(\mathcal{D}_{Y})$-}modules.
Therefore we obtain a functor 
\[
\int_{\varphi}:D(\text{mod}-\mathcal{D}_{X})\to D(\text{mod}-\mathcal{D}_{Y})
\]
(where $\text{mod}-\mathcal{D}$ stands for the category of right
$\mathcal{D}$-modules) defined as 
\[
\int_{\varphi}\mathcal{N}^{\cdot}:=R\varphi_{*}(\mathcal{N}^{\cdot}\otimes_{\mathcal{D}_{X}}^{L}\mathcal{D}_{X\to Y})
\]
Of course, as our aim is to study left $\mathcal{D}$-modules, and
as $\mathcal{D}$ is a non-commutative sheaf of rings, this does not
satisfy our original goal. However, the situation is salvaged by the
following: 
\begin{prop}
(c.f. \cite{key-49}, proposition 1.2.12) There is an equivalence
of categories $\mathcal{D}_{X}-\text{mod}\to\text{mod}-\mathcal{D}_{X}$.
On the underlying $\mathcal{O}_{X}$-modules, this functor is given
by $\mathcal{M}\to\mathcal{M}\otimes_{\mathcal{O}_{X}}\omega_{X}$.
This is referred to as the left-right interchange.
\end{prop}

Therefore, one may obtain a push-forward functor on left $\mathcal{D}$-modules
by applying the equivalence to right modules, applying the above functor
${\displaystyle \int_{\varphi}}$, and then applying the equivalence
from right to left $\mathcal{D}_{Y}$-modules. More efficiently, one
may use the left-right interchange to directly define from $\mathcal{D}_{X\to Y}$
a $(\varphi^{-1}(\mathcal{D}_{Y}),\mathcal{D}_{X})$-bimodule, denoted
$\mathcal{D}_{Y\leftarrow X}$ (c.f. the beginning of \prettyref{subsec:Operations-on-modules:Push!}
below for details) and then define 
\[
\int_{\varphi}:D(\mathcal{D}_{X}-\text{mod})\to D(\mathcal{D}_{Y}-\text{mod})
\]
via 
\[
\int_{\varphi}\mathcal{M}^{\cdot}:=R\varphi_{*}(\mathcal{D}_{Y\leftarrow X}\otimes_{\mathcal{D}_{X}}^{L}\mathcal{M}^{\cdot})
\]
Though one may easily prove many basic properties of the pushforward
from this formula, it does not seem to lend itself to any obvious
evaluation. For instance, if $Y=\text{Spec}(\mathbb{C})$ is a point,
and $\mathcal{M}^{\cdot}=\mathcal{M}$ is a left $\mathcal{D}_{X}$-module,
we would hope that the pushforward of $\mathcal{N}$ agrees with the
de Rham cohomology. That this is indeed the case (up to a harmless
homological shift) is contained in the following fundamental 
\begin{thm}
\label{thm:Basic-de-Rham-resolution}(c.f. \cite{key-49}, section
1.5) Suppose $Y$ is a point. Then $\mathcal{D}_{Y\leftarrow X}=\omega_{X}$
(with its natural right $\mathcal{D}_{X}$-module structure). Consider
the de Rham complex of $\mathcal{D}_{X}$, $(\mathcal{D}_{X}\otimes_{\mathcal{O}_{X}}\Omega_{X}^{\cdot},\nabla)$.
Giving each term $\mathcal{D}_{X}\otimes\Omega_{X}^{i}$ the structure
of a right $\mathcal{D}_{X}$-module via the right multiplication
of $\mathcal{D}_{X}$ on itself makes this complex into a complex
of right $\mathcal{D}_{X}$-modules. This complex is exact except
at the rightmost term (where $i=\text{dim}(X)$), and we have that
the action map 
\[
\mathcal{D}_{X}\otimes_{\mathcal{O}_{X}}\omega_{X}\to\omega_{X}
\]
induces an isomorphism $\mathcal{H}^{\text{dim}(X)}(\mathcal{D}_{X}\otimes_{\mathcal{O}_{X}}\Omega_{X}^{\cdot})\tilde{\to}\omega_{X}$. 

In other words, $(\mathcal{D}_{X}\otimes_{\mathcal{O}_{X}}\Omega_{X}^{\cdot},\nabla)$
is a locally free resolution of $\omega_{X}$ in the category of right
$\mathcal{D}_{X}$-modules. 

Plugging this resolution into the definition of the pushforward yields
\[
\int_{\varphi}\mathcal{M}\tilde{\to}\mathbb{H}_{\text{dR}}^{\cdot}(\mathcal{M})[\text{dim}(X)]
\]
\end{thm}

Along with the resolution of $\omega_{X}$ as a right $\mathcal{D}_{X}$-module,
there is (via the left-right-interchange) a corresponding resolution
of $\mathcal{O}_{X}$ as a left $\mathcal{D}_{X}$-module, called
the Spencer resolution. It allows one to prove the other basic interpretation
of the de Rham cohomology in terms of the category of modules with
flat connection:
\begin{cor}
\label{cor:Spencer-Res}There is a locally free resolution of $\mathcal{O}_{X}$,
in the category of left $\mathcal{D}_{X}$-modules, whose terms are
$\mathcal{D}_{X}\otimes_{\mathcal{O}_{X}}\mathcal{T}_{X}^{i}$ (here,
$\mathcal{T}_{X}^{i}$ is the $i$th exterior power of the tangent
sheaf); and the differential is obtained from the differential in
the de Rham complex of $\mathcal{D}_{X}$ via the left-right interchange.
Therefore, we have for any $\mathcal{M}\in\mathcal{D}_{X}-\text{mod}$
an isomorphism
\[
\text{Ext}{}_{\mathcal{D}_{X}}^{i}(\mathcal{O}_{X},\mathcal{M})\tilde{\to}\mathbb{H}_{dR}^{i}(\mathcal{M})
\]
\end{cor}

After identifying $\mathcal{D}$-modules with sheaves on the infinitesimal
site, this yields another proof of the fact that de Rham cohomology
of a connection agrees with its infinitesimal cohomology. 

In fact there is a generalization of this to any smooth morphism $\varphi:X\to Y$.
In that case, one may consider the relative connection $\nabla:\mathcal{D}_{X}\to\mathcal{D}_{X}\otimes_{\mathcal{O}_{X}}\Omega_{X/Y}^{1}$
and the associated relative de Rham complex 
\[
(\mathcal{D}_{X}\otimes_{\mathcal{O}_{X}}\Omega_{X/Y}^{\cdot},\nabla)
\]
This is again a locally free complex of right $\mathcal{D}_{X}$-modules,
which is a resolution of $\mathcal{D}_{Y\leftarrow X}$. Using this,
one may go on to show 
\begin{thm}
Let $\varphi:X\to Y$ be smooth of relative dimension $d$. Let $\mathcal{M}\in\mathcal{D}_{X}-\text{mod}$,
and $(\mathcal{M}\otimes_{\mathcal{O}_{X}}\Omega_{X/Y}^{\cdot},\nabla)$
the corresponding relative de Rham complex. Then 

1) There is an isomorphism 
\[
\int_{\varphi}\mathcal{M}[-d]\tilde{=}R\varphi_{*}(\mathcal{M}\otimes_{\mathcal{O}_{X}}\Omega_{X/Y}^{\cdot},\nabla)
\]
in the category $D(\mathcal{O}_{Y})-\text{mod}$. The resulting $\mathcal{D}_{Y}$-module
structure on the sheaf $R^{i}\varphi_{*}(\mathcal{M}\otimes_{\mathcal{O}_{X}}\Omega_{X/Y}^{\cdot},\nabla)$
corresponds to the Gauss-Manin connection. 

2) Let $\varphi^{!}=L\varphi^{*}[d]$. Then there is an isomorphism
\[
R\varphi_{*}R\mathcal{H}om_{\mathcal{D}_{X}}(\varphi^{!}\mathcal{N}^{\cdot},\mathcal{M}^{\cdot})\tilde{\to}R\mathcal{H}om{}_{\mathcal{D}_{Y}}(\mathcal{N}^{\cdot},\int_{\varphi}\mathcal{M}^{\cdot})
\]
for any $\mathcal{N}^{\cdot}\in D(\mathcal{D}_{Y}-\text{mod})$ and
any $\mathcal{M}^{\cdot}\in D(\mathcal{D}_{X}-\text{mod})$. In particular,
${\displaystyle \int_{\varphi}}$ is the right adjoint to $\varphi^{!}$. 
\end{thm}

For a nice exposition of this, we refer the reader to \cite{key-4},
chapter 4. Without going into any more detail, let us mention that
the theory of $\mathcal{D}$-modules allows one to study many other
interesting phenomena related to the topology of algebraic varieties.
For instance, there is a very rich theory of nearby and vanishing
cycles for $\mathcal{D}$-modules, there is a $\mathcal{D}$-module
duality which is generalization of Poincare duality, and that there
are deep connections with Hodge theory. 

Now let us turn to the main interest of this paper, the theory of
varieties in characteristic $p>0$. Fix now a perfect field $k$ of
characteristic $p>0$; letters $X$ and $Y$ will now denote smooth
varieties over $k$. Let $W(k)$ denote the $p$-typical Witt vectors
of $k$, and set $W_{r}(k)=W(k)/p^{r}$. In this case it was predicted
by Grothendieck that there should be a cohomology theory, the crystalline
cohomology, denoted $\mathbb{H}_{\text{crys}}^{\cdot}(X;W_{r}(k))$,
with coefficients in modules over $W_{r}(k)$ (for each $r$). The
crystalline cohomology is supposed to have the following property:
if $\mathfrak{X}_{r}$ is a flat lift of $X$ to a scheme over $W_{r}(k)$,
then there is an isomorphism 
\begin{equation}
\mathbb{H}_{\text{dR}}^{\cdot}(\mathfrak{X}_{r})\tilde{\to}\mathbb{H}_{\text{crys}}^{\cdot}(X;W_{r}(k))\label{eq:basic-crystalline-iso}
\end{equation}
Furthermore, the crystalline cohomology should be the cohomology of
the structure sheaf in a suitable site, the crystalline site of $X$.
One should then be able to take the limit over $r$ and obtain a cohomology
$\mathbb{H}_{\text{crys}}^{\cdot}(X)$ with values in $W(k)$. 

This prediction was borne out in Berthelot's monumental work \cite{key-60}.
In that work, he constructs the crystalline site of $X$ (it is modeled
after the infinitesimal site) and defines crystalline cohomology as
the cohomology of the structure sheaf in the crystalline site; and
proves, among many other things, the isomorphism \ref{eq:basic-crystalline-iso}.
As the crystalline cohomology is defined for any $X$, without using
a lift, the theory shows that the de Rham cohomology of such a lift
depends only on special fibre $X$. 

There is, in some cases, an interpretation of sheaves of $\mathcal{O}$-modules
on the crystalline site over $W_{r}(k)$ (we shall refer to them as
\emph{$W_{r}(k)$-crystals }from now on) in terms of flat connections.
Namely, if $\mathfrak{X}_{r}$ is a lift of $X$ as above, then there
is an equivalence from $W_{r}(k)$-crystals on $X$ to $\mathcal{O}_{\mathfrak{X}_{r}}$-modules
with flat connection (satisfying an additional condition called nilpotence).
This means that crystals always have such a description locally on
$X$. In general, however, $X$ admits no such lift, and so there
is no ``concrete'' description of the category of crystals analogous
to the theory of $\mathcal{D}_{X}$-modules sketched above for complex
varieties. 

In order to find such a description, the first step is to identify
the analogue of the de Rham complex for crystals. This task was accomplished
in the monumental work \cite{key-6} of Illusie (following ideas of
Bloch \cite{key-61} and Deligne). The required complex is called
the de Rham-Witt complex, denoted $W\Omega_{X}^{i}$. This is a complex
of sheaves on the formal scheme of Witt vectors of $X$, $W(X)$.
As $W(X)$ is a functorially defined lift of $X$ to mixed characteristic,
it had been suspected (since at least \cite{key-62}) that it might
be a good place to look for cohomology theories which take values
in $W(k)$, and the de Rham-Witt complex realizes this vision. The
theory of this complex is extensively developed in \cite{key-6},
but for now let us only mention the fact that there is a canonical
isomorphism
\[
\mathbb{H}_{\text{crys}}^{\cdot}(X)\tilde{\to}\mathbb{H}^{\cdot}(W\Omega_{X}^{i})
\]
Furthermore, Etesse (in \cite{key-56}) has constructed a functor
from crystals to modules equipped with a de Rham-Witt connection,
and has shown that, when the underlying $\mathcal{O}$-module $\mathcal{M}$
is a vector bundle, there is a comparison
\[
\mathbb{H}_{\text{crys}}^{\cdot}(\mathcal{M})\tilde{\to}\mathbb{H}_{dRW}^{\cdot}(\mathcal{M})
\]
where on the right we have the hypercohomology of the de Rham-Witt
complex with values in $\mathcal{M}$. 

It therefore seems natural to look for the analogue of the above theory
of $\mathcal{D}_{X}$-modules and try to interpret crystals in terms
of some kind of $\mathcal{D}$-modules on $W(X)$; one would hope
also for the analogues of the basic functorialities discussed above,
as well as suitable comparisons between crystalline cohomology, the
$\mathcal{D}$-module pushforward, and the de Rham-Witt cohomology. 

This paper accomplishes exactly such a construction. In the remainder
of this introduction, we will explain the idea behind our theory,
as well as stating the main theorems of the paper and giving an outline
of the contents. 

Let us begin by supposing that $A$ admits local coordinates $\{T_{1},\dots T_{n}\}$,
i.e., there is an etale map $\text{Spec}(A)\to\mathbb{A}_{k}^{n}$
. Then one has a concrete description of the ring of ($k$-linear)
differential operators $\mathcal{D}_{A}$ as the subring of $\text{End}_{k}(A)$
generated by multiplication by elements of $A$ and operators of the
form $\partial_{i}^{[j]}$ where $\partial_{i}^{[j]}$ is the unique
differential operator on $A$ satisfying
\[
\partial_{i}^{[j]}(T_{i}^{m})={m \choose j}T_{i}^{m-j}
\]
and $\partial_{i}^{[j]}(T_{l}^{m})=0$ for $l\neq i$. The collection
of operators $\partial_{i}^{[j]}$ is very well behaved; for instance,
we have for all $j$ the relation 
\[
\partial_{i}^{[pj]}\circ F=F\circ\partial_{i}^{[j]}
\]
where $F:A\to A$ is the Frobenius. The subring of $\mathcal{D}_{A}$
generated by $A$ and the operators $\{\partial_{i}^{[j]}\}_{j\leq p^{m}}$
is denoted\footnote{The notation in this area is not entirely standardized; see directly
below for our conventions on rings of differential operators} $\mathcal{\overline{D}}_{A}^{(m)}$, and the above relation is the
key to showing that the functor of pullback by Frobenius can be upgraded
to a functor 
\[
F^{*}:\mathcal{\overline{D}}_{A}^{(m)}-\text{mod}\to\mathcal{\overline{D}}_{A}^{(m+1)}-\text{mod}
\]
In fact, this functor is an equivalence of categories. 

Now consider the ring $W(A)$. It is a $W(k)$ algebra, equipped with
a lift of Frobenius $F$ and an operator $V$ which acts as $p\cdot F^{-1}$.
The filtration $V^{i}(W(A))$ is a decreasing filtration by ideals,
with respect to which $W(A)$ is complete. One has that $p^{i}\in V^{i}(W(A))$,
and $W(A)$ is also $p$-adically complete. There is a canonical isomorphism
\begin{equation}
W(A)/V(W(A))\tilde{\to}A\label{eq:Witt-lifts-A}
\end{equation}
On this algebra we construct a family of operators, denoted $\{\partial_{i}\}_{\lambda}$,
where $\lambda\in\mathbb{Q}^{+}$ is a positive rational number of
the form $j/p^{r}$ with $j$ a positive integer and $r\geq0$. For
all $\lambda$ the operators $\{\partial_{i}\}_{\lambda}$ preserve
the filtration $\{V^{i}(W(A))\}$. If $\lambda\in\mathbb{Z}$ then
$\{\partial_{i}\}_{\lambda}$ is a lift of $\partial_{i}^{[\lambda]}$,
via the isomorphism \prettyref{eq:Witt-lifts-A}. If $\text{val}_{p}(\lambda)=r<0$
then $\{\partial_{i}\}_{\lambda}(W(A))\subset V^{r}(W(A))$; this
means that for any such $\lambda$ and any $\alpha\in W(A)$ we can
define the operator $F^{-r}(\alpha)\cdot\{\partial_{i}\}_{\lambda}:W(A)\to W(A)$
(this is because, while $F^{-r}(\alpha)\notin W(A)$ in general, its
action on $V^{r}(W(A))$ is well-defined). 

Most importantly, the operators satisfy the relations 
\[
\{\partial_{i}\}_{\lambda}F=F\{\partial_{i}\}_{\lambda/p}
\]
and therefore
\[
\{\partial_{i}\}_{\lambda}V=V\{\partial_{i}\}_{p\lambda}
\]
As in the case of $\partial_{i}^{[j]}$, there is a formula for the
action of $\{\partial_{i}\}_{\lambda}$ on $W(A)$ in terms of coordinates;
more important than this is the fact there is an intrinsic construction
of these operators in terms of Hasse-Schmidt derivations on $W(A)$
(this is detailed in \prettyref{thm:Canonical-H-S} below). In particular,
the algebra of $W(k)$-linear endomorphisms of $W(A)$ generated\footnote{strictly speaking, we need to work with a kind of completion of this
algebra} by $W(A)$ itself and terms of the form $F^{-v}(\alpha)\cdot\{\partial_{i}\}_{\lambda}$
(where $v=\text{max}\{0,-\text{val}_{p}(\lambda)\}$ can be defined
independently of any choice of coordinates. We term this algebra the
Witt-differential operators of $W(A)$, denoted $\widehat{\mathcal{D}}_{W(A)}$.
For each integer $m$ (positive or negative), we have also the differential
operators of level $\leq m$; this is the subalgebra of $\widehat{\mathcal{D}}_{W(A)}$
generated\footnote{as above, generated in the topological sense}
by operators $F^{-v}(\alpha)\cdot\{\partial_{i}\}_{\lambda}$ for
which $\text{val}_{p}(\lambda)\leq m$. This algebra, denoted $\widehat{\mathcal{D}}_{W(A)}^{(m)}$
can be defined independently of the choice of coordinates (all of
this is detailed in \prettyref{subsec:Local-Coordinates} below). 

At this point the reader might (quite reasonably) ask: what is the
justification for this definition? And what is the relationship between
modules over this ring and crystals (or flat connections)? Both of
these questions are answered by the following construction: as $A$
is affine, it admits a lift to a $p$-adically complete, $W(k)$-flat
algebra $\mathcal{A}$. We may choose an endomorphism $F:\mathcal{A}\to\mathcal{A}$
which lifts the Frobenius map (in fact, as $A$ admits local coordinates,
we can even demand that $\mathcal{A}$ admits local coordinates $\{T_{1},\dots T_{n}\}$
for which $F(T_{i})=T_{i}^{p}$; we call such lifts \emph{coordinatized}).
By a basic property of the Witt vectors, such a lift induces a morphism
\[
\Phi:\mathcal{A}\to W(A)
\]
which intertwines the Frobenius lift $F$ with the Witt vector Frobenius
on $W(A)$. On $\mathcal{A}$ we have, for each $m\geq0$, Berthelot's
ring of arithmetic differential operators $\mathcal{\widehat{D}}_{\mathcal{A}}^{(m)}$.
For the purpose of this introduction we focus on the case $m=0$,
in which case $\mathcal{\widehat{D}}_{\mathcal{A}}^{(0)}$ is simply
the $p$-adic completion of subalgebra of $\text{End}_{W(k)}(\mathcal{A})$
generated by $\mathcal{A}$ and the continuous, $W(k)$-linear derivations
of $\mathcal{A}$. This algebra acts faithfully on $\mathcal{A}$
itself. 

Then we have 
\begin{thm}
\label{thm:Consider-the--bimodule}Consider the $(\widehat{\mathcal{D}}_{W(A)}^{(0)},\mathcal{\widehat{D}}_{\mathcal{A}}^{(0)})$-bimodule
$\text{Hom}_{W(k)}(\mathcal{A},W(A))$-this is a bimodule via the
obvious actions of $\widehat{\mathcal{D}}_{W(A)}^{(0)}$ on $W(A)$
and $\widehat{\mathcal{D}}_{\mathcal{A}}^{(0)}$ on $\mathcal{A}$.
Let $\Phi^{*}\mathcal{\widehat{D}}_{\mathcal{A}}^{(0)}$ denote the
completion of $W(A)\otimes_{\mathcal{A}}\mathcal{\widehat{D}}_{\mathcal{A}}^{(0)}$
along the filtration $V^{i}(W(A))\otimes_{\mathcal{A}}\mathcal{\widehat{D}}_{\mathcal{A}}^{(0)}$.
Then there is an embedding 
\[
\Phi^{*}\mathcal{\widehat{D}}_{\mathcal{A}}^{(0)}\to\text{Hom}_{W(k)}(\mathcal{A},W(A))
\]
which takes a tensor $\alpha\otimes P$ to the map $a\to\alpha\cdot\Phi(P(a))$.
The image of this embedding is exactly the $(\widehat{\mathcal{D}}_{W(A)}^{(0)},\mathcal{\widehat{D}}_{\mathcal{A}}^{(0)})$-bisubmodule
generated by $\Phi$. The object $\Phi^{*}\mathcal{\widehat{D}}_{\mathcal{A}}^{(0)}$
is faithfully flat as a right $\mathcal{\widehat{D}}_{\mathcal{A}}^{(0)}$-module,
and projective as a left $\widehat{\mathcal{D}}_{W(A)}^{(0)}$-module;
in fact it is a summand of $\widehat{\mathcal{D}}_{W(A)}^{(0)}$ itself. 
\end{thm}

This is a combination of \prettyref{prop:Construction-of-bimodule}
and \prettyref{thm:Projective!} below. This bimodule allows us to
closely relate modules over $\mathcal{\widehat{D}}_{\mathcal{A}}^{(0)}$
(which are, essentially, $\mathcal{A}$-modules with continuous flat
connection over $\mathcal{A}$) to modules over $\widehat{\mathcal{D}}_{W(A)}^{(0)}$.
Before giving the exact statement, let us mention that these results
are directly inspired by Berthelot's Frobenius descent. To state it,
recall that, for each $m\geq0$, Berthelot puts the structure of a
$(\mathcal{\widehat{D}}_{\mathcal{A}}^{(m+1)},\mathcal{\widehat{D}}_{\mathcal{A}}^{(m)})$
bimodule on $F^{*}\mathcal{\widehat{D}}_{\mathcal{A}}^{(m)}$. This
bimodule induces an equivalence of categories 
\[
\mathcal{M}\to F^{*}\mathcal{\widehat{D}}_{\mathcal{A}}^{(m)}\otimes_{\mathcal{\widehat{D}}_{\mathcal{A}}^{(m)}}\mathcal{M}:=F^{*}\mathcal{M}
\]
from $\mathcal{\widehat{D}}_{\mathcal{A}}^{(m)}-\text{mod}$ to $\mathcal{\widehat{D}}_{\mathcal{A}}^{(m+1)}-\text{mod}$.
The bimodule $F^{*}\mathcal{\widehat{D}}_{\mathcal{A}}^{(m)}$ can
be identified with the $(\mathcal{\widehat{D}}_{\mathcal{A}}^{(m+1)},\mathcal{\widehat{D}}_{\mathcal{A}}^{(m)})$
bi-submodule of $\text{End}_{W(k)}(\mathcal{A})$ generated by $F$.
So this theory shows that taking an $m$th power Frobenius map corresponds
to increasing the order of the differential operators by $m$. Roughly
speaking, the map $\Phi:\mathcal{A}\to W(A)$ is something like a
perfection map, or, an infinite-power Frobenius map. The Witt differential
operators, then, are exactly the (infinite order) differential operators
that correspond to Berthelot's arithmetic differential operators after
taking the pullback by $\Phi$. This is, in fact, how the definition
of these operators was originally arrived at. 

Now let us explain some important consequences of \prettyref{thm:Consider-the--bimodule}.
Reducing everything mod $p$, we obtain a $(\widehat{\mathcal{D}}_{W(A)}^{(0)}/p,\mathcal{D}_{A}^{(0)})$-bimodule
$\Phi^{*}\mathcal{\widehat{D}}_{\mathcal{A}}^{(0)}/p$ (here $\mathcal{D}_{A}^{(0)}$
are the pd-differential operators on $A$, which by definition are
the reduction mod $p$ of the arithmetic differential operators of
level $0$). We have 
\begin{thm}
\label{thm:Bimodule-unique!}Let $\Phi_{1},\Phi_{2}$ be two coordinatized
lifts of Frobenius on $\mathcal{A}$. 

1) There is a canonical isomorphism of $(\widehat{\mathcal{D}}_{W(A)}^{(0)}/p,\mathcal{D}_{A}^{(0)})$-bimodules
\[
\epsilon_{12}:\Phi_{1}^{*}\mathcal{\widehat{D}}_{\mathcal{A}}^{(0)}/p\tilde{\to}\Phi_{2}^{*}\mathcal{\widehat{D}}_{\mathcal{A}}^{(0)}/p
\]
If $\Phi_{3}$ is a third such lift we have the cocycle condition
$\epsilon_{23}\circ\epsilon_{12}=\epsilon_{13}$. 

2) Suppose $p>2$. Then there in fact a canonical isomorphism of $(\widehat{\mathcal{D}}_{W(A)}^{(0)},\mathcal{\widehat{D}}_{\mathcal{A}}^{(0)})$-bimodules
$\epsilon_{12}:\Phi_{1}^{*}\mathcal{\widehat{D}}_{\mathcal{A}}^{(0)}\tilde{\to}\Phi_{2}^{*}\mathcal{\widehat{D}}_{\mathcal{A}}^{(0)}$.
For a third lift $\Phi_{3}$ we have the cocycle condition $\epsilon_{23}\circ\epsilon_{12}=\epsilon_{13}$. 
\end{thm}

The second statement of this theorem is proved in \prettyref{thm:Uniqueness-of-bimodule},
by showing that $\Phi_{1}-\Phi_{2}$ is itself a differential operator
of the required type; this requires $p>2$. To prove the first statement
(for all $p$) we use a totally different technique, which relies
on the fact that $W(A)/p\to A$ is a square zero infinitesimal extension.
Such extensions are ubiquitous in the theory of flat connections,
because, as observed by Grothendieck, giving a flat connection on
a sheaf is essentially the same as giving a canonical extension of
that sheaf to every square zero extension (c.f. \cite{key-10}, intro,
for a detailed discussion of this). Our proof of part $1)$ of the
above theorem is essentially a rephrasing of this argument (see \prettyref{cor:Global-bimodule-for-U}
and \prettyref{cor:Global-bimodule-mod-p} for details). 

Now let us suppose that $X$ is an arbitrary smooth scheme over $k$.
We have the formal scheme $W(X)$, as well as the formal scheme $W(X)_{p=0}$,
whose structure sheaf is given by $\mathcal{O}_{W(X)}/p$ (in all
cases the topological space is $X$ itself). From the construction,
there is a sheaf of $\widehat{\mathcal{D}}_{W(X)}^{(0)}$, whose sections
over an open affine $\text{Spec}(A)$ which admits local coordinates
agree with $\widehat{\mathcal{D}}_{W(A)}^{(0)}$. In that case where
$X=\text{Spec}(A)$ is as above, we let $\mathfrak{X}=\text{Specf}(\mathcal{A})$.
Then $\Phi^{*}\mathcal{\widehat{D}}_{\mathcal{A}}^{(0)}$ sheafifies
to $\Phi^{*}\mathcal{\widehat{D}}_{\mathfrak{X}}^{(0)}$, a sheaf
of $(\widehat{\mathcal{D}}_{W(X)}^{(0)},\widehat{\mathcal{D}}_{\mathfrak{X}}^{(0)})$
bimodules. 

So, globalizing all of the above yields 
\begin{cor}
\label{cor:Bimodule-properties!}1) There is a sheaf of $(\widehat{\mathcal{D}}_{W(X)}^{(0)}/p,\mathcal{D}_{X}^{(0)})$
bimodules on $X$, denoted $\mathcal{B}_{X}^{(0)}$, whose sections
over an open affine $\text{Spec}(A)$ which admits local coordinates
are equal to $\Phi^{*}\mathcal{D}_{A}^{(0)}$. The functor $\mathcal{M}\to\mathcal{B}_{X}^{(0)}\otimes_{\mathcal{D}_{X}^{(0)}}\mathcal{M}$
from $\mathcal{D}_{X}^{(0)}-\text{mod}$ to $\mathcal{\widehat{D}}_{W(X)}^{(0)}/p-\text{mod}$
is exact and fully faithful, and the derived functor 
\[
\mathcal{B}_{X}^{(0)}\otimes_{\mathcal{D}_{X}^{(0)}}^{L}:D(\mathcal{D}_{X}^{(0)}-\text{mod})\to D(\mathcal{\widehat{D}}_{W(X)}^{(0)}/p-\text{mod})
\]
is fully faithful as well. A complex $\mathcal{M}^{\cdot}\in D(\mathcal{\widehat{D}}_{W(X)}^{(0)}/p-\text{mod})$
is called \textbf{accessible} if it is of the form $\mathcal{B}_{X}^{(0)}\otimes_{\mathcal{D}_{X}^{(0)}}^{L}\mathcal{N}^{\cdot}$
for some $\mathcal{N}^{\cdot}\in D(\mathcal{D}_{X}^{(0)}-\text{mod})$. 

2) Let $\mathcal{M}^{\cdot}\in D(\mathcal{\widehat{D}}_{W(X)}^{(0)}-\text{mod})$
be a cohomologically complete\footnote{This is a technical condition, which can be considered the analogue
for complexes of being $p$-adically complete. It ensures that the
functor $\otimes_{W(k)}^{L}k$ is well behaved } complex. Then $\mathcal{M}^{\cdot}$ is said to be \textbf{accessible
}if $\mathcal{M}^{\cdot}\otimes_{W(k)}^{L}k$ is accessible inside
$D(\mathcal{\widehat{D}}_{W(X)}^{(0)}/p-\text{mod})$. The complex
$\mathcal{M}^{\cdot}$ is accessible iff, for any open affine $\text{Spec}(A)\subset X$,
which admits local coordinates, and any coordinatized lift of Frobenius
$\Phi$, we have 
\[
\mathcal{M}^{\cdot}\tilde{\to}\Phi^{*}\mathcal{\widehat{D}}_{\mathfrak{X}}^{(0)}\widehat{\otimes}_{\widehat{\mathcal{D}}_{\mathfrak{X}}^{(0)}}^{L}\mathcal{N}^{\cdot}
\]
for a cohomologically complete complex $\mathcal{N}^{\cdot}\in D(\mathcal{\widehat{D}}_{\mathfrak{X}}^{(0)}-\text{mod})$
(here, the symbol $\Phi^{*}\mathcal{\widehat{D}}_{\mathfrak{X}}^{(0)}\widehat{\otimes}_{\widehat{\mathcal{D}}_{\mathfrak{X}}^{(0)}}^{L}$
denotes the derived tensor product, followed by the cohomological
completion). In particular, the latter condition is independent of
the choice of $\Phi$. The inclusion functor $D_{acc}(\mathcal{\widehat{D}}_{W(X)}^{(0)}-\text{mod})\to D_{cc}(\mathcal{\widehat{D}}_{W(X)}^{(0)}-\text{mod})$
admits a right adjoint. 

3) Suppose $p>2$, and suppose that $\mathfrak{X}$ is a smooth formal
scheme over $W(k)$ whose special fibre is $X$ (it might not exist
in general). Then there is an equivalence of categories 
\[
D_{cc}(\mathcal{\widehat{D}}_{\mathfrak{X}}^{(0)}-\text{mod})\to D_{\text{acc}}(\mathcal{\widehat{D}}_{W(X)}^{(0)}-\text{mod})
\]
where on the left, $D_{cc}$ stands for cohomologically complete complexes,
and on the right we have the accessible complexes inside $D(\mathcal{\widehat{D}}_{W(X)}^{(0)}-\text{mod})$. 
\end{cor}

These facts are proved below in \prettyref{subsec:Accessible-modules}
(and for the existence of the right adjoint see \prettyref{def:D-acc}).
There, one will also find analogues of the above for abelian categories,
at least under certain circumstances. For instance, one may work with
$\mathcal{\widehat{D}}_{W(X)}^{(0)}/p^{r}-\text{mod}$ (for some $r\geq1$)
and in this situation there is an abelian category of accessible modules;
denoted $\mathcal{\widehat{D}}_{W(X)}^{(0)}/p^{r}-\text{mod}_{\text{acc}}$,
and inside that there are categories of quasicoherent\footnote{We warn the reader that this is not quite the usual notion of quasicoherence
over a sheaf of rings. The same goes for coherent.} and coherent modules. There is also an abelian category of coherent
accessible modules over $\mathcal{\widehat{D}}_{W(X)}^{(0)}$ (c.f.
\prettyref{cor:D-b-coh}). These results have some other natural variants
as well. Most importantly, one can replace $\mathcal{\widehat{D}}_{W(X)}^{(0)}$
with a certain completion, $\mathcal{\widehat{D}}_{W(X),\text{crys}}^{(0)}$.
We have also the algebras $\mathcal{D}_{X,\text{crys}}^{(0)}$ and
$\mathcal{\widehat{D}}_{\mathfrak{X},\text{crys}}^{(0)}$, modules
over which correspond to sheaves with topologically nilpotent flat
connection. The analogue of \prettyref{cor:Bimodule-properties!}
holds in this context, and in fact the restriction $p>2$ is unnecessary
in part $3)$ (c.f. \prettyref{thm:Uniqueness-of-bimodule} below).
In particular, we obtain an embedding of the category of quasi-coherent
crystals on $X$ (over $W_{r}(k)$) into the category of accessible
$\mathcal{\widehat{D}}_{W(X)}^{(0)}/p^{r}$-modules, and the same
holds at the derived level (this is detailed in \prettyref{subsec:The-crystalline-version}
below). 

Another useful variant is the category of complete accessible modules
over \\
$\mathcal{\widehat{D}}_{W(X)}^{(0)}/p^{r}$; these are modules which
are given as the completion (along the topology defined by $V^{i}(\mathcal{O}_{W(X)})$)
of an accessible module; we show (c.f. \prettyref{lem:two-filtrations})
that this filtration is extremely well-behaved on accessible modules,
and therefore that the resulting completion functor is exact and conservative;
there is therefore a version in the derived category as well. 

With the basic categories defined, the rest of the paper is devoted
to the construction of the basic operations (the left-right interchange,
pullback and pushforward), the analogue of the de Rham resolution
(which we call the de Rham-Witt resolution), and the functor from
accessible modules to de Rham-Witt connections. We give an overview
of these sections now. 

The first basic operation we discuss is the left-right interchange
on accessible modules. To do so, we need to define the notion of an
accessible \emph{right }module over $\widehat{\mathcal{D}}_{W(X)}^{(0)}$.
Starting in the local case, we need the right handed version of the
bimodule $\Phi^{*}\mathcal{\widehat{D}}_{\mathcal{A}}^{(0)}$; one
may swiftly define it as $\text{Hom}_{\widehat{\mathcal{D}}_{W(X)}^{(0)}}(\Phi^{*}\mathcal{\widehat{D}}_{\mathcal{A}}^{(0)},\widehat{\mathcal{D}}_{W(A)}^{(0)}):=\Phi^{!}\mathcal{\widehat{D}}_{\mathcal{A}}^{(0)}$.
There is also a characterization of it as a certain submodule of $\text{Hom}_{\mathcal{A}}(W(A),\mathcal{\widehat{D}}_{\mathcal{A}}^{(0)})$
(in line with one would predict from Grothendieck duality; c.f. \prettyref{prop:Construction-of-Phi-!}
below). It is a $(\mathcal{\widehat{D}}_{\mathcal{A}}^{(0)},\widehat{\mathcal{D}}_{W(A)}^{(0)})$
bimodule and the analogues of \prettyref{thm:Bimodule-unique!} and
\prettyref{cor:Bimodule-properties!} hold without any essential change.
This already shows (via the left-right interchange on $\mathcal{\widehat{D}}_{\mathcal{A}}^{(0)}$-modules)
that there is \emph{local }left-right interchange; to obtain the global
interchange we need a little more. 

First of all, we have 
\[
\omega_{\mathcal{A}}\widehat{\otimes}_{\mathcal{\widehat{D}}_{\mathcal{A}}^{(0)}}^{L}\Phi^{!}\mathcal{\widehat{D}}_{\mathcal{A}}^{(0)}\tilde{=}\omega_{\mathcal{A}}\otimes{}_{\mathcal{\widehat{D}}_{\mathcal{A}}^{(0)}}\Phi^{!}\mathcal{\widehat{D}}_{\mathcal{A}}^{(0)}\tilde{\to}W\omega_{A}
\]
where $W\omega_{A}$ is the highest nonzero entry in the de Rham-Witt
complex. In particular $W\omega_{A}$ inherits the structure of a
right $\widehat{\mathcal{D}}_{W(A)}^{(0)}$-module, which can be shown
to be independent of all the choices involved. 

Next, we make the important 
\begin{defn}
\prettyref{def:D-acc} Let $\widehat{\mathcal{D}}_{W(X),\text{acc}}^{(0)}$
be the image of $\mathcal{\widehat{D}}_{W(X)}^{(0)}$ under the right
adjoint to the inclusion functor $D_{acc}(\mathcal{\widehat{D}}_{W(X)}^{(0)}-\text{mod})\to D_{cc}(\mathcal{\widehat{D}}_{W(X)}^{(0)}-\text{mod})$.
This sheaf admits the structure of a $(\mathcal{\widehat{D}}_{W(X)}^{(0)},\mathcal{\widehat{D}}_{W(X)}^{(0)})$
bimodule.
\end{defn}

Then, when $X=\text{Spec}(A)$ admits local coordinates, we have 
\[
\widehat{\mathcal{D}}_{W(X),\text{acc}}^{(0)}=\Phi^{*}\mathcal{\widehat{D}}_{\mathfrak{X}}^{(0)}\widehat{\otimes}_{\mathcal{\widehat{D}}_{\mathfrak{X}}^{(0)}}\Phi^{!}\mathcal{\widehat{D}}_{\mathfrak{X}}^{(0)}
\]
(here $\widehat{\otimes}$ denotes the $p$-adic completion). This
formula was inspired by Berthelot's isomorphism $\widehat{\mathcal{D}}_{\mathfrak{X}}^{(m+1)}\tilde{=}F^{!}F^{*}\mathcal{\widehat{D}}_{\mathfrak{X}}^{(m)}\tilde{=}F^{*}F^{!}\mathcal{\widehat{D}}_{\mathfrak{X}}^{(m)}$
in his Frobenius descent theory. So, one may say that, this bimodule,
rather than $\mathcal{\widehat{D}}_{W(X)}^{(0)}$ itself, can be considered
the fundamental object of the accessible theory\footnote{There are interesting aspects to Witt-differential operator theory
beyond the accessible realm; however, we don't discuss them much in
this paper, as they relate mostly to the case where the level of differential
operator is negative}. Now, using this fact, one may put two commuting structures of a
right $\mathcal{\widehat{D}}_{W(X)}^{(0)}$-module on the object 
\[
W\omega_{X}\widehat{\otimes}_{\mathcal{O}_{W(X)}}\widehat{\mathcal{D}}_{W(X),\text{acc}}^{(0)}
\]
(where the hat denotes suitable completion, c.f. \prettyref{prop:L-R-bimodule}
for details). Then we have 
\begin{cor}
The functor $\mathcal{M}^{\cdot}\to(W\omega_{X}\widehat{\otimes}_{\mathcal{O}_{W(X)}}\widehat{\mathcal{D}}_{W(X),\text{acc}}^{(0)})\otimes_{\widehat{\mathcal{D}}_{W(X)}^{(0)}}^{L}\mathcal{M}^{\cdot}$
yields an equivalence from left accessible to right accessible modules
over $\widehat{\mathcal{D}}_{W(X)}^{(0)}$. If $X=\text{Spec}(A)$
admits local coordinates, then this equivalence is compatible with
the classical equivalence from left to right $\widehat{\mathcal{D}}_{\mathfrak{X}}^{(0)}$-modules. 
\end{cor}

This is proved in \prettyref{cor:Left-right!} below. 

Now we turn to pullback and pushforward. Let $\varphi:X\to Y$ be
a morphism of smooth varieties over $k$. There is, by the functoriality
of the Witt vectors, a canonical morphism $W\varphi:W(X)\to W(Y)$,
which lifts $\varphi$. As in the classical case, the operations are
defined by suitable bimodules; we in fact have both a $(\widehat{\mathcal{D}}_{W(X)}^{(0)},W\varphi^{-1}(\widehat{\mathcal{D}}_{W(Y)}^{(0)}))$-bimodule
$\widehat{\mathcal{D}}_{W(X)\to W(Y),\text{acc}}^{(0)}$ and a $(W\varphi^{-1}(\widehat{\mathcal{D}}_{W(Y)}^{(0)}),\widehat{\mathcal{D}}_{W(X)}^{(0)})$
bimodule, $\widehat{\mathcal{D}}_{W(Y)\leftarrow W(X),\text{acc}}^{(0)}$
(which is obtained via the left right swap). The definitions are slightly
involved, however, one can say that they are closely modeled on the
analogous definitions in the classical case (c.f. \prettyref{prop:Construction-of-bimodule}
below). We have the following 
\begin{defn}
The functor $LW\varphi^{*}:D(\widehat{\mathcal{D}}_{W(Y)}^{(0)}-\text{mod})\to D(\widehat{\mathcal{D}}_{W(X)}^{(0)}-\text{mod})$
is defined as 
\[
\mathcal{M}^{\cdot}\to\widehat{\mathcal{D}}_{W(X)\to W(Y),\text{acc}}^{(0)}\otimes_{W\varphi^{-1}(\widehat{\mathcal{D}}_{W(Y)}^{(0)})}^{L}W\varphi^{-1}(\mathcal{M}^{\cdot})
\]
\end{defn}

The basic properties of this functor are given in the 
\begin{thm}
\label{thm:Basic-Pullback-properties}1) $LW\varphi^{*}$ takes $D_{\text{acc}}(\widehat{\mathcal{D}}_{W(Y)}^{(0)}-\text{mod})$
to $D_{\text{acc}}(\widehat{\mathcal{D}}_{W(X)}^{(0)}-\text{mod})$. 

2) Let $LW\varphi^{*}:D(\widehat{\mathcal{D}}_{W(Y)}^{(0)}/p-\text{mod})\to D(\widehat{\mathcal{D}}_{W(X)}^{(0)}/p-\text{mod})$
denote the functor defined by the bimodule $\widehat{\mathcal{D}}_{W(X)\to W(Y),\text{acc}}^{(0)}/p$.
Then, under the equivalence of categories 
\[
D(\mathcal{D}_{X}^{(0)}-\text{mod})\tilde{\to}D_{\text{acc}}(\widehat{\mathcal{D}}_{W(X)}^{(0)}/p-\text{mod})
\]
(and the analogous one for $Y$, c.f. \prettyref{cor:Bimodule-properties!},
part $1)$ above), the functor $LW\varphi^{*}$ corresponds to $L\varphi^{*}$,
the usual $\mathcal{D}^{(0)}$-module pullback. 

3) Let $p>2$, and suppose that $\mathfrak{X}$ and $\mathfrak{Y}$
are smooth formal schemes (over $W(k)$), lifting $X$ and $Y$, respectively,
and let $\varphi:\mathfrak{X}\to\mathfrak{Y}$ be a morphism lifting
$\varphi$. Then, under the equivalence of categories 
\[
D_{cc}(\mathcal{D}_{\mathfrak{X}}^{(0)}-\text{mod})\tilde{\to}D_{\text{acc}}(\widehat{\mathcal{D}}_{W(X)}^{(0)}-\text{mod})
\]
(and the analogous one for $Y$, c.f. \prettyref{cor:Bimodule-properties!},
part $3)$ above), the functor $LW\varphi_{\text{}}^{*}$ corresponds
to $L\varphi^{*}$, the usual $\mathcal{\widehat{D}}^{(0)}$-module
pullback. 
\end{thm}

These facts are proved in \prettyref{subsec:Operations-Pull} below.
Various standard consequences, such as the fact that pullback commutes
with composition, are also derived there; there is also a version
of part $3)$ for lifts of $X$ and $Y$ to $W_{r}(k)$ for some $r>1$.
If $X$ and $Y$ are both affine, and we have coordinatized lifts
of Frobenius on both of them, there is an ``explicit'' description
of $\widehat{\mathcal{D}}_{W(X)\to W(Y),\text{acc}}^{(0)}$, which
is \prettyref{prop:Pullbback-and-transfer} (this is the key to proving
the above theorem). 

Now let us turn to pushforward, where the situation is (mostly) similar. 
\begin{defn}
The functor ${\displaystyle \int_{W\varphi}:}D_{\text{acc}}(\widehat{\mathcal{D}}_{W(X)}^{(0)}-\text{mod})\to D_{\text{acc}}(\widehat{\mathcal{D}}_{W(Y)}^{(0)}-\text{mod})$
is defined by 
\[
\int_{W\varphi}\mathcal{M}^{\cdot}:=R(W\varphi)_{*}(\widehat{\mathcal{D}}_{W(Y)\leftarrow W(X),\text{acc}}^{(0)}\otimes_{\widehat{\mathcal{D}}_{W(X)}^{(0)}}^{L}\mathcal{M}^{\cdot})_{\text{acc}}
\]
where $()_{\text{acc}}$ denotes the right adjoint to the inclusion
from accessible $\widehat{\mathcal{D}}_{W(Y)}^{(0)}$-modules to all
$\widehat{\mathcal{D}}_{W(Y)}^{(0)}$-modules. There is also, for
$r\geq1$ the analogous functor 
\[
\int_{W\varphi}\mathcal{M}^{\cdot}:=R(W\varphi)_{*}(\widehat{\mathcal{D}}_{W(Y)\leftarrow W(X),\text{acc}}^{(m)}/p^{r}\otimes_{\widehat{\mathcal{D}}_{W(X)}^{(0)}/p^{r}}^{L}\mathcal{M}^{\cdot})_{\text{acc}}
\]
\end{defn}

Unlike in the case of pullback, I don't have a proof that the object
\[
R(W\varphi)_{*}(\widehat{\mathcal{D}}_{W(Y)\leftarrow W(X),\text{acc}}^{(m)}\otimes_{\widehat{\mathcal{D}}_{W(X)}^{(0)}}^{L}\mathcal{M}^{\cdot})
\]
is accessible when $\mathcal{M}^{\cdot}$ is- this question involves
tricky issues related to when the projection formula holds for some
extremely large objects. At any rate, with this definition we have
the following analogue of \prettyref{thm:Basic-Pullback-properties}:
\begin{thm}
1) Under the equivalence of categories 
\[
D(\mathcal{D}_{X}^{(0)}-\text{mod})\tilde{\to}D_{\text{acc}}(\widehat{\mathcal{D}}_{W(X)}^{(0)}/p-\text{mod})
\]
(and the analogous one for $Y$, c.f. \prettyref{cor:Bimodule-properties!},
part $1)$ above), the functor ${\displaystyle \int_{W\varphi}}$
corresponds to ${\displaystyle \int_{\varphi}}$, the usual $\mathcal{D}^{(0)}$-module
pushforward.

2) Let $p>2$, and suppose that $\mathfrak{X}$ and $\mathfrak{Y}$
are smooth formal schemes (over $W(k)$), lifting $X$ and $Y$, respectively,
and let $\varphi:\mathfrak{X}\to\mathfrak{Y}$ be a morphism lifting
$\varphi$. Then, under the equivalence of categories 
\[
D_{cc}(\mathcal{D}_{\mathfrak{X}}^{(0)}-\text{mod})\tilde{\to}D_{\text{acc}}(\widehat{\mathcal{D}}_{W(X)}^{(0)}-\text{mod})
\]
(and the analogous one for $Y$, c.f. \prettyref{cor:Bimodule-properties!},
part $3)$ above), the functor ${\displaystyle \int_{W\varphi}}$
corresponds to ${\displaystyle \int_{\varphi}}$, the usual $\mathcal{\widehat{D}}^{(0)}$-module
pushforward.

3) Fix $r\geq1$ and consider the functor 
\[
\widehat{\int}_{W\varphi}\mathcal{M}^{\cdot}:=R(W\varphi)_{*}(\widehat{\mathcal{D}}_{W(Y)\leftarrow W(X),\text{\ensuremath{c-}acc}}^{(0)}/p^{r}\widehat{\otimes}_{\widehat{\mathcal{D}}_{W(X)}^{(0)}/p^{r}}^{L}\mathcal{\widehat{M}}^{\cdot})
\]
where on the right we have the (derived) completion of $\mathcal{M}^{\cdot}$
with respect to $V^{i}(\mathcal{O}_{W(X)}/p^{r})$, and $\widehat{\mathcal{D}}_{W(Y)\leftarrow W(X),\text{\ensuremath{c-}acc}}^{(0)}/p^{r}$
is a suitable completion of $\widehat{\mathcal{D}}_{W(Y)\leftarrow W(X),\text{acc}}^{(0)}/p^{r}$.
Then ${\displaystyle \widehat{\int}_{W\varphi}\mathcal{M}^{\cdot}}$
is isomorphic to the (derived) completion of ${\displaystyle \int_{W\varphi}\mathcal{M}^{\cdot}}$
with respect to $V^{i}(\mathcal{O}_{W(Y)}/p^{r})$. In other words,
the completed pushforward is given by the ``naive'' formula. 
\end{thm}

As in the case of pullback, we can deduce many of the basic properties
of push-forward from this theorem (this is all done in \prettyref{subsec:Operations-on-modules:Push!}
below). 

Now, let us discuss how these functors are related. The main result
is
\begin{thm}
\label{thm:Smooth-Adunction}1) Let $\varphi:X\to Y$ be smooth of
relative dimension $d$, and let $W\varphi^{!}=LW\varphi^{*}[d]$.
Then there is an isomorphism 
\[
R\varphi_{*}R\mathcal{H}om_{\widehat{\mathcal{D}}_{W(X)}^{(0)}}(W\varphi^{!}\mathcal{N}^{\cdot},\mathcal{M}^{\cdot})\tilde{\to}R\mathcal{H}om{}_{\widehat{\mathcal{D}}_{W(Y)}^{(0)}}(\mathcal{N}^{\cdot},\int_{W\varphi}\mathcal{M}^{\cdot})
\]
for any $\mathcal{N}^{\cdot}\in D_{\text{acc}}(\widehat{\mathcal{D}}_{W(Y)}^{(0)}-\text{mod})$
and any $\mathcal{M}^{\cdot}\in D_{\text{acc}}(\widehat{\mathcal{D}}_{W(X)}^{(0)}-\text{mod})$.
In particular, ${\displaystyle \int_{W\varphi}}$ is the right adjoint
to $\varphi^{!}$. 

2) Fix $r\geq1$, and suppose $\mathcal{M}\in\widehat{\mathcal{D}}_{W(X)}^{(0)}-\text{mod}_{\text{qcoh}}$
is an accessible quasicoherent $\widehat{\mathcal{D}}_{W(X)}^{(0)}$-module
which is nilpotent. Let $\mathcal{\tilde{M}}$ denote the associated
crystal on $X$ over $W_{r}(k)$. Then, for each $i$ ${\displaystyle \mathcal{H}^{i}(\int_{W\varphi}\mathcal{M})}$
is an accessible quasicoherent $\widehat{\mathcal{D}}_{W(Y)}^{(0)}$-module
which is nilpotent, and there is a functorial isomorphism 
\[
\mathcal{H}^{i-d}\tilde{(\int_{W\varphi}\mathcal{M})}\tilde{\to}R^{i}\varphi_{*,\text{crys}}(\tilde{\mathcal{M}})
\]
 
\end{thm}

This is theorems \prettyref{cor:smooth-adjunction} and \prettyref{cor:Crystalline-Push}
below. 

To finish things off, let us discuss the de Rham-Witt theory. By a
theorem of Etesse, there is a natural functor from crystals over $W_{r}(k)$
to de Rham-Witt connections over $W_{r}\Omega_{X}^{1}$; and one may
easily show that it extends to a functor to continuous de Rham-Witt
connections over $W\Omega_{X}^{1}/p^{r}$ (this object, being flat
over $W_{r}(k)$, is better behaved). Using this theory we put a canonical
continuous de Rham-Witt connection on the object $\widehat{\mathcal{D}}_{W(X),c-\text{acc}}^{(0)}/p^{r}$
(this is the completion of $\widehat{\mathcal{D}}_{W(X),\text{acc}}^{(0)}/p^{r}$
along $V^{i}(\mathcal{O}_{W(X)}/p^{i})$). From this we deduce
\begin{thm}
\label{thm:Accessible-to-DRW} There is an equivalence of categories
from the (abelian) category of accessible quasicoherent $\widehat{\mathcal{D}}_{W(X)}^{(0)}/p^{r}$
modules to the category of quasicoherent modules over $W(X)_{p^{r}=0}$
with continuous flat connection over $W\Omega_{X}^{1}/p^{r}$. 
\end{thm}

The fact that this functor is fully faithful is actually quite straightforward;
the essential surjectivity is trickier. In fact, the case of vector
bundles is handled by the paper of Bloch \cite{key-24}, whose central
lemma also forms the core of our proof (c.f. \prettyref{lem:Basic-Bloch-lemma}
below). 

Finally, let us present the de Rham-Witt resolution. If $\varphi:X\to Y$
is a smooth morphism of relative dimension $d$, we have the relative
de Rham-Witt complex of Langer-Zink, as developed in \cite{key-27},
denoted $W\Omega_{X/Y}^{\cdot}$. The de Rham-Witt connection on $\widehat{\mathcal{D}}_{W(X),c-\text{acc}}^{(0)}/p^{r}$
yields also a continuous connection map
\[
\nabla:\widehat{\mathcal{D}}_{W(X),c-\text{acc}}^{(0)}/p^{r}\to\widehat{\mathcal{D}}_{W(X),c-\text{acc}}^{(0)}/p^{r}\widehat{\otimes}_{\mathcal{O}_{W(X)}/p^{r}}W\Omega_{X/Y}^{1}/p^{r}
\]
(as above the hat denotes completion with respect to $V^{i}(\mathcal{O}_{W(X)}/p^{r})$),
and therefore an associated de Rham-Witt complex $(\widehat{\mathcal{D}}_{W(X),c-\text{acc}}^{(0)}/p^{r}\widehat{\otimes}_{\mathcal{O}_{W(X)}/p^{r}}W\Omega_{X/Y}^{\cdot}/p^{r},\nabla)$.
We have 
\begin{thm}
The complex $(\widehat{\mathcal{D}}_{W(X),c-\text{acc}}^{(0)}/p^{r}\widehat{\otimes}_{\mathcal{O}_{W(X)}/p^{r}}W\Omega_{X/Y}^{\cdot}/p^{r},\nabla)$
is exact except at the right-most term (which is the $d$th) and we
have
\[
\mathcal{H}^{d}((\widehat{\mathcal{D}}_{W(X),c-\text{acc}}^{(0)}/p^{r}\widehat{\otimes}_{\mathcal{O}_{W(X)}/p^{r}}W\Omega_{X/Y}^{\cdot}/p^{r},\nabla))\tilde{\to}\widehat{\mathcal{D}}_{W(Y)\leftarrow W(X),\text{c-acc}}^{(0)}/p^{r}
\]
where $\widehat{\mathcal{D}}_{W(Y)\leftarrow W(X),\text{c-acc}}^{(0)}/p^{r}$
is the completion of $\widehat{\mathcal{D}}_{W(Y)\leftarrow W(X),\text{acc}}^{(0)}/p^{r}$
along $V^{i}(\mathcal{O}_{W(X)}/p^{r})$; when $Y$ is a point it
is simply $W\omega_{X}/p^{r}$. It follows that, for any $\mathcal{M}\in\widehat{\mathcal{D}}_{W(X)}^{(0)}/p^{r}-\text{mod}_{\text{acc}}$
there is an isomorphism 
\[
\widehat{\int}_{W\varphi_{p^{r}=0}}\mathcal{M}\tilde{\to}R\varphi_{*}(\widehat{\mathcal{M}}\widehat{\otimes}_{\mathcal{O}_{W(X)}/p^{r}}W\Omega_{X/Y}^{\cdot}/p^{r})
\]
of sheaves over $\mathcal{O}_{W(Y)}/p^{r}$; as above all of the completions
are along $V^{i}(\mathcal{O}_{W(X)}/p^{r})$. 
\end{thm}

This is \prettyref{thm:drW-resolution} below. The second part of
the theorem therefore implies that crystalline cohomology for an arbitrary
quasicoherent crystal can be computed via the de Rham-Witt complex
mod $p^{r}$. This result has antecedents in the work of Etesse \cite{key-56}
and Berthelot \cite{key-26}, who worked instead with the de Rham-Witt
complex $W_{r}\Omega_{X}^{\cdot}$. To get a comparison theorem when
working with that complex forces one to consider crystals which are
flat over $W_{r}(k)$ (Etesse considered only vector bundles). Despite
these differences, the proofs (including the one in this paper) all
rely essentially on the basic fact that the de-Rham-Witt complex mod
$p$ is quasi-isomorphic to the de Rham complex of $X$. In our strategy,
though, evaluating the complex on a single object (namely $\widehat{\mathcal{D}}_{W(X),c-\text{acc}}^{(0)}/p^{r}$)
allows one to deduce the result for the entire category, including
for those accessible modules which are not nilpotent. This appears
to be new.

\subsection{Notations and conventions}

Throughout the entire paper a perfect field $k$ of characteristic
$p>0$ is fixed; as indicated in the introduction, some of the results
depend on weather $p=2$, we will indicate this when it occurs. We
let $W(k)$ denote the $p$-typical Witt vectors of $k$. 

Letters $X,Y,Z$ denote a smooth varieties over $k$. Letters $A,B,C$
denote smooth $k$-algebras, i.e., finite type $k$-algebras such
that the induced morphism $\text{Spec}(A)\to\text{Spec}(k)$ is smooth.
Gothic letters $\mathfrak{X}$, $\mathfrak{Y}$,$\mathfrak{Z}$ denote
smooth formal schemes over $W(k)$, whose special fibres are denoted
by the corresponding standard letters. We use $\mathcal{A}$, $\mathcal{B}$,
$\mathcal{C}$ to denote smooth algebra over $W(k)$; we use $\mathfrak{X}_{r}$,
$\mathfrak{Y}_{r}$, $\mathfrak{Z}_{r}$ to denote smooth schemes
over $W_{r}(k)$ and $\mathcal{A}_{r}$, $\mathcal{B}_{r}$, $\mathcal{C}_{r}$
smooth algebras over $W_{r}(k)$. 

We shall make essential use of Berthelot's arithmetic differential
operators; c.f. \cite{key-1}. In particular, for $m\geq0$, $\widehat{\mathcal{D}}_{\mathfrak{X}}^{(m)}$
denotes the arithmetic differential operators of level $m$, a $W(k)$-flat
and $p$-adically complete sheaf of algebras whose reduction mod $p$
is denoted $\mathcal{D}_{X}^{(m)}$. When $m=0$ this algebra is the
usual PD differential operators of \cite{key-10}; for a quick definition
and exposition of the basic properties of $\mathcal{D}_{X}^{(0)}$
we recommend \cite{key-3}, chapter $2$. In addition, the construction
has been extended by Shiho to define algebras $\widehat{\mathcal{D}}_{\mathfrak{X}}^{(m)}$
for $m<0$; c.f. \cite{key-65} for details.

For all such $X$ over $k$ we have the formal scheme of $p$-typical
Witt vectors $W(X)$, whose underlying topological space is $X$ and
whose topology is defined by a collection of ideals $V^{i}(\mathcal{O}_{W(X)})$
(c.f. \cite{key-66},\cite{key-67} for a detailed account). None
of these are locally finitely generated, but we have the finite type
scheme $W_{m}(X)$ whose structure sheaf is $\mathcal{O}_{W(X)}/V^{m}(\mathcal{O}_{W(X)})$.
We also have the formal schemes $W(X)_{p^{r}=0}$ for any $r\geq1$,
whose underlying sheaf of rings are given by $\mathcal{O}_{W(X)}/p^{r}$.
By a quasicoherent sheaf on this scheme we mean an inverse limit of
quasicoherent sheaves on the finite type schemes $W_{m}(X)_{p^{r}=0}$.
The morphisms in this category are continuous maps of $\mathcal{O}_{W(X)}/p^{r}$-modules.
This category is additive, though not abelian and otherwise terribly
behaved. 

Finally, let us take a moment to discuss the notion of cohomological
completeness. This notion, which also goes under the name derived
completeness, has been developed extensively in, e.g., \cite{key-8}
\cite{key-43}, Tag 091N, \cite{key-41}. We will use as a reference
\cite{key-8}, chapter $1$ as it develops the notion in the context
we need- a noncommutative sheaf of rings $\mathcal{R}$ over a topological
space $X$, so that $\mathcal{R}$ possesses a global central element
$p$ (called $h$ in \cite{key-8}) and such that $\mathcal{R}$ has
no nonzero $p$-torsion. For us $\mathcal{R}$ will in fact be a flat
$W(k)$-module. In this case a complex $\mathcal{M}^{\cdot}\in D(\mathcal{R}-\text{mod})$
is called cohomologically complete if 
\[
R\mathcal{H}om_{\mathcal{R}}(\mathcal{R}[p^{-1}],\mathcal{M}^{\cdot})=R\mathcal{H}om_{W(k)}(W(k)[p^{-1}],\mathcal{M}^{\cdot})=0
\]
A $p$-torsion-free sheaf of modules is cohomologically complete iff
it is $p$-adically complete in the usual sense (c.f. \cite{key-8},
lemma 1.5.4). Further, any complex in the image of the natural functor
$D(\mathcal{R}/p^{n}-\text{mod})\to D(\mathcal{R}-\text{mod})$ is
cohomologically complete (for any $n\geq1$). The collection of such
complexes forms a thick triangulated subcategory, which we denote
$D_{cc}(\mathcal{R}-\text{mod})$. The main facts about $D_{cc}(\mathcal{R}-\text{mod})$
that we need are, firstly, the Nakayama lemma: if $\mathcal{M}^{\cdot}\in D_{cc}(\mathcal{R}-\text{mod})$,
then $\mathcal{M}^{\cdot}\otimes_{W(k)}^{L}k=0$ iff $\mathcal{M}^{\cdot}=0$
(this is \cite{key-8}, proposition 1.5.8.). Secondly, we shall use
the fact that the inclusion functor $D_{cc}(\mathcal{R}-\text{mod})\to D(\mathcal{R}-\text{mod})$
admits a right adjoint, the derived completion functor (c.f. \cite{key-8},
proposition 1.5.6), concretely, this functor is given by 
\[
\mathcal{M}^{\cdot}\to R\mathcal{H}om_{\mathcal{R}}(\mathcal{R}[p^{-1}]/\mathcal{R}[-1],\mathcal{M}^{\cdot}):=\widehat{\mathcal{M}^{\cdot}}
\]
For complexes over the various $W(k)$-torsion-free rings appearing
in this paper, the symbol $\widehat{?}$ will denote cohomological
completion. We should mention that this symbol is also sometimes used
for complexes which are annihilated by $p^{r}$, where it stands for
the (derived) completion with respect to $V^{i}(\mathcal{O}_{W(X)}/p^{r})$
(c.f. \prettyref{prop:completion} and the discussion directly above).
As complexes annihilated by $p^{r}$ are automatically cohomologically
complete (in the sense used in this paper), the two notions are essentially
disjoint, and hopefully this will not cause undo confusion.

\subsection{Acknowledgements}

I would like to thank Michel Gros for inviting me to speak about this
work at Rennes, and for helpful conversations afterwards. I would
also like to that Bernard Le Stum for friendly conversations about
related topics. I would like to gratefully acknowledge the support
of the NSF. 

\section{Higher Derivations on Witt Vectors}

In this section we define, for each $m\in\mathbb{Z}$, the sheaf $\mathcal{E}_{W(X)}^{(m)}$
of higher derivations on the formal scheme $W(X)$ (of level $m$),
as well as the algebra $\mathcal{\widehat{D}}_{W(X)}^{(m)}$ of Witt-differential
operators, which is (the completion of) the sheaf of algebras generated
by $\mathcal{O}_{W(X)}$ and $\mathcal{E}_{W(X)}^{(m)}$. The key
to defining $\mathcal{E}_{W(X)}^{(m)}$ is the observation that Hasse-Schmidt
derivations on $X$ can be lifted canonically to the Witt-vectors. 

In order to make all this precise, we recall the 
\begin{defn}
\label{def:H-S}Let $R$ be a commutative ring, and let $D_{0}:B\to A$
be a morphism of commutative $R$-algebras. A Hasse-Schmidt derivation
(or higher derivation) of length $n$ from $B$ to $A$ over $R$
is a sequence of $R$-linear operators $D=(D_{0},\dots,D_{n})$ (we
allow $n=\infty$ for such a sequence indexed by $\mathbb{N}$) such
that 
\[
D_{l}(xy)=\sum_{i+j=l}D_{i}(x)D_{j}(y)
\]

If $A=B$, we shall suppose that $D_{0}=Id$ unless otherwise specified.
In that case we simply refer to a Hasse-Schmidt derivation of $A$
(over $R$). The operator $D_{i}$ is referred to as the $i$th component
of $D$. 
\end{defn}

We refer the reader to \cite{key-22} and \cite{key-23} for more
details and some interesting applications (e..g to jet spaces in algebraic
geometry). At any rate it follows immediately from the definition
that $D_{1}\in\text{Der}_{R}(A,B)$, and then by a quick induction
it follows that each $D_{i}$ is an $R$-linear differential operator
from $B$ to $A$ of order $\leq i$. 

In the sequel, it will be convenient to use the elementary fact that
a Hasse-Schmidt derivation is equivalent to a map of $R$-algebras
\[
\varphi_{D}:B\to A[t]/t^{n+1}
\]
where the sequence $D=(D_{0},\dots,D_{n})$ corresponds to 
\[
\varphi_{D}(b)=\sum_{i=0}^{n}D_{i}(b)t^{i}
\]
(here, if $n=\infty$, then the target is taken to be the power series
ring $A[[t]]$). For instance, from this characterization and the
infinitesimal lifting property (c.f. \cite{key-15}, ch. 2, proposition
6), we see immediately that if $B\to B'$ is an etale morphism of
$R$-algebras such that the structure morphism $D_{0}:B\to A$ extends
to a morphism $D'_{0}:B'\to A$, then any Hasse-Schmidt derivation
$D=(D_{0},\dots,D_{n})$ from $B$ to $A$ over $R$ extends uniquely
to $D'=(D'_{0},\dots,D'_{n})$ from $B'$ to $A$ over $R$. In particular,
a Hasse-Schmidt derivation of $A$ over $R$ extends uniquely to any
localization $A_{g}$ (for $g\in A$). 

It turns out that, after applying appropriate Frobenius twists, higher
derivations are very well-behaved on Witt vectors. To explain this,
we start with the following basic result: 
\begin{lem}
\label{lem:Basic-Presentation-of-Witt}Let $A$ be a smooth $k$-algebra.
For any $r\geq0$, let $\mathcal{A}_{r+1}$ be a flat $W_{r+1}(k)$-algebra
such that $\mathcal{A}_{r+1}/p\tilde{=}A$. Then there is an embedding
\[
W_{r+1}(A)\to\mathcal{A}_{r+1}
\]
which is $F^{r}$-semilinear over $W_{r+1}(k)$. The map is given
as follows: if $(f_{1},f_{2}\dots,f_{r+1})\in W_{r+1}(A)$ we choose
any lifts $\tilde{f}_{i}$ of $f_{i}$ in $\mathcal{A}_{r+1}$ and
send
\[
(f_{1},f_{2}\dots,f_{r+1})\to\tilde{f}_{1}^{p^{r}}+p\tilde{f}_{2}^{p^{r-1}}+\dots+p^{r}\tilde{f}_{r+1}
\]
Further, if $F$ is any lift of the absolute Frobenius map to $\mathcal{A}_{r+1}$,
then the restriction of $F$ to the image of $W_{r+1}(A)$ is isomorphic
to the Witt-vector Frobenius; in particular, we have 
\[
F(\tilde{f}_{1}^{p^{r}}+p\tilde{f}_{2}^{p^{r-1}}+\dots+p^{r}\tilde{f}_{r+1})=\tilde{f}_{1}^{p^{r+1}}+p\tilde{f}_{2}^{p^{r}}+\dots+p^{r}\tilde{f}_{r+1}^{p}
\]
For each $i\geq0$, the ideal $V^{i}(W_{r+1}(A))$ is given by the
intersection of $p^{i}\mathcal{A}_{r+1}$ with $W_{r+1}(A)$. 

If $\varphi^{\#}:B\to A$ is a morphism of smooth $k$-algebras, and
we consider any lift to $\varphi_{r+1}^{\#}:\mathcal{B}_{r+1}\to\mathcal{A}_{r+1}$,
then $\varphi_{r+1}^{\#}(W_{r+1}(B))\subset W_{r+1}(A)$, and the
induced map 
\[
\varphi_{r+1}^{\#}:W_{r+1}(B)\to W_{r+1}(A)
\]
agrees with the functorial map $W\varphi^{\#}$ coming from Witt vector
theory. 
\end{lem}

\begin{proof}
By definition, the ghost map
\[
w_{r}:W_{r+1}(\mathcal{A}_{r+1})\to\mathcal{A}_{r+1}
\]
given by 
\[
(\tilde{f}_{1},\tilde{f}_{2}\dots,\tilde{f}_{r+1})\to\tilde{f}_{1}^{p^{r}}+p\tilde{f}_{2}^{p^{r-1}}+\dots+p^{r}\tilde{f}_{r+1}
\]
is a ring homomorphism, and since $p^{r+1}$ annihilates $\mathcal{A}_{r+1}$,
we see that $w_{r}$ factors through $W_{r+1}(\mathcal{A}_{r+1}/p)=W_{r+1}(A)$.
Call the image of this map $W'$. I claim that the surjection $W_{r+1}(A)\to W'$
is also injective. For, if
\[
\tilde{f}_{1}^{p^{r}}+p\tilde{f}_{2}^{p^{r-1}}+\dots+p^{r}\tilde{f}_{r+1}=0
\]
then, taking the image in $A$, we see that $f_{1}^{p^{r}}=0$ which
implies $f_{1}=0$ and so $\tilde{f}_{1}=p\tilde{g}_{1}$; but then
$\tilde{f}_{1}^{p^{r}}=0$. Therefore $p\tilde{f}_{2}^{p^{r-1}}+\dots+p^{r}\tilde{f}_{r+1}=0$
in $pA_{r+1}$ but since $A_{r+1}$ is flat over $W_{r+1}(k)$ this
implies $\tilde{f}_{2}^{p^{r-1}}+\dots+p^{r-1}\tilde{f}_{r+1}=0$
in $A_{r}=A_{r+1}/p^{r}$. Thus an induction on $r$ shows that each
$f_{i}=0$ (where as above $f_{i}$ is the image of $\tilde{f}_{i}$
in $A$) as required. 

Now let us show that the map $w_{r}$ is $F^{r}$-semilinear. Choose
a lift of Frobenius on $\mathcal{A}_{r+1}$ whose restriction to $W_{r+1}(k)$
is the Witt-vector Frobenius. By a universal property of Witt-vectors
(c.f., e.g., \cite{key-17} proposition 1.1.23) there is a unique
map 
\[
s:\mathcal{A}_{r+1}\to W_{r+1}(\mathcal{A}_{r+1})
\]
so that $w_{r}\circ S=F^{r}$. By our choice of $F$ the map $s$
is a morphism of $W_{r+1}(k)$-algebras; therefore $w_{r}$ is $F^{r}$-semilinear
over $W_{r+1}(k)$ as claimed. 

Next let $F$ be any lift of Frobenius on $A_{r+1}$. As $F$ is a
map of algebras we have 
\[
F(\tilde{f}_{1}^{p^{r}}+p\tilde{f}_{2}^{p^{r-1}}+\dots+p^{r}\tilde{f}_{r+1})=F(\tilde{f}_{1})^{p^{r}}+pF(\tilde{f}_{2})^{p^{r-1}}+\dots+p^{r}F(\tilde{f}_{r+1})
\]
so that $F(W')\subset W'$; further, by what we have just proved the
term $F(\tilde{f}_{1})^{p^{r}}+pF(\tilde{f}_{2})^{p^{r-1}}+\dots+p^{r}F(\tilde{f}_{r+1})$
depends only on the images $F(f_{i})=f_{i}^{p}$ in $A$; whence the
statement. The fact that 
\[
V^{i}(W_{r+1}(A))=p^{i}\mathcal{A}_{r+1}\cap W_{r+1}(A)
\]
follows immediately from the definition of $V^{i}(W_{r+1}(A))$. 

Finally, consider a map $\varphi^{\#}:B\to A$, which, by the infinitesimal
lifting property, can be lifted to a map $\varphi_{r+1}^{\#}:\mathcal{B}_{r+1}\to\mathcal{A}_{r+1}$.
Then 
\[
\varphi^{\#}(\tilde{f}_{1}^{p^{r}}+p\tilde{f}_{2}^{p^{r-1}}+\dots+p^{r}\tilde{f}_{r+1})=\varphi^{\#}(\tilde{f}_{1})^{p^{r}}+\dots+p^{r}\varphi^{\#}(\tilde{f}_{r+1})
\]
from which the result follows immediately. 
\end{proof}
In the sequel, we shall make extensive use of this map, writing $W_{r+1}(A^{(r)})\subset\mathcal{A}_{r+1}$
for the image. 

An important special case of this lemma is given by the 
\begin{example}
\label{exa:Affine-Space}Suppose $A=k[T_{1},\dots,T_{n}]$, and $\mathcal{A}_{r+1}=W_{r+1}(k)[T_{1},\dots,T_{n}]$.
Then the image of ${\displaystyle W_{r+1}(A^{(r)})\to\mathcal{A}_{r+1}}$
is the $W_{r+1}(k)$-submdule spanned by 
\[
\{p^{j}\prod_{i=1}^{n}T_{i}^{a_{i}p^{r-j}}\}
\]
where the $\{a_{i}\}$ are any natural numbers. Indeed, it is clear
that all of the displayed monomials are contained in $W_{r+1}(A^{(r)})$.
To see the converse, we use induction on $r$, the case $r=0$ being
trivial. Consider a term of the form $p^{j}f^{p^{r-j}}\in W_{r+1}(A^{(r)})$,
for $j<r$ (when $j=r$ the result is obvious). By induction we have
an expression 
\[
p^{j}f^{p^{r-1-j}}=\sum_{j=0}^{r-1}\sum_{i}b_{i}p^{j}T_{i}^{a_{i}p^{r-1-j}}\text{mod}(p^{r})
\]
We may consider the lift of Frobenius which takes $T_{i}\to T_{i}^{p}$.
Applying this map to the displayed equality yields 
\[
p^{j}f^{p^{r-j}}=\sum_{j=0}^{r-1}\sum_{i}b_{i}p^{j}T_{i}^{a_{i}p^{r-j}}\text{mod}(p^{r})
\]
which implies 
\[
p^{j}f^{p^{r-j}}=\sum_{j=0}^{r-1}\sum_{i}b_{i}p^{j}T_{i}^{a_{i}p^{r-j}}+\sum_{i}b_{i}p^{r}T^{a_{i}}
\]
as claimed. 
\end{example}

We will construct our canonical lifts of Hasse-Schmidt derivations
using the above construction of $W_{r+1}(A^{(r)})$; to apply it,
we make use of the following straightforward application of the infinitesimal
lifting property:
\begin{lem}
For any $r\geq0$, let $\mathcal{A}_{r+1},\mathcal{B}_{r+1}$ be flat
$W_{r+1}(k)$-algebras such that $\mathcal{A}_{r+1}/p\tilde{=}A$
and $\mathcal{B}_{r+1}/p\tilde{=}B$. Let $D=(D_{0},\dots,D_{n})$
be any Hasse-Schmidt derivation from $B$ to $A$, over $k$. Then
there is a lift to a Hasse-Schmidt derivation $\tilde{D}=(\tilde{D}_{0},\dots,\tilde{D}_{n})$,
over $W_{r+1}(k)$, from $\mathcal{B}_{r+1}$ to $\mathcal{A}_{r+1}$,
\end{lem}

In general, there will be many lifts of a given Hasse-Schmidt derivation.
However, upon restricting to Witt vectors we have the following essential
result: 
\begin{thm}
\label{thm:Canonical-H-S}Let $B,A$ be smooth $k$-algebras, and
let $D=(D_{0},D_{1},\dots):B\to A$ be a Hasse-Schmidt derivation
(of any length). Let $\tilde{D}=(\tilde{D}_{0},\tilde{D}_{1},\dots)$
be a lift of $D$ to a Hasse-Schmidt derivation from $\mathcal{B}_{r+1}$
to $\mathcal{A}_{r+1}$ (where, as above, these are flat lifts of
$B$ and $A$, respectively). Then $\tilde{D}$ takes the subalgebra
$W_{r+1}(B^{(r)})$ into $W_{r+1}(A^{(r)})$. Further, the induced
map $\tilde{D}:W_{r+1}(B^{(r)})\to W_{r+1}(A^{(r)})$ is independent
of the choice of lift $\tilde{D}$. 
\end{thm}

\begin{proof}
We begin by showing that $\tilde{D}$ takes $W_{r+1}(B^{(r)})$ into
$W_{r+1}(A^{(r)})$. We recall the following formula for the action
of Hasse-Schmidt derivations on powers, which is easily checked by
induction: 
\[
\tilde{D}_{j}(f^{n})=\sum_{i_{1}+\dots+i_{n}=j}\tilde{D}_{i_{1}}(f)\cdots\tilde{D}_{i_{n}}(f)
\]
so that we have 
\begin{equation}
\tilde{D}_{j}(p^{l}f^{p^{r-l}})=p^{l}\sum_{i_{1}+\dots+i_{p^{r-l}}=j}\tilde{D}_{i_{1}}(f)\cdots\tilde{D}_{i_{p^{r-l}}}(f)\label{eq:HS-action}
\end{equation}
Now, the set 
\[
\{(i_{1},\dots,i_{p^{r-l}})\in\mathbb{N}^{p^{r-l}}|i_{1}+\dots+i_{p^{r-l}}=j\}
\]
is acted upon by the symmetric group $S_{p^{r-l}}$, and, after grouping
like terms together, the coefficient of a term $\tilde{D}_{i_{1}}(f)\cdots\tilde{D}_{i_{p^{r-l}}}(f)$
in \prettyref{eq:HS-action} is precisely the size of the $S_{p^{r-l}}$
orbit of $(i_{1},\dots,i_{p^{r-l}})$. 

To calculate this size, suppose that there are $m$ distinct numbers
occurring in $(i_{1},\dots,i_{p^{r-l}})$, call them $\{j_{1},\dots,j_{m}\}$.
Let $\{C_{n}\}_{n=1}^{m}$ denote the corresponding $m$ subsets of
$\{1,\dots,p^{r-l}\}$, so that 
\[
C_{n}=\{t|i_{t}=j_{n}\}
\]
 The stabilizer of the action of $S_{p^{r-l}}$ on $(i_{1},\dots,i_{p^{r-l}})$
is then the group 
\[
S_{C_{1}}\times\cdots\times S_{C_{m}}
\]
of permutations which preserve each $C_{i}$. If we let $c_{i}=|C_{i}|$
then we deduce that 
\[
|\mathcal{O}_{(i_{1},\dots,i_{p^{r-l}})}|=\frac{(p^{r-l})!}{c_{1}!\cdots c_{m}!}
\]
where the left hand side denotes the size of the orbit. Summing up,
we deduce that 
\begin{equation}
\tilde{D}_{j}(p^{l}f^{p^{r-l}})=\sum_{\mathcal{O}}p^{l}\frac{(p^{r-l})!}{c_{1}!\cdots c_{m}!}(\tilde{D}_{j_{1}}(f))^{c_{1}}\cdots(\tilde{D}_{j_{m}}(f))^{c_{m}}\label{eq:Canonical-Lift}
\end{equation}
where $\mathcal{O}$ ranges over the set of orbits of $S_{p^{r-l}}$
on $\{(i_{1},\dots,i_{p^{r-l}})\in\mathbb{N}^{p^{r-l}}|i_{1}+\dots+i_{p^{r-l}}=j\}$. 

So, we must show that 
\[
p^{l}\frac{(p^{r-l})!}{c_{1}!\cdots c_{m}!}(\tilde{D}_{j_{1}}(f))^{c_{1}}\cdots(\tilde{D}_{j_{m}}(f))^{c_{m}}\in W_{r+1}(A^{(r)})
\]
Consider the map $W_{r+1}(k)[T_{1},\dots,T_{m}]\to A_{r+1}$ given
by sending $T_{i}\to\tilde{D}_{j_{i}}(f)$. This map takes $W_{r+1}(k[T_{1},\dots,T_{m}]^{(r)})$
to $W_{r+1}(A^{(r)})$, and so it suffices to prove that 
\[
p^{l}\frac{(p^{r-l})!}{c_{1}!\cdots c_{m}!}(T_{1})^{c_{1}}\cdots(T_{m})^{c_{m}}\in W_{r+1}(k[T_{1},\dots,T_{m}]^{(r)})
\]
On the other hand, the multinomial formula tells us that 
\[
p^{l}(T_{1}+\dots+T_{m})^{p^{r-l}}=\sum_{k_{1}+\dots k_{m}=p^{r-l}}p^{l}\frac{(p^{r-l})!}{k_{1}!\cdots k_{m}!}T_{1}^{k_{1}}\cdots T_{m}^{k_{m}}
\]
and by \prettyref{exa:Affine-Space} all of the terms of this sum
are in $W_{r+1}(k[T_{1},\dots,T_{m}]^{(r)})$, since the left hand
side is. So putting $k_{i}=c_{i}$ implies the result. 

To see the uniqueness, we note that we have just proved an equality
\[
p^{l}\frac{(p^{r-l})!}{c_{1}!\cdots c_{m}!}(\tilde{D}_{j_{1}}(f))^{c_{1}}\cdots(\tilde{D}_{j_{m}}(f))^{c_{m}}=b_{c_{1},\dots,c_{m}}p^{\alpha}(\tilde{D}_{j_{1}}(f))^{c'_{1}p^{r-\alpha}}\cdots(\tilde{D}_{j_{m}}(f))^{c'_{m}p^{r-\alpha}}
\]
where 
\[
\alpha=\text{val}(p^{l}\frac{(p^{r-l})!}{c_{1}!\cdots c_{m}!})
\]
and $b_{c_{1},\dots,c_{m}}$ is the image of ${\displaystyle p^{l}\frac{(p^{r-l})!}{c_{1}!\cdots c_{m}!}p^{-\alpha}}$
in $\mathbb{Z}/p^{r+1}$, and $c_{i}=c'_{i}p^{r-\alpha}$ for all
$i$. But by \prettyref{lem:Basic-Presentation-of-Witt}, such a term
depends only on the image of $\{\tilde{D}_{j_{1}}(f),\dots,\tilde{D}_{j_{m}}(f)\}$
in $A$, i.e., it depends only on the Hasse-Schmidt derivation $D$.
Since $\tilde{D}_{j}(p^{l}f^{p^{r-l}})$ is a sum of such terms, we
see that it is independent of the choice of lift $\tilde{D}$ as claimed. 
\end{proof}
If we unpack the proof of the theorem, we in fact obtain a formula
for the Hasse-Schmidt derivation as a sequence of maps $W_{r+1}(B^{(r)})\to W_{r+1}(A^{(r)})$;
if we then apply the isomorphisms $A\tilde{=}A^{(r)}$ and $B\tilde{=}B^{(r)}$
we obtain a Hasse-Schmidt derivation $W_{r+1}(B)\to W_{r+1}(A)$.
It will be useful to record this explicitly: 
\begin{cor}
\label{cor:Formula}Let $A,B$ be smooth $k$-algebras and $D$ a
Hasse-Schmidt derivation from $B$ to $A$ as above. Let $j\geq1$.
For $l\in\{0,\dots,r\}$, we have the set $\mathcal{S}_{l}=\{(i_{1},\dots,i_{p^{r-l}})\in\mathbb{N}^{p^{r-l}}|i_{1}+\dots+i_{p^{r-l}}=j\}$.
To each element $v\in\mathcal{S}_{l}$, we may associate a partition
$\{C_{i}\}_{i=1}^{m}$ of $\{1,\dots,p^{r-l}\}$. Let $c_{i}=|C_{i}|$.
The number $c_{i}$ depends only on the associated orbit $\mathcal{O}$
of the action of the symmetric group $S_{p^{r-l}}$; as does the set
$\{j_{1},\dots,j_{m}\}$ of distinct numbers\footnote{We index so that each element of $C_{i}$ is $j_{i}$}
occurring in $v$. Then we have

1) $\alpha_{\mathcal{O}}:={\displaystyle \text{val}(p^{l}\frac{(p^{r-l})!}{c_{1}!\cdots c_{m}!})}\geq l$
and if $\alpha_{\mathcal{O}}\leq r$ we have $p^{r-\alpha_{\mathcal{O}}}|c_{i}$
for each $i\in\{1,\dots m\}$. 

2) Let $c_{i}=c'_{i}p^{r-\alpha_{\mathcal{O}}}$, and let $b_{\mathcal{O}}$
be the image of ${\displaystyle p^{l}\frac{(p^{r-l})!}{c_{1}!\cdots c_{m}!}p^{-\alpha}}$
in $\mathbb{F}_{p}$. Define 
\[
\tilde{D}_{j}(0,\dots f_{l+1},\dots0)=\sum_{\mathcal{O}}(0,\dots,b_{\mathcal{O}}\prod_{i=1}^{m}D_{j_{i}}(f_{l+1})^{c'_{i}},\dots,0)
\]
where on the left hand side the term $f_{l+1}$ is in the $l+1$ position,
and on the right the term $b_{\mathcal{O}}\prod_{i=1}^{m}D_{j_{i}}(f_{l+1})$
occurs in the $\alpha_{\mathcal{O}}+1$ position (the term is taken
to be $0$ if $\alpha_{\mathcal{O}}>r$). If we extend this map to
all of $W_{r+1}(A)$ by addition of components, then $\tilde{D}$
is a $W_{r+1}(k)$-linear Hasse-Schmidt derivation from $W_{r+1}(B)$
to $W_{r+1}(A)$. 
\end{cor}

\begin{proof}
Everything is immediate from the proof of the previous theorem except
(perhaps) the $W_{r+1}(k)$-linearity; for a Hasse-Schmidt derivation
this is equivalent to demanding that $\tilde{D}_{j}(W(k))=0$ for
all $j\geq1$. But this is immediate from the formula; by definition
we have ${\displaystyle \sum_{i=1}^{m}c_{i}j_{i}=j}$, so that if
$j\geq1$ then at least one $j_{i}$ is $\geq1$ as well; and we have
$D_{j_{i}}(\lambda)=0$ for $\lambda\in k$ and $j_{i}\geq1$. 
\end{proof}
\begin{rem}
I an unaware of a \emph{direct }proof that the output of the formula
recorded above actually is a Hasse-Schmidt derivation; one deduces
this implicitly in the theorem from the fact that the lift $\tilde{D}$
is a Hasse-Schmidt derivation. In fact, it seems non-obvious that
the maps written above are even additive. 
\end{rem}

\begin{defn}
The Hasse-Schmidt derivations occurring in \prettyref{cor:Formula}
are called \emph{canonical} Hasse-Schmidt derivations. 
\end{defn}

\begin{rem}
In fact, the construction can be extended to the case where $A$ and
$B$ are finite type over a noetherian $k$-algebra $R$; in this
case, one obtains (by the same formula), $W_{r+1}(R)$-linear HS derivations
$W_{r+1}(B)\to W_{r+1}(A)$. We won't use this generalization in this
work.
\end{rem}

Until further notice we fix a morphism $B\to A$ of algebras, which
are each smooth over $k$. Let us show that the canonical Hasse-Schmidt
derivations enjoy many favorable properties. 
\begin{prop}
\label{prop:HS-and-F} Let $(D_{0},\dots D_{n})$ be a Hasse-Schmidt
derivation from $B$ to $A$ (over $R$). Then for any $j\leq n$
we have 
\[
\tilde{D}_{j}F=F\tilde{D}_{j/p}
\]
where $\tilde{D}$ denotes the canonical lift of $D$, $F$ is the
Witt-vector Frobenius (on $W_{r+1}(B)$ and $W_{r+1}(A)$, respectively),
and the right hand side is defined to be $0$ if $p$ does not divide
$j$. 
\end{prop}

\begin{proof}
Let us first show this result when $r=0$. In this case we have 
\[
D_{j}(F(f))=D_{j}(f^{p})=\sum_{i_{1}+\dots i_{p}=j}D_{i_{1}}(f)\cdots D_{i_{p}}(f)
\]
For each term in the sum, either $i_{1}=i_{2}=\cdots=i_{p}$, or there
are at least two distinct indices. In the latter case, grouping like
terms together and arguing as in the proof of \prettyref{thm:Canonical-H-S}
gives a coefficient of the form 
\[
\frac{p!}{c_{1}!\cdots c_{m}!}D_{j_{1}}(f)^{c_{1}}\dots D_{j_{m}}(f)^{c_{m}}=0
\]
since each $c_{i}<p$ and $p=0$ in $A$. Thus we obtain 
\[
D_{j}(f^{p})=D_{j/p}(f){}^{p}
\]
as required (where the right hand side is taken to be $0$ if $p$
does not divide $j$). 

Now let us consider the case of the action of $\tilde{D}_{j}$ from
$W_{r+1}(B)$ to $W_{r+1}(A)$. From \prettyref{cor:Formula}, we
obtain the formula 
\[
\tilde{D}_{j}(0,\dots F(f_{l+1}),\dots0)=\sum_{\mathcal{O}}(0,\dots,b_{\mathcal{O}}\prod_{i=1}^{m}D_{j_{i}}(F(f_{l+1}))^{c'_{i}},\dots,0)
\]
where the notation is as above. In particular, with $c_{i}=c'_{i}p^{r-\alpha_{\mathcal{O}}}$
we have ${\displaystyle \sum_{i=1}^{m}c_{i}j_{i}=j}$. From the case
$r=0$ of the proposition we therefore obtain an equality 
\[
\sum_{\mathcal{O}}(0,\dots,b_{\mathcal{O}}\prod_{i=1}^{m}D_{j_{i}}(F(f_{l+1}))^{c'_{i}},\dots,0)=\sum_{\mathcal{O}}(0,\dots,b_{\mathcal{O}}\prod_{i=1}^{m}F(D_{j_{i}/p}(f_{l+1}))^{c'_{i}},\dots,0)
\]
\[
=F(\sum_{\mathcal{O}}(0,\dots,b_{\mathcal{O}}\prod_{i=1}^{m}D_{j_{i}/p}(f_{l+1})^{c'_{i}},\dots,0))
\]
By definition any term in which $p$ does not divide all the $j_{i}$
is zero. So, if $p$ does not divide $j$, then from the equality
${\displaystyle \sum_{i=1}^{m}c_{i}j_{i}=j}$ we see that $p$ cannot
divide all the $j_{i}$, so that the sum is zero. 

If $p$ does divide $j$, then multiplication by $p$ induces a bijection
between $\{(i_{1},\dots,i_{p^{r-l}})\in\mathbb{N}^{p^{r-l}}|i_{1}+\dots+i_{p^{r-l}}=j/p\}$
and the subset of $\{(i_{1},\dots,i_{p^{r-l}})\in\mathbb{N}^{p^{r-l}}|i_{1}+\dots+i_{p^{r-l}}=j\}$
in which $p$ divides every term; and this bijection preserves the
numbers $\{c_{1},\dots,c_{m}\}$ and $\{c'_{1},\dots,c'_{m}\}$ associated
to each orbit. Thus from the displayed formula we see that $\tilde{D}_{j}F=F\tilde{D}_{j/p}$
as desired.
\end{proof}
We would like to consider the interaction of the canonical Hasse-Schmidt
derivations with the natural quotient maps on Witt vectors. To that
end we let $R:W_{r+1}(A)\to W_{r}(A)$ denote the natural projection.

Then we have
\begin{prop}
\label{prop:Move-Down}Let $D=(D_{0},\dots,D_{n})$ be a Hasse-Schmidt
derivation from $B$ to $A$, and let $\tilde{D}$ be the associated
canonical Hasse-Schmidt derivation from $W_{r+1}(B)$ to $W_{r+1}(A)$.
Then $\tilde{D}$ takes $V^{i}(W_{r+1}(B))$ into $V^{i}(W_{r+1}(A))$.
Furthermore, 

a) Under the isomorphisms
\[
R:W_{r+1}(A)/V^{r}(W_{r+1}(A))\tilde{\to}W_{r}(A)
\]
 and $R:W_{r+1}(B)/V^{r}(W_{r+1}(B))\tilde{\to}W_{r}(B)$ we have
$R\circ\tilde{D}_{j}=\tilde{D}_{j/p}\circ R$, where the right hand
side is zero if $p$ does not divide $j$. 

b) Now suppose $\text{val}(j)=a\leq r$. Then 
\[
\tilde{D}_{j}(W_{r+1}(B))\subset V^{r-a}(W_{r+1}(A))
\]
\end{prop}

\begin{proof}
The fact that $\tilde{D}$ takes $V^{i}(W_{r+1}(B))$ into $V^{i}(W_{r+1}(A))$
is immediate from \prettyref{cor:Formula}. 

Next, to prove $a)$, we first prove that if $p$ does not divide
$j$, then $R\circ\tilde{D}_{j}=0$. To see this, we use the formula
\[
\tilde{D}_{j}(0,\dots f_{l+1},\dots0)=\sum_{\mathcal{O}}(0,\dots,b_{\mathcal{O}}\prod_{i=1}^{m}D_{j_{i}}(f_{l+1})^{c'_{i}},\dots,0)
\]
Since, in this formula, ${\displaystyle \sum c_{i}j_{i}=j}$, the
assumption that $p$ does not divide $j$ implies that, for each term
in the sum, at least one $c_{i}$ is not divisible by $p$. Therefore
the fact that $p^{r-\alpha_{\mathcal{O}}}|c_{i}$ implies $\alpha_{\mathcal{O}}\geq r$
for each orbit $\mathcal{O}$, which shows that on the right hand
side each nonzero term is in the $r+1$ position. Therefore $R$ annihilates
this term. 

Now suppose $p$ divides $j$. Let $r>l$. For an $S_{p^{r-l}}$ orbit
inside $\mathcal{S}_{l}=\{(i_{1},\dots,i_{p^{r-l}})\in\mathbb{N}^{p^{r-l}}|i_{1}+\dots+i_{p^{r-l}}=j\}$,
we write $p|\mathcal{O}$ if each of the numbers $c_{i}$ is divisible
by $p$. If we choose a member of $\mathcal{O}$ for which each of
the sets $C_{i}$ is an interval in $\{1,\dots,p^{r-l}\}$ then $p|\mathcal{O}$
implies that each $C_{i}$ is an interval of the form $\{pm_{i}+1,\dots,pm_{i+1}\}$.
Thus to the orbit $\mathcal{O}$ we may associate the partition of
$\{1,\dots,p^{r-l-1}\}$ where $C'_{i}=\{m_{i}+1,\dots,m_{i+1}\}$.
Further, to each $C_{i}$ is associated the number $j_{i}$ such that
${\displaystyle \sum_{i=1}^{m}c_{i}j_{i}=j}$; this implies ${\displaystyle \sum_{i=1}^{m}(c_{i}/p)j_{i}=j/p}$,
so if we associate $j_{i}$ to each $C_{i}'$ we obtain an element
of $\mathcal{\tilde{S}}_{l}=\{(i_{1},\dots,i_{p^{r-l-1}})\in\mathbb{N}^{p^{r-l-1}}|i_{1}+\dots+i_{p^{r-l-1}}=j/p\}$.
Clearly, if we choose a different member of $\mathcal{O}$ the resulting
elements of $\mathcal{\tilde{S}}_{l}$ will be conjugate. Thus we
obtain a map 
\[
\{S_{p^{r-l}}\phantom{i}\text{orbits on}\phantom{i}\mathcal{S}_{l}\phantom{i}\text{with}\phantom{i}p|\mathcal{O}\}\to\{S_{p^{r-l-1}}\phantom{i}\text{orbits on}\phantom{i}\mathcal{\tilde{S}}_{l}\}
\]
and this map is a bijection. Indeed, given an orbit $\mathcal{O}'$
in $\tilde{S}_{l}$ we may choose an element for which the members
of the partition $C_{i}'$ consists of intervals $\{m_{i}+1,\dots,m_{i+1}\}$;
and to this partition we attach the partition $\{pm_{i}+1,\dots,pm_{i+1}\}$
of $\{1,\dots,p^{r-l}\}$, with the same associated number $j_{i}$
as before. 

Further, by the claim directly below, we have 
\[
p^{l}\frac{(p^{r-l})!}{c_{1}!\cdots c_{m}!}\equiv p^{l}\frac{(p^{r-l-1})!}{(c_{1}/p)!\cdots(c_{m}/p)!}\phantom{i}\text{mod}\phantom{i}p^{r}
\]
which tells us that $\alpha_{\mathcal{O}}=\alpha_{\mathcal{O}'}$
if $\alpha_{O}<r$, and in this case, $b_{\mathcal{O}}=b_{\mathcal{O}'}$.
Recall that the numbers $\{c_{i}\}$ satisfy $c_{i}=c'_{i}p^{r-\alpha_{\mathcal{O}}}$;
thus in this case we also have $c_{i}/p=c_{i}'p^{r-1-\alpha_{\mathcal{O}'}}$. 

Now consider $\tilde{D}_{j}$ acting on $(0,\dots f_{l+1},\dots0)$;
if $l=r$ then $R$ annihilates this term so suppose $r>l$. As above,
if $\mathcal{O}$ is an orbit such that $p$ does not divide $c_{i}$
for some $i$, then 
\[
R\circ(0,\dots,b_{\mathcal{O}}\prod_{i=1}^{m}D_{j_{i}}(f_{l+1})^{c'_{i}},\dots,0)=0
\]
and the same is true for any orbit $\mathcal{O}$ in which $\alpha_{\mathcal{O}}\geq r$.
Thus we have 
\[
R\circ\tilde{D}_{j}(0,\dots f_{l+1},\dots0)=\sum_{p|\mathcal{O}}R\circ(0,\dots,b_{\mathcal{O}}\prod_{i=1}^{m}D_{j_{i}}(f_{l+1})^{c'_{i}},\dots,0)
\]
\[
=\sum_{\{p|\mathcal{O},\alpha_{\mathcal{O}}<r\}}R\circ(0,\dots,b_{\mathcal{O}}\prod_{i=1}^{m}D_{j_{i}}(f_{l+1})^{c'_{i}},\dots,0)
\]
Now, we have 
\[
\sum_{\{p|\mathcal{O},\alpha_{\mathcal{O}}<r\}}R\circ(0,\dots,b_{\mathcal{O}}\prod_{i=1}^{m}D_{j_{i}}(f_{l+1})^{c'_{i}},\dots,0)=\sum_{\{\mathcal{O}',\alpha_{\mathcal{O}'}<r\}}(0,\dots,b_{\mathcal{O}'}\prod_{i=1}^{m}D_{j_{i}}(f_{l+1})^{c'_{i}},\dots,0)
\]
\[
=\sum_{\mathcal{O}'}(0,\dots,b_{\mathcal{O}'}\prod_{i=1}^{m}D_{j_{i}}(f_{l+1})^{c'_{i}},\dots,0)
\]
where the last equality follows because $\alpha_{\mathcal{O}'}\geq r$
implies that the term\linebreak{}
 $(0,\dots,b_{\mathcal{O}'}\prod_{i=1}^{m}D_{j_{i}}(f_{l+1})^{c'_{i}},\dots,0)$
is zero in $W_{r}(A)$. But the above discussion implies that this
last sum is $\tilde{D}_{j/p}\circ R(0,\dots f_{l+1},\dots0)$ as claimed. 

Finally, part $b)$ follows immediately from $a)$ and induction on
$r$. 
\end{proof}
To finish the proof we need the 
\begin{claim}
\label{claim:Multinomial-reduction-claim}Let $r>l$ and let $\{c_{1},\dots,c_{m}\}$
be positive integers such that $\sum c_{i}=p^{r-l}$. Suppose $c_{i}$
is divisible by $p$ for each $i$. We have 
\[
p^{l}\frac{(p^{r-l})!}{c_{1}!\cdots c_{m}!}\equiv p^{l}\frac{(p^{r-l-1})!}{(c_{1}/p)!\cdots(c_{m}/p)!}\phantom{i}\text{mod}\phantom{i}p^{r}
\]
\end{claim}

\begin{proof}
Consider the polynomial ring $\mathbb{Z}/p^{r+1}[T_{1},\dots,T_{m}]$.
The multinomial formula yields 
\begin{equation}
p^{l}(T_{1}+\dots+T_{m})^{p^{r-l}}=\sum_{k_{1}+\dots k_{m}=p^{r-l}}p^{l}\frac{(p^{r-l})!}{k_{1}!\cdots k_{m}!}T_{1}^{k_{1}}\cdots T_{m}^{k_{m}}\label{eq:first-multi}
\end{equation}
If we take the image of $p^{l}(T_{1}+\dots+T_{m})^{p^{r-l}}$ in $\mathbb{Z}/p^{r}[T_{1},\dots,T_{m}]$,
then we have 
\[
p^{l}(T_{1}+\dots+T_{m})^{p^{r-l}}=p^{l}((T_{1}+\dots+T_{m})^{p})^{p^{r-1-l}}=F_{1}(p^{l}(T_{1}+\dots+T_{m})^{p^{r-1-l}})
\]
where $F_{1}$ is an appropriately chosen lift of Frobenius. Since
$p^{l}(T_{1}+\dots+T_{m})^{p^{r-1-l}}\in W_{r}(\mathbb{F}_{p}[T_{1},\dots T_{m}]^{(r-1)})\subset\mathbb{Z}/p^{r}[T_{1},\dots,T_{m}]$,
the same equation holds if we replace $F_{1}$ with any lift of Frobenius,
by \prettyref{lem:Basic-Presentation-of-Witt}. Now, applying the
multinomial formula again yields 
\[
p^{l}(T_{1}+\dots+T_{m})^{p^{r-1-l}}=\sum_{j_{1}+\dots j_{m}=p^{r-l-1}}p^{l}\frac{(p^{r-l-1})!}{j_{1}!\cdots j_{m}!}T_{1}^{j_{1}}\cdots T_{m}^{j_{m}}
\]
and so if we apply the map $F$ for which $F(T_{i})=T_{i}^{p}$ we
obtain 
\[
F(p^{l}(T_{1}+\dots+T_{m})^{p^{r-1-l}})=\sum_{j_{1}+\dots j_{m}=p^{r-l-1}}p^{l}\frac{(p^{r-l-1})!}{j_{1}!\cdots j_{m}!}T_{1}^{pj_{1}}\cdots T_{m}^{pj_{m}}
\]
and if we consider the term where $j_{i}=c_{i}/p$ and compare with
\prettyref{eq:first-multi} we obtain the result. 
\end{proof}
With these preliminaries in hand, we can give the definition of Witt-differential
operators over the algebra $A$. We shall be working with the space
\\
$\text{Hom}_{W(k)}^{V}(W(B),W(A))$ of continuous,  $W(k)$-linear
endomorphisms of $W(A)$ which preserve $V^{i}(W(A))$ for each $i\geq0$.
We note that $W(A)\to\text{Hom}_{W(k)}^{V}(W(B),W(A))$ via the map
which takes $a\in W(A)$ to $a\cdot\varphi^{\#}$, and $W(B)\to\text{Hom}_{W(k)}^{V}(W(B),W(A))$
by the map which takes $b$ to $\varphi^{\#}(b\cdot)$. 

For each $r\geq0$, we define 
\[
\text{Hom}_{W(k)}^{V}(W_{r+1}(B),W_{r+1}(A))\subset\text{Hom}_{W(k)}(W_{r+1}(B),W_{r+1}(A))
\]
 to be the subspace of $W(k)$-linear endomorphisms of $W_{r+1}(A)$
which preserve the filtration $V^{i}(W_{r+1}(A))$. 

We have the natural restriction map
\[
R:\text{Hom}_{W(k)}^{V}(W_{r+1}(B),W_{r+1}(A))\to\text{Hom}_{W(k)}^{V}(W_{r}(B),W_{r}(A))
\]
and from the definitions it follows directly that 
\[
\text{Hom}_{W(k)}^{V}(W(B),W(A))\tilde{=}\lim_{r}\text{Hom}_{W(k)}^{V}(W_{r+1}(B),W_{r+1}(A))
\]

We have the following elementary observation: 
\begin{lem}
For each $j\geq0$, and each $r\geq0$ (including $r=\infty$) there
is a well-defined operation 
\[
W_{r+1}(A)\times V^{j}(W_{r+1}(A))\to V^{j}(W_{r+1}(A))
\]
given by 
\[
(a,V^{j}(b))\to V^{j}(a\cdot b):=F^{-j}(a)\cdot V^{j}(b)
\]
and this operation is $F^{-j}$-semilinear over $W_{r+1}(k)$, and
preserves the filtration $V^{i}(W_{r+1}(A))$. 
\end{lem}

To justify the notation, note that when $r=\infty$, the product $F^{-j}(x)\cdot y$
actually makes sense inside $W(A)[p^{-1}]$, and agrees with the operation
as defined above. 
\begin{defn}
\label{def:Elementary-Witt-Ring} Let $m\in\mathbb{Z}$, and $r\geq-m$.
An element 
\[
\varphi\in\text{Hom}_{W(k)}^{V}(W_{r+1}(B),W_{r+1}(A))
\]
is an elementary Witt-differential operator of level $\leq m$ if
there is a canonical Hasse-Schmidt derivation $D=(\tilde{D}_{0},\dots,\tilde{D}_{i})$
of length $0<i\leq p^{r+m}$ on $W_{r+1}(A)$ so that $\varphi=F^{\text{val}(i)-r}(a)\cdot\tilde{D}_{i}$
for some $a\in W_{r+1}(A)$ (this is a well-defined operator by \prettyref{prop:Move-Down}). 

We define $\mathcal{E}W_{r+1}^{(m)}(B;A)$ to be the $W_{r+1}(A)$-submodule
of \\
$\text{Hom}_{W(k)}^{V}(W_{r+1}(B),W_{r+1}(A))$ generated by the elementary
Witt-differential operators of level $\leq m$. By \prettyref{prop:Move-Down},
we have that the map 
\[
R:\mathcal{E}W_{r+1}^{(m)}(B;A)\to\text{Hom}_{W(k)}^{V}(W_{r}(B),W_{r}(A))
\]
has image contained in $\mathcal{E}W_{r}(A)$; thus we may define
\[
\mathcal{E}W^{(m)}(B;A):=\lim_{r}\mathcal{E}W_{r+1}^{(m)}(B;A)\subset\text{Hom}_{W(k)}^{V}(W(B),W(A))
\]
When $A=B$ and $\varphi^{\#}$ is the identity map we write $\mathcal{E}W_{r+1}^{(m)}(A)$
and $\mathcal{E}W^{(m)}(A)$ for $\mathcal{E}W_{r+1}^{(m)}(B;A)$
and $\mathcal{E}W^{(m)}(B;A)$, respectively. In this case, The $W(A)$-module
$\mathcal{E}W^{(m)}(A)$ is naturally filtered by 
\[
\mathcal{E}W^{(m),r}(A)=\{\varphi\in\mathcal{E}W^{(m)}(A)|\varphi(W(A))\subset V^{r}(W(A))\}
\]
\end{defn}

Now we give the 
\begin{defn}
\label{def:WDO-Affine}The ring $\mathcal{D}_{W(A)}^{(m)}$ of (uncompleted)
Witt-differential operators of level $\leq m$ on $A$ is defined
to be the subalgebra of continuous $W(k)$-linear endomorphisms of
$W(A)$ generated by multiplication by elements of $W(A)$ and the
elements of $\mathcal{E}W_{A}^{(m)}$. This ring possesses a filtration
by two sided ideals 
\[
I^{(r)}:=(V^{r}(W(A)),\mathcal{E}W^{(m),r}(A))
\]
and we define $\mathcal{D}_{W_{r}(A)}^{(m)}:=\mathcal{D}_{W(A)}^{(m)}/I^{(r)}$. 

We define $\widehat{\mathcal{D}}_{W(A)}^{(m)}$, the ring of Witt-differential
operators on $A$, of level $\leq m$, to be the completion of $\mathcal{D}_{W(A)}^{(m)}$
along $\{I^{(r)}\}$; i.e., the inverse limit of the $D_{W_{r}(A)}^{(m)}$.
Since each element of $I^{(r)}$ takes $W(A)$ to $V^{r}(W(A))$,
this ring also acts on $W(A)$.

Finally, we let $\widehat{\mathcal{D}}_{W(A)}$ be the inductive limit
of $\widehat{\mathcal{D}}_{W(A)}^{(m)}$ under the obvious maps $\widehat{\mathcal{D}}_{W(A)}^{(m)}\to\widehat{\mathcal{D}}_{W(A)}^{(m+1)}$;
similarly $\mathcal{D}_{W(A)}$ is the inductive limit of $\mathcal{D}_{W(A)}^{(m)}$. 
\end{defn}

This definition is inspired by Berthelot's definition of arithmetic
differential operators of level $\leq m$ (c.f. \cite{key-1},\cite{key-2},
as well as \cite{key-16} for a more operator-theoretic point of view)
and, as we shall see below, yields closely related categories of modules,
at least when $m\geq0$. 

\subsection{\label{subsec:Local-Coordinates}Local Coordinates }

In this subsection we shall give a presentation of the modules $\mathcal{E}W^{(m)}(B;A)$
and the algebra $\widehat{\mathcal{D}}_{W(A)}^{(m)}$ for $m\in\mathbb{Z}$
in terms of local coordinates on $B$ (or $A$ in the latter case)
and deduce some basic structural results. Therefore, throughout this
subsection, we shall suppose that that there is an etale map $k[T_{1},\dots,T_{n}]\to B$.
This map yields Hasse-Schmidt derivations $B\to A$, which we shall
denote by $\partial_{l}^{[i]}$; these are defined on $k[T_{1},\dots T_{n}]$
as 
\[
\partial_{l}^{[i]}(T_{l}^{j})={j \choose i}\varphi^{\#}(T_{l}^{j-i})
\]
and 
\[
\partial_{l}^{[i]}(T_{l'}^{j})=0
\]
for $l\neq l'$ and $i>0$. By \cite{key-22} these operators extend
uniquely to Hasse-Schmidt derivations on $B$. When $B=A$ and $\varphi^{\#}=Id$
these Hasse-Schmidt derivations are iterative; i.e, we have
\[
\partial_{l}^{[i]}\partial_{l}^{[j]}={i+j \choose i}\partial_{l}^{[i+j]}
\]
Further, one sees immediately from the definition of these operators
that they all mutually commute; i.e., 
\[
\partial_{l}^{[i]}\partial_{l'}^{[j]}=\partial_{l'}^{[j]}\partial_{l}^{[i]}
\]
for all $l,l',i,j$. 

If $I=(i_{1},\dots,i_{n})$ is a multi-index, we shall also use the
notation $\partial^{[I]}=\partial_{1}^{[i_{1}]}\cdots\partial_{n}^{[i_{n}]}$.
If $J$ is another multi-index, then we also have $\partial^{J}=\partial_{1}^{j_{1}}\cdots\partial_{n}^{j_{n}}$
and $(\partial^{[I]})^{J}=(\partial_{1}^{[i_{1}]})^{j_{1}}\cdots(\partial_{n}^{[i_{n}]})^{j_{n}}$. 

Now we give the 
\begin{defn}
\label{def:basic-ops}Let $\alpha\in W(A)$ and $r,j\in\mathbb{Z}$
such that $j>0$. We define the operator $\{\partial_{l}\}_{jp^{r}}$
as follows: it is equal to $\varphi^{\#}(\partial_{l}^{[jp^{r'+r}]}):W_{r'+1}(B)\to W_{r'+1}(A)$
where $\partial_{l}^{[jp^{r'+r}]}$ is the canonical lift of the HS
derivation $\partial_{l}^{[jp^{r'+r}]}$ to $W_{r'+1}(B)$. The inverse
limit of these operators is well defined on $W(B)$ by \prettyref{prop:Move-Down}. 

Now suppose $r>0$ and let $v_{j}:=\text{val}(j)$. Then we define
the operator 
\[
F^{v_{j}-r}(\alpha)\cdot\{\partial_{l}\}_{j/p^{r}}
\]
as the operator $F^{v_{j}-r}(\alpha)\cdot\varphi^{\#}(\partial_{l}^{[jp^{r'-r}]})$
from $W_{r'+1}(B)\to W_{r'+1}(A)$; here, we denote by $\partial_{l}^{[jp^{r'-r}]}$
the canonical lift of $\partial_{l}^{[jp^{r'-r}]}$ to $W_{r'+1}(B)$
for any $r'$ such that $r'\geq r-\text{val}(j)$, for $r'<r-\text{val}(j)$
this operator is taken to be $0$. By \prettyref{prop:Move-Down},
we have 
\[
\partial_{l}^{[jp^{r'-r}]}(W_{r'+1}(B))\subset V^{r-v_{j}}(W_{r'+1}(A))
\]
so that this operator is well-defined on $W_{r'+1}(B)$. 

Further, we note that by \prettyref{prop:Move-Down}, $a)$, we have;
for $R:W_{r'+1}(A)\to W_{r'}(A)$
\[
R(F^{v_{j}}(\alpha)\partial_{l}^{[jp^{r'-r}]})=F^{v_{j}}(\alpha)\partial_{l}^{[jp^{r'-1-r}]}R
\]
so that the inverse limit of these operators is in fact well-defined
on $W(B)$. We note that when $1\leq j\leq p^{r}$, we have $F^{v_{j}-r}(\alpha)\cdot\{\partial_{l}\}_{j/p^{r}}\in\mathcal{E}W(A)$.
We may extend this definition to the case $j=0$ by defining $v_{0}=r$;
then we have 
\[
F^{v_{j}-r}(\alpha)\cdot\{\partial_{l}\}_{j/p^{r}}=\alpha
\]
for $j=0$ and any $r\geq0$. 

If $J=(j_{1},\dots,j_{n})$ denotes a multi-index in $\mathbb{N}$
for each $j_{l}$, let $v_{J}=\text{min}_{1\leq l\leq n}\text{val}\{j_{l}\}$.
Suppose this minimum is realized at index $l'$. Then we define
\[
F^{v_{J}-r}(\alpha)\cdot\{\partial\}_{J/p^{r}}:=F^{v_{J}-r}(\alpha)\prod_{l=1}^{n}\{\partial_{l}\}_{j_{l}/p^{r}}=F^{v_{l'}-r}(\alpha)\{\partial_{l'}\}_{j_{l'}/p^{r}}\cdot\prod_{l=1,l\neq l'}^{n}\{\partial_{l}\}_{j_{l}/p^{r}}
\]
and we note that this product is in $\mathcal{D}_{W(A)}$ whenever
each $j_{i}\leq p^{r}$. If $J=(0,\dots,0)$ then we set $v_{J}=r$,
this gives ${\displaystyle F^{v_{J}-r}(\alpha)\cdot\{\partial\}_{J/p^{r}}=\alpha}$
as above. 
\end{defn}

We shall show that (products of) such operators form a basis $\widehat{\mathcal{D}}_{W(A)}^{(0)}$,
in a suitable sense. Namely, 
\begin{thm}
\label{thm:Basis}Let $m\in\mathbb{Z}$. Let $Q\in\widehat{\mathcal{D}}_{W(A)}^{(m)}$.
Then we may write 
\[
Q=\sum_{I}(\sum_{r=1}^{\infty}\sum_{J_{I}}F^{-r}(\alpha_{J_{I}})\{\partial\}_{J_{I}/p^{r}}+\sum_{K_{I}}\alpha_{K_{I}}\{\partial\}_{K_{I}})\{\partial\}_{p^{m}}^{I}
\]
where $\alpha_{J_{I}},\alpha_{K_{I}}\in W(A)$, $I=(i_{1},\dots,i_{n})\in\mathbb{N}^{n}$,
$\{\partial\}_{p^{m}}^{I}:=\{\partial_{1}\}_{p^{m}}^{i_{1}}\cdots\{\partial_{n}\}_{p^{m}}^{i_{n}}$,
$J_{I}=(j_{1},\dots,j_{n})$ satisfies $0\leq j_{l}<p^{r}$ for each
$j_{l}$ and $(p,j_{l})=1$ for at least one $l$; unless $r=1$ where
we allow $J_{I}=0$, and $K_{I}=(k_{1},\dots,k_{n})$ satisfies $0\leq k_{i}<p^{m}$
with at least one $k_{i}\neq0$ (if $m\leq0$ this sum is taken to
be empty). We demand that the terms 
\[
\sum_{r=1}^{\infty}\sum_{J_{I}}F^{-r}(\alpha_{J_{I}})\{\partial\}_{J_{I}/p^{r}}+\sum_{K_{I}}\alpha_{K_{I}}\{\partial\}_{K_{I}}
\]
approach $0$ (in the topology of $\widehat{\mathcal{D}}_{W(A)}^{(m)}$)
as $I\to\infty$. Furthermore, the elements $\alpha_{I_{J}}$ are
unique. 
\end{thm}

We shall break the proof of this up into several steps, beginning
by showing that we can write the elementary Witt-differential operators
(of level $0$) in the required form. 
\begin{lem}
\label{lem:Reps-of-HS}Let $\tilde{D}$ be a canonical Hasse-Schmidt
derivation of length $\leq p^{r}$ from $W_{r+1}(B)$ to $W_{r+1}(A)$.
Then, for any $0<j\leq p^{r}$ the associated canonical Hasse-Schmidt
derivation $\tilde{D}_{j}$ on $W_{r+1}(A)$ is equal to 
\[
\sum_{I}F^{v_{I}-r}(\beta_{I})\cdot\{\partial\}_{I/p^{r}}
\]
where $I=(i_{1},\dots,i_{n})\backslash\{(0,\cdots,0)\in\mathbb{N}^{n}$
satisfies ${\displaystyle \sum_{l=1}^{n}i_{l}\leq p^{r}}$, and $\beta_{I}\in W_{r+1}(A)$.
Furthermore, if $v_{I}\geq\text{val}(j)$, then we have $\beta_{I}\in V^{v_{I}-\text{val}(j)}(W_{r+1}(A))$. 
\end{lem}

\begin{proof}
We recall that the operator $\tilde{D}_{j}$ is obtained by choosing
any lift of the Hasse-Schmidt derivation $D$ to a Hasse-Schmidt derivation
$\mathcal{B}_{r+1}\to\mathcal{A}_{r+1}$ and restricting to $W_{r+1}(B^{(r)})$.
If we write 
\[
D_{j}=\sum_{I}\alpha_{I}\partial^{[I]}
\]
for $\alpha_{I}\in A$ (and $\partial^{[I]}=\varphi^{\#}(\partial_{1}^{[i_{1}]}\cdots\partial_{n}^{[i_{n}]})$)
then since $D_{j}$ is a differential operator of order $j\leq p^{r}$
we obtain $|I|\leq p^{r}$ for all $I$ such that $\alpha_{I}\neq0$. 

Now, write
\[
\tilde{D}_{j}=\sum_{I}\tilde{\alpha}_{I}\partial^{[I]}
\]
for $\tilde{\alpha}_{I}\in A_{r+1}$ lifting $\alpha_{I}$. To prove
the lemma, we must show that, after possibly altering $\tilde{\alpha}_{I}$
in a way that does not change the action on $W_{r+1}(A^{(r)})$, we
have $F^{r-v_{I}}(\tilde{\alpha}_{I})\in W_{r+1}(A^{(r)})$ (where
$F:\mathcal{A}_{r+1}\to\mathcal{A}_{r+1}$ is a lift of Frobenius);
and that $\tilde{\alpha}_{I}\in p^{v_{I}-\text{val(j)}}\mathcal{A}_{r+1}$
if $v_{I}>\text{val}(j)$.

Let $r\geq0$. We'll first show that $\tilde{\alpha}_{I}\in p^{v_{I}-\text{val(j)}}\mathcal{A}_{r+1}$
if $v_{I}>\text{val}(j)$. If this does not hold, let $I_{0}$ be
the least element (in the lexicographic ordering) amongst $I$ such
that $\tilde{\alpha}_{I}\notin p^{v_{I}-\text{val(j)}}\mathcal{A}_{r+1}$
and $v_{I}>\text{val}(j)$. 

Then we have $\partial^{[I_{0}]}(p^{r-v_{I_{0}}}T^{I_{0}})=p^{r-v_{I_{0}}}$,
while $\partial^{[J]}(p^{r-v_{I}}T^{I_{0}})=0$ for all other $J$
with $\tilde{\alpha}_{J}\notin p^{v_{J}-\text{val(j)}}\mathcal{A}_{r+1}$
and $v_{J}>\text{val}(j)$. For $K$ such that $v_{K}\leq\text{val}(j)$
we have $\partial^{[K]}(p^{r-v_{I}}T^{I_{0}})\subset V^{r-v_{K}}(W_{r+1}(A^{(r)}))\subset V^{r-\text{val}(j)}(W_{r+1}(A^{(r)}))$
(by \prettyref{prop:Move-Down} applied to the HS derivations $\partial_{m}^{[i]}$).
So, in order for $\tilde{D}_{j}(p^{r-v_{I_{0}}}T^{I_{0}})\subset V^{r-\text{val}(j)}(W_{r+1}(A^{(r)}))$
we must have 
\[
p^{r-v_{I_{0}}}\tilde{\alpha}_{I_{0}}\in p^{r-\text{val}(j)}\mathcal{A}_{r+1}
\]
which yields $\tilde{\alpha}_{I_{0}}\in p^{v_{I_{0}}-\text{val}(j)}\mathcal{A}_{r+1}$
after all. 

Now we check that, after possibly altering $\tilde{\alpha}_{I}$ in
a way that does not change the action on $W_{r+1}(A^{(r)})$, we have
$F^{r-v_{I}}(\tilde{\alpha}_{I})\in W_{r+1}(A^{(r)})$ for each $I$
such that $v_{I}\leq r$. Note that every term of the form $\tilde{\alpha}_{I}\partial^{[I]}$,
for which $F^{r-v_{I}}(\tilde{\alpha}_{I})\in W_{r+1}(A^{(r)})$,
takes $W_{r+1}(B^{(r)})$ to $W_{r+1}(A^{(r)})$ (as explained above
in \prettyref{def:basic-ops}). So let $I_{0}$ be the least element
(in the lexicographic ordering) amongst $I$ such that $F^{r-v_{I}}(\tilde{\alpha}_{I})\notin W_{r+1}(A^{(r)})$.
Then 
\[
\tilde{D}_{j}(p^{r-v_{I_{0}}}T^{I_{0}})=\sum_{I}\tilde{\alpha}_{I}\partial^{[I]}(p^{r-v_{I_{0}}}T^{I_{0}})
\]
\[
=\sum_{\{I|F^{r-v_{I}}(\tilde{\alpha}_{I})\in W_{r+1}(A^{(r)})\}}\tilde{\alpha}_{I}\partial^{[I]}(p^{r-v_{I_{0}}}T^{I_{0}})+\sum_{\{I|F^{r-v_{I}}(\tilde{\alpha}_{I})\notin W_{r+1}(A^{(r)})\}}\tilde{\alpha}_{I}\partial^{[I]}(p^{r-v_{I_{0}}}T^{I_{0}})
\]
\[
=\beta+p^{r-v_{I_{0}}}\tilde{\alpha}_{I_{0}}
\]
where $\beta\in W_{r+1}(A^{(r)})$; the last equality comes from noting
that each $\tilde{\alpha}_{I}\partial^{[I]}(p^{r-v_{I_{0}}}T^{I_{0}})$
in the first sum is contained in $W_{r+1}(A^{(r)})$, and in the second
sum each term is zero except the one coming from $I_{0}$. Since $\tilde{D}_{j}$
itself takes $W_{r+1}(B^{(r)})$ to $W_{r+1}(A^{(r)})$, we see that
$\beta+p^{r-v_{I_{0}}}\tilde{\alpha}_{I_{0}}\in W_{r+1}(A^{(r)})$,
which forces $p^{r-v_{I_{0}}}\tilde{\alpha}_{I_{0}}\in W_{r+1}(A^{(r)})$.
Writing 
\[
p^{r-v_{I_{0}}}\tilde{\alpha}_{I_{0}}=\sum_{i=r-v_{I_{0}}}^{r}p^{i}a^{p^{r-i}}
\]
we obtain 
\[
\tilde{\alpha}_{I_{0}}=\sum_{i=0}^{v_{I_{0}}}p^{i}a^{p^{v_{I_{0}}-i}}+p^{v_{I_{0}}+1}\gamma
\]
for some $\gamma\in\mathcal{A}_{r+1}$. Therefore
\[
F^{r-v_{I_{0}}}(\tilde{\alpha}_{I_{0}})=\sum_{i=0}^{v_{I_{0}}}p^{i}a^{p^{r-i}}+p^{v_{I_{0}}+1}F^{r-v_{I_{0}}}(\gamma)
\]
And since $\partial^{[I_{0}]}(W_{r+1}(B^{(r)}))\subset V^{r-v_{I_{0}}}(W_{r+1}(A^{(r)}))$,
we see that $p^{v_{I_{0}}+1}\gamma\partial^{[I_{0}]}$ annihilates
$W_{r+1}(B^{(r)})$. Therefore, altering $\tilde{\alpha}_{I_{0}}$
to $\tilde{\alpha}_{I_{0}}-p^{v_{I_{0}}+1}\gamma$, we have that $F^{r-v_{I_{0}}}(\tilde{\alpha}_{I_{0}})\in W_{r+1}(A^{(r)})$.
Continuing up the lexicographic ordering in this way, we obtain the
result. 
\end{proof}
Next we have 
\begin{lem}
\label{lem:Reps-of-phi}Every element $\phi\in\mathcal{E}W_{r+1}^{(0)}(B;A)$
admits a representation 
\[
\phi=\sum_{I}F^{v_{I}-r}(\alpha_{I})\{\partial\}_{I/p^{r}}
\]
where $\alpha_{I}\in W_{r+1}(A)$, $I=(i_{1},\dots,i_{n})\in\mathbb{N}^{n}\backslash\{(0,\dots,0)\}$,
with ${\displaystyle \sum_{l=1}^{n}i_{l}\leq p^{r}}$; and, as above,
$v_{I}=\text{min}\{\text{val}(i_{l})\}$. Furthermore, 
\[
\sum_{I}F^{v_{I}-r}(\alpha_{I})\{\partial\}_{I/p^{r}}=\sum_{I}F^{v_{I}-r}(\alpha'_{I})\{\partial\}_{I/p^{r}}
\]
iff $\alpha'_{I}-\alpha_{I}\in V^{v_{I}+1}(W_{r+1}(A))$ for all $I$. 
\end{lem}

\begin{proof}
In the previous lemma we showed that such a representation exists
for $\phi$ if $\phi$ is itself a canonical Hasse-Schmidt derivation;
more precisely, if $\phi=\tilde{D}_{j}$ for some $j\leq p^{r}$ then
\[
\phi=\sum_{I}F^{v_{I}-r}(\beta_{I})\{\partial\}_{I/p^{r}}
\]
where the sum ranges over $I$ such that ${\displaystyle \sum_{l=1}^{n}i_{l}\leq j}$,
and we have $\beta_{I}\in V^{v_{I}-\text{val}(j)}(W_{r+1}(A^{(r)}))$
whenever $v_{I}>\text{val}(j)$. We must show that a similar representation
exists for operators of the form $F^{\text{val}(j)-r}(\alpha)\cdot\phi$
for some $\alpha\in W_{r+1}(A)$. We write 
\[
F^{\text{val}(j)-r}(\alpha)\cdot\phi
\]
\[
=\sum_{\{I|v_{I}\leq\text{val}(j)\}}F^{\text{val}(j)-r}(\alpha)F^{v_{I}-r}(\beta_{I})\{\partial\}_{I/p^{r}}+\sum_{\{I|v_{I}>\text{val}(j)\}}F^{\text{val}(j)-r}(\alpha)F^{v_{I}-r}(\beta_{I})\{\partial\}_{I/p^{r}}
\]
\[
=\sum_{\{I|v_{I}\leq\text{val}(j)\}}F^{\text{val}(j)-v_{I}}(\alpha)\cdot F^{v_{I}-r}(\alpha\beta_{I})\{\partial\}_{I/p^{r}}+\sum_{\{I|v_{I}>\text{val}(j)\}}F^{\text{val}(j)-r}(\alpha)F^{v_{I}-r}(\beta_{I})\{\partial\}_{I/p^{r}}
\]
Now, the first sum is already in the required form since we may write
\[
F^{\text{val}(j)-v_{I}}(\alpha)\cdot F^{v_{I}-r}(\alpha\beta_{I})=F^{v_{I}-r}(F^{r-\text{val}(j)-2v_{I}}(\alpha)\cdot\beta_{I})
\]
and $r-\text{val}(j)-2v_{I}=(r-v_{I})+(\text{val}(j)-v_{I})\geq0$. 

On the other hand, for $I$ such that $v_{I}>\text{val}(j)$, we may,
by the previous lemma, write $\beta_{I}=V^{v_{I}-\text{val}(j)}(\beta'_{I})$
so that 
\[
\sum_{\{I|v_{I}>\text{val}(j)\}}F^{\text{val}(j)-r}(\alpha)F^{v_{I}-r}(\beta_{I})\{\partial\}_{I/p^{r}}
\]
\[
=\sum_{\{I|v_{I}>\text{val}(j)\}}F^{\text{val}(j)-r}(\alpha)F^{v_{I}-r}(V^{v_{I}-\text{val}(j)}(\beta'_{I}))\{\partial\}_{I/p^{r}}
\]
\[
=\sum_{\{I|v_{I}>\text{val}(j)\}}F^{v_{I}-r}(V^{v_{I}-\text{val}(j)}(\alpha\beta'_{I}))\{\partial\}_{I/p^{r}}
\]
which is a sum of the required form; therefore, the claimed representation
exists. 

Now we deal with the uniqueness. First note that, if $\alpha_{I}'-\alpha_{I}\in V^{v_{I}+1}(W_{r+1}(A))$,
then, since the operator $\{\partial\}_{I/p^{r}}$ takes $W_{r+1}(B)$
into $V^{r-v_{I}}(W_{r+1}(A))$ (by \prettyref{prop:Move-Down}),
we have that $F^{v_{I}-r}(\alpha'_{I}-\alpha_{I})\cdot\{\partial\}_{I/p^{r}}$
acts as $0$. So we must show the other implication; namely, that
if an operator 
\[
\sum_{I}F^{v_{I}-r}(\beta_{I})\{\partial\}_{I/p^{r}}
\]
acts as $0$, then $\beta_{I}\in V^{v_{I}+1}(W_{r+1}(A))$. This is
equivalent to the assertion that, if a differential operator
\[
\sum_{I}\tilde{\beta}_{I}\partial^{[I]}
\]
(with $F^{r-v_{I}}(\tilde{\beta}_{I})\in W_{r+1}(A^{(r)})$) is zero
on $W_{r+1}(B^{(r)})$, then 
\[
F^{r-v_{I}}(\tilde{\beta}_{I})\in V^{v_{I}+1}(W_{r+1}(A^{(r)}))
\]
for all $I$; we shall in fact show that $\tilde{\beta}_{I}\in p^{v_{I}+1}A_{r+1}$,
which immediately implies the above statement.

Now, for any $I$ we may consider $p^{r-v_{I}}T^{I}\in W_{r+1}(A^{(r)})$.
If $I_{1}$ is the least (in the lexicographic ordering) index such
that $\tilde{\beta}_{I_{1}}\neq0$, we have
\[
0=(\sum_{I}\tilde{\beta}_{I}\partial^{[I]})(p^{r-v_{I_{1}}}T^{I_{1}})=p^{r-v_{I_{1}}}\tilde{\beta}_{I_{1}}
\]
which implies $\tilde{\beta}_{I_{1}}\in p^{v_{I}+1}A_{r+1}$ as claimed.
Further, by the forward implication, this implies that the term $\tilde{\beta}_{I_{1}}\partial^{[I_{1}]}$
acts as zero on $W_{r+1}(B^{(r)})$ , we see that the operator ${\displaystyle \sum_{I>I_{1}}\tilde{\beta}_{I}\partial^{(I)}}$
acts as zero on $W_{r+1}(B^{(r)})$ as well. Thus we may replace ${\displaystyle \sum_{I}\tilde{\beta}_{I}\partial^{(I)}}$
by ${\displaystyle \sum_{I>I_{1}}\tilde{\beta}_{I}\partial^{(I)}}$
and run the argument again; continuing in this way implies the result. 
\end{proof}
From this, we deduce 
\begin{cor}
\label{cor:Reps-of-phi-redux}Every element $\phi\in\mathcal{E}W^{(0)}(B;A)$
may be written as 
\[
\phi=\sum_{r=0}^{\infty}\sum_{I}F^{-r}(\alpha_{I})\{\partial\}_{I/p^{r}}
\]
where $\alpha_{I}\in W(A)$, $I=(i_{1},\dots,i_{n})\in\mathbb{N}^{n}\backslash\{(0,\dots,0)\}$,
with ${\displaystyle \sum_{l=1}^{n}i_{l}\leq p^{r}}$, and $v_{I}=0$
for all $I$. Furthermore, this representation is unique. 
\end{cor}

\begin{proof}
By definition 
\[
\mathcal{E}W^{(0)}(B;A)=\lim_{r}\mathcal{E}W_{r+1}^{(0)}(B;A)
\]
and, if $\phi_{r}$ is the image of $\phi$ in $\mathcal{E}W_{r+1}^{(0)}(B;A)$,
then, by slightly rewriting the previous lemma, we have an expression
\[
\phi_{r}=\sum_{r'=0}^{r}\sum_{I}F^{-r'}(\alpha_{I,r})\{\partial\}_{I/p^{r'}}
\]
with $v_{I}=0$ for all $I$ (in this expression, terms of the form
$F^{-r'}(\alpha_{I,r})\{\partial\}_{I/p^{r'}}$ correspond to terms
with $v_{I}=r-r'$ in the statement of the previous lemma). 

Thus, by the uniqueness statement of the previous lemma, the image
\[
\overline{\alpha_{I,r}}\in W_{r+1}(A)/V^{r-r'+1}(W_{r+1}(A))
\]
is well-defined. So, fixing an index $I/p^{r'}$, we obtain an element
\[
\alpha_{I}:=\lim_{r}\overline{\alpha}_{I,r}\in\lim_{r}W_{r+1}(A)/V^{r-r'+1}(W_{r+1}(A))\tilde{=}W(A)
\]
so that the action of ${\displaystyle \sum_{r=0}^{\infty}\sum_{I}F^{v_{I}-r}(\alpha_{I})\{\partial\}_{I/p^{r}}}$
agrees with $\phi$ on $W(A)$. 

To see that this representation is unique, suppose that a term of
the form ${\displaystyle \sum_{r=0}^{\infty}\sum_{I}F^{v_{I}-r}(\alpha_{I})\{\partial\}_{I/p^{r}}}$
acts as zero on $W(A)$. Consider a term $F^{v_{I}-r'}(\alpha_{I})\{\partial\}_{I/p^{r'}}$
for some index $r'$. Then, for each $r\geq r'$, the uniqueness statement
of the previous lemma implies that $\alpha_{I}\in V^{r-r'+1}(W_{r+1}(A))$.
Since this is true for all $r$, we see $\alpha_{I}=0$ as claimed. 
\end{proof}
Now we can deduce from this the analogous presentation for $\mathcal{E}W^{(m)}(B;A)$: 
\begin{cor}
\label{cor:Reps-of-phi-for-m}Let $m\in\mathbb{Z}$. Every element
$\phi\in\mathcal{E}W^{(m)}(B;A)$ may be written as 
\[
\phi=\sum_{r=0}^{\infty}\sum_{I}F^{-r}(\alpha_{I})\{\partial\}_{I/p^{r}}+\sum_{J}\alpha_{J}\{\partial\}_{J}
\]
where, in the first sum, $\alpha_{I}\in W(A)$, $I=(i_{1},\dots,i_{n})\in\mathbb{N}^{n}\backslash\{(0,\dots,0)\}$,
with ${\displaystyle \sum_{l=1}^{n}i_{l}\leq p^{r}}$, and $v_{I}=0$
for all $I$, and, in the second sum, we have $\alpha_{J}\in W(A)$
and ${\displaystyle \sum_{l=1}^{n}j_{l}\leq p^{m}}$. Furthermore,
this representation is unique. 
\end{cor}

\begin{proof}
By \prettyref{prop:HS-and-F} we have $\tilde{D}_{p^{m}j}F^{m}=F^{m}\tilde{D}_{j}$
for any canonical Hasse-Schmidt derivation from $W_{r+1}(B)$ to $W_{r+1}(A)$.
Therefore if $\phi$ is a canonical Hasse-Schmidt derivation of length
$\leq p^{m}$, we have $\phi F^{m}=F^{m}\psi$ where $\psi\in\mathcal{E}W_{r+1}^{(0)}(B;A)$.
So, by \prettyref{lem:Reps-of-phi} we have 
\[
\phi\circ F^{m}=F^{m}(\sum_{I}F^{v_{I}-r}(\alpha_{I})\{\partial\}_{I/p^{r}})=\sum_{I}F^{m+v_{I}-r}(\alpha_{I})\{\partial\}_{p^{m}I/p^{r}}\circ F^{m}
\]
where the notation is as in loc cit. So for an arbitrary element of
the form ${\displaystyle \sum_{i}F^{m_{i}}(\alpha_{i})\phi_{i}\in}\mathcal{E}W_{r+1}^{(m)}(B;A)$
(where $\phi_{i}$ is the $i$th component of a canonical Hasse-Schmidt
derivation, and $m_{i}=\text{min}\{0,\text{val}(i)-r\}$) we have
\[
\sum_{i}F^{m_{i}}(\alpha_{i})\phi_{i}\circ F^{m}=\sum_{i}\sum_{I}F^{m_{i}}(\alpha_{i})\cdot F^{m+v_{I}-r}(\alpha_{I,i})\{\partial\}_{p^{m}I/p^{r}}\circ F^{m}
\]
As in the proof of \prettyref{lem:Reps-of-phi}, taking the inverse
limit and re-indexing yields an expression 
\[
\phi\circ F^{m}=(\sum_{r=0}^{\infty}\sum_{I}F^{-r}(\alpha_{I})\{\partial\}_{I/p^{r}}+\sum_{J}\alpha_{J}\{\partial\}_{J})\circ F^{m}
\]
for $\phi\in\mathcal{E}W^{(m)}(B;A)$; this easily implies the result
(as $F^{m}$ is an isomorphism on $W(A)[p^{-1}]$). 
\end{proof}
Now we turn to proving the full statement of \prettyref{thm:Basis}.
To do this, we need to also analyze products. We start with the elementary 
\begin{lem}
\label{lem:products}Inside $\mathcal{D}_{W(A)}$ for any $i\geq1$,
$r\geq0$, we have 
\[
\{\partial_{l}\}_{1/p^{r}}^{i}=u\cdot i!\{\partial_{l}\}_{i/p^{r}}
\]
where $u\in\mathbb{Z}_{p}$ is a unit (which depends on $r$ and $i$).
In particular, for $I=(i_{1},\dots,i_{n})$ with each $i_{j}<p^{r}$
we have 
\[
\{\partial\}_{I/p^{r}}=u'\prod_{l,r'}\{\partial_{l}\}_{1/p^{r'}}^{i_{r'}}
\]
where $r'\in\{1\dots,r-1\}$, $0\leq i_{r'}<p$, and $u'$ is a unit
in $\mathbb{Z}_{p}$. 
\end{lem}

\begin{proof}
Let $r'\geq r$. Then $\{\partial_{l}\}_{i/p^{r}}$ acts on $W_{r'+1}(A)$
via the action of $\partial_{l}^{[ip^{r'-r}]}$ on $W_{r+1}(A^{(r)})\subset\mathcal{A}_{r'+1}$.
But we have 
\[
(\partial_{l}^{[p^{r'-r}]})^{i}=\prod_{j=1}^{i}{jp^{r'-r} \choose p^{r'-r}}\partial_{l}^{[ip^{r'-r}]}
\]
(this follows easily from the fact that $\{\partial_{l}^{[i]}\}$
form an iterative HS derivation). Now, we have that 
\[
{jp^{r'-r} \choose p^{r'-r}}=\prod_{m=0}^{p^{r'-r}-1}\frac{(jp^{r'-r}-m)}{(p^{r'-r}-m)}=j\cdot\prod_{m=1}^{p^{r'-r}-1}\frac{(jp^{r'-r}-m)}{(p^{r'-r}-m)}
\]
For each $m\in\{1,\dots,p^{r'-r}-1\}$, we have $\text{val}(p^{r'-r}-m)=\text{val}(m)=\text{val}(jp^{r'-r}-m)$.
Therefore 
\[
\text{val}{jp^{r'-r} \choose p^{r'-r}}=\text{val}(j)
\]
and we see that 
\[
(\partial_{l}^{[p^{r'-r}]})^{i}=u_{r'}\cdot i!\partial_{l}^{[ip^{r'-r}]}
\]
where $u_{r'}$ is a unit in $\mathbb{Z}/p^{r+1}$. The fact that
$u_{r'+1}\equiv u_{r'}$ mod $p^{r'}$ follows from the claim directly
below, therefore we may set ${\displaystyle u:=\lim_{r'}u_{r'}\in\mathbb{Z}_{p}}$
to prove the result.
\end{proof}
To finish the proof, we need to show
\begin{claim}
\label{claim:basic-id}We have the identity 
\begin{equation}
{lp^{r} \choose p^{j}}\equiv{lp^{r-1} \choose p^{j-1}}\phantom{q}\text{mod}\phantom{q}p^{r}\label{eq:basic-id-1}
\end{equation}
for all natural numbers $l$, $r$, and $j$. 
\end{claim}

\begin{proof}
This is not difficult to check directly. However, since ${\displaystyle d^{([p]}(T^{lp^{r}})={lp^{r} \choose p^{j}}}T^{lp^{r}-p^{j}}$
inside $\mathbb{Z}/p^{r+1}[T]$, and ${\displaystyle d^{[p^{j-1}]}(T^{lp^{r-1}})={lp^{r-1} \choose p^{j-1}}}T^{lp^{r-1}-p^{j-1}}$
inside $\mathbb{Z}/p^{r}[T]$, this equality also follows from $R\circ d^{[p^{j}]}=d^{[p^{j-1}]}\circ R$,
which is \prettyref{prop:Move-Down}. 
\end{proof}
For the course of the next few lemmas, we write $\widehat{\mathcal{D}}_{W(A),b}^{(m)}$
for the set of elements in $\widehat{\mathcal{D}}_{W(A)}^{(m)}$ which
have a representation as in \prettyref{thm:Basis}. If 
\[
Q=\sum_{I}(\sum_{r=1}^{\infty}\sum_{J_{I}}F^{-r}(\alpha_{J_{I}})\{\partial\}_{J_{I}/p^{r}}+\sum_{K_{I}}\alpha_{K_{I}}\{\partial\}_{K_{I}})\{\partial\}_{p^{m}}^{I}
\]
is an element of $\widehat{\mathcal{D}}_{W(A),b}^{(m)}$, we say that
$Q\in\widehat{\mathcal{D}}_{W(A),b,i}^{(m)}$ if $Q$ admits a representation
as above so that for each $r<i$ and each associated $J_{I}$, we
have $\alpha_{J_{I}}\in V^{i-r}(W(A))$; we demand also that $\alpha_{K_{I}}\in V^{i}(W(A))$
for all $K_{I}$. Note that if we have for each $s$ an element $\phi_{s}\in\widehat{D}_{W(A),b,s}$
then 
\[
\sum_{s=i}^{\infty}\phi_{s}\in\widehat{\mathcal{D}}_{W(A),b,i}^{(m)}
\]
for all $i\geq0$. 
\begin{lem}
\label{lem:big-products}Fix $m\in\mathbb{Z}$. Suppose $\phi\in\mathcal{E}W^{(m)}(A)$,
and $Q\in\widehat{\mathcal{D}}_{W(A),b}^{(m)}$. Then $\phi\cdot Q\in\widehat{\mathcal{D}}_{W(A),b}^{(m)}$.
Further, if $\phi\in\mathcal{E}W^{m,(r)}(A)$, then $\phi\cdot Q\in\widehat{\mathcal{D}}_{W(A),b.r}^{(m)}$.
If $Q\in\widehat{\mathcal{D}}_{W(A),b,r}^{(m)}$, then $\phi\cdot Q\in\widehat{\mathcal{D}}_{W(A),b.r}^{(m)}$
for any $\phi$. Similarly, $a\cdot Q\in\widehat{\mathcal{D}}_{W(A),b,r}^{(m)}$
for any $Q$ if $a\in V^{r}(W(A))$. 
\end{lem}

\begin{proof}
By \prettyref{cor:Reps-of-phi-for-m}, if $\phi\in\mathcal{E}W^{(m)}(A)$
we can write 
\[
\phi=\sum_{r=0}^{\infty}\sum_{I}F^{-r}(\alpha_{I})\{\partial\}_{I/p^{r}}+\sum_{J}\alpha_{J}\{\partial\}_{J}\in\widehat{\mathcal{D}}_{W(A),b}^{(m)}
\]
Therefore we must compute
\[
\phi\cdot Q=
\]
 
\begin{equation}
(\sum_{r=0}^{\infty}\sum_{I}F^{-r}(\alpha_{I})\{\partial\}_{I/p^{r}}+\sum_{J}\alpha_{J}\{\partial\}_{J})(\sum_{I}(\sum_{r=1}^{\infty}\sum_{J_{I}}F^{-r}(\alpha_{J_{I}})\{\partial\}_{J_{I}/p^{r}}+\sum_{K_{I}}\alpha_{K_{I}}\{\partial\}_{K_{I}})\{\partial\}_{p^{m}}^{I})\label{eq:phi-times-Q}
\end{equation}
In order to compute this term, we first note that if $D=(D_{0},\dots,D_{j})$
is any Hasse-Schmidt derivation (on an arbitrary ring $R$), we have
an equality of operators
\[
[D_{j},f]=\sum_{i=1}^{j}D_{i}(f)D_{j-i}=\sum_{i=0}^{j-1}D_{j-i}(f)D_{i}
\]
(where $f\in R$ denotes the operator $g\to fg$). Now choose any
$j$ so that $j\leq p^{r+m}$. We have, for any $r'\geq r-\text{val}(j)$
and $\alpha\in W_{r'+1}(A)$, 
\[
[\partial_{l}^{[jp^{r'-r}]},\alpha]=\sum_{i=0}^{jp^{r'-r}-1}\partial_{l}^{[jp^{r'-r}-i]}(\alpha)\partial_{l}^{[i]}
\]
which may be rewritten as 
\[
[\{\partial_{l}\}_{j/p^{r}},\alpha]=\sum_{i=0}^{jp^{r'-r}-1}\{\partial_{l}\}_{(jp^{r'-r}-i)/p^{r'}}(\alpha)\cdot\{\partial_{l}\}_{i/p^{r'}}
\]
\[
=\sum_{i=0}^{jp^{r'-r}-1}\{\partial_{l}\}_{(jp^{r'-r}-i)/p^{r'}}(\alpha)\cdot\{\partial_{l}\}_{i/p^{r'}}
\]
So that we obtain the equality 
\[
[\{\partial_{l}\}_{j/p^{r}},\alpha]=\{\partial_{l}\}_{j/p^{r}}(\alpha)+\sum_{r'\geq r-\text{val}(j)}^{\infty}\sum_{i=1}^{jp^{r'-r}-1}\{\partial_{l}\}_{(jp^{r'-r}-i)/p^{r'}}(\alpha)\cdot\{\partial_{l}\}_{i/p^{r'}}
\]
where on the right hand side we sum over $i$ with $(i,p)=1$. This
is an equality of operators on $W(A)$; with both sides in $\mathcal{D}_{W(A)}^{(m)}$
(in fact the right hand side is in $\mathcal{E}W_{A}$). Thus we see
from this expression that $[\{\partial_{l}\}_{j/p^{r}},\alpha]\in\widehat{\mathcal{D}}_{W(A),b,r}^{(m)}$,
and consequently that $F^{-r}(\alpha)\{\partial\}_{I/p^{r}}\cdot\alpha_{I_{0}}\in\widehat{\mathcal{D}}_{W(A),b,r}^{(m)}$
for any $\alpha,\alpha_{I_{0}}\in W(A)$. 

Next, we consider a product of the form $F^{-r}(\alpha)\{\partial_{l}\}_{j/p^{r}}\cdot F^{-s}(\beta)\{\partial_{t}\}_{j'/p^{s}}$,
where $j'$ is not divisible by $p$, and $s\geq r$ (the case $s<r$
is essentially identical; as is the case where we need to multiply
$\alpha_{J}\{\partial\}_{J}\cdot F^{-s}(\beta)\{\partial_{t}\}_{j'/p^{s}}$).
To compute this term, we may, by extending linearly, regard both $\{\partial_{l}\}_{j/p^{r}}$
and $F^{-s}(\beta)\{\partial_{t}\}_{j'/p^{s}}$ as endomorphisms of
$W(A)[p^{-1}]$. Since $F^{-s}(\beta)\in W(A)[p^{-1}]$, the previous
equality implies an equality
\[
[\{\partial_{l}\}_{j/p^{r}},F^{-s}(\beta)]
\]
\[
=\{\partial_{l}\}_{j/p^{r}}(F^{-s}(\beta))+\sum_{r'\geq r-\text{val}(j)}^{\infty}\sum_{i=1}^{jp^{r'-r}-1}\{\partial_{l}\}_{(jp^{r'-r}-i)/p^{r'}}(F^{-s}(\beta))\cdot\{\partial_{l}\}_{i/p^{r'}}
\]
of operators on $W(A)[p^{-1}]$. Therefore 
\begin{equation}
[\{\partial_{l}\}_{j/p^{r}},F^{-s}(\beta)\{\partial_{m}\}_{j'/p^{s}}]\label{eq:bracket-1}
\end{equation}
\[
=\{\partial_{l}\}_{j/p^{r}}(F^{-s}(\beta))\{\partial_{t}\}_{j'/p^{s}}
\]
\[
+\sum_{r'\geq r-\text{val}(j)}^{\infty}\sum_{i=1}^{jp^{r'-r}-1}\{\partial_{l}\}_{(jp^{r'-r}-i)/p^{r'}}(F^{-s}(\beta))\cdot\{\partial_{l}\}_{i/p^{r'}}\{\partial_{t}\}_{j'/p^{s}}
\]
(using that $\{\partial_{l}\}_{j/p^{r}}$ and $\{\partial_{t}\}_{i/p^{s}}$
commute); as above we suppose $(i,p)=1$ in the right hand sum. 

Now, \prettyref{prop:HS-and-F}, applied to the operators $\{\partial_{l}\}_{j/p^{r}}$,
yields a relation 
\[
\{\partial_{l}\}_{j/p^{r}}F=F\{\partial_{l}\}_{j/p^{r+1}}
\]
which implies 
\[
F^{-1}\{\partial_{l}\}_{j/p^{r}}=\{\partial_{l}\}_{j/p^{r+1}}F^{-1}
\]
in $W(A)[p^{-1}]$. We note that this makes sense for all $j,r\in\mathbb{Z}$
such that $j>0$ and $r\geq0$; in particular we have 
\[
\{\partial_{l}\}_{(j/p^{r}-i/p^{r'})}(F^{-s}(\beta))=F^{-s}(\{\partial_{l}\}_{(jp^{r'-r}-i)p^{s}/p^{r'}}(\beta))
\]
So that 
\[
F^{-r}(\alpha)\{\partial_{l}\}_{j/p^{r}}\cdot F^{-s}(\beta)\{\partial_{t}\}_{j'/p^{s}}=F^{-s}(\beta F^{s-r}(\alpha))\{\partial_{t}\}_{j'/p^{s}}\{\partial_{l}\}_{j/p^{r}}
\]
\[
+F^{-s}(\{\partial_{l}\}_{p^{s}j/p^{r}}(\beta)\cdot F^{s-r}(\alpha))\{\partial_{t}\}_{j'/p^{s}}
\]
\[
+\sum_{r'\geq r-\text{val}(j)}^{\infty}\sum_{i=1}^{jp^{r'-r}-1}F^{-s}(\{\partial_{l}\}_{(jp^{r'-r}-i)p^{s}/p^{r'}}(\beta)F^{s-r}(\alpha))\cdot\{\partial_{l}\}_{i/p^{r'}}\{\partial_{t}\}_{j'/p^{s}}
\]
as operators on $W(A)$. We note that all of the terms involved both
sides of this equality are in $\mathcal{D}_{W(A)}^{(m)}$; indeed,
since $\mathcal{D}_{W(A)}^{(m)}$ is an algebra, this is clear for
the product $F^{-r}(\alpha)\{\partial_{l}\}_{j/p^{r}}\cdot F^{-s}(\beta)\{\partial_{t}\}_{i/p^{s}}$,
and it is also clear for the terms $F^{-s}(\beta F^{s-r}(\alpha))\{\partial_{t}\}_{j'/p^{s}}\{\partial_{l}\}_{j/p^{r}}$
and $F^{-s}(\{\partial_{l}\}_{p^{s}j/p^{r}}(\beta)\cdot F^{s-r}(\alpha))\{\partial_{t}\}_{j'/p^{s}}$,
therefore it is true for the last term as well.

So, since $\mathcal{D}_{W(A)}^{(m)}\subset\text{End}_{W(k)}(W(A))$,
the equality of \prettyref{eq:bracket-1} is in fact an equality of
elements of $\mathcal{D}_{W(A)}^{(m)}$, and so we obtain that the
image of $F^{-r}(\alpha)\{\partial_{l}\}_{j/p^{r}}\cdot F^{-s}(\beta)\{\partial_{t}\}_{i/p^{s}}$
in $\widehat{\mathcal{D}}_{W(A)}^{(m)}$ is contained in $\widehat{\mathcal{D}}_{W(A),b}^{(m)}$. 

Now, we claim that $F^{-r}(\alpha)\{\partial_{l}\}_{j/p^{r}}\cdot F^{-s}(\beta)\{\partial_{t}\}_{i/p^{s}}$
is actually contained in $\widehat{\mathcal{D}}_{W(A),b,s}^{(m)}$.
This is obvious in the above sum unless $r'=s$ and $l=t$. In that
case, we have 
\[
\{\partial_{l}\}_{i/p^{s}}\{\partial_{l}\}_{j'/p^{s}}=u\cdot{i+j' \choose i}\{\partial_{l}\}_{(i+j')/p^{s}}
\]
by \prettyref{lem:products}. Let $\nu=\text{val}(i+j')$. Then we
need to show that 
\[
\text{val}{i+j' \choose i}\geq\nu
\]
This only has content if $v>0$, so we assume this from now on; let
$\alpha=(i+j')/p$. To see this, we consider the element $(T_{1}+T_{2})^{\alpha}$
inside $\mathbb{Z}/p^{v}[T_{1},T_{2}]$. Since $\text{val}(\alpha)=\nu-1$,
we have $(T_{1}+T_{2})^{\alpha}\in W_{v}(\mathbb{F}_{p}[T_{1},T_{2}]^{(\nu-1)})$.
Therefore $F(T_{1}+T_{2})^{\alpha}=(T_{1}+T_{2})^{i+j'}$ for any
lift of Frobenius on $\mathbb{Z}/p^{v}[T_{1},T_{2}]$; choosing $F$
to be the lift which takes $T_{i}\to T_{i}^{p}$, we see that the
coefficient of $T_{1}^{a}T_{2}^{b}$ in $(T_{1}+T_{2})^{i+j'}$ is
zero if either $a$ or $b$ is coprime to $p$. As the coefficient
of $T_{1}^{i}T_{2}^{j'}$ is ${\displaystyle {i+j' \choose i}}$ ,
and as $i$ is coprime to $p$, we deduce that ${\displaystyle {i+j' \choose i}}$
is zero in $\mathbb{Z}/p^{v}$ as required. 

Now the result follows, via additivity, from the remark above that
${\displaystyle \sum_{s=i}^{\infty}\phi_{s}\in\widehat{\mathcal{D}}_{W(A),b,i}^{(m)}}$
whenever $\phi_{s}\in\widehat{\mathcal{D}}_{W(A),b,s}^{(m)}$. 
\end{proof}
From this we deduce 
\begin{cor}
\label{cor:existence-of-basis}Every element of $\widehat{\mathcal{D}}_{W(A)}^{(m)}$
is contained in $\widehat{\mathcal{D}}_{W(A),b}^{(m)}$. Further,
the ideal $\widehat{I^{(i)}}$ is equal to $\widehat{\mathcal{D}}_{W(A),b.i}^{(m)}$. 
\end{cor}

\begin{proof}
Since $W(A)\subset\widehat{\mathcal{D}}_{W(A),b}^{(m)}$ and $\mathcal{E}W^{(m)}(A)\subset\widehat{\mathcal{D}}_{W(A),b}^{(0)}$,
we see that the image of $\mathcal{D}_{W(A)}^{(m)}\to\widehat{\mathcal{D}}_{W(A)}^{(m)}$
is contained in $\widehat{\mathcal{D}}_{W(A),b}^{(m)}$. By the last
statement of \prettyref{lem:big-products}, the image of $I^{(i)}$
is contained in $\widehat{\mathcal{D}}_{W(A),b,i}^{(m)}$. Since an
arbitrary sum ${\displaystyle \sum_{s=0}^{\infty}Q_{s}}$ converges
if $Q_{s}\in\widehat{D}_{W(A),b,s}$, we deduce that $\widehat{\mathcal{D}}_{W(A)}^{(m)}=\widehat{\mathcal{D}}_{W(A),b}^{(m)}$,
with $\widehat{I^{(i)}}\subset\widehat{\mathcal{D}}_{W(A),b,i}^{(m)}$.
But the containment $\widehat{\mathcal{D}}_{W(A),b,i}^{(m)}\subset\widehat{I^{(i)}}$
is clear from the definitions. 
\end{proof}
Finally, we need to address uniqueness. It will follow from the a
generalization of \prettyref{lem:Reps-of-phi}; to state this, set,
for $i\in\mathbb{N}$, ${\displaystyle f_{i}:=\text{val}(i!)}$, and,
for a multi-index $I$, set ${\displaystyle f_{I}:=\sum_{j=1}^{n}f_{i_{j}}}$. 
\begin{lem}
\label{lem:Injectivity-Lemma}An element 
\[
Q=\sum_{I}(\sum_{r=1}^{\infty}\sum_{J_{I}}F^{-r}(\alpha_{J_{I}})\{\partial\}_{J_{I}/p^{r}}+\sum_{K_{I}}\alpha_{K_{I}}\{\partial\}_{K_{I}})\{\partial\}_{p^{m}}^{I}
\]
(where the notation is as in \prettyref{thm:Basis}) acts trivially
on $W_{r'+1}(A)$ iff we have $\alpha_{J_{I}}\in V^{r'+1-r-f_{I}}(W(A))$
for each $I$ and $J_{I}$ (where we take the superscript to be $0$
if $r'+1-r-f_{I}<0$), and $\alpha_{K_{I}}\in V^{r'+1}(W(A))$ for
all $K_{I}$. 
\end{lem}

\begin{proof}
We consider the operator 
\[
\sum_{I}(\sum_{r=-m}^{r'}\sum_{J_{I}}F^{-r}(\alpha_{J_{I}})\partial^{[p^{r'-r}J_{I}]}+\sum_{K_{I}}\alpha_{K_{I}}\partial^{[p^{r'}K_{I}]})(\partial^{[p^{r'+m}]})^{I}
\]
we note that by our conventions on sums in \prettyref{thm:Basis}
we have $\alpha_{J_{I}},\alpha_{K_{I}}\to0$ as $I\to\infty$, so
for the purposes of considering the action on $W_{r'+1}(A)$ we may
regard this as being a finite sum. We claim that 
\[
((\partial^{[p^{r'+m}]})^{I})(W_{r'+1}(A^{(r')}))\subset V^{f_{I}}(W_{r'+1}(A^{(r')}))
\]
To show this it suffices to show $(\partial_{l}^{[p^{r'+m}]})^{i}(W_{r'+1}(A^{(r')}))\subset p^{f_{i}}(W_{r'+1}(A^{(r')}))$
for each $l$ and $i$. We have, by \prettyref{lem:products} the
relation 
\[
(\partial_{l}^{[p^{r'+m}]})^{i}=up^{f_{i}}\partial_{l}^{[ip^{r'+m}]}
\]
for an appropriate unit $u\in\mathbb{Z}_{p}$; the operator $\partial_{l}^{[ip^{r'}]}$
is the canonical lift of the Hasse-Schmidt derivation $\partial_{l}^{[ip^{r'}]}$
and hence preserves $W_{r'+1}(A^{(r')})$, so that $(\partial_{l}^{[p^{r'}]})^{i}$
must take $W_{r'+1}(A^{(r')})$ to $p^{f_{i}}(W_{r'+1}(A^{(r')}))$
as required.

Next, we claim that $F^{-r}(\alpha_{J_{I}})\partial^{[p^{r'-r}J_{I}]}(V^{f_{I}}(W_{r'+1}(A^{(r')})))=0$
if \\
$\alpha_{J_{I}}\in V^{r'-r+1-f_{I}}(W(A))$; this follows directly
from the definitions as in \prettyref{lem:Reps-of-phi}. This, along
with the fact that $\alpha_{K_{I}}\partial^{[p^{r'}K_{I}]}(W_{r'+1}(A^{(r')}))=0$
if \\
$\alpha_{K_{I}}\in V^{r'+1}(W_{r'+1}(A^{(r')}))$, gives us the reverse
direction of the lemma. 

For the forward direction, consider the (finite) set of indices $\{p^{r'-r}J_{I}+p^{r'+m}I|\alpha_{J_{I}}\neq0\}$,
union with $\{p^{r'K_{I}}+p^{r'+m}I|\alpha_{K_{I}}\neq0\}$. Suppose
the least (in the lexicographic ordering) index in this set has the
form $p^{r'-r}J_{I_{1}}+p^{r'}I_{1}$ (the case where the least element
has the form $p^{r'K_{I}}+p^{r'+m}I$ is similar, but simpler). Then
we have 
\[
0=\sum_{I}(\sum_{r=-m}^{r'}\sum_{J_{I}}F^{-r}(\alpha_{J_{I}})\partial^{[p^{r'-r}J_{I}]}+\sum_{K_{I}}\alpha_{K_{I}}\partial^{[p^{r'}K_{I}]})(\partial^{[p^{r'+m}]})^{I}(p^{r}T^{p^{r'-r}J_{I_{1}}+p^{r'}I_{1}})
\]
\[
=u\cdot p^{r+f_{I}}F^{-r}(\alpha_{J_{I}})
\]
for some $u$ such that $\text{val}(u)=0$; here we have used that
$\partial^{[p^{r'-r}J_{I}]}(\partial^{[p^{r'}]})^{I}(T^{I'})=0$ whenever
$I'<p^{r'-r}J_{I}+p^{r'}I$, and that $\partial^{[p^{r'-r}J_{I_{1}}]}T^{p^{r'-r}J_{I_{1}}}=1$
while $(\partial^{[p^{r'}]})^{I_{1}}T^{I_{1}}=up^{f_{I}}$. For this
to hold we must have $\alpha_{J_{I_{1}}}\in V^{r'+1-r-f_{I}}(W_{r'+1}(A^{(r')}))$;
this shows, using the direction of the lemma already proved, that
$F^{-r}(\alpha_{J_{I_{1}}})\partial^{[p^{r'-r}J_{I_{1}}]}(\partial^{[p^{r'}]})^{I_{1}}$
acts as zero on $W_{r'+1}(A^{(r')})$. Thus we may subtract this element
from 
\[
\sum_{I}\sum_{r=0}^{r'}\sum_{J_{I}}F^{-r}(\alpha_{J_{I}})\partial^{[p^{r'-r}J_{I}]}(\partial^{[p^{r'}]})^{I}
\]
and obtain another operator which acts as zero; continuing in this
way shows that $\alpha_{J_{I}}\in V^{r'+1-r-f_{I}}(W_{r'+1}(A^{(r')}))$
for all $\alpha_{J}$ as required. 
\end{proof}
Now we put everything together for the 
\begin{proof}
(of \prettyref{thm:Basis}) The existence, for each $Q\in\widehat{\mathcal{D}}_{W(A)}^{(m)}$,
of the claimed representation is given by \prettyref{cor:existence-of-basis}.
As for the uniqueness, we note that if any element of the form 
\[
\sum_{I}(\sum_{r=1}^{\infty}\sum_{J_{I}}F^{-r}(\alpha_{J_{I}})\{\partial\}_{J_{I}/p^{r}}+\sum_{K_{I}}\alpha_{K_{I}}\{\partial\}_{K_{I}})\{\partial\}_{p^{m}}^{I}
\]
acts as zero on $W(A)$, then by the previous lemma we have, for each
index $J_{I}$, $\alpha_{J_{I}}\in V^{r'-r+1-f_{I}}(W(A))$ for all
$r'\geq0$; but this implies $\alpha_{J_{I}}=0$. 
\end{proof}
The proof actually gave us a little more:
\begin{cor}
The natural map $\widehat{\mathcal{D}}_{W(A)}^{(m)}\to\text{End}_{W(k)}(W(A))$
is injective. 
\end{cor}

In addition, using the description of $I^{(1)}$ in \prettyref{cor:existence-of-basis},
we obtain
\begin{cor}
For any $m\geq0$ there is an isomorphism $\widehat{\mathcal{D}}_{W(A)}^{(m)}/I^{(1)}\tilde{\to}\mathcal{D}_{A}^{(m)}$. 
\end{cor}

As well as 
\begin{cor}
\label{cor:Defn-of-D-infty}For each $m$ the obvious map $\widehat{\mathcal{D}}_{W(A)}^{(m)}\to\widehat{\mathcal{D}}_{W(A)}^{(m+1)}$
is injective. For each $m$ the natural map $\widehat{\mathcal{D}}_{W(A)}^{(m)}\to\widehat{\mathcal{D}}_{W(A)}$
is injective. 
\end{cor}

To close out this section, we would like to record for later use the
analogue of \prettyref{thm:Basis} when we have a morphism $\varphi^{\#}:B\to A$
of smooth $k$-algebras. By the functoriality of the Witt vectors
there is a morphism $W\varphi^{\#}:W(B)\to W(A)$. Letting $X=\text{Spec}(A)$
and $Y=\text{Spec}(B)$, we obtain a morphism of affine formal schemes
$W\varphi:W(X)\to W(Y)$. 
\begin{defn}
Let $\mathcal{D}_{W(X)\to W(Y)}^{(m)}$ be the $(\mathcal{D}_{W(A)},\mathcal{D}_{W(B)})$
bi-submodule of \\
$\text{Hom}_{W(k)}(W(B),W(A))$ generated by $\mathcal{E}W^{(m)}(B;A)$.
Let $\widehat{\mathcal{D}}_{W(X)\to W(Y)}^{(m)}$ be the completion
of this bimodule along the filtration 
\[
F^{l}(\mathcal{D}_{W(X)\to W(Y)}^{(m)})=\{I^{(i)}\mathcal{D}_{W(X)\to W(Y)}^{(m)}I^{(j)}\}_{i+j\geq l}
\]
\end{defn}

Suppose that $B$ possesses local coordinates $\{T_{1},\dots,T_{d}\}$.
Then we have 
\begin{cor}
\label{cor:Basis-for-bimodule}Every element of $\widehat{\mathcal{D}}_{W(X)\to W(Y)}^{(m)}$
can be uniquely expressed as 
\[
\sum_{I}(\sum_{r=0}^{\infty}\sum_{J_{I}}F^{-r}(\alpha_{J_{I}})\cdot W\varphi^{\#}(\{\partial\}_{J_{I}/p^{r}})\{\partial\}^{I})
\]
where $\alpha_{J_{I}}\in W(A)$, $\{\partial\}_{J_{I}/p^{r}}$ and
$\{\partial\}^{I}$ are in $\widehat{\mathcal{D}}_{W(B)}$, and the
convergence conditions are as in \prettyref{thm:Basis}. In particular,
for any $m\geq0$ we have 
\[
I^{(1)}\backslash\widehat{\mathcal{D}}_{W(X)\to W(Y)}/I^{(1)}\tilde{\to}A\otimes_{B}\mathcal{D}_{B}^{(m)}
\]
The latter object is the usual transfer bimodule in the theory of
$\mathcal{D}^{(m)}$-modules. 
\end{cor}

\begin{proof}
It follows directly from the description of $\mathcal{E}W^{(m)}(B;A)$
in \prettyref{cor:Reps-of-phi-for-m} the set of such sums are contained
in $\widehat{\mathcal{D}}_{W(X)\to W(Y)}$; denote this set by $\widehat{\mathcal{D}}_{W(X)\to W(Y),b}$
and as above we say that $Q\in\widehat{\mathcal{D}}_{W(X)\to W(Y),b,i}$
if $Q$ admits a representation as above so that for each $r<i$ and
each associated $J_{I}$, we have $\alpha_{J_{I}}\in V^{i-r}(W(A))$.
So, applying \prettyref{thm:Basis}, we have to show that for any
element $\phi$ of $\mathcal{E}W^{(0),i}(A)$, we have that $\phi\circ W\varphi^{\#}$
is contained in $\widehat{D}_{W(A),b,i}$. 

Using the local coordinates on $B$, we see that the argument(s) of
\prettyref{lem:Reps-of-HS}, \prettyref{lem:Reps-of-phi}, and \prettyref{cor:Reps-of-phi-redux}
carry over to this situation with (essentially) no change, and we
conclude that $\widehat{\mathcal{D}}_{W(X)\to W(Y)}=\widehat{\mathcal{D}}_{W(X)\to W(Y),b}$. 

The uniqueness of the representation of an element of $\widehat{\mathcal{D}}_{W(X)\to W(Y)}$
as a sum follows exactly as in \prettyref{lem:Injectivity-Lemma},
and the last sentence follows immediately from the description and
the definition of the ideal $I^{(1)}$. 
\end{proof}
To finish out this section we show how to turn $\widehat{\mathcal{D}}_{W(A)}$
into a sheaf on the etale site of $X=\text{Spec}(A)$. This immediately
leads to a definition of $\widehat{\mathcal{D}}_{W(X)}$ for any smooth
$X$ over $k$. We need to show:
\begin{prop}
\label{prop:construction-of-transfer}1) Let $\psi^{\#}:B\to A$ be
an etale morphism of smooth $k$-algebras, where $B$ admits local
coordinates. Then there is an isomorphism 
\[
W(A)\widehat{\otimes}_{W(B)}\widehat{\mathcal{D}}_{W(B)}\tilde{\to}\widehat{\mathcal{D}}_{W(A)}
\]
where on the left hand side the $\widehat{\otimes}$ denotes completion
with respect to the filtration $V^{i}(W(A))$. In particular, we have
for each $r\geq0$ 
\[
W_{r+1}(A)\otimes_{W_{r+1}(B)}\mathcal{D}_{W_{r+1}(B)}\tilde{\to}\mathcal{D}_{W_{r+1}(A)}
\]

2) Let $\varphi^{\#}:C\to B$ any morphism, and $\psi^{\#}:B\to A$
as above. Let $X=\text{Spec}(B)$, $Y=\text{Spec}(C)$, and $U=\text{Spec}(A)$,
we obtain morphisms of affine formal schemes $W\varphi:W(X)\to W(Y)$
and $W\psi:W(U)\to W(X)$, as well as the composition $W(U)\to W(Y)$.
Then there is an isomorphism 
\[
W(A)\widehat{\otimes}_{W(B)}\widehat{\mathcal{D}}_{W(X)\to W(Y)}\tilde{\to}\widehat{\mathcal{D}}_{W(U)\to W(Y)}
\]
\end{prop}

\begin{proof}
1) As HS derivations extend canonically over etale morphisms, one
sees easily that there is a morphism $\mathcal{E}W_{B}\to\mathcal{E}W_{A}$
which extends to a morphism $\widehat{\mathcal{D}}_{W(B)}\to\widehat{\mathcal{D}}_{W(A)}$;
extending this by linearity and completing gives the morphism of the
statement. The fact that it is an isomorphism follows directly from
\prettyref{thm:Basis} for both $\widehat{\mathcal{D}}_{W(A)}$ and
$\widehat{\mathcal{D}}_{W(B)}$; noting the fact that 
\[
W(A)\otimes_{W(B)}F^{-r}(W(B))\tilde{\to}F^{-r}(W(A))
\]
where $F^{-r}(W(B))$ and $F^{-r}(W(A))$ are regarded as sub-modules
of $W(B)[p^{-1}]$ and $W(A)[p^{-1}]$, respectively; this fact, in
turn, can be easily checked using local coordinates for $B$ and the
fact that $A$ is etale over $B$. Now the last sentence follows from
the identification of $\widehat{I^{(r+1)}}$ with $\widehat{D}_{W(A),b,r+1}$. 

2) This follows exactly as in $1)$ but using \prettyref{cor:Basis-for-bimodule}
instead of \prettyref{thm:Basis}. 
\end{proof}
This leads to the 
\begin{defn}
\label{def:D-and-Transfer}1) Let $X$ be a smooth scheme over $k$.
We define $\widehat{\mathcal{D}}_{W(X)}$ as the unique sheaf of rings
(in the Zariski topology of $X$) which assigns to each open affine
$\text{Spec}(A)\subset X$ the ring $\widehat{\mathcal{D}}_{W(A)}$.
By the previous lemma this is an inverse limit of quasicoherent sheaves
on $\{W_{r+1}(X)\}$. 

2) Let $\varphi:X\to Y$ a morphism of smooth schemes. We define $\widehat{\mathcal{D}}_{W(X)\to W(Y)}$
as the unique sheaf of $(\widehat{\mathcal{D}}_{W(X)},W\varphi^{-1}(\widehat{\mathcal{D}}_{W(Y)}))$
bimodules which assigns, to each open affine $\text{Spec}(A)\subset X$
such that $\varphi(\text{Spec}(A))\subset\text{Spec}(B)\subset Y$,
the bimodule $\widehat{\mathcal{D}}_{W(\text{Spec}(A))\to W(\text{Spec}(B))}$.
This is an inverse limit of quasicoherent sheaves on $\{W_{r+1}(X)\}$. 
\end{defn}

\subsection{Compliment: The algebra $\widehat{\mathcal{D}}_{W(X),\text{crys}}^{(0)}$}

In this subsection we wish to discuss the analogue of the above results
and constructions for a slightly different algebra, which will turn
out to be a completion of $\widehat{\mathcal{D}}_{W(X)}^{(0)}$. As
above we start in the case where $X=\text{Spec}(A)$. 
\begin{defn}
The operator filtration on $\widehat{\mathcal{D}}_{W(A)}^{(0)}$ is
the filtration defined by 
\[
F^{i}(\widehat{\mathcal{D}}_{W(A)}^{(0)})=\{P\in\widehat{\mathcal{D}}_{W(A)}^{(0)}|P(W(A)\subset V^{i}(W(A))\}
\]
This is a filtration by two sided ideals, and we denote the completion
along it by $\widehat{\mathcal{D}}_{W(X),\text{crys}}^{(0)}$. 
\end{defn}

There is a description of $\widehat{\mathcal{D}}_{W(X),\text{crys}}^{(0)}$
in local coordinates: 
\begin{prop}
Suppose $A$ admits local coordinates, and let the notation be as
in \prettyref{thm:Basis}. Then an element 
\[
\sum_{I}(\sum_{r=0}^{\infty}\sum_{J_{I}}F^{-r}(\alpha_{J_{I}})\{\partial\}_{J_{I}/p^{r}})\{\partial\}^{I}\in\widehat{\mathcal{D}}_{W(A)}^{(0)}
\]
is contained in $F^{i}(\widehat{\mathcal{D}}_{W(A)}^{(0)})$ iff,
whenever $f_{I}<i$ and $r<i$, we have $\alpha_{J_{I}}\in V^{i-(r+f_{I})}(W(A))$
(where we set $V^{i-(r+f_{I})}(W(A)=W(A)$ for $i-(r+f_{I})<0$). 

Therefore, every element of $\widehat{\mathcal{D}}_{W(A),\text{crys}}^{(0)}$
has a unique expression 
\[
\sum_{I}(\sum_{r=0}^{\infty}\sum_{J_{I}}F^{-r}(\alpha_{J_{I}})\{\partial\}_{J_{I}/p^{r}})\{\partial\}^{I}
\]
where $I=(i_{1},\dots,i_{n})\in\mathbb{N}^{n}$, $\{\partial\}^{I}:=\{\partial_{1}\}_{1}^{i_{1}}\cdots\{\partial_{n}\}_{1}^{i_{n}}$,
and $J_{I}=(j_{1},\dots,j_{n})$ satisfies $0\leq j_{l}<p^{r}$ for
each $j_{l}$ and $(p,j_{l})=1$ for at least one $l$; unless $r=0$
and $J_{I}=0$. 
\end{prop}

\begin{proof}
This follows from \prettyref{lem:Injectivity-Lemma} above.
\end{proof}
Exactly as above, this implies that the sheafification $\widehat{\mathcal{D}}_{W(X),\text{crys}}^{(0)}$
has the property that $\Gamma(U,\widehat{\mathcal{D}}_{W(X),\text{crys}}^{(0)})=\widehat{\mathcal{D}}_{W(A),\text{crys}}^{(0)}$
whenever $U=\text{Spec}(A)$ has local coordinates. We also see that 
\begin{cor}
The completion map $\widehat{\mathcal{D}}_{W(X)}^{(0)}\to\widehat{\mathcal{D}}_{W(X),\text{crys}}^{(0)}$
is injective.
\end{cor}

\section{Accessibility}

In this chapter we will define the main category of interest in this
paper, the category of accessible modules over $\widehat{\mathcal{D}}_{W(X)}^{(m)}$
(here, $m\geq0$). In particular, we shall prove \prettyref{thm:Consider-the--bimodule},
\prettyref{thm:Bimodule-unique!}, \prettyref{cor:Bimodule-properties!},
as well as several other related results. To get things off the ground,
we need to defined and study the fundamental bimodule $\Phi^{*}\widehat{\mathcal{D}}_{\mathcal{A}}^{(m)}$
where $\mathcal{A}$ is, as usual, a smooth $W(k)$ algebra lifting
$A$ which admits local coordinates. We let $F:\mathcal{A}\to\mathcal{A}$
be some coordinatized lift of Frobenius and $\Phi:\mathcal{A}\to W(A)$
the associated map to Witt vectors. 
\begin{lem}
1) The action map $\widehat{\mathcal{D}}_{\mathcal{A}}^{(m)}\to\mathcal{E}nd_{W(k)}(\mathcal{A})$
is injective. 

2) Let $\Phi^{*}\widehat{\mathcal{D}}_{\mathcal{A}}^{(m)}$ denote
the completion of $W(A)\otimes_{\mathcal{A}}\widehat{\mathcal{D}}_{\mathcal{A}}^{(m)}$
along the filtration $V^{i}(W(A))\otimes_{\mathcal{A}}\widehat{\mathcal{D}}_{\mathcal{A}}^{(m)}$.
Then there is an embedding 
\[
\Phi^{*}\widehat{\mathcal{D}}_{\mathcal{A}}^{(m)}\to\mathcal{H}om_{W(k)}(\mathcal{A},W(A))
\]
which takes $1\otimes1$ to $\Phi$. This embedding preserves the
natural right $\widehat{\mathcal{D}}_{\mathcal{A}}^{(m)}$-module
structures on both sides. 
\end{lem}

\begin{proof}
1) This is a well-known fact, whose proof boils down to a (vastly)
simplified version of \prettyref{lem:Injectivity-Lemma}. 

2) There is a map $\widehat{\mathcal{D}}_{\mathcal{A}}^{(m)}\to\mathcal{H}om_{W(k)}(\mathcal{A},W(A))$
given by taking an element $Q\in\widehat{\mathcal{D}}_{\mathcal{A}}^{(m)}$
to $\Phi\circ Q:\mathcal{A}\to W(A)$ (where we regard $Q$ as acting
on $\mathcal{A}$). Since $\Phi$ is injective this map is injective
also; in addition, the right action of $\widehat{\mathcal{D}}_{\mathcal{A}}^{(m)}$
on itself clearly corresponds to the right action of $\widehat{\mathcal{D}}_{\mathcal{A}}^{(m)}$
on $\mathcal{H}om_{W(k)}(\mathcal{A},W(A))$ through the action on
$\mathcal{A}$. Further, the (left) action of $W(A)$ on itself gives
$\mathcal{H}om_{W(k)}(\mathcal{A},W(A))$ the structure of a $W(A)$-module,
so that there is an induced map 
\[
\iota:W(A)\otimes_{\mathcal{A}}\widehat{\mathcal{D}}_{\mathcal{A}}^{(m)}\to\mathcal{H}om_{W(k)}(\mathcal{A},W(A))
\]
and, for elements $a_{i}\in V^{i}(W(A))\otimes_{\mathcal{A}}\widehat{\mathcal{D}}_{\mathcal{A}}^{(m)}$,
the map $\iota(a_{i})$ has image in $V^{i}(W(A))$, so that the sum
${\displaystyle \sum_{i}\iota(a_{i})}$ is well-defined in $\mathcal{H}om_{W(k)}(\mathcal{A},W(A))$.
Thus there is an induced map, which we also call $\iota$, $\iota:\widehat{\mathcal{D}}_{\mathcal{A}}^{(m)}\to\mathcal{H}om_{W(k)}(\mathcal{A},W(A))$,
which clearly preserves the right $\widehat{\mathcal{D}}_{\mathcal{A}}^{(m)}$-module
structures. 

To show that this map is injective, let $\{T_{1},\dots,T_{n}\}$ be
local coordinates on $\mathcal{A}$. Then any element of $W(A)$ can
be written uniquely as ${\displaystyle \sum_{r=0}^{\infty}\sum_{I}V^{r}(T^{I})\cdot\Phi(a_{I})}$
where, for any $r$, $I=(i_{1},\dots,i_{n})$ ranges over elements
of $\mathbb{N}^{n}$ such that $0\leq i_{j}<p^{r}$ and $(p,i_{j})=1$
for at least one index $j$; and $a_{I}\in\mathcal{A}$. Thus any
element $Q\in\Phi^{*}\widehat{\mathcal{D}}_{\mathcal{A}}^{(m)}$ can
be written uniquely as ${\displaystyle \sum_{r=0}^{\infty}\sum_{I}V^{r}(T^{I})\cdot Q_{I}}$
where $Q_{I}\in\Phi^{*}\widehat{\mathcal{D}}_{\mathcal{A}}^{(m)}$
and $I$ satisfies the same conditions as above. But such an element
acts as $0$ on $\mathcal{A}$ iff each $Q_{I}$ acts as zero on $\mathcal{A}$
which implies that $Q=0$; therefore $\iota$ is injective as claimed. 
\end{proof}
The natural action of $\widehat{\mathcal{D}}_{W(A)}^{(m)}$ on $W(A)$
endows the module $\mathcal{H}om_{W(k)}(\mathcal{A},W(A))$ with a
left action by $\widehat{\mathcal{D}}_{W(A)}^{(m)}$. 
\begin{prop}
\label{prop:Construction-of-bimodule}The submodule $\Phi^{*}\widehat{\mathcal{D}}_{\mathcal{A}}^{(m)}$
of $\mathcal{H}om_{W(k)}(\mathcal{A},W(A))$ is preserved under the
action of $\widehat{\mathcal{D}}_{W(A)}^{(m)}$. Thus $\Phi^{*}\widehat{\mathcal{D}}_{\mathcal{A}}^{(m)}$
has the structure of a $(\widehat{\mathcal{D}}_{W(A)}^{(m)},\widehat{\mathcal{D}}_{\mathcal{A}}^{(m)})$-bimodule.
If $m\geq0$ the map $\mathcal{D}_{W(\mathcal{A})}^{(0)}\to\Phi^{*}\widehat{\mathcal{D}}_{\mathcal{A}}^{(0)}$
which takes $1\to\Phi$ is onto. 
\end{prop}

\begin{proof}
Let us consider the first statement. We shall suppose that $m=0$,
the other cases being extremely similar. We begin by considering an
element of the form $\{\partial\}_{J/p^{r}}$, where $r\geq0$ and
$0\leq j_{i}\leq p^{r}$ for all $i\in\{1,\dots n\}$). We shall show
that $\{\partial\}_{J/p^{r}}\circ\Phi\in\text{Hom}_{W(k)}(\mathcal{A},W(A))$
agrees with an element of $\Phi^{*}\widehat{\mathcal{D}}_{\mathcal{A}}^{(0)}$. 

The map $\{\partial\}_{J/p^{r}}\circ\Phi$ induces, for each $r'\geq r$,
a map $\mathcal{A}/p^{r'+1}\to W_{r'+1}(A)$. Further, $\mathcal{A}/p^{r'+1}$
is etale over a polynomial ring $W_{r'+1}(k)[T_{1},\dots,T_{n}]$.
Therefore we must compute the action of $\{\partial\}_{J/p^{r}}$
on a monomial $T_{1}^{k_{1}}\cdots T_{n}^{k_{n}}\in W_{r'+1}(k)[T_{1},\dots,T_{n}]$.
We have 
\[
\{\partial\}_{J/p^{r}}(T_{1}^{k_{1}}\cdots T_{n}^{k_{n}})=\prod_{i=1}^{n}\binom{k_{i}p^{r'}}{j_{i}p^{r'-r}}\cdot T^{k_{i}-j_{i}/p^{r}}\in W_{r'+1}(k[T_{1},\dots,T_{n}])
\]
 When $j_{i}\neq0$ we have the identity 
\[
\binom{k_{i}p^{r'}}{j_{i}p^{r'-r}}=\prod_{l=1}^{j_{i}p^{r'-r}}\frac{k_{i}p^{r'}-j_{i}p^{r'-r}+l}{l}=\prod_{l=1}^{j'_{i}p^{r'-r+v_{i}}}\frac{k_{i}p^{r'}-j'_{i}p^{r'-r+v_{i}}+l}{l}
\]
where we have written $j_{i}=j_{i}'p^{v_{i}}$ where $v_{i}=\text{val}(j_{i})$.
Our aim is to show that this expression is a polynomial in $k_{i}$,
with coefficients in $\mathbb{Z}/p^{r'+1}$. If $j_{i}=p^{r}$, then
we have 
\[
\binom{k_{i}p^{r'}}{j_{i}p^{r'-r}}=\binom{k_{i}p^{r'}}{p^{r'}}=\prod_{l=1}^{p^{r'}}\frac{(k_{i}-1)p^{r'}+l}{l}
\]
\[
=k_{i}\cdot\prod_{l=1}^{p^{r'}-1}((k_{i}-1)\frac{p^{r'}}{l}+1)
\]
which is evidently a polynomial in $k_{i}$, with coefficients in
$\mathbb{Z}/p^{r'+1}$. We remark that if $q_{i,r'}(t)$ denotes this
polynomial, then $t|q_{i,r'}(t)$. 

Next, if $j_{i}<p^{r}$ then $v_{i}<r$ and we have 
\begin{equation}
\frac{k_{i}p^{r'}-j'_{i}p^{r'-r+v_{i}}-l}{l}=\frac{p^{r'-r+v_{i}}(k_{i}p^{r-v_{i}}-j'_{i})-l}{l}\label{eq:term}
\end{equation}
and since $\text{val}(j_{i}')=0$ we have $\text{val}(k_{i}p^{r-v_{i}}-j'_{i})=0$
and so, if $\text{val}(l)\leq r'-r+v_{i}$ we have 
\[
\frac{p^{r'-r+v_{i}}(k_{i}p^{r-v_{i}}-j'_{i})-l}{l}=\frac{p^{r'-r+v_{i}}}{l}(k_{i}p^{r-v_{i}}-j'_{i})-1\in\mathbb{Z}/p^{r'+1}
\]
However, if $\text{val}(l)>r'-r+v_{i}$ then the term \prettyref{eq:term},
considered as a element of $\mathbb{Q}$, is not in $\mathbb{Z}_{p}$.
For each such $l$ we define $l_{j'}:=j'p^{r'-r+v_{i}}-l$; since
$1\leq l\leq j'p^{r'-r+v_{i}}$ the same is true of $l_{j'}$, and
we have $\text{val}(l_{j'})=r'-r+v_{i}$. We have that 
\[
\frac{k_{i}p^{r'}-j'_{i}p^{r'-r+v_{i}}-l}{l}\cdot\frac{k_{i}p^{r'}-j'_{i}p^{r'-r+v_{i}}-l_{j'}}{l_{j'}}
\]
\[
=\frac{k_{i}p^{r'}-l_{j'}}{l_{j'}}\cdot\frac{k_{i}p^{r'}-l}{l}=(k_{i}\frac{p^{r'}}{l_{j'}}-1)(k_{i}\frac{p^{r'}}{l}-1)
\]
So, in sum, we have 
\[
\binom{k_{i}p^{r'}}{j_{i}p^{r'-r}}=\prod(\frac{p^{r'-r+v_{i}}}{l}(k_{i}p^{r-v_{i}}-j'_{i})-1)\cdot\prod_{\{l|\text{val}(l)>r'-r+v_{i}\}}(k_{i}\frac{p^{r'}}{l_{j'}}-1)(k_{i}\frac{p^{r'}}{l}-1)
\]
where the first product ranges over $l\in\{1,\dots,j'_{i}p^{r'-r+v_{i}}\}$
such that $\text{val}(l)\leq r'-r+v_{i}$ and $l$ is not of the form
$l_{j'}$. This expression is evidently a polynomial in $k_{i}$,
with coefficients in $\mathbb{Z}/p^{r'+1}$; as above we denote this
polynomial by $q_{i,r'}(t)$ . When $l=j'_{i}p^{r'-r+v_{i}}$ we get
the term ${\displaystyle k_{i}\frac{p^{r-v_{i}}}{j_{i}'}}$, so that
again $t|q_{i,r'}(t)$. By inspection one sees that the image of $q_{i,r'}(t)$
in $\mathbb{Z}/p^{r'}[t]$ is $q_{i,r'-1}(t)$. 

Now we note that, for any polynomial $q'(t)\in\mathbb{Z}/p^{r'+1}[t]$,
and for a fixed index $i\in\{1,\dots n\}$, the endomorphism of $W_{r'+1}(k)[T_{1},\dots,T_{n}]$
given by 
\[
\prod_{j=1}^{n}T_{j}^{k_{j}}\to k_{i}q'(k_{i})\prod_{j\neq i}T_{j}^{k_{j}}\cdot T_{i}^{k_{i}-1}
\]
is given by the action of an operator in $D_{\mathcal{A}/p^{r'+1}}^{(0)}$;
indeed we have that $\frac{\partial}{\partial T_{i}}\cdot q'{\displaystyle (T_{i}\frac{\partial}{\partial T_{i}})(\prod_{j=1}^{n}T_{j}^{k_{j}})=k_{i}q'(k_{i})\cdot\prod_{j\neq i}T_{j}^{k_{j}}\cdot T_{i}^{k_{i}-1}}$.
So we have
\[
\{\partial\}_{J/p^{r}}(T_{1}^{k_{1}}\cdots T_{n}^{k_{n}})=(\prod_{\{i|j_{i}\neq0\}}p^{r-v_{i}}T_{i}^{1-j_{i}/p^{r}}\cdot q_{i,r'}(k_{i}))(T_{1}^{k_{1}-\epsilon_{1}}\cdots T_{n}^{k_{n}-\epsilon_{n}})
\]
(where $\epsilon_{i}=1$ if $j_{i}\neq0$ and $\epsilon_{i}=0$ if
$j_{i}=0$), which, since $t|q_{i,r'}(t)$, is given by the application
of an operator 
\[
\prod_{\{i|j_{i}\neq0\}}p^{r-v_{i}}T_{i}^{1-j_{i}/p^{r}}\cdot Q_{i,r'}
\]
for appropriate $Q_{i,r'}\in D_{\mathcal{A}/p^{r'+1}}^{(0)}$; since
$\mathcal{A}/p^{r'+1}$ is etale over $W_{r'+1}(k)[T_{1},\dots,T_{n}]$
and both $\{\partial\}_{J/p^{r}}$ and ${\displaystyle \prod_{\{i|j_{i}\neq0\}}p^{r-v_{i}}T_{i}^{1-j_{i}/p^{r}}\cdot Q_{i,r'}}$
are differential operators, we deduce that they are equal on all of
$\mathcal{A}/p^{r'+1}$. Furthermore the image of $Q_{i,r'}$ in $D_{\mathcal{A}/p^{r'}}^{(0)}$
is $Q_{i,r'-1}$ since this is true for the polynomials $q_{i,r'}$.
Therefore, we have that the operator 
\[
\prod_{\{i|j_{i}\neq0\}}p^{r-v_{i}}T_{i}^{1-j_{i}/p^{r}}\cdot Q_{i}=\lim_{r'\to\infty}\prod_{\{i|j_{i}\neq0\}}p^{r-v_{i}}T_{i}^{1-j_{i}/p^{r}}\cdot Q_{i,r'}
\]
acts as $\{\partial\}_{J/p^{r}}$ on $\mathcal{A}$, which shows that
$\{\partial\}_{J/p^{r}}\circ\Phi$ agrees with an element of $\Phi^{*}\mathcal{D}_{\mathcal{A}}^{(0)}$. 

Now suppose, in addition, that $J=(j_{1},\dots,j_{n})$ has the property
that at least one $j_{i}$, say $j_{i'}$, satisfies $\text{val}(j_{i'})=0$.
Then we may consider the operator $F^{-r}(\alpha)\cdot\{\partial\}_{J/p^{r}}$,
and the previous discussion shows that this operator acts as 
\[
V^{r}(\alpha T^{p^{r}-j_{i'}})Q_{i'}\cdot\prod_{i\neq i'}p^{r-v_{i}}T_{i}^{1-j_{i}/p^{r}}\cdot Q_{i}\in\Phi^{*}\mathcal{D}_{\mathcal{A}}^{(0)}
\]
This implies further that, for any $Q\in\mathcal{D}_{\mathcal{A}}^{(0)}$,
the operator $F^{-r}(\alpha)\{\partial\}_{J/p^{r}}\circ\Phi\circ Q$
agrees with $V^{r}(\alpha T^{p^{r}-j_{i'}})Q_{i'}\cdot\prod_{i\neq i'}p^{r-v_{i}}T_{i}^{1-j_{i}/p^{r}}\cdot Q_{i}\cdot Q$,
as an operator on $\mathcal{A}$, and is therefore also contained
in $\Phi^{*}\mathcal{D}_{\mathcal{A}}^{(0)}$. Thus we see that, for
any sum ${\displaystyle \sum_{r=1}^{\infty}F^{-r}(\alpha)\cdot\{\partial\}_{J/p^{r}}}$,
we have 
\[
\sum_{r=1}^{\infty}F^{-r}(\alpha)\cdot\{\partial\}_{J/p^{r}}\circ\Phi\circ Q={\displaystyle \sum_{r=1}^{\infty}}V^{r}(\alpha T^{p^{r}-j_{i'}})Q_{i'}\cdot\prod_{i\neq i'}p^{r-v_{i}}T_{i}^{1-j_{i}/p^{r}}\cdot Q_{i}\cdot Q
\]
which is a convergent sum in $\Phi^{*}\mathcal{D}_{\mathcal{A}}^{(0)}$. 

Next consider any element of the form $\{\partial\}^{I}=\{\partial_{1}\}^{i_{1}}\cdots\{\partial_{n}\}^{i_{n}}$.
Suppose, inductively, that $\{\partial_{1}\}^{i_{1}-1}\cdots\{\partial_{n}\}^{i_{n}}\circ\Phi\in\Phi^{*}\mathcal{D}_{\mathcal{A}}^{(0)}$,
and write 
\[
\{\partial_{1}\}^{i_{1}-1}\cdots\{\partial_{n}\}^{i_{n}}\circ\Phi=\sum_{r=0}^{\infty}V^{r}(\alpha_{r})\cdot\Phi\circ Q_{r}
\]
for $Q_{r}\in\mathcal{D}_{\mathcal{A}}^{(0)}$. Then 
\[
\{\partial_{1}\}\sum_{r=0}^{\infty}V^{r}(\alpha_{r})\cdot\Phi\circ Q_{r}=\sum_{r=0}^{\infty}[\{\partial_{1}\},V^{r}(\alpha_{r})]\cdot\Phi\circ Q_{r}+\sum_{r=0}^{\infty}V^{r}(\alpha_{r})\{\partial\}_{1}\circ\Phi\circ Q_{r}
\]
and since ${\displaystyle [\{\partial_{1}\},V^{r}(\alpha_{r})]=\sum_{l=1}^{\infty}\sum_{j\leq p^{l}}F^{-l}(\beta_{jl})\{\partial\}_{j/p^{l}}}$
(as in the proof of \prettyref{eq:phi-times-Q}), the previous remarks
imply that $\{\partial\}^{I}\circ\Phi\in\Phi^{*}\mathcal{D}_{\mathcal{A}}^{(0)}$.
As above this implies directly that $\{\partial\}^{I}\circ\Phi\circ Q\in\Phi^{*}\mathcal{D}_{\mathcal{A}}^{(0)}$
for all $Q\in\mathcal{D}_{\mathcal{A}}^{(0)}$, and then, using the
representation of an element of $\widehat{D}_{W(A)}^{(0)}$ as in
\prettyref{thm:Basis}, the first result follows. 

Now we consider the second statement, again when $m=0$ ($m>0$ is
essentially identical). We may, by the fact that $\Phi^{*}\widehat{\mathcal{D}}_{\mathcal{A}}^{(0)}$
is $p$-adically complete, prove the analogous fact for the reduction
$\Phi^{*}\widehat{\mathcal{D}}_{\mathcal{A}}^{(0)}/p$; let us call
this object $\Phi^{*}\mathcal{D}_{A}^{(0)}$. We have, for each $1\leq i\leq n$,
that $\{\partial_{i}\}\cdot(1\otimes1)=1\otimes\partial_{i}$ inside
$\Phi^{*}\mathcal{D}_{A}^{(0)}$ (this follows directly from the above
description of the action of $\{\partial_{i}\}$). So we see that
the image of $\widehat{\mathcal{D}}_{W(A)}^{(0)}$ acting on $1\otimes1\in\Phi^{*}\mathcal{D}_{A}^{(0)}$
contains every element of the form $\alpha\otimes P$ for $P\in\mathcal{D}_{A}^{(0)}$.
As every element of $\Phi^{*}\mathcal{D}_{A}^{(0)}$ is a (convergent)
sum of such elements, the result follows. 
\end{proof}
Let us note here that this bimodule sheafifies well:
\begin{lem}
\label{lem:Bimodule-is-a-sheaf}Let $\mathcal{A}$ and $\Phi:\mathcal{A}\to W(A)$
be as above, and let $X=\text{Spec}(A)$, $\mathfrak{X}=\text{Specf}(\mathcal{A})$.
For each affine open $U\subset X$ we obtain, by complete localization
of the map $\Phi$, a map $\Phi:\mathcal{O}_{\mathfrak{X}}(U)\to W(\mathcal{O}_{X}(U))$,
and, therefore, a bimodule $\Phi^{*}\mathcal{O}_{\mathfrak{X}}(U)$.
Thus, regarding $\Phi$ as a morphism $\mathcal{O}_{\mathfrak{X}}\to W(\mathcal{O}_{X})$,
this sheaf is the sheaf of bi-sub-modules of $\mathcal{H}om_{W(k)}(\mathcal{O}_{\mathfrak{X}},\mathcal{O}_{W(X)})$
locally generated by $\Phi$. It is also the inverse limit of the
quasicoherent sheaves $(W_{r}(A)\otimes_{\mathcal{A}}\mathcal{D}_{\mathcal{A}_{r}}^{(0)})^{\sim}$
on $\text{Spec}(W_{r}(A))$, and hence quasicoherent on $W(X)$. 
\end{lem}

\begin{proof}
Everything is clear from the above except, perhaps, the quasicoherence
statement. For that, recall that, for a $\{\mathcal{F}_{n}\}$ a surjective
system of quasicoherent sheaves and $U\subset X$ any open affine,
we have ${\displaystyle \lim_{n}\mathcal{F}_{n}(U)\tilde{=}(\lim_{n}\mathcal{F}_{n})(U)}$
(by, e.g., \cite{key-8}, lemma 1.1.6). But for each such $U$ the
theorem above implies 
\[
\lim_{r}(W_{r}(A)\otimes_{\mathcal{A}}\mathcal{D}_{\mathcal{A}_{r}}^{(0)})^{\sim}(U)\tilde{=}\Phi^{*}\mathcal{O}_{\mathfrak{X}}(U)
\]
and the result follows.
\end{proof}
Before proceeding, we would like to give the crystalline version of
these results. Fix $r\geq1$, and consider $\mathfrak{X}_{r}$ a smooth
formal scheme over $W_{r}(k)$ with special fibre $X$. We recall
that the differential operators $\mathcal{D}_{\mathfrak{X}_{r}}^{(0)}$
posses a two sided ideal $\mathcal{I}_{r}$ consisting of sections
$P$ which annihilate $\mathcal{O}_{\mathfrak{X}_{r}}$. If $\mathfrak{X}_{r}$
is affine and possess local coordinates $\{T_{1},\dots,T_{n}\}$,
then this ideal is generated by sections of the form $\{p^{j}\partial_{i}^{p^{r-j}}\}$
for $1\leq i\leq n$ and $0\leq j\leq r-1$. The completion of $\mathcal{D}_{\mathfrak{X}_{r}}^{(0)}$
along the powers of $\mathcal{I}_{r}$ is denoted $\mathcal{D}_{\mathfrak{X}_{r},\text{crys}}^{(0)}$;
if $\mathfrak{X}_{r}=\text{Spec}(\mathcal{A}_{r})$ then we denote
$\mathcal{D}_{\mathcal{A}_{r},\text{crys}}^{(0)}$ the global sections
of $\mathcal{D}_{\mathfrak{X}_{r},\text{crys}}^{(0)}$. Taking the
inverse limit over $r$, we obtain the algebra $\mathcal{\widehat{D}}_{\mathcal{A},\text{crys}}^{(0)}$,
which is a further completion of $\widehat{\mathcal{D}}_{\mathcal{A}}^{(0)}$;
similarly we have the sheaf $\mathcal{\widehat{D}}_{\mathfrak{X},\text{crys}}^{(0)}$.
One can show without much difficulty that elements of this ring have
unique expressions of the form 
\[
\sum_{I}a_{I}\partial^{I}
\]
where $a_{I}\in\mathcal{A}$. Since $\partial^{I}\to0$ as $I\to\infty$
as operators on $\mathcal{A}$, this algebra acts on $\mathcal{A}$,
and the action is faithful. 

Now, repeating the proof of the above theorems essentially verbatim,
we obtain 
\begin{thm}
1) Let $\Phi^{*}\widehat{\mathcal{D}}_{\mathcal{A},\text{crys}}^{(0)}$
denote the completion of $W(A)\otimes_{\mathcal{A}}\widehat{\mathcal{D}}_{\mathcal{A},\text{crys}}^{(0)}$
along the filtration $V^{i}(W(A))\otimes_{\mathcal{A}}\widehat{\mathcal{D}}_{\mathcal{A},\text{crys}}^{(0)}$.
Then there is an embedding 
\[
\Phi^{*}\widehat{\mathcal{D}}_{\mathcal{A},\text{crys}}^{(0)}\to\mathcal{H}om_{W(k)}(\mathcal{A},W(A))
\]
which takes $1\otimes1$ to $\Phi$. This embedding preserves the
natural right $\widehat{\mathcal{D}}_{\mathcal{A},\text{crys}}^{(0)}$-module
structures on both sides. 

2) The submodule $\Phi^{*}\widehat{\mathcal{D}}_{\mathcal{A},\text{crys}}^{(0)}$
of $\mathcal{H}om_{W(k)}(\mathcal{A},W(A))$ is preserved under the
action of $\widehat{\mathcal{D}}_{W(A)}^{(0)}$. Thus $\Phi^{*}\widehat{\mathcal{D}}_{\mathcal{A},\text{crys}}^{(0)}$
has the structure of a $(\widehat{\mathcal{D}}_{W(A)}^{(0)},\widehat{\mathcal{D}}_{\mathcal{A},\text{crys}}^{(0)})$-bimodule.

3) Let $X=\text{Spec}(A)$, $\mathfrak{X}=\text{Specf}(\mathcal{A})$.
Then there is a $(\widehat{\mathcal{D}}_{W(X)}^{(0)},\widehat{\mathcal{D}}_{\mathfrak{X},\text{crys}}^{(0)})$,
which is an inverse limit of quasicoherent sheaves on $W(X)$.
\end{thm}

Now we want to compare two bimodules $\Phi^{*}\widehat{\mathcal{D}}_{\mathcal{A}}^{(m)}$
and $\Psi^{*}\widehat{\mathcal{D}}_{\mathcal{A}}^{(m)}$ for two morphisms
$\Phi,\Psi:\mathcal{A}\to W(A)$ coming from two lifts of Frobenius
on $\mathcal{A}$. In fact, assuming $p>2$ the situation is as nice
as possible. We begin by showing 
\begin{lem}
Let $\Phi$ and $\Psi$ be two morphisms $\mathcal{A}\to W(A)$, coming
from coordinatized lifts of Frobenius on $\mathcal{A}$. Make $W(A)$
into an $\mathcal{A}$-module via $\Phi$. Then, for any $n\geq1$,
the map $\delta_{n}:\mathcal{A}/p^{n}\to W(A)/p^{n}$, given by the
reduction mod $p^{n}$ of $\delta=\Phi-\Psi$, is a differential operator
of degree $\leq n$ over $\mathcal{A}$. 
\end{lem}

\begin{proof}
We start by noting the formula 
\[
\delta(ab)=a\delta(b)+b\delta(a)-\delta(a)\delta(b)
\]
for any $a,b\in\mathcal{A}$. Indeed, we have 
\[
\delta(ab)=\Phi(ab)-\Psi(ab)=\Phi(a)\Phi(b)-\Psi(a)\Psi(b)
\]
\[
=\Phi(a)\Phi(b)-\Phi(a)\Psi(b)+\Phi(a)\Psi(b)-\Psi(a)\Psi(b)
\]
\[
=\Phi(a)\delta(b)+\delta(a)\Psi(b)=\Phi(a)\delta(b)+\delta(a)(-\delta(b)+\Phi(b))
\]
\[
=a\delta(b)+b\delta(a)-\delta(a)\delta(b)
\]
as claimed. Now, the map $\delta$ takes values in the ideal $V(W(A))$,
as the reduction mod $V(W(A))$ of both $\Phi$ and $\Psi$ is, by
construction, the natural quotient map $\mathcal{A}\to A$. So, after
reduction to $W(A)/p$ we have $\delta(a)\delta(b)=0$; in other words,
$\delta_{1}:A\to W(A)/p$ is a derivation. 

Now we suppose $n\geq2$. For any $a\in\mathcal{A}/p^{n}$ we have
the operator $\delta_{n,a}:\mathcal{A}/p^{n}\to W(A)/p^{n}$ defined
by $\delta_{n,a}(b)=\delta_{n}(ab)-a\cdot\delta_{n}(b)$. Iterating,
we obtain for any sequence $(a_{1},\dots a_{r})$ in $\mathcal{A}/p^{n}$
an operator $\delta_{n,a_{1},\dots,a_{r}}:\mathcal{A}/p^{n}\to W(A)/p^{n}$
defined inductively by 
\[
\delta_{n,a_{1},\dots,a_{r}}(b)=\delta_{n,a_{1},\dots,a_{r-1}}(a_{r}\cdot b)-a_{r}\cdot\delta_{n,a_{1},\dots,a_{r-1}}(b)
\]
To show that $\delta_{n}$ is a differential operator of order $\leq n$
we must show that $\delta_{n,a_{1},\dots,a_{n+1}}=0$ for any sequence
$(a_{1},\dots a_{n+1})$ of length $n+1$. 

On the other hand, for the sequence $(a_{1},\dots,a_{r})$ we can
define the operator $\epsilon_{a_{1},\dots a_{r}}:\mathcal{A}/p^{n}\to W(A)/p^{n}$
\[
\epsilon_{a_{1},\dots a_{r}}(b)=\delta_{n}(a_{1})\cdots\delta_{n}(a_{r})\delta_{n}(b)
\]
I claim that for each $r\geq1$, the map
\[
\delta_{a_{1},\dots,a_{r}}+(-1)^{r+1}\epsilon_{a_{1},\dots a_{r}}:\mathcal{A}/p^{n}\to W(A)/p^{n}
\]
is an $\mathcal{A}/p^{n}$-linear map. We proceed by induction on
$r$; when $r=1$ this follows from 
\[
\delta_{n,a}(b):=\delta_{n}(ab)-a\cdot\delta_{n}(b)=b\cdot\delta_{n}(a)-\delta_{n}(a)\delta_{n}(b)
\]
Supposing the result holds for $r$, we have 
\[
\delta_{n,a_{1},\dots,a_{r+1}}(b):=\delta_{n,a_{1},\dots,a_{r}}(a_{r+1}\cdot b)-a_{r+1}\cdot\delta_{n,a_{1},\dots,a_{r}}(b)
\]
By induction, we have that 
\[
\delta_{n,a_{1},\dots,a_{r}}=(-1)^{r+1}\epsilon_{a_{1},\dots,a_{r}}+\varphi
\]
where $\varphi:\mathcal{A}/p^{n}\to W(A)/p^{n}$ is an $\mathcal{A}/p^{n}$-linear
map. So 
\[
\delta_{n,a_{1},\dots,a_{r}}(a_{r+1}\cdot b)=(-1)^{r+1}\epsilon_{a_{1},\dots,a_{r}}(a_{r+1}\cdot b)+a_{r+1}\varphi(b)
\]
and 
\[
a_{r+1}\cdot\delta_{n,a_{1},\dots,a_{r}}(b)=(-1)^{r+1}a_{r+1}\cdot\epsilon_{a_{1},\dots,a_{r}}(b)+a_{r+1}\varphi(b)
\]
So we see that 
\[
\delta_{n,a_{1},\dots,a_{r+1}}(b)=(-1)^{r+1}(\epsilon_{a_{1},\dots,a_{r}}(a_{r+1}\cdot b)-a_{r+1}\cdot\epsilon_{a_{1},\dots,a_{r}}(b))
\]
But 
\[
\epsilon_{a_{1},\dots,a_{r}}(a_{r+1}\cdot b)=\delta_{n}(a_{1})\cdots\delta_{n}(a_{r})\delta_{n}(a_{r+1}b)
\]
\[
=b\cdot\delta_{n}(a_{1})\cdots\delta_{n}(a_{r})\delta_{n}(a_{r+1})+a_{r+1}\cdot\delta_{n}(a_{1})\cdots\delta_{n}(a_{r})\delta_{n}(b)-\delta_{n}(a_{1})\cdots\delta_{n}(a_{r})\delta_{n}(a_{r+1})\delta_{n}(b)
\]
\[
=b\cdot\delta_{n}(a_{1})\cdots\delta_{n}(a_{r})\delta_{n}(a_{r+1})+a_{r+1}\epsilon_{a_{1},\dots,a_{r}}(b)-\epsilon_{a_{1},\dots,a_{r+1}}(b)
\]
so that 
\[
\epsilon_{a_{1},\dots,a_{r}}(a_{r+1}\cdot b)-a_{r+1}\cdot\epsilon_{a_{1},\dots,a_{r}}(b)=b\cdot\delta_{n}(a_{1})\cdots\delta_{n}(a_{r})\delta_{n}(a_{r+1})-\epsilon_{a_{1},\dots,a_{r+1}}(b)
\]
Therefore
\[
\delta_{n,a_{1},\dots,a_{r+1}}(b)=(-1)^{r+1}(\epsilon_{a_{1},\dots,a_{r}}(a_{r+1}\cdot b)-a_{r+1}\cdot\epsilon_{a_{1},\dots,a_{r}}(b))
\]
\[
=(-1)^{r+1}(b\cdot\delta_{n}(a_{1})\cdots\delta_{n}(a_{r})\delta_{n}(a_{r+1})-\epsilon_{a_{1},\dots,a_{r+1}}(b))
\]
\[
=(-1)^{r+2}\epsilon_{a_{1},\dots,a_{r+1}}(b)+(-1)^{r+1}b\cdot\delta_{n}(a_{1})\cdots\delta_{n}(a_{r})\delta_{n}(a_{r+1})
\]
as required. 

Finally, to finish the proof, we note that 
\[
\epsilon_{a_{1},\dots a_{n}}(b)=\delta_{n}(a_{1})\cdots\delta_{n}(a_{n})\delta_{n}(b)=0
\]
for all $b$, as the product of $n+1$ elements of the form $V(\alpha)$
(in $W(A)$) is contained in the ideal generated by $p^{n}$. Thus
$\delta_{n,a_{1},\dots,a_{n}}$ is a linear operator, and the result
follows. 
\end{proof}
Using this, we can prove 
\begin{thm}
\label{thm:Uniqueness-of-bimodule}Let $\Phi$ and $\Psi$ be two
morphisms $\mathcal{A}\to W(A)$, coming from coordinatized lifts
of Frobenius on $\mathcal{A}$.

1) Suppose $p>2$. Then there is a canonical isomorphism of $(\widehat{\mathcal{D}}_{W(A)}^{(0)},\widehat{\mathcal{D}}_{\mathcal{A}}^{(0)})$-bimodules
$\epsilon_{\Phi,\Psi}:\Phi^{*}\widehat{\mathcal{D}}_{\mathcal{A}}^{(0)}\tilde{\to}\Psi^{*}\widehat{\mathcal{D}}_{\mathcal{A}}^{(0)}$.
If $\Xi$ is a third such lift then we have the cocycle condition
$\epsilon_{\Xi,\Phi}\circ\epsilon_{\Phi,\Psi}=\epsilon_{\Xi,\Psi}$. 

2) For any $p$, there is a canonical isomorphism of $(\widehat{\mathcal{D}}_{W(A),\text{crys}}^{(0)},\widehat{\mathcal{D}}_{\mathcal{A},\text{crys}}^{(0)})$-bimodules
$\epsilon_{\Phi,\Psi}:\Phi^{*}\widehat{\mathcal{D}}_{\mathcal{A},\text{crys}}^{(0)}\tilde{\to}\Psi^{*}\widehat{\mathcal{D}}_{\mathcal{A},\text{crys}}^{(0)}$.
If $\Xi$ is a third such lift then we have the cocycle condition
$\epsilon_{\Xi,\Phi}\circ\epsilon_{\Phi,\Psi}=\epsilon_{\Xi,\Psi}$.
When $p>2$ this isomorphism is a completion of the one constructed
in part $1)$ above. 

3) Let $m>0$. Then for all $p$ there is a canonical isomorphism
of $(\widehat{\mathcal{D}}_{W(A)}^{(m)},\widehat{\mathcal{D}}_{\mathcal{A}}^{(m)})$-bimodules
$\epsilon_{\Phi,\Psi}:\Phi^{*}\widehat{\mathcal{D}}_{\mathcal{A}}^{(m)}\tilde{\to}\Psi^{*}\widehat{\mathcal{D}}_{\mathcal{A}}^{(m)}$;
we again have the cocycle condition for a third such lift.
\end{thm}

\begin{proof}
1) By construction (c.f. \prettyref{prop:Construction-of-bimodule}),
we have canonical inclusions $\Psi^{*}\widehat{\mathcal{D}}_{\mathcal{A}}^{(0)},\Phi^{*}\widehat{\mathcal{D}}_{\mathcal{A}}^{(0)}\subset\text{Hom}_{W(k)}(\mathcal{A},W(A))$.
We shall show that, when $p>2$, these inclusions have the same image;
this immediately proves part $1)$. 

To proceed, we shall show that $\Psi:\mathcal{A}\to W(A)$ is contained
in the image of $\Phi^{*}\widehat{\mathcal{D}}_{\mathcal{A}}^{(0)}$
in $\text{Hom}_{W(k)}(\mathcal{A},W(A))$. This implies directly that
the image of $\Psi^{*}\widehat{\mathcal{D}}_{\mathcal{A}}^{(0)}$
is contained in image of $\Phi^{*}\widehat{\mathcal{D}}_{\mathcal{A}}^{(0)}$.
By symmetry, we also have the reverse inclusion, and hence the required
equality. 

Since $\Phi:\mathcal{A}\to W(A)$ is contained in the image of $\Phi^{*}\widehat{\mathcal{D}}_{\mathcal{A}}^{(0)}$
by definition, it suffices to show that $\delta=\Phi-\Psi$ is contained
in the image of $\Phi^{*}\widehat{\mathcal{D}}_{\mathcal{A}}^{(0)}$.
By the previous lemma, the reduction $\delta_{n}:\mathcal{A}/p^{n}\to W(A)/p^{n}$
is a differential operator of order $\leq n$ (as in that lemma, we
regard $W(A)$ as being a $\mathcal{A}$-module via $\Phi$). Therefore
it suffices to construct a sequence of operators $(\gamma_{n})$ in
the image of $\Phi^{*}\widehat{\mathcal{D}}_{\mathcal{A}}^{(0)}$
in $\text{Hom}_{W(k)}(\mathcal{A},W(A))$, such that $\gamma_{n}-\delta_{n}=0$
on $\mathcal{A}/p^{n}$ and such that $(\gamma_{n})$ is convergent
in $\Phi^{*}\widehat{\mathcal{D}}_{\mathcal{A}}^{(0)}$. Then $\gamma:=\lim\gamma_{n}$
will agree with $\delta$ and the result will follow. 

We construct $\gamma_{n}$ by induction on $n$. We know that $\delta_{1}:A\to W(A)/p$
is a derivation. So, we can write 
\[
\delta_{1}=\sum_{i=1}^{n}\sum_{(I,r)}\bar{a}_{I,i}(\overline{p^{r}T^{I/p^{r}}})\cdot\partial_{i}
\]
where the second sum is over pairs $(I,r)$ where $I$ is a multi-index,
each of whose entries is $<p^{r}$, and at least one entry of which
is coprime to $p$; and $\bar{a}_{I,i}\in A$ (as usual we write $p^{r}T^{I/p^{r}}\in W(A)$
for $V^{r}(T^{I})$). Choosing lifts of the $\bar{a}_{I,i}$ to $a_{I,i}\in\mathcal{A}$,
we obtain a derivation $\gamma_{1}:\mathcal{A}\to W(A)$
\[
\gamma_{1}=\sum_{i=1}^{n}\sum_{(I,r)}a_{I,i}p^{r}T^{I/p^{r}}\partial_{i}
\]
so $\gamma_{1}-\delta_{1}=0$ on $\mathcal{A}/p$, and $\gamma_{1}\in\Phi^{*}\widehat{\mathcal{D}}_{\mathcal{A}}^{(0)}$
by construction. 

Suppose we have constructed $(\gamma_{1},\gamma_{2},\dots,\gamma_{n})$.
Then the operator $\gamma_{n}-\delta:\mathcal{A}\to W(A)$ takes values
in $p^{n}W(A)$. Looking at the reduction mod $p^{n+1}$ of this operator,
we obtain a map 
\[
\gamma_{n}-\delta:\mathcal{A}/p^{n+1}\to W(A)/p^{n+1}W(A)
\]
Since $\gamma_{n}\in\Phi^{*}\widehat{\mathcal{D}}_{\mathcal{A}}^{(0)}$
has order $\leq n$ by induction, and $\delta_{n+1}$ is a differential
operator of order $\leq n+1$, we see that this map takes the form
\[
\gamma_{n}-\delta=\sum_{|J|\leq n+1}\sum_{(I,r)}p^{n}\bar{a}_{I,J}(\overline{p^{r}T^{I/p^{r}}})\partial^{[J]}
\]
where the notation is as above. Now, 
\[
\partial^{[J]}=\frac{\partial_{1}^{j_{1}}}{j_{1}!}\cdots\frac{\partial_{d}^{j_{d}}}{j_{d}!}
\]
with $j_{1}+\dots+j_{d}\leq n+1$. Recalling the formula 
\[
\text{val}(j!)=\frac{j-s_{p}(j)}{p-1}
\]
where $s_{p}(j)$ is the sum of the digits in the $p$-adic expansion
of $j$, we see that 
\[
\text{val}_{p}(j_{1}!\cdots j_{d}!)=\sum_{i=1}^{d}\text{val}_{p}(j_{i}!)=\sum_{i=1}^{d}\frac{j_{i}-s_{p}(j_{i})}{p-1}\leq\frac{n}{p-1}
\]
where in the last inequality we used that $j_{1}+\dots+j_{d}\leq n+1$.
Therefore we have 
\[
p^{n}\partial^{[J]}=\alpha\cdot p^{n-\text{val}_{p}(j_{1}!\cdots j_{d}!)}\partial^{J}
\]
for some $\alpha\in\mathbb{Z}_{p}$. So we have 
\[
\gamma_{n}-\delta=\sum_{|J|\leq n+1}\sum_{I}p^{n-\text{val}_{p}(j_{1}!\cdots j_{d}!)}\alpha a_{I,J}(p^{r}T^{I/p^{r}})\partial^{J}
\]
and we can thus define $\gamma_{n+1}$ as 
\[
\gamma_{n}-\sum_{|J|\leq n+1}\sum_{I}p^{n-\text{val}_{p}(j_{1}!\cdots j_{d}!)}\alpha a_{I,J}(p^{r}T^{I/p^{r}})\partial^{J}
\]
which makes sense as a map $\mathcal{A}\to W(A)$ (and is clearly
an element of $\Phi^{*}\widehat{\mathcal{D}}_{\mathcal{A}}^{(0)}$). 

Further, since $p>2$, we see that 
\[
n-\text{val}_{p}(j_{1}!\cdots j_{d}!)\to\infty
\]
as $n\to\infty$. Thus we see that, in this case, the sequence $(\gamma_{n})$
is convergent inside $\Phi^{*}\widehat{\mathcal{D}}_{\mathcal{A}}^{(0)}$
as claimed. 

2) The proof proceeds identically to that of part $1)$, until we
need to evaluate $n-\text{val}_{p}(j_{1}!\cdots j_{d}!)$. At that
point, we simply have the estimate 
\[
\text{val}_{p}(j_{1}!\cdots j_{d}!)\leq n
\]
(indeed, they can be equal in characteristic $2$). Thus the sequence
of operators $(\gamma_{n})$ may not converge in $\Phi^{*}\widehat{\mathcal{D}}_{\mathcal{A}}^{(0)}$,
but only in $\Phi^{*}\widehat{\mathcal{D}}_{\mathcal{A},\text{crys}}^{(0)}$.
Thus we obtain the isomorphism $\Phi^{*}\widehat{\mathcal{D}}_{\mathcal{A},\text{crys}}^{(0)}\tilde{\to}\Psi^{*}\widehat{\mathcal{D}}_{\mathcal{A},\text{crys}}^{(0)}$
as desired. 

3) We again repeat the proof, but note that we can now use operators
of the form 
\[
\sum_{|J|\leq n+1}\sum_{I}p^{n-\text{val}_{p}(j_{1}!\cdots j_{d}!)}\alpha a_{I,J}(p^{r}T^{I/p^{r}})(\frac{\partial^{p^{m}}}{p^{m}!})^{J}
\]
and when $m>1$ obtain the required convergence even when $p=2$. 
\end{proof}
Later on, we shall use this result to explain the relationship between
$\mathcal{D}_{\mathfrak{X}_{r}}^{(m)}$-modules on a smooth scheme
$\mathfrak{X}_{r}$ over $W(k)/p^{r}$ and modules over $\widehat{\mathcal{D}}_{W(X)}^{(0)}/p^{r}$.
To that end, we shall record the following result, which follows directly
from the construction: 
\begin{cor}
\label{cor:Iso-mod-p^n}Let $m\geq0$, and let $p>2$ if $m=0$. Let
$\Phi$ and $\Psi$ be two morphisms $\mathcal{A}\to W(A)$, coming
from coordinatized lifts of Frobenius on $\mathcal{A}$. Suppose that
these maps agree after reduction mod $p^{r}$. Then the reduction
mod $p^{r}$ of the canonical isomorphism $\epsilon_{\Phi,\Psi}:\Phi^{*}\widehat{\mathcal{D}}_{\mathcal{A}}^{(m)}/p^{r}\tilde{\to}\Psi^{*}\widehat{\mathcal{D}}_{\mathcal{A}}^{(m)}/p^{r}$
constructed above agrees with the obvious identification $\Phi^{*}(\widehat{\mathcal{D}}_{\mathcal{A}}^{(m)}/p^{r})\tilde{\to}\Psi^{*}(\widehat{\mathcal{D}}_{\mathcal{A}}^{(m)}/p^{r})$
coming from the fact that $\Phi=\Psi$ on $\mathcal{A}/p^{n}$. 
\end{cor}

We should also remark here on the relation between these constructions
and some facts from the usual crystalline cohomology theory. Recall
that the algebra $\mathcal{D}_{\mathcal{A}_{r},\text{crys}}^{(0)}$
is a major player in that story, under the name of the HPD differential
operators. They are constructed as the dual of the divided power envelope
of the diagonal of $\mathfrak{X}_{r}$ (c.f. \cite{key-10}, chapter
4 for details). 

We consider the two maps 
\[
\Phi,\Psi:\mathcal{A}_{r}\to W(A)/p^{r}
\]
Via the standard divided power structure on $W(A)$, both of these
maps may be considered as (inverse limits of) divided power thickenings.
Therefore the theory of the crystalline site (or, in this case, the
description of $\mathcal{D}_{\mathcal{A}_{r},\text{crys}}^{(0)}$
as the HPD differential operators) yields a canonical isomorphism
\[
\eta_{r}:\Phi^{*}\mathcal{D}_{\mathcal{A}_{r},\text{crys}}^{(0)}\tilde{\to}\Psi^{*}\mathcal{D}_{\mathcal{A}_{r},\text{crys}}^{(0)}
\]
Taking the inverse limit over $r$yields an isomorphism 
\[
\eta:\Phi^{*}\widehat{\mathcal{D}}_{\mathcal{A},\text{crys}}^{(0)}\tilde{\to}\Psi^{*}\widehat{\mathcal{D}}_{\mathcal{A},\text{crys}}^{(0)}
\]
and we have 
\begin{cor}
\label{cor:Crystalline-Iso}The isomorphism $\eta$ agrees with the
isomorphism $\epsilon$ constructed above in \prettyref{thm:Uniqueness-of-bimodule}. 
\end{cor}

\begin{proof}
This will follow if we know that the maps $\Phi^{*}\widehat{\mathcal{D}}_{\mathcal{A},\text{crys}}^{(0)}\to\text{Hom}_{W(k)}(\mathcal{A},W(A))$
and $\Psi^{*}\widehat{\mathcal{D}}_{\mathcal{A},\text{crys}}^{(0)}\to\text{Hom}_{W(k)}(\mathcal{A},W(A))$
coming from the action of differential operators on $\mathcal{A}$
are compatible with $\eta$. But the analogous fact is true for $\eta_{r}$,
simply from the construction of $\eta_{r}$ which is given by factoring
$\Phi\times\Psi$ through a divided power neighborhood of the diagonal
(c.f. \cite{key-10}, chapters 3 and 4). Taking the inverse limit
over $r$ we see that this is true for $\eta$ as well. 
\end{proof}
To finish off this section, we shall develop another fundamental property
of the bimodules $\Phi^{*}\widehat{\mathcal{D}}_{\mathcal{A}}^{(m)}$,
namely 
\begin{thm}
\label{thm:Projective!} 1) As a right $\widehat{\mathcal{D}}_{\mathcal{A}}^{(m)}$-module,
$\Phi^{*}\widehat{\mathcal{D}}_{\mathcal{A}}^{(m)}$ is faithfully
flat. 

2) Let $m\geq0$. As a left $\widehat{\mathcal{D}}_{W(A)}$-module,
$\Phi^{*}\widehat{\mathcal{D}}_{\mathcal{A}}^{(m)}$ is projective
and cyclic; in particular, it is a summand of $\widehat{\mathcal{D}}_{W(A)}^{(m)}$
itself. 

The analogous statements hold for $\Phi^{*}\widehat{\mathcal{D}}_{\mathcal{A},\text{crys}}^{(m)}$
as a $(\widehat{\mathcal{D}}_{W(A),\text{crys}}^{(m)},\Phi^{*}\widehat{\mathcal{D}}_{\mathcal{A},\text{crys}}^{(m)})$
bimodule. 
\end{thm}

This result will allow us, in the next section, to define categories
of $\widehat{\mathcal{D}}_{W(X)}^{(m)}$-modules which are surprisingly
well behaved. 
\begin{proof}
(of part $1)$ As $\Phi^{*}\widehat{\mathcal{D}}_{\mathcal{A}}^{(m)}$
is $p$-torsion free and $p$-adically complete, by \cite{key-8},
corollary 1.6.7 it suffices to show that $\Phi^{*}\widehat{\mathcal{D}}_{\mathcal{A}}^{(m)}/p$
is faithfully flat over $\mathcal{D}_{A}^{(m)}$. However, $\Phi^{*}\widehat{\mathcal{D}}_{\mathcal{A}}^{(m)}/p$
is an inverse limit of free modules over $\mathcal{D}_{A}^{(m)}$,
namely modules of the form 
\[
\bigoplus_{I<p^{r}}V^{r}(T^{I})\cdot\mathcal{D}_{A}^{(m)}
\]
where $I=(i_{1},\dots,i_{n})$ is a multi-index and $I<p^{r}$ means
that each $i_{j}<p^{r}$. As the maps in this inverse system are surjective,
one deduces immediately that $\Phi^{*}\widehat{\mathcal{D}}_{\mathcal{A}}^{(m)}/p$
is faithfully flat over $\mathcal{D}_{A}^{(m)}$ as needed. An identical
proof works for the crystalline version of the theory. 
\end{proof}
To prove part $2)$, we will have to work a bit harder. We start with
the special case in which $A=k[T]$ ; the lift of this algebra is
given by $W(k)<<T>>$ and we shall use the standard Frobenius lift
$T\to T^{p}$. In this context we write ${\displaystyle d=\frac{d}{dt}}$.
For $i\in\{0,\dots,p^{r}-1\}$, we write $\pi_{(i,p^{r})}:k[T]\to k[T]$
for the $k$-linear projection operator which takes $T^{i+ap^{r}}$
to itself (for all $a$) and which takes all other monomials to zero.
Then
\begin{lem}
\label{lem:The-operator-=00005Cpi-r}The operator $\pi_{(i,p^{r})}$
is a differential operator of order $<p^{r}$, namely 
\[
\pi_{(i,p^{r})}=T^{i}d^{[p^{r}-1]}T^{p^{r}-1-i}=\sum_{l=i}^{p^{r}-1}{p^{r}-1-i \choose p^{r}-1-l}T^{l}\cdot d^{[l]}
\]
\end{lem}

\begin{proof}
We begin with the case $i=0$. Then, for any $a\geq0$ we have
\[
d^{[p^{r}-1]}T^{p^{r}-1}(T^{p^{r}a})={p^{r}a+p^{r}-1 \choose p^{r}-1}T^{p^{r}a}
\]
But 
\[
{p^{r}a+p^{r}-1 \choose p^{r}-1}=\prod_{m=1}^{p^{r}-1}\frac{p^{r}a+m}{m}=\prod_{m=1}^{p^{r}-1}(1+a\frac{p^{r}}{m})=1
\]
since $\text{val}_{p}(m)<r$ for all terms in the product; so we see
$d^{[p^{r}-1]}T^{p^{r}-1}(T^{p^{r}a})=T^{p^{r}a}$. On the other hand,
for any natural number of the form $p^{r}a+b$ where $1\leq b\leq p^{r}-1$,
we have 
\[
d^{[p^{r}-1]}T^{p^{r}-1}(T^{p^{r}a+b})=T^{p^{r}(a+1)}\cdot d^{[p^{r}-1]}(T^{b-1})=0
\]
where we have used the well known identity (over $k$) $[d^{[p^{r}-1]},T^{p^{r}(a+1)}]=0$.
This proves the result when $i=0$. 

Now we consider the general case. Working for a moment over the ring
$k[T,T^{-1}]$, it follows from the case $i=0$ that the operator
$\pi_{(i,p^{r})}$ has the form \\
${\displaystyle T^{i-p^{r}}(d^{(p^{r}-1)}T^{p^{r}-1})T^{p^{r}-i}}$.
Now write $2p^{r}-1-i=p^{r}+b$ for $b\in\{0,\dots,p^{r}-2\}$. Then
\[
T^{i-p^{r}}(d^{[p^{r}-1]}T^{p^{r}-1})T^{p^{r}-i}=T^{i-p^{r}}d^{[p^{r}-1]}T^{p^{r}+b}=T^{i}d^{[p^{r}-1]}T^{b}
\]
is the operator of the required form, where in the last equality we
have used $[d^{[p^{r}-1]},T^{p^{r}}]=0$. 

Finally to prove the equality 
\[
T^{i}d^{[p^{r}-1]}T^{p^{r}-1-i}=\sum_{l=i}^{p^{r}-1}{p^{r}-1-i \choose p^{r}-1-l}T^{l}\cdot d^{[l]}
\]
we note that by the Hasse-Schmidt property we have
\[
[d^{[p^{r}-1]},T^{p^{r}-1-i}]=\sum_{l=0}^{p^{r}-2}d^{[p^{r}-1-l]}(T^{p^{r}-1-i})\cdot d^{[l]}
\]
\[
=\sum_{l=0}^{p^{r}-2}{p^{r}-1-i \choose p^{r}-1-l}T^{l-i}\cdot d^{[l]}
\]
from which the equality follows directly. 
\end{proof}
Our goal is to study the actions of these operators when lifted to
$W_{r+1}(A)$; we start with the 
\begin{defn}
Let $f:\mathbb{Z}_{\geq0}\to\mathbb{Z}/p^{r+1}\mathbb{Z}$ be a function,
and let $j>0$ be some natural number. We say that $f$ is polynomial
in arithmetic progressions (of length $j$) if, for each $1\leq a\leq j$,
there is a polynomial $p_{a}\in(\mathbb{Z}/p^{r+1}\mathbb{Z})[x]$
such that $f(a+bj)=p_{a}(b)$ for all $b\in\mathbb{Z}_{\geq0}$. 
\end{defn}

We note that if $f$ is polynomial in arithmetic progressions (of
length $j$), then $f$ is also polynomial in arithmetic progressions
of length $j'$ for all $j'$ which are divisible by $j$. 

Then the result we need is 
\begin{prop}
\label{prop:polynomial-in-AP}Let $1\leq i<p^{n}$. Fix $r\geq0$.
Then each operator of the form $T^{i}d^{[i]}$ on $W_{r+1}(k)[T]$
takes the form 
\[
T^{i}d^{[i]}(T^{a})=f(a)T^{a}
\]
where $f:\mathbb{Z}_{\geq0}\to\mathbb{Z}/p^{r+1}\mathbb{Z}$ is polynomial
in arithmetic progressions of length $p^{n-1}$. 
\end{prop}

\begin{proof}
We begin with the case of the operator $T^{p^{m}}d^{[p^{m}]}$ for
$m<n$. Fix $a\in\{0,\dots,p^{m}-1\}$. In this case, we have 
\[
T^{p^{m}}d^{[p^{m}]}(T^{a+bp^{m}})={a+bp^{m} \choose p^{m}}T^{a+bp^{m}}
\]
We shall show that ${\displaystyle {a+bp^{m} \choose p^{m}}}$ is
polynomial in $b$ with coefficients in $\mathbb{Z}_{p}$, and therefore
$f$ is arithmetic progressions of length $p^{m}$ (as remarked above,
this implies that it is also polynomial in arithmetic progressions
of length $p^{n-1}$). We have 
\[
{a+bp^{m} \choose p^{m}}=\prod_{l=1}^{p^{m}}\frac{a+(b-1)p^{m}+l}{l}
\]
For each $l\in\{1,\dots,p^{m}\}$ set 
\[
a+l=t+\epsilon_{l}p^{m}
\]
where $t\in\{1,\dots,p^{m}\}$ and $\epsilon_{t}\in\{0,1\}$. Note
that, as $l$ ranges over $\{1,\dots,p^{m}\}$, $t$ also ranges over
$\{1,\dots,p^{m}\}$. Therefore we have 
\[
\prod_{l=1}^{p^{m}}\frac{a+(b-1)p^{m}+l}{l}=\prod_{t=1}^{p^{m}}\frac{(b-1)p^{m}+t+\epsilon_{l}p^{m}}{t}
\]
\[
=\prod_{t=1}^{p^{m}}\frac{(b-1+\epsilon_{t})p^{m}+t}{t}=b\prod_{t=1}^{p^{m}-1}\frac{(b-1+\epsilon_{t})p^{m}+t}{t}
\]
Now, in the last expression, we have $\text{val}_{p}(t)=\text{val}_{p}((b-1+\epsilon_{t})p^{m}+t)$.
It follows that this product is a polynomial in $b$ with coefficients
in $\mathbb{Z}_{p}$ as desired. 

In order to handle the general case, we recall that, if $i<p^{n}$
has $p$-adic expansion ${\displaystyle i=\sum_{j=0}^{m}a_{j}p^{j}}$,
then we have 
\[
d^{[i]}=u\cdot(d^{[p^{j}]})^{a_{j}}\cdots(d)^{a_{0}}
\]
where $u$ is a unit in $\mathbb{Z}_{p}$. From this, and the standard
relation 
\[
d^{[m]}T^{m}\cdot d^{[n]}T^{n}=
\]
\[
\sum_{l=0}^{m}{n \choose m-l}{l+n \choose l}T^{(l+n)}d^{[l+n]}
\]
we see that $T^{i}d^{[i]}$ is itself a polynomial (with coefficients
in $\mathbb{Z}_{p}$) is the operators $\{T^{p^{m}}d^{[p^{m}]}\}_{m=0}^{p^{n-1}}$,
and the result follows. 
\end{proof}
\begin{cor}
Let $\Xi:k[T]\to k[T[$ be an operator which satisfies $\Xi(T^{a})=f(a)T^{a}$,
where $f:\mathbb{Z}_{\geq0}\to\mathbb{Z}/p\mathbb{Z}$ is polynomial
in arithmetic progressions of length $p^{n-1}$. Then $\Xi$ is given
by a differential operator of the form 
\[
\Xi=\sum_{i=0}^{p^{n}-1}g_{i}(T^{i}d^{[i]})
\]
where $g_{i}$ are polynomials in $\mathbb{Z}/p\mathbb{Z}[T]$. 
\end{cor}

\begin{proof}
By (the proof of) the previous proposition, we have that 
\[
T^{p^{m}}d^{[p^{m}]}T^{(a+bp^{m})}=b\cdot T^{(a+bp^{m})}
\]
for each $0\leq a<p^{m}-1$. Thus, if $g$ is any polynomial in $\mathbb{Z}/p\mathbb{Z}$,
the operator $g(T^{p^{m}}d^{[p^{m}]})$ takes $T^{(a+bp^{m})}$ to
$g(b)T^{a+bp^{m}}$. Furthermore, the projector $\pi_{(a,p^{m})}$
is an operator of the form ${\displaystyle \sum_{i=0}^{p^{m}-1}c_{i}T^{i}d^{[i]}}$
for $c_{i}\in\mathbb{Z}$ (by \prettyref{lem:The-operator-=00005Cpi-r}).
So the composition $g(T^{p^{m}}d^{[p^{m}]})\cdot\pi_{a/p^{m}}$ acts
with eigenvalue $g(b)$ on each element of the form $T^{a+bp^{m}}$
and is zero on all other monomials; as above we see that this product
is a polynomial in the $\{T^{i}d^{[i]}\}$. Adding up such operators
over $a\in\{0,\dots,p^{m}-1\}$ yields the result. 
\end{proof}
Now, for any $r\geq0$, consider for any $i\in\{0,\dots,p^{r}-1\}$
the projection operator $\pi_{i/p^{r}}$ on $W_{r+1}(A)$ which takes
$p^{m}T^{a/p^{m}}$ to itself whenever $p^{r-m}a\equiv i\phantom{i}\text{mod}\phantom{i}p^{r}$
and $0$ otherwise. Then we have
\begin{prop}
\label{prop:construction-of-projector}Let $r\geq0$. For each $i\in\{1,\dots,p^{r}-1\}$
the projector $\pi_{i/p^{r}}$ (on $W_{r+1}(A)$) is given by the
action of an element in $\mathcal{D}_{W(A)}^{(0)}$, of the form 
\[
\psi_{i/p^{r}}=\sum_{i=1}^{p^{r}-1}g_{i}(T^{i/p^{r}}\{d\}_{i/p^{r}})
\]
where $g_{i}\in\mathbb{Z}_{p}[x]$. We have that $\psi_{i/p^{r}}\in V^{r-\text{val}_{p}(i)}(\mathcal{D}_{W(A)})$.
Therefore, the element 
\[
\pi_{r}=1-\sum_{i=1}^{p^{r}-1}\psi_{i/p^{r}}
\]
acts as a projector from $W_{r+1}(A)$ to $\mathcal{A}_{r}$. Further,
one has $\pi_{r}\equiv\pi_{r-1}\phantom{i}\text{mod}\phantom{i}V^{r}(\mathcal{\widehat{D}}_{W(A)}^{(0)})$.
Therefore the sequence $\{\pi_{r}\}$, approaches a limit, $\pi\in\mathcal{\widehat{D}}_{W(A)}^{(0)}$,
which acts as a projector from $W(A)$ to $\mathcal{A}$. 
\end{prop}

\begin{proof}
We proceed by induction on $r$, the case $r=0$ being trivial. Suppose
the theorem holds at level $r-1$. Consider any $i\in\{1,\dots,p^{r}-1\}$
with $\text{val}(i)=0$. We will work now with the copy of $W_{r+1}(A^{(r)})$
contained in $\mathcal{A}_{r+1}=W_{r+1}(k)[T]$ and construct the
projector there. 

The projector $\pi_{(i,p^{r})}$ on $k[T]$ is given by an expression
of the form ${\displaystyle \sum_{j=0}^{p^{r}-1}c_{j}T^{j}d^{[j]}}$
(for $c_{j}\in\mathbb{Z}$); where, since $\text{val}(i)=0$, we have
$c_{j}=0$ whenever $\text{val}(j)\neq0$ (to see this, note that
the restriction of the operator to the subalgebra $k[T^{p}]$ is $0$).
Now, consider the action of the term ${\displaystyle \sum_{j=0}^{p^{r}-1}c_{j}T^{j}d^{[j]}}$
on $W_{r+1}(A^{(r)})\subset\mathcal{A}_{r+1}$. As $\text{val}(j)=0$
for all nonzero $c_{j}$, this operator takes every term of the form
$p^{m}T^{ap^{r-m}}$ to a term of the form $f_{m}(a)\cdot p^{r}T^{ap^{r-m}}$
for some $f_{m}(a)\in\mathbb{Z}/p\mathbb{Z}$. To analyze the functions
$f_{m}$, we shall apply \prettyref{prop:polynomial-in-AP}. In particular,
we see that $a\to f_{m}(a)$ is polynomial on arithmetic progressions
of length $p^{m}$, and is equal to zero on all terms of the form
$a=pa'$ (when $m>0$). 

Applying the previous result, we see that there is an operator of
the form 
\[
\sum_{i=0}^{p^{r+1}-1}g_{i,.m}(p^{r-\text{val}(i)}T^{i}d^{[i]})
\]
where $g_{i,m}\in\mathbb{Z}_{p}[X]$, which takes $p^{m}T^{ap^{r-m}}\to f_{m}(a)\cdot p^{r}T^{ap^{r-m}}$
for all $a$. By construction this operator is contained in $V_{r}(\mathcal{D}_{W(A)}^{(0)})$,
and we have 
\[
\psi_{i/p^{r}}=\sum_{j=0}^{p^{r}-1}c_{j}T^{j}d^{[j]}-\sum_{m=0}^{r-1}\sum_{i=0}^{p^{r+1}-1}g_{i}(p^{r-\text{val}(i)}T^{i}d^{[i]})
\]
Now suppose $\text{val}(i)>0$. Write $i=pi'$. By induction we have
already constructed the operator $\psi_{i'/p^{r-1}}$ on $W_{r}(A)$,
and it is an operator of the form ${\displaystyle \sum_{i=0}^{p^{r}-1}h_{i}(T^{i/p^{r-1}}\{d\}_{i/p^{r-1}})}$
for polynomials $h_{i}\in\mathbb{Z}_{p}[X]$. Translating to $W_{r}(A^{(r-1)})\subset W_{r}(k)[T]$
it is an operator of the form ${\displaystyle {\displaystyle \sum_{i=0}^{p^{r}-1}h_{i}(T^{i}d^{[i]})}}$.
Lifting to $W_{r+1}(A^{(r))})\subset W_{r+1}(k)[T]$ yields the operator
${\displaystyle \sum_{i=0}^{p^{r}-1}h_{i}(T^{pi}d^{[pi]})}$. By construction,
the operator 
\[
\psi_{i/p^{r}}-\sum_{i=0}^{p^{r}-1}h_{i}(T^{pi}d^{[pi]})
\]
takes each term of the form $p^{m}T^{ap^{r-m}}$ to a term of the
form $f_{m}(a)\cdot p^{r}T^{ap^{r-m}}$ for some $f_{m}(a)\in\mathbb{Z}/p\mathbb{Z}$
which is polynomial on arithmetic progressions of length $p^{m}$
(again by \prettyref{prop:polynomial-in-AP}). Arguing as above, we
find a term in $V_{r}(\mathcal{D}_{W(A)}^{(0)})$ which acts as ${\displaystyle \psi_{i/p^{r}}-\sum_{i=0}^{p^{r}-1}h_{i}(T^{pi}d^{[pi]})}$
on $W_{r+1}(A^{(r)})$, and thus construct $\psi_{i/p^{r}}$ as required. 
\end{proof}
Now we can give the 
\begin{cor}
\label{cor:Projector!}Let $A$ be any smooth affine algebra which
is equipped with local coordinates, and let $\Phi:\mathcal{A}\to W(A)$
be the map coming from a coordinatized Frobenius lift. Then there
is an element $\pi\in\mathcal{\widehat{D}}_{W(A)}^{(m)}$ which acts
as a projector onto $\Phi(\mathcal{A})$, i.e., $\pi^{2}=\pi$ and
$\pi(W(A))=\Phi(\mathcal{A})$. 

Further, the map $\mathcal{\widehat{D}}_{W(A)}^{(m)}\to\Phi^{*}\mathcal{\widehat{D}}_{\mathcal{A}}^{(m)}$
which takes $1\to\Phi$ induces an isomorphism $\mathcal{\widehat{D}}_{W(A)}^{(m)}\cdot\pi\tilde{\to}\Phi^{*}\mathcal{\widehat{D}}_{\mathcal{A}}^{(m)}$;
therefore, $\widehat{\Phi}^{*}\mathcal{\widehat{D}}_{\mathcal{A}}^{(m)}$
is a summand of $\mathcal{D}_{W(A)}^{(m)}$. The analogous facts are
true for $\mathcal{\widehat{D}}_{W(A),\text{crys}}^{(0)}$ and $\Phi^{*}\mathcal{\widehat{D}}_{\mathcal{A},\text{crys}}^{(0)}$. 
\end{cor}

Note that this immediately proves part $2)$ of \prettyref{thm:Projective!}
\begin{proof}
If $A=k[T_{1},\dots,T_{d}]$, then we have inclusions $\mathcal{\widehat{D}}_{W(k[T_{i}])}^{(0)}\to\mathcal{\widehat{D}}_{W(A)}^{(0)}$
coming from sending $p^{r}T_{i}^{a/p^{r}}\to p^{r}T_{i}^{a/p^{r}}$
and $\{d\}_{a/p^{r}}\to\{\partial_{i}\}_{a/p^{r}}$; and we can set
${\displaystyle \pi=\prod_{i=1}^{d}\pi_{i}}$ where $\pi_{i}\in\mathcal{D}_{W(k[T_{i}])}^{(0)}$
is the element constructed directly above. In this case we see that
$\pi:W(A)\to\mathcal{A}=W(k)<<T_{1},\dots,T_{d}>>$ is the standard
coordinate projection; i.e., the continuous map which takes $p^{r}T^{I/p^{r}}$to
$0$ whenever $r\geq1$ and $T^{I}\to T^{I}$ for $I\in(\mathbb{Z}^{\geq0})^{d}$. 

For the general case, the choice of local coordinates yields an etale
map \\
$k[T_{1},\dots,T_{d}]\to A$. By the etale local nature of Witt-differential
operators (\prettyref{prop:construction-of-transfer}), we see that
$\pi$ extends uniquely to $\mathcal{D}_{W(A)}^{(0)}$; this is the
required element. 

To obtain the last statement, note that the element $\Phi\in\widehat{\Phi}^{*}\mathcal{\widehat{D}}_{\mathcal{A}}^{(0)}$
has annihilator consisting of $\{\varphi\in\mathcal{\widehat{D}}_{W(A)}^{(0)}|\varphi\cdot\Phi(\mathcal{A})=0\}$;
this is also the left annihilator of $\pi$; this yields the isomorphism
$\mathcal{D}_{W(A)}^{(0)}\cdot\pi\tilde{\to}\widehat{\Phi}^{*}\mathcal{\widehat{D}}_{\mathcal{A}}^{(0)}$.
Since $\pi$ is an idempotent the rest of the statement follows directly.
We deduce the statement for higher $m$ by looking at the image of
$\pi$ in $\mathcal{\widehat{D}}_{W(A)}^{(m)}$. 

Similarly, to handle the crystalline version, we recall that there
is an injective map $\mathcal{\widehat{D}}_{W(A)}^{(0)}\to\mathcal{\widehat{D}}_{W(A),\text{crys}}^{(0)}$,
and so the image of $\pi$ in $\mathcal{\widehat{D}}_{W(A),\text{crys}}^{(0)}$
does the job. 
\end{proof}

\subsection{\label{subsec:Accessible-modules}Accessible modules}

In this section we use the fundamental isomorphism of \prettyref{thm:Uniqueness-of-bimodule}
to define and study our categories of interest. 

Let us begin in positive characteristic. Translating \prettyref{cor:Iso-mod-p^n}
to this situation (c.f. \prettyref{cor:Global-bimodule-mod-p} for
a proof that works in all characteristics) yields 
\begin{prop}
\label{prop:Properties-of-B_W}Let $m\geq0$. There is a well-defined
$(\mathcal{\widehat{D}}_{W(X)}^{(m)}/p,\mathcal{D}_{X}^{(m)})$ bimodule,
denoted $\mathcal{B}_{X}^{(m)}$, which is locally projective over
$\mathcal{\widehat{D}}_{W(X)}^{(m)}/p$, and locally faithfully flat
as a right $\mathcal{D}_{X}^{(0)}$-module. The associated functor
\[
\mathcal{B}_{X}^{(m)}\otimes_{\mathcal{D}_{X}^{(0)}}?:\mathcal{D}_{X}^{(m)}-\text{mod}\to\mathcal{\widehat{D}}_{W(X)}^{(m)}/p-\text{mod}
\]
is exact and fully faithful, and admits an exact right adjoint $\mathcal{H}om_{\mathcal{\widehat{D}}_{W(X)}^{(m)}/p}(\mathcal{B}_{X}^{(m)},?)$.
We have 
\[
\mathcal{H}om_{\mathcal{\widehat{D}}_{W(X)}^{(m)}/p}(\mathcal{B}_{X}^{(m)},\mathcal{B}_{X}^{(m)}\otimes_{\mathcal{D}_{X}^{(m)}}\mathcal{M})\tilde{\to}\mathcal{M}
\]
for all $\mathcal{M}\in\mathcal{D}_{X}^{(m)}-\text{mod}$. Therefore,
the functor 
\[
\mathcal{B}_{X}\otimes_{\mathcal{D}_{X}^{(m)}}^{L}:D(\mathcal{D}_{X}^{(m)}-\text{mod})\to D(\mathcal{\widehat{D}}_{W(X)}^{(m)}/p-\text{mod})
\]
is fully faithful and we have 
\[
R\mathcal{H}om_{\mathcal{\widehat{D}}_{W(X)}^{(m)}/p}(\mathcal{B}_{X}^{(m)},\mathcal{B}_{X}^{(m)}\otimes_{\mathcal{D}_{X}^{(m)}}\mathcal{M}^{\cdot})\tilde{\to}\mathcal{M}^{\cdot}
\]
for all $\mathcal{M}^{\cdot}\in D(\mathcal{D}_{X}^{(m)}-\text{mod})$. 
\end{prop}

\begin{proof}
The existence of the bimodule follows immediately from \prettyref{thm:Uniqueness-of-bimodule}
and \prettyref{cor:Iso-mod-p^n} (or \prettyref{cor:Global-bimodule-mod-p}).
By \prettyref{thm:Projective!}, we see that it is locally faithfully
flat as a right $\mathcal{D}_{X}^{(m)}$-module and locally projective
over $\mathcal{\widehat{D}}_{W(X)}^{(m)}/p$. The rest follows directly
once we prove that the adjunction map
\begin{equation}
\mathcal{M}\to\mathcal{H}om_{\mathcal{\widehat{D}}_{W(X)}^{(m)}/p}(\mathcal{B}_{X}^{(m)},\mathcal{B}_{X}^{(m)}\otimes_{\mathcal{D}_{X}^{(m)}}\mathcal{M})\label{eq:Adjunction}
\end{equation}
is an isomorphism for all $\mathcal{M}\in\mathcal{D}_{X}^{(m)}-\text{mod}$.
To see this, we may work locally and assume $X=\text{Spec}(A)$. Let
$\Phi:\mathcal{A}\to W(A)$ be the morphism coming from a coordinatized
lift of Frobenius on $\mathcal{A}$, and $\pi:W(A)\to\mathcal{A}$
the associated projector. Now note that, as a right $\mathcal{D}_{X}^{(m)}$-module,
we have 
\[
\mathcal{B}_{X}^{(m)}=(1-\pi)\mathcal{B}_{X}^{(m)}\oplus\pi\mathcal{B}_{X}^{(m)}=(1-\pi)\mathcal{B}_{X}^{(m)}\oplus\mathcal{D}_{X}^{(m)}
\]
so that 
\[
\mathcal{B}_{X}^{(m)}\otimes_{\mathcal{D}_{X}^{(m)}}\mathcal{M}=((1-\pi)\mathcal{B}_{X}^{(m)}\otimes_{\mathcal{D}_{X}^{(m)}}\mathcal{M})\oplus\mathcal{M}
\]
As the functor $\mathcal{H}om_{\mathcal{\widehat{D}}_{W(X)}^{(m)}/p}(\mathcal{B}_{X}^{(m)},\mathcal{N})$
agrees with $\{n\in\mathcal{N}|(1-\pi)n=0\}$, we see that \prettyref{eq:Adjunction}
is an isomorphism as required. 
\end{proof}
This yields the 
\begin{defn}
\label{def:Accessible}1) A module $\mathcal{N}\in\mathcal{\widehat{D}}_{W(X)}^{(m)}/p-\text{mod}$
is accessible if is of the form $\mathcal{B}_{X}^{(m)}\otimes_{\mathcal{D}_{X}^{(m)}}\mathcal{M}$
for some $\mathcal{M}\in\mathcal{D}_{X}^{(m)}-\text{mod}$. 

2) A complex $\mathcal{N}^{\cdot}\in D(\mathcal{\widehat{D}}_{W(X)}^{(m)}/p-\text{mod})$
is accessible if is of the form $\mathcal{B}_{X}^{(m)}\otimes_{\mathcal{D}_{X}^{(m)}}^{L}\mathcal{M}^{\cdot}$
for some $\mathcal{M}^{\cdot}\in D(\mathcal{D}_{X}^{(m)}-\text{mod})$.

3) Let $r\geq1$. A complex $\mathcal{N}^{\cdot}\in D(\mathcal{\widehat{D}}_{W(X)}^{(m)}/p^{r}-\text{mod})$
is accessible if $\mathcal{N}^{\cdot}\otimes_{W_{r}(k)}^{L}k$ is
accessible in $D(\mathcal{\widehat{D}}_{W(X)}^{(m)}/p-\text{mod})$.
Similarly, a complex $\mathcal{N}^{\cdot}\in D_{cc}(\mathcal{\widehat{D}}_{W(X)}^{(m)}-\text{mod})$
is accessible if $\mathcal{N}^{\cdot}\otimes_{W(k)}^{L}k$ is accessible
in $D(\mathcal{\widehat{D}}_{W(X)}^{(m)}/p-\text{mod})$. 
\end{defn}

As the full subcategory of accessible complexes in $D(\mathcal{\widehat{D}}_{W(X)}^{(m)}/p-\text{mod})$
is a thick triangulated subcategory, the same is true for the accessible
complexes in $D(\mathcal{\widehat{D}}_{W(X)}^{(m)}/p^{n}-\text{mod})$
and $D(\mathcal{\widehat{D}}_{W(X)}^{(m)}-\text{mod})$. We denote
these categories by \\
$D_{\text{acc}}(\mathcal{\widehat{D}}_{W(X)}^{(m)}/p^{n}-\text{mod})$
and $D_{\text{acc}}(\mathcal{\widehat{D}}_{W(X)}^{(m)}-\text{mod})$,
respectively. 

In the presence of a lift of Frobenius, we have the following characterization
of accessibility. Before giving it, let us set some notation: Let
$A$ be smooth affine which possesses local coordinates and let $\Phi:\mathcal{A}\to W(A)$
be the morphism coming from a coordinatized lift of Frobenius. For
$\mathcal{M}^{\cdot}\in D(\mathcal{D}_{X_{n}}^{(m)}-\text{mod})$,
let $\Phi^{*}\mathcal{M}^{\cdot}:=\Phi^{*}\mathcal{D}_{X_{n}}^{(m)}\otimes_{\mathcal{D}_{X_{n}}^{(m)}}^{L}\mathcal{M}^{\cdot}$.
For $\mathcal{M}^{\cdot}\in D_{cc}(\mathcal{D}_{\mathfrak{X}}^{(m)}-\text{mod})$,
let $\Phi^{*}\mathcal{M}^{\cdot}:=\Phi^{*}\mathcal{D}_{\mathfrak{X}}^{(m)}\widehat{\otimes}_{\mathcal{D}_{\mathfrak{X}}^{(m)}}^{L}\mathcal{M}^{\cdot}$,
the cohomological completion of $\Phi^{*}\mathcal{M}^{\cdot}:=\Phi^{*}\mathcal{D}_{\mathfrak{X}}^{(m)}\otimes_{\mathcal{D}_{\mathfrak{X}}^{(m)}}^{L}\mathcal{M}^{\cdot}$.
Then we have
\begin{thm}
\label{thm:Local-Accessible}Let $A$, $\mathcal{A}$ and $\Phi:\mathcal{A}\to W(A)$
be the morphism coming from a coordinatized lift of Frobenius. Let
$X=\text{Spec}(A)$. Then $\mathcal{N}^{\cdot}\in D(\mathcal{\widehat{D}}_{W(X)}^{(m)}/p^{r}-\text{mod})$
is accessible iff $\mathcal{N}^{\cdot}\tilde{=}\Phi^{*}\mathcal{M}^{\cdot}$
for some $\mathcal{M}^{\cdot}\in D(\mathcal{D}_{\mathfrak{X}_{r}}^{(m)}-\text{mod})$
(where $\mathfrak{X}_{r}=\text{Spec}(\mathcal{A}/p^{r})$). 

Similarly, $\mathcal{N}^{\cdot}\in D_{cc}(\mathcal{\widehat{D}}_{W(X)}^{(m)}-\text{mod})$
is accessible iff $\mathcal{N}^{\cdot}\tilde{=}\Phi^{*}\mathcal{M}^{\cdot}$
for some $\mathcal{M}^{\cdot}\in D_{cc}(\mathcal{D}_{\mathfrak{X}}^{(m)}-\text{mod})$. 
\end{thm}

\begin{proof}
Let $\mathcal{N}^{\cdot}\in D(\mathcal{\widehat{D}}_{W(X)}^{(m)}/p-\text{mod})$.
It follows from \prettyref{prop:Properties-of-B_W} that $\mathcal{N}^{\cdot}$
is accessible iff the adjunction 
\[
\mathcal{N}^{\cdot}\to\mathcal{B}_{W(X)}^{(m)}\otimes_{\mathcal{D}_{X}^{(m)}}^{L}R\mathcal{H}om_{\mathcal{\widehat{D}}_{W(X)}^{(m)}/p}(\mathcal{B}_{W(X)}^{(m)},\mathcal{N}^{\cdot})
\]
is an isomorphism. 

Now let $\mathcal{N}^{\cdot}\in D(\mathcal{\widehat{D}}_{W(X)}^{(m)}/p^{r}-\text{mod})$.
Suppose $\mathcal{N}^{\cdot}$ is accessible. We have the bimodule
$\mathcal{B}_{\mathfrak{X}_{r}}^{(m)}=\Phi^{*}\mathcal{D}_{\mathfrak{X}_{r}}^{(m)}$,
which, by \prettyref{thm:Projective!}, is locally projective as a
left $\widehat{\mathcal{D}}_{W(X)}^{(m)}/p^{n}$-module and is locally
faithfully flat as a right $\mathcal{D}_{\mathfrak{X}_{r}}^{(m)}$-module.
Therefore 
\[
R\mathcal{H}om_{\mathcal{\widehat{D}}_{W(X)}^{(m)}/p^{r}}(\Phi^{*}\mathcal{D}_{X_{n}}^{(m)},\mathcal{N}^{\cdot})\otimes_{W_{r}(k)}^{L}k\tilde{\to}R\mathcal{H}om_{\mathcal{\widehat{D}}_{W(X)}^{(m)}/p}(\mathcal{B}_{W(X)}^{(m)},\mathcal{N}^{\cdot}\otimes_{W_{r}(k)}^{L}k)
\]
Therefore, 
\[
(\Phi^{*}\mathcal{D}_{\mathfrak{X}_{r}}^{(m)}\otimes_{\mathcal{D}_{\mathfrak{X}_{r}}^{(m)}}^{L}R\mathcal{H}om_{\mathcal{\widehat{D}}_{W(X)}^{(m)}/p^{r}}(\Phi^{*}\mathcal{D}_{\mathfrak{X}_{r}}^{(m)},\mathcal{N}^{\cdot}))\otimes_{W_{r}(k)}^{L}k
\]
\[
\tilde{\to}\mathcal{B}_{X}^{(m)}\otimes_{\mathcal{D}_{X}^{(m)}}^{L}R\mathcal{H}om_{\mathcal{\widehat{D}}_{W(X)}^{(m)}/p}(\mathcal{B}_{W(X)}^{(m)},\mathcal{N}^{\cdot}\otimes_{W_{r}(k)}^{L}k)
\]
so that the adjunction map 
\[
\mathcal{N}^{\cdot}\to\Phi^{*}\mathcal{D}_{\mathfrak{X}_{r}}^{(m)}\otimes_{\mathcal{D}_{\mathfrak{X}_{r}}^{(m)}}^{L}R\mathcal{H}om_{\mathcal{\widehat{D}}_{W(X)}^{(m)}/p^{r}}(\Phi^{*}\mathcal{D}_{\mathfrak{X}_{r}}^{(m)},\mathcal{N}^{\cdot})
\]
becomes an isomorphism after applying $\otimes_{W_{r}(k)}^{L}k$,
and therefore is already an isomorphism by Nakayama. Thus we have
$\mathcal{N}^{\cdot}\tilde{=}\Phi^{*}\mathcal{M}^{\cdot}$ for some
$\mathcal{M}^{\cdot}\in\text{Qcoh}(\mathcal{D}_{\mathfrak{X}_{r}}^{(m)})$;
the converse direction is clear. 

The statement for $\mathcal{N}^{\cdot}\in D_{cc}(\mathcal{\widehat{D}}_{W(X)}^{(m)}-\text{mod})$
is essentially identical, using the Nakayama lemma for cohomologically
complete complexes and the fact that 
\[
R\mathcal{H}om_{\mathcal{\widehat{D}}_{W(X)}^{(m)}}(\Phi^{*}\mathcal{D}_{\mathfrak{X}_{n}}^{(m)},\mathcal{N}^{\cdot})
\]
is cohomologically complete if $\mathcal{N}^{\cdot}$ is (since this
sheaf is a summand of $\mathcal{N}^{\cdot}$). 
\end{proof}
From this, we obtain
\begin{cor}
\label{cor:Qcoh-mod-p}Let $D_{\text{acc},\text{qcoh}}(\mathcal{\widehat{D}}_{W(X)}^{(m)}/p^{r}-\text{mod})$
denote the full subcategory of $D_{\text{acc}}(\mathcal{\widehat{D}}_{W(X)}^{(m)}/p^{r}-\text{mod})$
consisting of complexes $\mathcal{N}^{\cdot}$ such that 
\[
\mathcal{N}^{\cdot}\otimes_{W_{r}(k)}^{L}k\tilde{\to}\mathcal{B}_{X}^{(m)}\otimes_{\mathcal{D}_{X}^{(m)}}^{L}\mathcal{M}^{\cdot}
\]
where $\mathcal{H}^{i}(\mathcal{M}^{\cdot})\in\text{Qcoh}(\mathcal{D}_{X}^{(m)})$
for all $i$. Then, if $X=\text{Spec}(A)$ as above, we have $\mathcal{N}^{\cdot}\in D_{\text{acc},\text{qcoh}}(\mathcal{\widehat{D}}_{W(X)}^{(m)}/p^{r}-\text{mod})$
iff $\mathcal{N}^{\cdot}\tilde{=}\Phi^{*}\mathcal{M}^{\cdot}$ for
some $\mathcal{M}^{\cdot}\in D_{\text{qcoh}}(\mathcal{D}_{\mathfrak{X}_{r}}^{(m)}-\text{mod})$. 
\end{cor}

\begin{proof}
By the previous theorem, it suffices to show the following: let $\mathcal{M}^{\cdot}\in D(\mathcal{O}_{\mathfrak{X}_{r}}-\text{mod})$.
Then $\mathcal{M}^{\cdot}\in D(\text{QCoh}(\mathfrak{X}_{r}))$ iff
$\mathcal{M}^{\cdot}\otimes_{W_{r}(k)}^{L}k\in D(\text{QCoh}(\mathfrak{X}_{r}))$.
The forward direction is obvious. For the reverse, consider the short
exact sequence 
\[
0\to W_{r-1}(k)\to W_{r}(k)\to k\to0
\]
where on the left we identify $pW_{n}(k)\tilde{=}W_{n-1}(k)$. This
yields a distinguished triangle 
\[
\mathcal{M}^{\cdot}\otimes_{W_{r}(k)}^{L}W_{r-1}(k)\to\mathcal{M}^{\cdot}\to\mathcal{M}^{\cdot}\otimes_{W_{r}(k)}^{L}k
\]
As $(\mathcal{M}^{\cdot}\otimes_{W_{r}(k)}^{L}W_{r-1}(k))\otimes_{W_{r-1}(k)}^{L}k\tilde{\to}\mathcal{M}^{\cdot}\otimes_{W_{r}(k)}^{L}k$,
by induction we may assume $\mathcal{M}^{\cdot}\otimes_{W_{r}(k)}^{L}W_{r-1}(k)\in D(\text{Qcoh}(\mathfrak{X}_{r-1}))$;
hence both edges of the triangle can be regarded as members of $D(\text{Qcoh}(\mathfrak{X}_{r}))$,
therefore $\mathcal{M}^{\cdot}\in D(\text{Qcoh}(\mathfrak{X}_{r}))$
as well. 
\end{proof}
We shall also find it useful to consider the category $D_{\text{acc},\text{qcoh}}(\mathcal{\widehat{D}}_{W(X)}^{(m)}-\text{mod})$
consisting of complexes $\mathcal{N}^{\cdot}$ such that 
\[
\mathcal{N}^{\cdot}\otimes_{W_{r}(k)}^{L}k\tilde{\to}\mathcal{B}_{X}^{(m)}\otimes_{\mathcal{D}_{X}^{(m)}}^{L}\mathcal{M}^{\cdot}
\]
where $\mathcal{H}^{i}(\mathcal{M}^{\cdot})\in\text{Qcoh}(\mathcal{D}_{X}^{(m)})$
for all $i$. I don't know a local characterization of it in terms
of the functor $\Phi^{*}$ as above. However, we do have: 
\begin{cor}
\label{cor:D-b-coh}Let $D_{\text{acc},\text{coh}}^{b}(\mathcal{\widehat{D}}_{W(X)}^{(m)}-\text{mod})$
denote the full subcategory of $D_{\text{acc}}(\mathcal{\widehat{D}}_{W(X)}^{(m)}-\text{mod})$
consisting of bounded complexes $\mathcal{N}^{\cdot}$ such that 
\[
\mathcal{N}^{\cdot}\otimes_{W(k)}^{L}k\tilde{\to}\mathcal{B}_{X}^{(m)}\otimes_{\mathcal{D}_{X}^{(m)}}^{L}\mathcal{M}^{\cdot}
\]
where $\mathcal{H}^{i}(\mathcal{M}^{\cdot})\in\text{Coh}(\mathcal{D}_{X}^{(m)})$
for all $i$. Then, if $X=\text{Spec}(A)$ as above, we have $\mathcal{N}^{\cdot}\in D_{\text{acc},\text{coh}}^{b}(\mathcal{\widehat{D}}_{W(X)}^{(m)}-\text{mod})$
iff $\mathcal{N}^{\cdot}\tilde{=}\Phi^{*}\mathcal{P}^{\cdot}$ for
some $\mathcal{P}^{\cdot}\in D_{\text{coh}}^{b}(\mathcal{\widehat{D}}_{\mathfrak{X}}^{(m)}-\text{mod})$.
Furthermore, a bounded complex $\mathcal{N}^{\cdot}$ is contained
in $D_{\text{acc},\text{coh}}^{b}(\mathcal{\widehat{D}}_{W(X)}^{(m)}-\text{mod})$
iff each $\mathcal{H}^{i}(\mathcal{N}^{\cdot})$, considered as a
complex in degree $0$, is contained in $D_{\text{acc},\text{coh}}^{b}(\mathcal{\widehat{D}}_{W(X)}^{(m)}-\text{mod})$.
In particular, the elements of $D_{\text{acc},\text{coh}}^{b}(\mathcal{\widehat{D}}_{W(X)}^{(m)}-\text{mod})$
in homological degree zero form an abelian subcategory of $D_{\text{acc}}(\mathcal{\widehat{D}}_{W(X)}^{(m)}-\text{mod})$. 
\end{cor}

\begin{proof}
Note that, as the functor $\otimes_{W(k)}^{L}k$ has finite homological
dimension, we have that the complex $\mathcal{M}^{\cdot}$ in the
statement is bounded. Further, $\mathcal{N}^{\cdot}\tilde{=}\Phi^{*}\mathcal{P}^{\cdot}$
for some $\mathcal{P}^{\cdot}\in D_{\text{coh}}^{b}(\mathcal{D}_{\mathfrak{X}}^{(m)}-\text{mod})$
clearly implies 
\[
\mathcal{N}^{\cdot}\otimes_{W(k)}^{L}k\tilde{\to}\mathcal{B}_{X}^{(m)}\otimes_{\mathcal{D}_{X}^{(m)}}^{L}(\mathcal{P}^{\cdot}\otimes_{W(k)}^{L}k)
\]
which gives the converse direction of the if and only if. For the
forward direction, note that $\mathcal{N}^{\cdot}\tilde{=}\Phi^{*}\mathcal{P}^{\cdot}$
for some cohomologically complete complex $\mathcal{P}^{\cdot}$ which
satisfies $\mathcal{P}^{\cdot}\otimes_{W(k)}^{L}k\in D_{\text{coh}}^{b}(\mathcal{D}_{X}^{(m)}-\text{mod})$;
therefore , applying \cite{key-8}, theorem 1.6.4, we have $\mathcal{P}^{\cdot}\in D_{\text{coh}}^{b}(\mathcal{\widehat{D}}_{\mathfrak{X}}^{(m)}-\text{mod})$. 

Now, we have that $\Phi^{*}\mathcal{D}_{\mathfrak{X}}^{(m)}\widehat{\otimes}_{\mathcal{D}_{\mathfrak{X}}^{(m)}}^{L}\widehat{\mathcal{D}}_{\mathfrak{X}}^{(m)}$
is the cohomological completion of
\[
\Phi^{*}\mathcal{D}_{\mathfrak{X}}^{(m)}\otimes_{\mathcal{D}_{\mathfrak{X}}^{(m)}}^{L}\widehat{\mathcal{D}}_{\mathfrak{X}}^{(m)}=\Phi^{*}\mathcal{D}_{\mathfrak{X}}^{(m)}
\]
which, being $p$-torsion-free and $p$-adically complete, is already
cohomologically complete. Thus, for any locally free coherent $\mathcal{\widehat{D}}_{\mathfrak{X}}^{(m)}$-module
$\mathcal{P}$, we have 
\[
\Phi^{*}\mathcal{D}_{\mathfrak{X}}^{(m)}\widehat{\otimes}_{\mathcal{D}_{\mathfrak{X}}^{(m)}}^{L}\mathcal{P}\tilde{=}\Phi^{*}\mathcal{D}_{\mathfrak{X}}^{(m)}\otimes_{\mathcal{D}_{\mathfrak{X}}^{(m)}}^{L}\mathcal{P}=\Phi^{*}\mathcal{D}_{\mathfrak{X}}^{(m)}\otimes_{\mathcal{D}_{\mathfrak{X}}^{(m)}}\mathcal{P}
\]
As $\mathcal{\widehat{D}}_{\mathfrak{X}}^{(m)}$ has finite homological
dimension, we have that any coherent $\mathcal{\widehat{D}}_{\mathfrak{X}}^{(m)}$-module
$\mathcal{M}$ is quasi-isomorphic to a finite complex of locally
free coherent $\mathcal{\widehat{D}}_{\mathfrak{X}}^{(m)}$-modules,
and so, as $\Phi^{*}\mathcal{D}_{\mathfrak{X}}^{(m)}$ is flat over
$\mathcal{D}_{\mathfrak{X}}^{(m)}$, we see
\[
\Phi^{*}\mathcal{D}_{\mathfrak{X}}^{(m)}\widehat{\otimes}_{\mathcal{D}_{\mathfrak{X}}^{(m)}}^{L}\mathcal{M}\tilde{=}\Phi^{*}\mathcal{D}_{\mathfrak{X}}^{(m)}\otimes_{\mathcal{D}_{\mathfrak{X}}^{(m)}}\mathcal{M}
\]
for all such $\mathcal{M}$; by induction on the cohomological length
we see that 
\[
\Phi^{*}\mathcal{D}_{\mathfrak{X}}^{(m)}\widehat{\otimes}_{\mathcal{D}_{\mathfrak{X}}^{(m)}}^{L}\mathcal{M}^{\cdot}\tilde{=}\Phi^{*}\mathcal{D}_{\mathfrak{X}}^{(m)}\otimes_{\mathcal{D}_{\mathfrak{X}}^{(m)}}^{L}\mathcal{M}^{\cdot}
\]
for all $\mathcal{M}^{\cdot}\in D_{\text{coh}}^{b}(\mathcal{\widehat{D}}_{\mathfrak{X}}^{(m)}-\text{mod})$.
Therefore, a bounded complex $\mathcal{N}^{\cdot}$ is contained in
$D_{\text{acc},\text{coh}}^{b}(\mathcal{\widehat{D}}_{W(X)}^{(m)}-\text{mod})$
iff we have $\mathcal{H}^{i}(\mathcal{N}^{\cdot})=\Phi^{*}\mathcal{D}_{\mathfrak{X}}^{(m)}\otimes_{\mathcal{D}_{\mathfrak{X}}^{(m)}}\mathcal{M}$
for some coherent $\mathcal{\widehat{D}}_{\mathfrak{X}}^{(m)}$-module
$\mathcal{M}$; indeed the forward direction of this follows immediately
from the above discussion, and the converse follows by induction on
the cohomological length; the rest of the corollary follows immediately.
\end{proof}
\begin{rem}
\label{rem:=00005CPhi-pull-for-bounded-torsion}In the course of the
above proof we showed that 
\[
\Phi^{*}\mathcal{D}_{\mathfrak{X}}^{(m)}\widehat{\otimes}_{\mathcal{D}_{\mathfrak{X}}^{(m)}}^{L}\mathcal{M}\tilde{=}\Phi^{*}\mathcal{D}_{\mathfrak{X}}^{(m)}\otimes_{\mathcal{D}_{\mathfrak{X}}^{(m)}}\mathcal{M}
\]
for any coherent $\mathcal{D}_{\mathfrak{X}}^{(m)}$-module $\mathcal{M}$.
In fact this isomorphism holds whenever $\mathcal{M}$ is a cohomologically
complete $\mathcal{D}_{\mathfrak{X}}^{(m)}$-module with bounded $p$-torsion;
i.e., there is some $N\in\mathbb{N}$ such that, if a section $m$
is killed by a power of $p$, then it is killed by $p^{N}$. Let $\mathcal{M}_{\text{tors}}$
denote this subsheaf. Then, since modules of bounded torsion are cohomologically
complete, we see that $\Phi^{*}\mathcal{D}_{\mathfrak{X}}^{(m)}\widehat{\otimes}_{\mathcal{D}_{\mathfrak{X}}^{(m)}}^{L}\mathcal{M}_{\text{tors}}=\Phi^{*}\mathcal{D}_{\mathfrak{X}}^{(m)}\otimes_{\mathcal{D}_{\mathfrak{X}}^{(m)}}^{L}\mathcal{M}_{\text{tors}}=\Phi^{*}\mathcal{D}_{\mathfrak{X}}^{(m)}\otimes_{\mathcal{D}_{\mathfrak{X}}^{(m)}}\mathcal{M}_{\text{tors}}$.
Further, since $\mathcal{M}/\mathcal{M}_{\text{tors}}$ is $p$-torsion-free,
so is $\Phi^{*}\mathcal{D}_{\mathfrak{X}}^{(m)}\otimes_{\mathcal{D}_{\mathfrak{X}}^{(m)}}\mathcal{M}/\mathcal{M}_{\text{tors}}$
(by the flatness of $\Phi^{*}\mathcal{D}_{\mathfrak{X}}^{(m)}$) and
we have that $\Phi^{*}\mathcal{D}_{\mathfrak{X}}^{(m)}\widehat{\otimes}_{\mathcal{D}_{\mathfrak{X}}^{(m)}}^{L}\mathcal{M}/\mathcal{M}_{\text{tors}}$
is simply the $p$-adic completion of $\Phi^{*}\mathcal{D}_{\mathfrak{X}}^{(m)}\otimes_{\mathcal{D}_{\mathfrak{X}}^{(m)}}\mathcal{M}/\mathcal{M}_{\text{tors}}$,
which lives in homological degree $0$. Therefore the result follows
for $\mathcal{M}$ from the short exact sequence 
\[
0\to\mathcal{M}_{\text{tors}}\to\mathcal{M}\to\mathcal{M}/\mathcal{M}_{\text{tors}}\to0
\]
\end{rem}

When working mod $p^{r}$ we can do even better:
\begin{cor}
1) Let $\mathcal{N}^{\cdot}\in D(\mathcal{\widehat{D}}_{W(X)}^{(m)}/p^{r}-\text{mod})$.
Then $\mathcal{N}^{\cdot}$ is accessible iff $\mathcal{H}^{i}(\mathcal{N}^{\cdot})$
(considered as a complex concentrated in degree $0$) is accessible
for all $i$. Thus the full subcategory of $\mathcal{\widehat{D}}_{W(X)}^{(m)}/p^{r}-\text{mod}$
consisting of accessible modules (when considered as complexes in
degree $0$) is abelian. The same holds for accessible quasicoherent
and coherent modules.
\end{cor}

\begin{proof}
The statement is local, so we can apply \prettyref{cor:Qcoh-mod-p}
using the exactness of $\Phi^{*}$ and the functor $\mathcal{H}om_{\mathcal{\widehat{D}}_{W(X)}^{(m)}/p^{r}}(\Phi^{*}\mathcal{D}_{\mathfrak{X}_{r}}^{(m)},)$. 
\end{proof}
We denote by $\mathcal{\widehat{D}}_{W(X)}^{(m)}/p^{r}-\text{mod}_{\text{acc}}$,
$\mathcal{\widehat{D}}_{W(X)}^{(m)}/p^{r}-\text{mod}_{\text{acc},\text{qcoh}}$,
and $\mathcal{\widehat{D}}_{W(X)}^{(m)}/p^{r}-\text{mod}_{\text{acc},\text{coh}}$
the abelian category of accessible modules, accessible quasicoherent
modules, and accessible coherent modules, respectively. 

Now let us note what happens in the presence of a global smooth lift
$\mathfrak{X}_{r}$ of $X$ (over $W_{r}(k)$; we allow $r=\infty$
here to cover the case of a smooth formal $\mathfrak{X}$ over $W(k)$). 
\begin{cor}
\label{cor:Bimodule-over-X-r}Suppose that $p>2$ if $m=0$. There
is a $(\mathcal{\widehat{D}}_{W(X)}^{(m)}/p^{r},\mathcal{D}_{\mathfrak{X}_{r}}^{(m)})$-bimodule
$\mathcal{B}_{\mathfrak{X}_{r}}^{(m)}$ which is locally isomorphic
to $\Phi^{*}\mathcal{D}_{\mathfrak{X}_{r}}^{(m)}$ whenever we have
$X=\text{Spec}(A)$ as above. This bimodule induces an equivalence
of categories 
\[
\mathcal{D}_{\mathfrak{X}_{r}}^{(m)}-\text{mod}\to\mathcal{\widehat{D}}_{W(X)}^{(m)}/p^{r}-\text{mod}_{\text{acc}}
\]
as well as a derived equivalence 
\[
D(\mathcal{D}_{\mathfrak{X}_{r}}^{(m)}-\text{mod})\to D_{\text{acc}}(\mathcal{\widehat{D}}_{W(X)}^{(m)}/p^{r}-\text{mod})
\]
and when $r=\infty$ and equivalence 
\[
D_{cc}(\widehat{\mathcal{D}}_{\mathfrak{X}}^{(m)}-\text{mod})\to D_{\text{acc}}(\mathcal{\widehat{D}}_{W(X)}^{(m)}-\text{mod})
\]
\end{cor}

This follows immediately from the above discussions and \prettyref{cor:Iso-mod-p^n}
and \prettyref{thm:Uniqueness-of-bimodule}. 

To finish off this section we make some remarks about the completions
in this theory, which will be useful when discussing right $\mathcal{D}$-modules
and the de Rham-Witt resolution. Let us begin by noting the 
\begin{lem}
\label{lem:two-filtrations}Let $A$ be smooth affine which possesses
local coordinates, and let $\Phi:\mathcal{A}\to W(A)$ be the morphism
coming from a coordinatized lift of Frobenius. Let $X=\text{Spec}(A)$.
For $r\geq1$ consider the filtrations $\{V^{m}(W(A)/p^{r})\}_{m\geq1}$
and 
\[
F^{m}(W(A)/p^{r})=\{\sum_{(I,p)=1}a_{I}p^{s}T^{I/p^{s}}|a_{I}=0\phantom{i}\text{for}\phantom{i}s<m\}
\]
(for $m\geq1$). These filtrations satisfy $F^{m}(W(A)/p^{r})\subset V^{m}(W(A)/p^{r})$
(for all $m$) and $V^{m+r-1}(W(A)/p^{r})\subset F^{m}(W(A)/p^{r})$
(for all $m$).
\end{lem}

\begin{proof}
The first inclusion is obvious since $p^{s}T^{I/p^{s}}\in V^{m}$
whenever $s\geq m$. For the second, set $s=r+m-1$ and note that
for $I\in\mathbb{Z}_{\geq0}^{d}$
\[
V^{s}(T^{I})=p^{s}T^{I/p^{s}}
\]
Write $T^{I}=T^{I_{1}}T^{I_{2}}$ where the valuation of each entry
in $I_{1}$ is $\geq s$ and the valuation of each entry in $I_{2}$
is $<s$ (if there are no entries in the second category we take $I_{2}=\emptyset$).
Let $J_{1}=I_{1}/p^{s}$. Then 
\[
p^{s}T^{I/p^{s}}=T^{J_{1}}\cdot p^{s}T^{I_{2}/p^{s}}
\]
 If $I_{2}=\emptyset$ then $p^{s}T^{I/p^{s}}=p^{s}T^{J_{1}}=0$ (as
$s\geq r$). If $I_{2}\neq\emptyset$ then let $t=\text{min}\{\text{val}_{p}(i)|i\in I_{2}\}$,
and set $J_{2}=I_{2}/p^{t}$. Then 
\[
p^{s}T^{I_{2}/p^{s}}=p^{t}p^{s-t}T^{J_{2}/p^{s-t}}
\]
and if $t<r$ then $s-t\geq m$. So we have either $p^{s}T^{I_{2}/p^{s}}=0$
(because $t\geq r$) or $p^{s}T^{I_{2}/p^{s}}\in F^{m}(W(A)/p^{r})$;
so we see that $V^{r+m-1}(T^{I})\in F^{m}(W(A)/p^{r})$ for all $I$.
As $W(A)/p^{r}$ is the completion of the span over $\mathcal{A}/p^{r}$
of terms of the form $T^{I}$ and $V^{j}(T^{I})$ for some $j>0$,
the result follows. 
\end{proof}
This allows us to consider limits in the theory. Recall that if $\mathcal{F}_{i}$
is a sequence of objects in a triangulated category with maps $\mathcal{F}_{i}\to\mathcal{F}_{i-1}$
an object $\mathcal{L}$ is said to be a homotopy limit of the $\mathcal{F}_{i}$
if there is a distinguished triangle 
\[
\mathcal{L}\to\prod_{i}\mathcal{F}_{i}\to\prod_{i}\mathcal{F}_{i}
\]
where the second map is $\text{Id }-\text{Shift}$. Even if they exist,
homotopy limits are not functorial in a general triangulated category.
However, in this paper all of our triangulated categories will be
the derived category of modules over some sheaf of rings on $X$,
or a full subcategory thereof. In this setting, functorial cones and
limits do exists; indeed if $\mathcal{R}$ is a sheaf of rings on
$X$ then, by a fundamental result of Spaltenstein, we have 
\[
D(\mathcal{R}-\text{mod})\tilde{\to}K(\text{k-inj}(\mathcal{R}))
\]
where $\text{k-inj}(\mathcal{R})$ is the category of k-injective
complexes in $\mathcal{R}-\text{mod}$, and $K()$ is the homotopy
category thereof. In other words, we may replace any element $\mathcal{M}^{\cdot}\in D(\mathcal{R}-\text{mod})$
with a k-injective complex, and this replacement is unique up to homotopy.
As the category $K(\text{k-inj}(\mathcal{R}))$ admits functorial
cones, so too does the category $D(\mathcal{R}-\text{mod})$. So in
this setting we will write $\text{holim}(\mathcal{F}_{i})$ for the
functorially defined homotopy limit of the $\mathcal{F}_{i}$. 

Specializing further, suppose $\mathcal{M}^{\cdot}\in\mathcal{\widehat{D}}_{W(X)}^{(m)}/p^{r}-\text{mod}$.
Then we define the derived completion of $\mathcal{M}^{\cdot}$ along
$V^{i}(\mathcal{O}_{W(X)}/p^{r})$ to be 
\[
\widehat{\mathcal{M}}^{\cdot}:=\text{holim}(\mathcal{M}^{\cdot}\otimes_{\mathcal{O}_{W(X)/p^{r}}}^{L}(\mathcal{O}_{W(X)}/p^{r})/V^{i}(\mathcal{O}_{W(X)}/p^{r}))
\]

In the accessible, quasicoherent case, this object admits a more direct
description: 
\begin{prop}
\label{prop:completion}Let $r\geq1$. Let $(\mathcal{M}^{j},d)$
be a complex of $\mathcal{\widehat{D}}_{W(X)}^{(m)}/p^{r}$ modules,
and suppose its image $\mathcal{M}^{\cdot}$ in $D(\mathcal{\widehat{D}}_{W(X)}^{(m)}/p^{r}-\text{mod})$
is accessible and quasicoherent. Form the complex $\mathcal{\mathcal{N}}^{j}$
whose terms are the completions 
\[
\mathcal{\widehat{M}}^{j}:=\lim_{i}\mathcal{M}^{j}/V^{i}(\mathcal{O}_{W(X)}/p^{r})
\]
Then there is an isomorphism 
\[
\mathcal{N}^{\cdot}\tilde{\to}\widehat{\mathcal{M}}^{\cdot}
\]
In particular, if $X=\text{Spec}(A)$ possesses local coordinates
and $\Phi:\mathcal{A}_{r}\to W(A)/p^{r}$ is as above, then for any
$\mathcal{M}^{\cdot}\in D_{\text{acc}}(\mathcal{\widehat{D}}_{W(X)}^{(m)}/p^{r}-\text{mod})$,
if $\mathcal{M}^{\cdot}=\Phi^{*}\mathcal{N}^{\cdot}$ then 
\[
R\lim_{i}(\mathcal{M}^{\cdot}\otimes_{\mathcal{O}_{W(X)/p^{r}}}^{L}(\mathcal{O}_{W(X)}/p^{r})/V^{i})\tilde{=}\widehat{\Phi}^{*}\mathcal{N}^{\cdot}
\]
where $\widehat{\Phi}^{*}$ denotes pullback followed by completion.
For each $\mathcal{N}^{i}$, $\widehat{\Phi}^{*}\mathcal{N}^{i}$
is an inverse limit of a direct sum of copies of $\mathcal{N}^{i}$. 
\end{prop}

\begin{proof}
Let $\mathcal{M}^{\cdot}\to\mathcal{F}^{\cdot}$ be a K-flat resolution
in $\mathcal{\widehat{D}}_{W(X)}^{(m)}/p^{r}-\text{mod}$. Then $\mathcal{M}^{\cdot}\otimes_{\mathcal{O}_{W(X)/p^{r}}}^{L}(\mathcal{O}_{W(X)}/p^{r})/V^{i}$
is represented by the complex $\mathcal{F}^{\cdot}\otimes_{\mathcal{O}_{W(X)/p^{r}}}(\mathcal{O}_{W(X)}/p^{r})/V^{i})$
which is the complex whose terms are $\mathcal{F}^{j}/V^{i}(\mathcal{O}_{W(X)}/p^{r})$.
Thus we obtain a map of complexes whose terms are $\mathcal{M}^{j}/V^{i}(\mathcal{O}_{W(X)}/p^{r})\to\mathcal{F}^{j}/V^{i}(\mathcal{O}_{W(X)}/p^{r})$.
Taking the inverse limit we obtain a map of complexes 
\[
\mathcal{N}{}^{\cdot}\to(\mathcal{\widehat{F}})^{\cdot}
\]
and the latter is a complex representing $\widehat{\mathcal{M}}$.
To show this is an isomorphism, we may work locally and assume $X=\text{Spec}(A)$
possesses local coordinates. Choose a quasi-isomorphism $\mathcal{F}^{\cdot}\to\mathcal{H}^{\cdot}$,
where $\mathcal{H}^{\cdot}$ is a k-flat complex whose terms are direct
sums of $\Phi^{*}\mathcal{D}_{\mathfrak{X}_{r}}^{(0)}$ (this is possible
as $\mathcal{M}^{\cdot}$ is accessible and quasicoherent). It suffices
to show that the induced map $(\widehat{\mathcal{M}})^{\cdot}\to(\mathcal{\widehat{H}})^{\cdot}$
is an isomorphism.

To see it, apply $\otimes_{W_{r}(k)}^{L}k$. We can evaluate it using
the fact that each $\mathcal{M}^{i}\tilde{=}\Phi^{*}\mathcal{N}^{i}$,
and that the completion along $V^{i}$ is equal to the completion
along the filtration $F^{i}$ (by the lemma above); this makes $\widehat{\mathcal{M}}^{i}$
an inverse limit of a surjective system of modules, each of which
is a direct sum of copies of $\mathcal{M}^{i}$. This shows that 
\[
(\widehat{\mathcal{M}})^{\cdot}\otimes_{W_{r}(k)}^{L}k\tilde{=}\widehat{\Phi}^{*}(\mathcal{N}^{\cdot}\otimes_{W(k)}^{L}k)
\]
By the same token, if $\mathcal{H}^{i}={\displaystyle \bigoplus_{J_{i}}}\Phi^{*}\mathcal{D}_{\mathfrak{X}_{r}}^{(0)}$,
then 
\[
(\mathcal{\widehat{H}})^{i}\otimes_{W_{r}(k)}^{L}k\tilde{=}\widehat{\Phi}^{*}({\displaystyle \bigoplus_{J_{i}}}\mathcal{D}_{X}^{(0)})
\]
and since $\widehat{\Phi}^{*}$ is exact and conservative we obtain
the result. 
\end{proof}

\subsection{\label{subsec:The-crystalline-version}The crystalline version}

To relate the theory considered here to crystals, we run the above
program with $\mathcal{\widehat{D}}_{W(X),crys}^{(m)}$ in place of
$\mathcal{\widehat{D}}_{W(X)}^{(m)}$. All of the theorems above go
through without any change, and we obtain 
\begin{thm}
\label{thm:Crystalline-version!}1) There is a well-defined $(\mathcal{\widehat{D}}_{W(X),\text{crys}}^{(0)}/p,\mathcal{D}_{X,\text{crys}}^{(0)})$
bimodule, denoted $\mathcal{B}_{X,\text{crys}}^{(0)}$ which is locally
projective over $\mathcal{\widehat{D}}_{W(X),\text{crys}}^{(0)}/p$,
and locally faithfully flat as a right $\mathcal{D}_{X,\text{crys}}^{(0)}$-module.
The associated functor 
\[
\mathcal{B}_{X,\text{crys}}^{(0)}\otimes_{\mathcal{D}_{X,\text{crys}}^{(0)}}?:\mathcal{D}_{X,\text{crys}}^{(0)}-\text{mod}\to\mathcal{\widehat{D}}_{W(X),\text{crys}}^{(0)}/p-\text{mod}
\]
is exact and fully faithful, and admits an exact right adjoint $\mathcal{H}om_{\mathcal{\widehat{D}}_{W(X),\text{crys}}^{(0)}/p}(\mathcal{B}_{X,\text{crys}}^{(0)},?)$.
We have 
\[
\mathcal{H}om_{\mathcal{\widehat{D}}_{W(X),\text{crys}}^{(0)}/p}(\mathcal{B}_{X,\text{crys}}^{(0)},\mathcal{B}_{X,\text{crys}}^{(0)}\otimes_{\mathcal{D}_{X,\text{crys}}^{(0)}}\mathcal{M})\tilde{\to}\mathcal{M}
\]
for all $\mathcal{M}\in\mathcal{D}_{X,\text{crys}}^{(0)}-\text{mod}$.
Therefore, the functor 
\[
\mathcal{B}_{X,\text{ crys}}\otimes_{\mathcal{D}_{X,\text{crys}}^{(0)}}^{L}:D(\mathcal{D}_{X,\text{crys}}^{(0)}-\text{mod})\to D(\mathcal{\widehat{D}}_{W(X),\text{crys}}^{(0)}/p-\text{mod})
\]
is fully faithful and we have 
\[
R\mathcal{H}om_{\mathcal{\widehat{D}}_{W(X),\text{crys}}^{(0)}/p}(\mathcal{B}_{X,\text{crys}}^{(0)},\mathcal{B}_{X,\text{crys}}^{(0)}\otimes_{\mathcal{D}_{X,\text{crys}}^{(0)}}\mathcal{M}^{\cdot})\tilde{\to}\mathcal{M}^{\cdot}
\]
for all $\mathcal{M}^{\cdot}\in D(\mathcal{D}_{X,\text{crys}}^{(0)}-\text{mod})$.

2) A module $\mathcal{M}^{\cdot}\in D(\mathcal{\widehat{D}}_{W(X),\text{crys}}^{(0)}/p-\text{mod})$
is called accessible if it is of the form $\mathcal{B}_{X,\text{crys}}^{(0)}\otimes_{\mathcal{D}_{X,\text{crys}}^{(0)}}^{L}\mathcal{N}^{\cdot}$
for some $\mathcal{N}^{\cdot}\in D(\mathcal{D}_{X,\text{crys}}^{(0)}-\text{mod})$.
Let $\mathcal{M}^{\cdot}\in D_{cc}(\mathcal{\widehat{D}}_{W(X),\text{crys}}^{(0)}-\text{mod})$.
Then $\mathcal{M}^{\cdot}$ is said to be accessible\textbf{ }if $\mathcal{M}^{\cdot}\otimes_{W(k)}^{L}k$
is accessible inside $D(\mathcal{\widehat{D}}_{W(X)}^{(0)}/p-\text{mod})$;
similarly $\mathcal{M}^{\cdot}\in D(\mathcal{\widehat{D}}_{W(X),\text{crys}}^{(0)}/p^{r}-\text{mod})$
is said to accessible if $\mathcal{M}^{\cdot}\otimes_{W_{r}(k)}^{L}k$
is accessible inside $D(\mathcal{\widehat{D}}_{W(X)}^{(0)}/p-\text{mod})$.

The complex $\mathcal{M}^{\cdot}$ is accessible iff, for any open
affine $\text{Spec}(A)\subset X$, which admits local coordinates,
and any coordinatized lift of Frobenius $\Phi$, we have 
\[
\mathcal{M}^{\cdot}\tilde{\to}\Phi^{*}\mathcal{\widehat{D}}_{\mathfrak{X,\text{crys}}}^{(0)}\widehat{\otimes}_{\widehat{\mathcal{D}}_{\mathfrak{X},\text{crys}}^{(0)}}^{L}\mathcal{N}^{\cdot}
\]
for a complex $\mathcal{N}^{\cdot}\in D_{cc}(\mathcal{\widehat{D}}_{\mathfrak{X}}^{(0)}-\text{mod})$
In particular, the latter condition is independent of the choice of
$\Phi$; one has the analogous statement for $D(\mathcal{\widehat{D}}_{W(X),\text{crys}}^{(0)}/p^{r}-\text{mod})$,
in which case $\mathcal{M}^{\cdot}$ is accessible iff $\mathcal{H}^{i}(\mathcal{M}^{\cdot})$
is for all $i$. 

3) Suppose that $\mathfrak{X}$ is a smooth formal scheme over $W(k)$
whose special fibre is $X$ (it might not exist in general). Then
there is an equivalence of categories 
\[
D_{cc}(\mathcal{\widehat{D}}_{\mathfrak{X,\text{crys}}}^{(0)}-\text{mod})\to D_{\text{acc}}(\mathcal{\widehat{D}}_{W(X),\text{crys}}^{(0)}-\text{mod})
\]
and the analogous fact holds for schemes $\mathfrak{X}_{r}$ which
are smooth over $W_{r}(k)$. 
\end{thm}

In addition, we have a new subcategory to consider: 
\begin{defn}
An object $\mathcal{M}$ in $\mathcal{\widehat{D}}_{W(X),\text{crys}}^{(0)}/p^{r}-\text{mod}_{\text{acc}}$
is said to be locally nilpotent if each local section is annihilated
by $F^{m}(\mathcal{\widehat{D}}_{W(X),\text{crys}}^{(0)}/p^{r})$
for all $m>>0$ (here, $F^{m}$ is the image of the operator filtration
on $\mathcal{\widehat{D}}_{W(X),\text{crys}}^{(0)}$). If $X=\text{Spec}(A)$
possesses local coordinates this is equivalent to $\mathcal{M}\tilde{=}\Phi^{*}\mathcal{N}$
where $\mathcal{N}$ is locally nilpotent over $\mathcal{D}_{\mathfrak{X}_{r},\text{crys}}^{(0)}$
in the usual sense (each section is killed by a power of the ideal
$\mathcal{I}_{r}$). One makes the same definition over $\mathcal{\widehat{D}}_{W(X)}^{(0)}/p^{r}$,
and the category of locally nilpotent accessible $\mathcal{\widehat{D}}_{W(X)}^{(0)}/p^{r}$-modules
is equivalent to the category of locally nilpotent accessible $\mathcal{\widehat{D}}_{W(X),\text{crys}}^{(0)}/p^{r}$-modules;
indeed, the local nilpotence condition ensures that the $\mathcal{\widehat{D}}_{W(X)}^{(0)}/p^{r}$-module
structure extends uniquely to a $\mathcal{\widehat{D}}_{W(X),\text{crys}}^{(0)}/p^{r}$-module. 
\end{defn}

We use the subscript $\text{ln}$ to denote locally nilpotent objects
in a given category. Then we have 
\begin{prop}
The map $\mathcal{\widehat{D}}_{W(X)}^{(0)}/p^{r}\to\mathcal{\widehat{D}}_{W(X),\text{crys}}^{(0)}/p^{r}$
yields a functor 
\[
\mathcal{\widehat{D}}_{W(X),\text{crys}}^{(0)}/p^{r}-\text{mod}_{\text{acc},\text{qcoh,ln}}\to\mathcal{\widehat{D}}_{W(X)}^{(0)}/p^{r}-\text{mod}_{\text{acc},\text{qcoh}}
\]
This functor is fully faithful, and its image consists of all sheaves
$\mathcal{M}$ such which are accessible, quasicoherent, and locally
nilpotent.
\end{prop}

\begin{proof}
The map of algebras $\mathcal{\widehat{D}}_{W(X)}^{(0)}/p^{r}\to\mathcal{\widehat{D}}_{W(X),\text{crys}}^{(0)}/p^{r}$
yields a forgetful functor $D(\mathcal{\widehat{D}}_{W(X),\text{crys}}^{(0)}/p^{r}-\text{mod})\to D(\mathcal{\widehat{D}}_{W(X)}^{(0)}/p^{r}-\text{mod})$.
Suppose $\mathcal{M}$ is an accessible, quasicoherent, and locally
nilpotent module over $\mathcal{\widehat{D}}_{W(X),\text{crys}}^{(0)}/p^{r}$.
Restricting to $X=\text{Spec}(A)$ for some $A$ which possesses local
coordinates, we may write 
\[
\mathcal{M}\tilde{\to}\Phi^{*}\mathcal{D}_{\mathfrak{X}_{r},\text{crys}}^{(0)}\otimes_{\mathcal{D}_{\mathfrak{X}_{r},\text{crys}}^{(0)}}^{L}\mathcal{N}
\]
As $\mathcal{N}$ is quasicoherent and locally nilpotent over $\mathcal{D}_{\mathfrak{X}_{r},\text{crys}}^{(0)}$,
it is a union of its coherent $\mathcal{D}_{\mathfrak{X}_{r},\text{crys}}^{(0)}$-submodules
which are nilpotent (i.e. the entire module is annihilated by some
power of $\mathcal{I}_{r}$). Let $\mathcal{N}'$ be such a coherent
submodule. Then the natural map
\[
\Phi^{*}\mathcal{D}_{\mathfrak{X}_{r}}^{(0)}\otimes_{\mathcal{D}_{\mathfrak{X}_{r}}^{(0)}}^{L}\mathcal{N}'\to\Phi^{*}\mathcal{D}_{\mathfrak{X}_{r},\text{crys}}^{(0)}\otimes_{\mathcal{D}_{\mathfrak{X}_{r},\text{crys}}^{(0)}}^{L}\mathcal{N}'
\]
is an isomorphism. To see this, note that $\mathcal{N}'$ is also
coherent over $\mathcal{D}_{\mathfrak{X}_{r}}^{(0)}$. Therefore both
sides are isomorphic to ${\displaystyle \lim_{\leftarrow}(\mathcal{O}_{W(X)}/p^{r})/V^{m}\otimes_{\mathcal{O}_{\mathfrak{X}_{r}}}\mathcal{N}'}$
(use a presentation of $\mathcal{N}'$ over $\mathcal{D}_{\mathfrak{X}_{r}}^{(0)}$
on the left hand side and over $\mathcal{D}_{\mathfrak{X}_{r},\text{crys}}^{(0)}$
on the right hand side). As the tensor product commutes with inductive
limits, we see that 
\[
\Phi^{*}\mathcal{D}_{\mathfrak{X}_{r}}^{(0)}\otimes_{\mathcal{D}_{\mathfrak{X}_{r}}^{(0)}}^{L}\mathcal{N}\tilde{\to}\Phi^{*}\mathcal{D}_{\mathfrak{X}_{r},\text{crys}}^{(0)}\otimes_{\mathcal{D}_{\mathfrak{X}_{r},\text{crys}}^{(0)}}^{L}\mathcal{N}
\]
In other words, $\mathcal{M}$ is also accessible and quasicoherent
when regarded as a module over $\mathcal{\widehat{D}}_{W(X)}^{(0)}/p^{r}$.
Thus the natural functor 
\[
\mathcal{\widehat{D}}_{W(X),\text{crys}}^{(0)}/p^{r}-\text{mod}_{\text{acc},\text{qcoh,ln}}\to\mathcal{\widehat{D}}_{W(X)}^{(0)}/p^{r}-\text{mod}_{\text{acc},\text{qcoh}}
\]
has image in $\mathcal{\widehat{D}}_{W(X)}^{(0)}/p^{r}-\text{mod}_{\text{acc},\text{qcoh}}$
as claimed, and is clearly onto the category of accessible, quasicoherent,
and locally nilpotent modules. For the full faithfulness, for any
two objects $\mathcal{M}_{1},\mathcal{M}_{2}$ in $\mathcal{\widehat{D}}_{W(X),\text{crys}}^{(0)}/p^{r}-\text{mod}_{\text{acc},\text{qcoh,ln}}$,
we consider the morphism of sheaves 
\[
\mathcal{H}om_{\mathcal{\widehat{D}}_{W(X),\text{crys}}^{(0)}/p^{r}}(\mathcal{M}_{1},\mathcal{M}_{2})\to\mathcal{H}om_{\mathcal{\widehat{D}}_{W(X)}^{(0)}/p^{r}}(\mathcal{M}_{1},\mathcal{M}_{2})
\]
we will be done if we can show it is an isomorphism; to do so, we
can work locally and suppose $X=\text{Spec}(A)$ and $\mathcal{M}_{i}=\Phi^{*}\mathcal{N}_{i}$
(for $i=1,2$). Then we need to show that 
\[
\mathcal{H}om_{\mathcal{\mathcal{D}}_{\mathfrak{X}_{r},\text{crys}}^{(0)}}(\mathcal{N}_{1},\mathcal{N}_{2})\to\mathcal{H}om_{\mathcal{D}_{\mathfrak{X}_{r}}^{(0)}}(\mathcal{N}_{1},\mathcal{N}_{2})
\]
is an isomorphism; but this is clear. 
\end{proof}
Now we give the relation with crystals in the usual sense. It reads
\begin{thm}
\label{thm:Embedding-of-crystals}Consider the category $\text{Crys}_{W_{r}(k)}(X)$
of crystals on $X$ over $W_{r}(k)$. There is an exact functor
\[
\eta:\text{Crys}_{W_{r}(k)}(X))\to\mathcal{\widehat{D}}_{W(X),\text{crys}}^{(0)}/p^{r}-\text{mod}_{\text{acc},}
\]
which is fully faithful. Let 
\[
\epsilon:\text{Crys}_{W_{r}(k)}(X))\to\mathcal{\widehat{D}}_{W(X)}^{(0)}/p^{r}-\text{mod}
\]
denote the composition of the restriction of $\eta$ with the forgetful
functor to $\mathcal{\widehat{D}}_{W(X)}^{(0)}/p^{r}-\text{mod}$.
Then, upon restriction to $\text{Qcoh}(\text{Crys}_{W_{r}(k)}(X)))$
$\epsilon$ lands in $\mathcal{\widehat{D}}_{W(X)}^{(0)}/p^{r}-\text{mod}_{\text{acc,qcoh}}$.
The functor
\[
\epsilon:D_{\text{qcoh}}(\text{Crys}_{W_{r}(k)}(X))\to D_{\text{acc},\text{qcoh}}(\mathcal{\widehat{D}}_{W(X)}^{(0)}/p^{r}-\text{mod})
\]
is also fully faithful. The image of $\epsilon$ consists of all complexes
$\mathcal{M}^{\cdot}$ such that, for each $i$, $\mathcal{H}^{i}(\mathcal{M}^{\cdot})$
is accessible, quasicoherent, and locally nilpotent.
\end{thm}

\begin{proof}
Let $\tilde{\mathcal{M}}$ be an element of $\text{Crys}_{W_{r}(k)}(X)$.
For every open subset of the form $U=\text{Spec}(A)$ (which possesses
local coordinates), and every flat lift $\mathcal{A}_{r}$ to $W_{r}(k)$,
$\tilde{\mathcal{M}}$ produces a canonically defined sheaf with locally
nilpotent flat connection on $\mathcal{A}_{r}$, denoted $(\mathcal{N},\nabla)$.
For any map $\Phi:\mathcal{A}_{r}\to W(A)/p^{r}$ coming from a coordinatized
lift of Frobenius, we have $\widehat{\Phi}^{*}\mathcal{N}$, the completion
(along $V^{i}(\mathcal{O}_{W(X)}/p^{r})$) of $\Phi^{*}\mathcal{N}$;
this is exactly the sheaf over $W(X)_{p^{r}=0}$ that the theory of
the crystalline site attaches to $\mathcal{N}$ (coming from the fact
that $\Phi:\mathcal{A}_{r}\to W(A)/p^{r}$ is an inverse limit of
pd thickenings). 

For any other such map $\Psi$, the crystalline theory yields a canonical
isomorphism $\widehat{\Phi}^{*}\mathcal{N}\tilde{\to}\widehat{\Psi}^{*}\mathcal{N}$,
and, using \prettyref{cor:Crystalline-Iso}, we have that this map
agrees with the canonical isomorphism coming from theory of accessible
$\mathcal{\widehat{D}}_{W(A),\text{crys}}^{(0)}/p^{r}$-modules (this
is because the morphism is realized via $\Phi^{*}\mathcal{D}_{\mathfrak{X}_{r}}^{(0)}\widehat{\otimes}_{\mathcal{D}_{\mathfrak{X}_{r}}^{(0)}}\mathcal{N}\tilde{\to}\Psi^{*}\mathcal{D}_{\mathfrak{X}_{r}}^{(0)}\widehat{\otimes}_{\mathcal{D}_{\mathfrak{X}_{r}}^{(0)}}\mathcal{N}$).
So the sheaf $\tilde{\mathcal{M}}_{W(X)_{p^{r}=0}}$ on $W(X)_{p^{r}=0}$
which the theory of the crystalline site attaches to $\tilde{\mathcal{M}}$
carries a canonical $\mathcal{\widehat{D}}_{W(X),\text{crys}}^{(0)}/p^{r}$
action. This sheaf is not accessible, however, there is an exact functor
\[
\mathcal{\widehat{D}}_{W(X),\text{crys}}^{(0)}/p^{r}-\text{mod}\to\mathcal{\widehat{D}}_{W(X),\text{crys}}^{(0)}/p^{r}-\text{mod}_{\text{acc}}
\]
which is right adjoint to the inclusion functor (c.f. the discussion
right below \prettyref{def:D-acc} below); we denote it $\mathcal{M}\to\mathcal{M}_{\text{acc}}$;
if $X=\text{Spec}(A)$ then $(\widehat{\Phi}^{*}\mathcal{N})_{\text{acc}}=\Phi^{*}\mathcal{N}$.
Therefore we define the functor 
\[
\eta(\tilde{\mathcal{M}})=(\tilde{\mathcal{M}}_{W(X)_{p^{r}=0}})_{\text{acc}}
\]
Clearly its restriction to quasicoherent crystals lands in quasicoherent
accessible modules (this is the argument of the previous proposition).
There is a canonical map $\mathcal{H}om_{\text{crys}}(\tilde{\mathcal{M}}_{1},\tilde{\mathcal{M}}_{2})\to\mathcal{H}om_{\mathcal{\widehat{D}}_{W(X),\text{crys}}^{(0)}/p^{r}}(\eta(\tilde{\mathcal{M}}_{1}),\eta(\tilde{\mathcal{M}}_{2}))$
of sheaves in the Zariski topology of $X$; to show it is an isomorphism
we may work locally, but there the claim reduces to the full-faithfulness
of $\Phi^{*}$. 

Now, after restricting $\eta$ to quasicoherent sheaves, we obtain
$\epsilon$, which is also fully faithful. Let us consider the derived
version. Let $\mathcal{M}_{i}^{\cdot}$ ($i=1,2$) be elements of
$D_{\text{qcoh}}(\text{Crys}_{W_{r}(k)}(X)))$; and replace $\mathcal{M}_{2}$
with a K-injective resolution, $\mathcal{K}^{\cdot}$ Then we have
\[
R\mathcal{H}om_{\text{crys}}(\mathcal{M}_{1}^{\cdot},\mathcal{M}_{2}^{\cdot})=\mathcal{H}om_{\text{crys}}(\mathcal{M}_{1},\mathcal{K}^{\cdot})\to\mathcal{H}om_{\mathcal{\widehat{D}}_{W(X)}^{(0)}/p^{r}}(\epsilon(\mathcal{M}_{1}),\epsilon(\mathcal{K}^{\cdot}))
\]
where the latter two $\mathcal{H}om$ indicate Hom in the homotopy
category of chain complexes. Now we have a map
\[
R\mathcal{H}om_{\mathcal{\widehat{D}}_{W(X)}^{(0)}/p^{r}}(\epsilon(\mathcal{M}_{1}),\epsilon(\mathcal{K}^{\cdot}))\to\mathcal{H}om_{\mathcal{\widehat{D}}_{W(X)}^{(0)}/p^{r}}(\epsilon(\mathcal{M}_{1}),\epsilon(\mathcal{K}^{\cdot}))
\]
and we will show that both of these maps are isomorphisms. This is
a local question, which boils down to the following: if $\mathcal{N}_{i}^{\cdot}$
($i=1,2$) are elements of the derived category $\text{Qcoh}_{\text{ln}}(\mathcal{D}_{\mathfrak{X}_{r}}^{(0)})$
(the category of quasicoherent $\mathcal{D}_{\mathfrak{X}_{r}}^{(0)}$-modules
which are locally nilpotent), then the map 
\[
R\mathcal{H}om_{\text{Qcoh}_{\text{ln}}(\mathcal{D}_{\mathfrak{X}_{r}}^{(0)})}(\mathcal{N}_{1}^{\cdot},\mathcal{N}_{2}^{\cdot})\to R\mathcal{H}om_{\mathcal{D}_{\mathfrak{X}_{r}}^{(0)}-\text{mod}}(\mathcal{N}_{1}^{\cdot},\mathcal{N}_{2}^{\cdot})
\]
is an isomorphism. This follows from by regarding a quasicoherent
$\mathcal{D}_{\mathfrak{X}_{r}}^{(0)}$-module as a quasicoherent
sheaf on the center $\mathcal{Z}(\mathcal{D}_{\mathfrak{X}_{r}}^{(0)})$,
and applying the functor of local cohomology with support along the
zero section (compare, e.g. \cite{key-3}, lemma 3.1.7). 
\end{proof}
It is worth noting that we also have equivalences 
\[
D(\text{Qcoh}(\text{Crys}_{W_{r}(k)}(X)))\tilde{\to}D_{\text{qcoh}}(\text{Crys}_{W_{r}(k)}(X))
\]
and 
\[
D(\mathcal{\widehat{D}}_{W(X)}^{(0)}/p^{r}-\text{mod}_{\text{acc,qcoh,ln}})\tilde{\to}D_{\text{acc},\text{qcoh},\text{ln}}(\mathcal{\widehat{D}}_{W(X)}^{(0)}/p^{r}-\text{mod})
\]
These can be proved in a very similar way to the classical statement
\[
D(\text{QCoh}(X))\tilde{\to}D_{\text{qcoh}}(\mathcal{O}_{X}-\text{mod})
\]
c.f. \cite{key-19}, corollary 5.5.

\subsection{Frobenius descent}

In this subsection we explain how Berthelot's fundamental theorem
on the independence of the Frobenius in arithmetic $\mathcal{D}$-module
theory also follows from our results. To set things up, let us recall
the main results of that theory: 
\begin{thm}
1) Let $A$ be a smooth algebra which possesses local coordinates;
let $\mathcal{A}$ be a lift of $A$ and let $F:\mathcal{A}\to\mathcal{A}$
be a coordinatized lift of Frobenius. Let consider the $(\widehat{\mathcal{D}}_{\mathcal{A}}^{(m+1)},\widehat{\mathcal{D}}_{\mathcal{A}}^{(m)})$
bisubmodule of $\text{End}_{W(k)}(\mathcal{A})$ generated by $F$.
Then the natural map $F^{*}\widehat{\mathcal{D}}_{\mathcal{A}}^{(m)}\to\text{End}_{W(k)}(\mathcal{A})$
which takes $a\otimes P\to a\cdot F(P)$ is an isomorphism onto this
bisubmodule. Furthermore, this bimodule induces an equivalence of
categories 
\[
\mathcal{M}\to F^{*}\mathcal{\widehat{D}}_{\mathcal{A}}^{(m)}\otimes_{\mathcal{\widehat{D}}_{\mathcal{A}}^{(m)}}\mathcal{M}:=F^{*}\mathcal{M}
\]
from $\mathcal{\widehat{D}}_{\mathcal{A}}^{(m)}-\text{mod}$ to $\mathcal{\widehat{D}}_{\mathcal{A}}^{(m+1)}-\text{mod}$.

2) If $F_{1},F_{2}$ are two lifts of Frobenius as a above, then (assuming
$p>2$ when $m=0$), there is a canonical isomorphism of bimodules
\[
F_{1}^{*}\mathcal{\widehat{D}}_{\mathcal{A}}^{(m)}\tilde{\to}F_{2}^{*}\mathcal{\widehat{D}}_{\mathcal{A}}^{(m)}
\]
In particular, if $X$ is an arbitrary smooth scheme over $k$ and
$\mathfrak{X}$ is a lift, then there is a globally defined equivalence
of categories $F^{*}:\mathcal{D}_{\mathfrak{X}}^{(m)}-\text{mod}\to\mathcal{D}_{\mathfrak{X}}^{(m+1)}-\text{mod}$. 
\end{thm}

This theorem is proved in \cite{key-2}, section 2.3. (for part $1)$)
and \cite{key-2}, theorem 2.2.5 for part $2)$; c.f. also \cite{key-21},
corollary 13.3.8. 

Part $1)$ can be proved rather rapidly from the methods of this paper.
Our construction of $\Phi^{*}\mathcal{\widehat{D}}_{\mathcal{A}}^{(m)}$
is a more elaborate version of the construction of $F^{*}\mathcal{\widehat{D}}_{\mathcal{A}}^{(m)}$,
and, the proof of \prettyref{thm:Projective!} is a more elaborate
version of the proof that $F^{*}\mathcal{\widehat{D}}_{\mathcal{A}}^{(m)}$
is projective over $\mathcal{D}_{\mathcal{A}}^{(m+1)}$. Once this
is known, the crucial isomorphism 
\[
\mathcal{\widehat{D}}_{\mathcal{A}}^{(m+1)}\tilde{\to}\text{End}_{\mathcal{\widehat{D}}_{\mathcal{A}}^{(m),\text{opp}}}(F^{*}\mathcal{\widehat{D}}_{\mathcal{A}}^{(m)})
\]
can be proved after reduction mod $p$, where it is an elementary
computation (c.f. \cite{key-16}, proposition 4.18 for the case $m=0$). 

Now let us address part $2)$. We in fact have the following for Witt
differential operators: 
\begin{lem}
There is a morphism of sheaves of algebras $F:\widehat{\mathcal{D}}_{W(X)}^{(m)}\to\widehat{\mathcal{D}}_{W(X)}^{(m+1)}$
whose restriction to $\mathcal{O}_{W(X)}$ is the Witt-vector Frobenius
map. 
\end{lem}

\begin{proof}
It suffices to show this locally, so without loss of generality we
assume $X=\text{Spec}(A)$ possesses local coordinates. Then for an
operator $P\in\widehat{\mathcal{D}}_{W(A)}^{(m)}$ we consider the
operator $F\circ P\circ F^{-1}$. When $P$ is multiplication by an
element $a\in W(A)$, this operator is $F(a)$. Using \prettyref{prop:HS-and-F},
we see that $F\{\partial_{i}\}_{\lambda}F^{-1}=\{\partial_{i}\}_{p\lambda}$
for all $\lambda$; so that \prettyref{thm:Basis} ensures us that
$F\circ P\circ F^{-1}\in\widehat{\mathcal{D}}_{W(A)}^{(m+1)}$ as
required.
\end{proof}
Now, we can prove
\begin{thm}
\label{thm:Indep-of-Frobenius!}Let $(\mathcal{A},F)$ be as above,
and $X=\text{Spec(A)}$; let $\Phi:\mathcal{A}\to W(A)$ the associated
map. Then for $\mathcal{M}^{\cdot}\in D_{\text{acc}}(\widehat{\mathcal{D}}_{W(X)}^{(m)}-\text{mod})$,
if $\mathcal{M}^{\cdot}=\Phi^{*}\mathcal{N}^{\cdot}$ then 
\[
F^{*}\mathcal{M}^{\cdot}\tilde{\to}\Phi^{*}(F^{*}\mathcal{N}^{\cdot})
\]
where on the left we have the pullback with respect to the algebra
morphism $\widehat{\mathcal{D}}_{W(X)}^{(m)}\to\widehat{\mathcal{D}}_{W(X)}^{(m+1)}$,
and on the right we have Berthelot's Frobenius pullback. In particular
for two such morphisms $F_{1},F_{2}$ we obtain a canonical isomorphism
$F_{1}^{*}\mathcal{N}^{\cdot}\tilde{\to}F_{2}^{*}\mathcal{N}^{\cdot}$. 
\end{thm}

\begin{proof}
By definition we have 
\[
F^{*}\mathcal{M}^{\cdot}=\widehat{\mathcal{D}}_{W(X)}^{(m+1)}\widehat{\otimes}_{\widehat{\mathcal{D}}_{W(X)}^{(m)}}^{L}\mathcal{M}^{\cdot}
\]
where the action of $\widehat{\mathcal{D}}_{W(X)}^{(m)}$ on $\widehat{\mathcal{D}}_{W(X)}^{(m+1)}$
is via the map $F$. As $\mathcal{M}^{\cdot}=\Phi^{*}\widehat{\mathcal{D}}_{\mathfrak{X}}^{(m)}\widehat{\otimes}_{\widehat{\mathcal{D}}_{\mathfrak{X}}^{(m)}}^{L}\mathcal{N}^{\cdot}$
we are left to ponder 
\[
\widehat{\mathcal{D}}_{W(X)}^{(m+1)}\widehat{\otimes}_{\widehat{\mathcal{D}}_{W(X)}^{(m)}}^{L}\Phi^{*}\widehat{\mathcal{D}}_{\mathfrak{X}}^{(m)}
\]
As $\Phi^{*}\widehat{\mathcal{D}}_{\mathfrak{X}}^{(m)}$ is locally
projective over $\widehat{\mathcal{D}}_{W(X)}^{(m)}$ this is concentrated
in a single degree, and is in fact a summand of $\widehat{\mathcal{D}}_{W(X)}^{(m+1)}$.
We are considering $\widehat{\mathcal{D}}_{W(X)}^{(m+1)}$ as a $(\widehat{\mathcal{D}}_{W(X)}^{(m+1)},\widehat{\mathcal{D}}_{W(X)}^{(m)})$
bimodule; and it is exactly the bisubmodule of $\mathcal{E}nd_{W(k)}(\mathcal{O}_{W(X)})$
which is locally generated by $F$. As $\Phi^{*}\widehat{\mathcal{D}}_{\mathfrak{X}}^{(m)}$
is the $(\widehat{\mathcal{D}}_{W(X)}^{(m)},\widehat{\mathcal{D}}_{\mathfrak{X}}^{(m)})$
bisubmodule of $\mathcal{H}om_{W(k)}(\mathcal{O}_{\mathfrak{X}},\mathcal{O}_{W(X)})$
locally generated by $\Phi$, we conclude that 
\[
\widehat{\mathcal{D}}_{W(X)}^{(m+1)}\widehat{\otimes}_{\widehat{\mathcal{D}}_{W(X)}^{(m)}}^{L}\Phi^{*}\widehat{\mathcal{D}}_{\mathfrak{X}}^{(m)}
\]
is the $(\widehat{\mathcal{D}}_{W(X)}^{(m+1)},\widehat{\mathcal{D}}_{\mathfrak{X}}^{(m)})$
bisubmodule of $\mathcal{H}om_{W(k)}(\mathcal{O}_{\mathfrak{X}},\mathcal{O}_{W(X)})$
locally generated by $F\circ\Phi$. By the same token we have that
\[
\Phi^{*}\widehat{\mathcal{D}}_{\mathfrak{X}}^{(m+1)}\widehat{\otimes}_{\widehat{\mathcal{D}}_{\mathfrak{X}}^{(m+1)}}^{L}F^{*}\widehat{\mathcal{D}}_{\mathfrak{X}}^{(m)}
\]
is the $(\widehat{\mathcal{D}}_{W(X)}^{(m+1)},\widehat{\mathcal{D}}_{\mathfrak{X}}^{(m)})$
bisubmodule of $\mathcal{H}om_{W(k)}(\mathcal{O}_{\mathfrak{X}},\mathcal{O}_{W(X)})$
locally generated by $\Phi\circ F$. As $F\circ\Phi=\Phi\circ F$
we obtain that these two bimodules are isomorphic and so 
\[
\widehat{\mathcal{D}}_{W(X)}^{(m+1)}\widehat{\otimes}_{\widehat{\mathcal{D}}_{W(X)}^{(m)}}^{L}\mathcal{M}^{\cdot}=\widehat{\mathcal{D}}_{W(X)}^{(m+1)}\widehat{\otimes}_{\widehat{\mathcal{D}}_{W(X)}^{(m)}}^{L}\Phi^{*}\widehat{\mathcal{D}}_{\mathfrak{X}}^{(m)}\widehat{\otimes}_{\widehat{\mathcal{D}}_{\mathfrak{X}}^{(m)}}^{L}\mathcal{N}^{\cdot}
\]
\[
\tilde{\to}\Phi^{*}\widehat{\mathcal{D}}_{\mathfrak{X}}^{(m+1)}\widehat{\otimes}_{\widehat{\mathcal{D}}_{\mathfrak{X}}^{(m+1)}}^{L}F^{*}\widehat{\mathcal{D}}_{\mathfrak{X}}^{(m)}\widehat{\otimes}_{\widehat{\mathcal{D}}_{\mathfrak{X}}^{(m)}}^{L}\mathcal{N}^{\cdot}
\]
\[
=\Phi^{*}(F^{*}\mathcal{N}^{\cdot})
\]
as required. 
\end{proof}
Restricting oneself to the subcategory of coherent accessible modules,
we obtain an isomorphism of functors $F_{1}^{*}\tilde{\to}F_{2}^{*}$
on that category. This recovers part $2)$ of Berthelot's theorem
above. In addition, one may use this technique to get a rather quick
proof that the category of coherent modules in Berthelot's theory
of overconvergent $\mathcal{D}$-modules, admits a Frobenius action.
We will return to this topic elsewhere. 

\section{Operations on Accessible Modules}

In this chapter we develop the basic operations on the accessible
$\widehat{\mathcal{D}}_{W(X)}^{(m)}$-modules; we assume throughout
that the level $m\geq0$. 

\subsection{\label{subsec:Operations-Left-Right}Operations on modules: Right
Modules and the left-right interchange}

In this section, we'll construct the category of accessible \emph{right
}$\mathcal{D}_{W(X)}^{(m)}$-modules and explain the left-right interchange
in this context. Before doing so, we give a quick review of the left-right
interchange in the classical situation: on $\mathfrak{X}$, the line
bundle $\omega_{\mathfrak{X}}$ carries the natural structure of a
right $\widehat{\mathcal{D}}_{\mathfrak{X}}^{(m)}$ module. One can
define the structure of a sheaf of algebras on 
\[
\omega_{\mathfrak{X}}\otimes_{\mathcal{O}_{\mathfrak{X}}}\mathcal{\widehat{D}}_{\mathfrak{X}}^{(m)}\otimes_{\mathcal{O}_{\mathfrak{X}}}\omega_{\mathfrak{X}}^{-1}
\]
via 
\[
(s_{1}\otimes\Phi_{1}\otimes t_{1})\cdot(s_{2}\otimes\Phi_{2}\otimes t_{2})=s_{1}\otimes\Phi_{1}<t_{1},s_{1}>\Phi_{2}\otimes t_{2}
\]
snd we have an there is an isomorphism of sheaves of algebras 
\[
\omega_{\mathfrak{X}}\otimes_{\mathcal{O}_{\mathfrak{X}}}\mathcal{\widehat{D}}_{\mathfrak{X}}^{(m)}\otimes_{\mathcal{O}_{\mathfrak{X}}}\omega_{\mathfrak{X}}^{-1}\tilde{=}\mathcal{\widehat{D}}_{\mathfrak{X}}^{(m),\text{op}}
\]
(this uses the right action of $\mathcal{\widehat{D}}_{\mathfrak{X}}^{(m)}$
on $\omega_{\mathfrak{X}}$). One obtains from this the fact that
$\mathcal{M}\to\omega_{\mathfrak{X}}\otimes_{\mathcal{O}_{\mathfrak{X}}}\mathcal{M}$
is an equivalence of categories from $\mathcal{\widehat{D}}_{\mathfrak{X}}^{(m)}-\text{mod}$
to $\mathcal{\widehat{D}}_{\mathfrak{X}}^{(m),\text{op}}-\text{mod}$.
This extends to an equivalence on the derived categories, which preserves
cohomologically complete objects on each side. We wish to extend this
to the accessible $\widehat{\mathcal{D}}_{W(X)}^{(m)}$-modules. 

Let's begin with the local situation: 
\begin{prop}
\label{prop:Construction-of-Phi-!}Let $X=\text{Spec}(A)$ where $A$
possesses local coordinates; let $F:\mathcal{A}\to\mathcal{A}$ a
coordinatized lift of Frobenius, and let $\Phi:\mathcal{A}\to W(A)$;
let $\pi:W(A)\to\mathcal{A}$ be the associated projection as in \prettyref{cor:Projector!}.
Then 
\[
\Phi^{!}(\mathcal{\widehat{D}}_{\mathcal{A}}^{(m)}):=\text{Hom}_{\mathcal{\widehat{D}}_{W(A)}^{(m)}}(\Phi^{*}\widehat{\mathcal{D}}_{\mathcal{A}}^{(m)},\mathcal{\widehat{D}}_{W(A)}^{(m)})\tilde{=}\pi\cdot\mathcal{\widehat{D}}_{W(A)}^{(m)}
\]
is a $(\mathcal{\widehat{D}}_{\mathcal{A}}^{(m)},\mathcal{\widehat{D}}_{W(A)}^{(m)})$
bimodule, which is a projective as a right $\mathcal{\widehat{D}}_{W(A)}^{(m)}$-module.
Further, there is an isomorphism of right $\mathcal{\widehat{D}}_{\mathcal{A}}^{(m)}$-modules
\[
\tilde{\text{Hom}}_{\mathcal{A}}(W(A),\mathcal{\widehat{D}}_{\mathcal{A}}^{(m)})\tilde{\to}\Phi^{!}(\mathcal{\widehat{D}}_{\mathcal{A}}^{(m)})
\]
where $\tilde{\text{Hom}}_{\mathcal{A}}(W(A),\mathcal{\widehat{D}}_{\mathcal{A}}^{(m)})\subset\text{Hom}_{\mathcal{A}}(W(A),\mathcal{\widehat{D}}_{\mathcal{A}}^{(m)})$
consists of those $\mathcal{A}$-linear morphisms $\epsilon:W(A)\to\mathcal{\widehat{D}}_{\mathcal{A}}^{(m)}$
which satisfy $\epsilon(V^{r}(W(A)))\subset p^{r}\mathcal{\widehat{D}}_{\mathcal{A}}^{(m)}$. 
\end{prop}

\begin{proof}
The first line follows immediately from $\Phi^{*}\widehat{\mathcal{D}}_{\mathcal{A}}^{(m)}=\mathcal{\widehat{D}}_{W(A)}^{(m)}\cdot\pi$;
the $(\mathcal{\widehat{D}}_{\mathcal{A}}^{(m)},\mathcal{\widehat{D}}_{W(A)}^{(m)})$-bimodule
structure is deduced from the right action of $\mathcal{\widehat{D}}_{W(A)}^{(m)}$
on\\
$\text{Hom}_{\mathcal{\widehat{D}}_{W(A)}^{(m)}}(\Phi^{*}\widehat{\mathcal{D}}_{\mathcal{A}}^{(m)},\mathcal{\widehat{D}}_{W(A)}^{(m)})$
and the right action of $\widehat{\mathcal{D}}_{\mathcal{A}}^{(m)}$
on $\Phi^{*}\widehat{\mathcal{D}}_{\mathcal{A}}^{(m)}$. So the only
thing that needs to be done is to construct the isomorphism 
\[
\tilde{\text{Hom}}_{\mathcal{A}}(W(A),\mathcal{\widehat{D}}_{\mathcal{A}}^{(m)})\tilde{\to}\Phi^{!}(\mathcal{\widehat{D}}_{\mathcal{A}}^{(m)})
\]
First, define $\tilde{\text{Hom}}_{\mathcal{A}}(W(A),\mathcal{A})$
to be the set of $\mathcal{A}$ linear maps $\epsilon:W(A)\to\mathcal{A}$
satisfying $\epsilon(V^{r}(W(A)))\subset p^{r}\mathcal{A}$. By the
lemma below we have that if $\epsilon\in\tilde{\text{Hom}}_{\mathcal{A}}(W(A),\mathcal{A})$,
then $\Phi\circ\epsilon\in\pi\cdot\mathcal{\widehat{D}}_{W(A)}^{(m)}$;
thus we obtain an inclusion $\iota:\tilde{\text{Hom}}_{\mathcal{A}}(W(A),\mathcal{A})\subset\Phi^{!}(\mathcal{\widehat{D}}_{\mathcal{A}}^{(m)})$.
Now, I claim that there is an isomorphism 
\begin{equation}
\tilde{\text{Hom}}_{\mathcal{A}}(W(A),\mathcal{\widehat{D}}_{\mathcal{A}}^{(m)})\tilde{=}\tilde{\text{Hom}}_{\mathcal{A}}(W(A),\mathcal{A})\widehat{\otimes}_{\mathcal{A}}\mathcal{\widehat{D}}_{\mathcal{A}}^{(m)}\label{eq:iso-for-shriek}
\end{equation}
where $\widehat{\otimes}$ here denotes the completion with respect
to the filtration 
\[
F^{m}(\tilde{\text{Hom}}_{\mathcal{A}}(W(A),\mathcal{A})):=\{\epsilon\in\tilde{\text{Hom}}_{\mathcal{A}}(W(A),\mathcal{A})|\epsilon(W(A))\subseteq p^{m}\mathcal{A}\}
\]
To prove this note that, via the inclusion $\mathcal{A}\subset\mathcal{\widehat{D}}_{\mathcal{A}}^{(m)}$
we have an inclusion $\tilde{\text{Hom}}_{\mathcal{A}}(W(A),\mathcal{A})\subset\tilde{\text{Hom}}_{\mathcal{A}}(W(A),\mathcal{\widehat{D}}_{\mathcal{A}}^{(m)})$,
which induces a morphism 
\[
\tilde{\text{Hom}}_{\mathcal{A}}(W(A),\mathcal{A})\otimes_{\mathcal{A}}\mathcal{\widehat{D}}_{\mathcal{A}}^{(m)}\to\tilde{\text{Hom}}_{\mathcal{A}}(W(A),\mathcal{\widehat{D}}_{\mathcal{A}}^{(m)})
\]
whose completion induces \prettyref{eq:iso-for-shriek} (that it is
an isomorphism is easily checked in local coordinates). 

So, via the right action of $\mathcal{\widehat{D}}_{\mathcal{A}}^{(m)}$
on $\Phi^{!}(\mathcal{\widehat{D}}_{\mathcal{A}}^{(m)})$ we obtain
a map
\[
\tilde{\text{Hom}}_{\mathcal{A}}(W(A),\mathcal{A})\otimes_{\mathcal{A}}\mathcal{\widehat{D}}_{\mathcal{A}}^{(m)}\to\Phi^{!}(\mathcal{\widehat{D}}_{\mathcal{A}}^{(m)})
\]
given by $(\epsilon,P)\to\iota(\epsilon)\cdot P$. As elements of
$F^{m}(\tilde{\text{Hom}}_{\mathcal{A}}(W(A),\mathcal{A}))$ are contained
in $V^{r}(\mathcal{\widehat{D}}_{W(A)}^{(0)})$ (as shown in the lemma
directly below), we may complete to obtain a map 
\[
a:\tilde{\text{Hom}}_{\mathcal{A}}(W(A),\mathcal{A})\widehat{\otimes}_{\mathcal{A}}\mathcal{\widehat{D}}_{\mathcal{A}}^{(m)}\to\Phi^{!}(\mathcal{\widehat{D}}_{\mathcal{A}}^{(m)})
\]
Identifying $\Phi^{!}(\mathcal{\widehat{D}}_{\mathcal{A}}^{(m)})\tilde{=}\pi\cdot\mathcal{\widehat{D}}_{W(A)}^{(m)}$
with a set of $W(k)$-linear morphisms from $W(A)$ to $\mathcal{A}$,
this map is given by the composition 
\[
(\epsilon,P)\to P\circ\epsilon:W(A)\to\mathcal{A}
\]
It follows directly that the map $a$ is injective. To see the surjectivity,
it suffices to show that the image of $a$ is preserved under the
right action of $\mathcal{\widehat{D}}_{W(A)}^{(m)}$ (since $\pi$
is clearly contained in the image). The proof of this is given in
the lemma directly below
\end{proof}
In the proof above, we used the 
\begin{lem}
1) Let $\phi_{I/p^{r}}:W(A)\to\mathcal{A}$ be the $\mathcal{A}$-linear
map which takes $p^{r}T^{I/p^{r}}$ to $p^{r}$ and which takes $p^{r'}T^{J/p^{r'}}$
to $0$ for all $J\neq I$. Then $\Phi\circ\phi_{I/p^{r}}\in\pi\cdot\mathcal{\widehat{D}}_{W(A)}^{(0)}$. 

2) Let $P\in\mathcal{\widehat{D}}_{W(A)}^{(0)}$. Then $\phi_{I/p^{r}}(P\cdot):W(A)\to\mathcal{A}$
is contained in the image of $\tilde{\text{Hom}}_{\mathcal{A}}(W(A),\mathcal{A})\widehat{\otimes}_{\mathcal{A}}\mathcal{\widehat{D}}_{\mathcal{A}}^{(m)}$
in $\text{Hom}_{W(k)}(W(A),\mathcal{A})$. 
\end{lem}

\begin{proof}
These will be variants of the proofs of \prettyref{cor:Projector!}
and \prettyref{prop:Construction-of-bimodule}. 

1) Start with case $A=k[T]$, $\mathcal{A}=W(k)<<T>>$, and the Frobenius
lift is the standard one. Let $i\in\{1,\dots p^{r}-1\}$ so that $\text{val}_{p}(i)=0$.
Let $r'\geq r$. As in \prettyref{prop:construction-of-projector},
let $\pi_{i/p^{r}}:W_{r'+1}(A)\to W_{r'+1}(A)$ denote the projection
operator which takes any monomial of the form $p^{r}T^{j+i/p^{r}}$
to itself (for each $j\in\mathbb{Z}_{\geq0}$), and all other monomials
to $0$. By induction on $r'$ we shall construct an operator satisfying
\begin{equation}
\Phi_{r'}\circ\phi_{i/p^{r}}^{r'}:=\{d\}_{i/p^{r}}\circ\pi_{i/p^{r}}-\sum_{i=1}^{p^{r'+1}-1}f_{i}(T^{i/p^{r'}}\{d\}_{i/p^{r'}})\circ\{d\}_{i/p^{r}}\circ\pi_{i/p^{r}}\label{eq:operator!}
\end{equation}
for some $f_{i}\in\mathbb{Z}_{p}[X]$, so that $\phi_{i/p^{r}}^{r'}\equiv\phi_{i/p^{r}}^{r'-1}\phantom{i}\text{mod}\phantom{i}V^{r}(\mathcal{\widehat{D}}_{W(A)}^{(0)})$.
The inverse limit of these operators is then the required $\phi_{i/p^{r}}$. 

We will work now with the copy of $W_{r'+1}(A^{(r')})$ contained
in $\mathcal{A}_{r'+1}=W_{r'+1}(k)[T]$ and construct the operator
there. We will proceed by induction; when $r'=r$ we have that 
\[
d^{[i]}\circ\pi_{i/p^{r}}
\]
is already equal to $\phi_{i/p^{r}}^{r}$. Assuming by induction that
$\phi_{i/p^{r}}^{r'-1}$ has been constructed which satisfies \prettyref{eq:operator!},
consider the operator 
\[
d^{[p^{r'-r}i]}\circ\pi_{i/p^{r}}-\sum_{i=1}^{p^{r'}-1}f_{i}(T^{i}d^{[i]})\circ d^{[p^{r'-r}i]}\circ\pi_{i/p^{r}}
\]
on $\mathcal{A}_{r'+1}$. Arguing exactly as in the proof of \prettyref{prop:polynomial-in-AP},
and employing the induction hypothesis, we see that the evaluation
of this operator on a term of the form $p^{r}T^{p^{r'-r}i}T^{ap^{r'}}$
yields $p^{r'}h(a)T^{ap^{r'}}$ where $h$ is a polynomial in $\mathbb{Z}_{p}[X]$.
So we can add a term of the form 
\[
-p^{r'}h(T^{p^{r'+1}}d^{[p^{r'+1}]})\circ d^{[p^{r'-r}i]}\circ\pi_{i/p^{r}}
\]
to construct $\phi^{r'}$, as desired. 

To construct $\phi_{I/p^{r}}$ in general, use the inclusions $\mathcal{\widehat{D}}_{W(k[T_{i}])}^{(0)}\to\mathcal{\widehat{D}}_{W(A)}^{(0)}$
(as noted in \prettyref{cor:Projector!}) and take products of operators
in $\mathcal{\widehat{D}}_{W(k[T_{i}])}^{(0)}$. 

$2)$ As $\tilde{\text{Hom}}_{\mathcal{A}}(W(A),\mathcal{A})\widehat{\otimes}_{\mathcal{A}}\mathcal{\widehat{D}}_{\mathcal{A}}^{(m)}\to\text{Hom}_{W(k)}(W(A),\mathcal{A})$
is injective, we identify $\tilde{\text{Hom}}_{\mathcal{A}}(W(A),\mathcal{A})\widehat{\otimes}_{\mathcal{A}}\mathcal{\widehat{D}}_{\mathcal{A}}^{(m)}$
with its image. By an $F$-conjugation argument, we can assume $m=0$
from now on. As $A$ is etale over $k[T_{1},\dots,T_{n}]$, we can
assume without loss of generality that $A=k[T_{1},\dots,T_{n}]$ with
the coordinate lift of Frobenius. 

We begin by considering $\phi_{I/p^{m}}(\{\partial_{j}\}_{1/p^{r}}\cdot)$;
without loss of generality suppose $j=1$. Choose $l\geq r,m$ and
consider $\phi_{I/p^{m}}$ acting on $W_{l+1}(A)$; as usual we work
with the copy of $W_{l+1}(A^{(l)})\subset\mathcal{A}_{l+1}=W(k)[T_{1},\dots,T_{n}]$.
Then $\phi_{I/p^{m}}$ is nonzero on monomials of the form $T^{K}$
where $K=(p^{l-m}i_{1}+p^{l}j_{1},\dots,p^{l-m}i_{n}+p^{l}j_{n})$
where $j_{i}\in\mathbb{Z}_{\geq0}$ for all $i$, and zero on all
other monomials. In this context, $\{\partial_{1}\}_{1/p^{r}}$ becomes
the operator $\partial_{1}^{[p^{l-r}]}$. The operator $\phi_{I/p^{m}}(\partial_{1}^{[p^{l-r}]}\cdot)$
is nonzero on monomials of the form $T^{K'}$where $K'=(p^{l-m}i_{1}+p^{l}j_{1}+p^{l-r},\dots,p^{l-m}i_{n}+p^{l}j_{n})$,
and zero on all the others.

We have 
\[
d_{1}^{[p^{l-r}]}T_{1}^{p^{l-m}i_{1}+p^{l}j_{1}+p^{l-r}}T_{2}^{p^{l-m}i_{2}+p^{l}j_{2}}\cdots T_{n}^{p^{l-m}i_{n}+p^{l}j_{n}}
\]
\[
={p^{l-m}i_{1}+p^{l}j_{1}+p^{l-r} \choose p^{l-r}}T_{1}^{p^{l-m}i_{1}+p^{l}j_{1}}T_{2}^{p^{l-m}i_{2}+p^{l}j_{2}}\cdots T_{n}^{p^{l-m}i_{n}+p^{l}j_{n}}
\]
However, 
\[
{p^{l-m}i_{1}+p^{l-r}(p^{r}j_{1}+1) \choose p^{l-r}}=f_{l}(p^{r}j_{1}+1)
\]
is a polynomial in $j_{1}$ (by \prettyref{prop:polynomial-in-AP}).
So we deduce that 
\[
\phi_{I/p^{m}}(\partial_{1}^{[p^{l-r}]}\cdot)=p^{a}\lim_{l}f_{l}(1+p^{r}T_{1}\partial_{1})\circ\phi_{I'/p^{b}}
\]
where ${\displaystyle \lim_{l}f_{l}(1+p^{r}T_{1}\partial_{1})}$ is
understood as an element of $\widehat{\mathcal{D}}_{\mathcal{A}}^{(0)}$,
$I'/p^{b}$ is the term $I/p^{m}+(1/p^{r},0,\dots,0)$, and $a=|b-\text{max}\{r,m\}|$.;
thus we see that $\phi_{I/p^{m}}(\partial_{1}^{[p^{l-r}]}\cdot)\in\tilde{\text{Hom}}_{\mathcal{A}}(W(A),\mathcal{A})\widehat{\otimes}_{\mathcal{A}}\mathcal{\widehat{D}}_{\mathcal{A}}^{(0)}$.
As any element of $\tilde{\text{Hom}}_{\mathcal{A}}(W(A),\mathcal{A})\widehat{\otimes}_{\mathcal{A}}\mathcal{\widehat{D}}_{\mathcal{A}}^{(0)}$
can be written as a sum of the form 
\[
\sum_{I,r}\phi_{I/p^{r}}\cdot P_{I}
\]
we see that $\tilde{\text{Hom}}_{\mathcal{A}}(W(A),\mathcal{A})\widehat{\otimes}_{\mathcal{A}}\mathcal{\widehat{D}}_{\mathcal{A}}^{(0)}$
is closed under the action of $\{\partial_{j}\}_{1/p^{r}}$, and therefore
under the action of every term of the form $\{\partial\}_{J/p^{r}}$.
By \prettyref{thm:Basis}, it suffices to show that it is closed under
the action of terms of the form $T^{K/p^{r}}\{\partial\}_{J/p^{r}}$.
That follows from the above by noting that the operator 
\[
\phi_{I/p^{m}}(T^{K/p^{r}}\cdot):V^{r}(W(A))\to\mathcal{A}
\]
is the restriction to $V^{r}(W(A))$ of a sum of operators of the
form $\phi_{J/p^{l}}$ (and this is straightforward to verify). 
\end{proof}
From the above we obtain 
\begin{defn}
Set $\mathcal{H}om_{\mathcal{\widehat{D}}_{W(X)}^{(m)}/p}(\mathcal{B}_{X}^{(m)},\mathcal{\widehat{D}}_{W(X)}^{(m)}/p):=\mathcal{B}_{X}^{(m),r}$.
This is a sheaf of $(\widehat{\mathcal{D}}_{X}^{(m)},\mathcal{\widehat{D}}_{W(X)}^{(m)}/p)$
bimodules, which is locally isomorphic to $\Phi^{!}(\mathcal{\widehat{D}}_{X}^{(m)})$
for a coordinatized lift of Frobenius $\Phi$. 
\end{defn}

In particular, this bimodule always exists when $n=1$. We can thus
copy over \prettyref{def:Accessible} and obtain 
\begin{defn}
\label{def:right-accessible}1) A module $\mathcal{N}\in\mathcal{\widehat{D}}_{W(X)}^{(m),\text{op}}/p-\text{mod}$
is accessible if is of the form $\mathcal{M}\otimes_{\mathcal{D}_{X}^{(m)}}\mathcal{B}_{X}^{(m),r}$
for some $\mathcal{M}\in\mathcal{D}_{X}^{(m),\text{op}}-\text{mod}$. 

2) A complex $\mathcal{N}^{\cdot}\in D(\mathcal{\widehat{D}}_{W(X)}^{(m),\text{op}}/p-\text{mod})$
is accessible if is of the form $\mathcal{M}^{\cdot}\otimes_{\mathcal{D}_{X}^{(m)}}^{L}\mathcal{B}_{X}^{(m),r}$
for some $\mathcal{M}^{\cdot}\in D(\mathcal{D}_{X}^{(m),\text{op}}-\text{mod})$.

3) Let $r\geq1$. A complex $\mathcal{N}^{\cdot}\in D(\mathcal{\widehat{D}}_{W(X)}^{(m),\text{op}}/p^{r}-\text{mod})$
is accessible if $\mathcal{N}^{\cdot}\otimes_{W_{r}(k)}^{L}k$ is
accessible in $D(\mathcal{\widehat{D}}_{W(X)}^{(m),\text{op}}/p-\text{mod})$.
Similarly, a complex $\mathcal{N}^{\cdot}\in D_{cc}(\mathcal{\widehat{D}}_{W(X)}^{(m),\text{op}}-\text{mod})$
is accessible if $\mathcal{N}^{\cdot}\otimes_{W(k)}^{L}k$ is accessible
in $D(\mathcal{\widehat{D}}_{W(X)}^{(m),\text{op}}/p-\text{mod})$. 
\end{defn}

The analogues of all the basic results on accessibility (\prettyref{thm:Local-Accessible}
through \prettyref{cor:Bimodule-over-X-r}) all hold without any change,
and with identical proofs, for right modules.

With this in place, we turn to defining the fundamental bimodule $\mathcal{\widehat{D}}_{W(X),\text{acc}}^{(m)}$
and then describing the left-right swap for $\mathcal{\widehat{D}}_{W(X)}^{(m)}$-modules. 

First suppose we are working locally, i.e., $X=\text{Spec}(A)$ possesses
local coordinates. Then the functor 
\[
\mathcal{N}^{\cdot}\to\Phi^{*}\mathcal{\widehat{D}}_{\mathfrak{X}}^{(m)}\widehat{\otimes}_{\mathcal{\widehat{D}}_{\mathfrak{X}}^{(m)}}^{L}R\mathcal{H}om_{\mathcal{\widehat{D}}_{W(X)}^{(m)}}(\Phi^{*}\mathcal{\widehat{D}}_{\mathfrak{X}}^{(m)},\mathcal{N}^{\cdot})
\]
\[
\tilde{\to}\Phi^{*}\mathcal{\widehat{D}}_{\mathfrak{X}}^{(m)}\widehat{\otimes}_{\mathcal{\widehat{D}}_{\mathfrak{X}}^{(m)}}^{L}(\Phi^{!}\mathcal{\widehat{D}}_{\mathfrak{X}}^{(m)}\widehat{\otimes}_{\mathcal{\widehat{D}}_{W(X)}^{(m)}}^{L}\mathcal{N}^{\cdot})
\]
is the right adjoint to the inclusion $D_{acc}(\mathcal{\widehat{D}}_{W(X)}^{(m)}-\text{mod})\to D_{cc}(\mathcal{\widehat{D}}_{W(X)}^{(m)}-\text{mod})$.
Applying this functor to $\widehat{\mathcal{D}}_{W(X)}^{(m)}$ itself,
we obtain $\Phi^{*}\mathcal{\widehat{D}}_{\mathfrak{X}}^{(m)}\widehat{\otimes}_{\mathcal{\widehat{D}}_{\mathfrak{X}}^{(m)}}^{L}\Phi^{!}\mathcal{\widehat{D}}_{\mathfrak{X}}^{(m)}$.
As $\Phi^{!}\mathcal{\widehat{D}}_{\mathfrak{X}}^{(m)}$ is $p$-torsion
free, we see from \prettyref{rem:=00005CPhi-pull-for-bounded-torsion}
that $\Phi^{*}\mathcal{\widehat{D}}_{\mathfrak{X}}^{(m)}\widehat{\otimes}_{\mathcal{\widehat{D}}_{\mathfrak{X}}^{(m)}}^{L}\Phi^{!}\mathcal{\widehat{D}}_{\mathfrak{X}}^{(m)}$
is concentrated in degree $0$; i.e., it is a sheaf. By the uniqueness
of adjoints, we obtain for any pair of lifts $\Phi_{1},\Phi_{2}$,
a canonical isomorphism 
\begin{equation}
\Phi_{1}^{*}\mathcal{\widehat{D}}_{\mathfrak{X}}^{(m)}\widehat{\otimes}_{\mathcal{\widehat{D}}_{\mathfrak{X}}^{(m)}}^{L}\Phi_{1}^{!}\mathcal{\widehat{D}}_{\mathfrak{X}}^{(m)}\tilde{\to}\Phi_{2}^{*}\mathcal{\widehat{D}}_{\mathfrak{X}}^{(m)}\widehat{\otimes}_{\mathcal{\widehat{D}}_{\mathfrak{X}}^{(m)}}^{L}\Phi_{2}^{!}\mathcal{\widehat{D}}_{\mathfrak{X}}^{(m)}\label{eq:canonical-iso}
\end{equation}
of sheaves; for a third lift $\Phi_{3}$ we necessarily have the cocycle
condition. Therefore, for arbitrary $X$, the object $\Phi^{*}\mathcal{\widehat{D}}_{\mathfrak{X}}^{(m)}\widehat{\otimes}_{\mathcal{\widehat{D}}_{\mathfrak{X}}^{(m)}}^{L}\Phi^{!}\mathcal{\widehat{D}}_{\mathfrak{X}}^{(m)}$
glues to form a sheaf on $X$. 
\begin{defn}
\label{def:D-acc}The sheaf of bimodules constructed just above is
called $\mathcal{\widehat{D}}_{W(X),\text{acc}}^{(m)}$. 
\end{defn}

As the natural map 
\[
\Phi^{*}\mathcal{\widehat{D}}_{\mathfrak{X}}^{(m)}\widehat{\otimes}_{\mathcal{\widehat{D}}_{\mathfrak{X}}^{(m)}}^{L}\Phi^{!}\mathcal{\widehat{D}}_{\mathfrak{X}}^{(m)}\to\mathcal{\widehat{D}}_{W(X)}^{(m)}
\]
coming from the adjunction is compatible with the isomorphism \prettyref{eq:canonical-iso},
we obtain a morphism of bimodules $\mathcal{\widehat{D}}_{W(X),\text{acc}}^{(m)}\to\mathcal{\widehat{D}}_{W(X)}^{(m)}$. 

This, in turn, implies that the right adjoint to the inclusion $D_{acc}(\mathcal{\widehat{D}}_{W(X)}^{(m)}-\text{mod})\to D_{cc}(\mathcal{\widehat{D}}_{W(X)}^{(m)}-\text{mod})$
is defined on any $X$, namely, it is given by the functor 
\[
\mathcal{N}^{\cdot}\to\mathcal{\widehat{D}}_{W(X),\text{acc}}^{(m)}\widehat{\otimes}_{\mathcal{\widehat{D}}_{W(X)}^{(m)}}^{L}\mathcal{N}^{\cdot}
\]
Note that 
\[
\mathcal{N}^{\cdot}\to\mathcal{N}^{\cdot}\widehat{\otimes}_{\mathcal{\widehat{D}}_{W(X)}^{(m)}}^{L}\mathcal{\widehat{D}}_{W(X),\text{acc}}^{(m)}
\]
is the right adjoint to the inclusion $D_{acc}(\text{mod}-\mathcal{\widehat{D}}_{W(X)}^{(m)})\to D_{cc}(\text{mod}-\mathcal{\widehat{D}}_{W(X)}^{(m)})$
of the category of right accessible $\mathcal{\widehat{D}}_{W(X)}^{(m)}$
modules (this discussion completes the proof of \prettyref{cor:Bimodule-properties!},
part $2$). Note the the same argument, applied to $\mathcal{\widehat{D}}_{W(X),\text{acc}}^{(m)}/p^{r}$,
shows that the inclusion $D_{acc}(\mathcal{\widehat{D}}_{W(X)}^{(m)}/p^{r}-\text{mod})\to D(\mathcal{\widehat{D}}_{W(X)}^{(m)}/p^{r}-\text{mod})$
admits a right adjoint as well; in this case the functor preserves
the abelian subcategories of accessible modules as well. We note here
also that replacing everywhere $\mathcal{\widehat{D}}_{W(X)}^{(m)}$
with $\mathcal{\widehat{D}}_{W(X),\text{crys}}^{(m)}$ yields the
analogous results for those categories of modules.

Now, let's give an application of the construction of $\Phi^{!}$: 
\begin{prop}
\label{prop:omega-n-is-D-module}For each $m\geq0$ the sheaf $W\omega_{X}$
admits the structure of an accessible right $\mathcal{\widehat{D}}_{W(X)}^{(m)}$-module.
If $X=\text{Spec}(A)$ then choosing a coordinatized lift of Frobenius,
and the corresponding map $\Phi:\mathcal{A}\to W(A)$, we have $W\omega_{X}\tilde{=}\Phi^{!}(\omega_{\mathfrak{X}})$.
For each $m\geq0$ we have $F^{!}W\omega_{X}\tilde{\to}W\omega_{X}$. 
\end{prop}

The proof of this takes a few steps. Echoing the construction of the
functor $F^{*}$ for left $\mathcal{D}$-modules, we have: 
\begin{cor}
Let $\mathfrak{X}=\text{Specf}(\mathcal{A})$ be equipped with a coordinatized
lift of Frobenius $F$, and let $\Phi:\mathcal{A}\to W(A)$ be the
associated morphism. Let $\mathcal{A}^{F}$ denote $\mathcal{A}$
as an $\mathcal{A}$-module via $F$ and let $\pi:\mathcal{A}\to F(\mathcal{A})$
be an $F(\mathcal{A})$-linear projection.

1) Let $F^{!}\widehat{\mathcal{D}}_{\mathcal{A}}^{(m)}$ denote the
$(\widehat{\mathcal{D}}_{\mathcal{A}}^{(m)},\widehat{\mathcal{D}}_{\mathcal{A}}^{(m+1)})$
bi-sub-module of $\text{Hom}_{W(k)}(\mathcal{A},\mathcal{A})$ generated
by $\pi$. Then the natural map
\[
\text{Hom}_{\mathcal{A}}(\mathcal{A}^{F},\mathcal{A})\otimes_{\mathcal{A}}\widehat{\mathcal{D}}_{\mathcal{A}}^{(m)}\to\text{Hom}_{W(k)}(\mathcal{A},\mathcal{A})
\]
which takes $(\phi,P)$ to $P\circ\phi$ induces isomorphisms
\[
\text{Hom}_{\mathcal{A}}(\mathcal{A}^{F},\widehat{\mathcal{D}}_{\mathcal{A}}^{(m)})\tilde{=}\text{Hom}_{\mathcal{A}}(\mathcal{A}^{F},\mathcal{A})\otimes_{\mathcal{A}}\widehat{\mathcal{D}}_{\mathcal{A}}^{(m)}\tilde{\to}F^{!}\widehat{\mathcal{D}}_{\mathcal{A}}^{(m)}
\]
Sheafifying to $\mathfrak{X}=\text{Specf}(\mathcal{A})$, we obtain
a sheaf of bimodules $F^{!}\widehat{\mathcal{D}}_{\mathfrak{X}}^{(m)}$.
Thus there is a functor $\mathcal{M}\to F^{!}\mathcal{M}:=\mathcal{M}\otimes_{\widehat{\mathcal{D}}_{\mathfrak{X}}^{(m)}}F^{!}(\widehat{\mathcal{D}}_{\mathfrak{X}}^{(m)})$
which is in fact an equivalence of categories $\widehat{\mathcal{D}}_{\mathfrak{X}}^{(m)}-\text{mod}\to\widehat{\mathcal{D}}_{\mathfrak{X}}^{(m+1)}-\text{mod}$.
When $\mathcal{M}\in\text{Coh}(\widehat{\mathcal{D}}_{\mathfrak{X}}^{(m)})$
we have 
\[
F^{!}\mathcal{M}\tilde{\to}\text{Hom}_{\mathcal{O}_{\mathfrak{X}}}(F_{*}\mathcal{O}_{\mathfrak{X}},\mathcal{M})
\]

2) Let $F^{!}\widehat{\mathcal{D}}_{W(A)}^{(m)}$ denote the $(\widehat{\mathcal{D}}_{W(A)}^{(m)},\widehat{\mathcal{D}}_{W(A)}^{(m+1)})$
bi-sub-module of $\text{Hom}_{W(k)}(W(A),W(A))$ generated by the
projection $\pi:W(A)\to\mathcal{A}^{(1)}$. Sheafifying, we obtain
a sheaf of bimodules $F^{!}\widehat{\mathcal{D}}_{W(X)}^{(m)}$, and
the induced functor $\mathcal{M}\to F^{!}\mathcal{M}:=\mathcal{M}\otimes_{\widehat{\mathcal{D}}_{W(X)}^{(m)}}\widehat{\mathcal{D}}_{W(X)}^{(m+1)}$
is an equivalence from right-accessible $\widehat{\mathcal{D}}_{W(X)}^{(m)}$-modules
to right-accessible $\widehat{\mathcal{D}}_{W(X)}^{(m+1)}$-modules.
In fact, if $\mathcal{M}\in\text{Qcoh}(\widehat{\mathcal{D}}_{W(X)}^{(m)})$
we have $F^{!}\Phi_{m}^{!}\mathcal{M}\tilde{\to}\Phi_{m+1}^{!}(F^{!}\mathcal{M})$
where on the left $F$ is the Witt-vector Frobenius and on the right
it is the lift of Frobenius on $\mathfrak{X}$. 
\end{cor}

The proof of this result is extremely similar to that of \prettyref{thm:Indep-of-Frobenius!}
and \prettyref{prop:Construction-of-Phi-!}, and will be left to the
reader. 

It will be necessary to also verify an additional compatibility. To
state it, we recall a few basic facts from Grothendieck duality theory
(c.f. \cite{key-54}, as well as \cite{key-21}, appendix for a useful
discussion of the Cartier isomorphism in this context). Recall that
if $f:X\to Y$ is a proper morphism of locally noetherian separated
schemes then we have a functor $f^{!}:D^{b}\text{Coh}(Y)\to D^{b}\text{Coh}(X)$
which satisfies 
\[
R\mathcal{H}om_{\mathcal{O}_{Y}}(Rf_{*}\mathcal{N},\mathcal{M})\tilde{\to}Rf_{*}R\mathcal{H}om_{\mathcal{O}_{X}}(\mathcal{N},f^{!}\mathcal{M})
\]
for any $\mathcal{M}\in D^{b}\text{Coh}(Y)$ and $\mathcal{N}\in D^{b}\text{Coh}(X)$.
If $f$ is a finite morphism finite then we can in fact define
\[
f^{!}(\mathcal{M}):=f^{-1}\mathcal{H}om_{\mathcal{O}_{Y}}(f_{*}\mathcal{O}_{X},\mathcal{M})
\]
from $\text{Coh}(Y)$ to $\text{Coh}(X)$ (the $\mathcal{O}_{X}$-module
structure is given by the action of $\mathcal{O}_{X}$ on $f^{-1}(f_{*}\mathcal{O}_{X})$).
In most cases of interest in this paper $f$ will also be a topological
isomorphism, in which case we will suppress the $f^{-1}$ in the formula. 

Let $\mathfrak{X}_{r}$ be a flat lift of $X$ over $W_{r}(k)$; denote
the morphism to $W_{r}(k)$ by $p_{r}$. Then there is a canonical
isomorphism $p_{r}^{!}(W_{r}(k))\tilde{\to}\omega_{\mathfrak{X}_{r}}$.
It follows that, if $F:\mathfrak{X}_{r}\to\mathfrak{X}_{r}$ is a
lift of Frobenius, we have 
\[
\mathcal{H}om_{\mathcal{O}_{\mathfrak{X}_{r}}}(F_{*}(\mathcal{O}_{\mathfrak{X}_{r}}),\omega_{\mathfrak{X}_{r}}):=F^{!}(\omega_{\mathfrak{X}_{r}})\tilde{\to}\omega_{\mathfrak{X}_{r}}
\]
\begin{cor}
\label{cor:F-!-Compatibility}Let notation be as in the previous corollary.

1) Let $\mathcal{M}\in\text{mod}-\widehat{\mathcal{D}}_{\mathfrak{X}}^{(m)}$
be coherent and $p$-torsion free. Then there is a natural isomorphism
\[
\Phi^{!}(\mathcal{M})\tilde{\to}\tilde{\mathcal{H}om}_{\mathcal{O}_{\mathfrak{X}}}(\mathcal{O}_{W(X)},\mathcal{M})
\]
where $\tilde{\mathcal{H}om}_{\mathcal{O}_{\mathfrak{X}}}(\mathcal{O}_{W(X)},\mathcal{M})\subset\mathcal{H}om_{\mathcal{O}_{\mathfrak{X}}}(\mathcal{O}_{W(X)},\mathcal{M})$
consists of those morphisms which satisfy $\epsilon(V^{m}(\mathcal{O}_{W(X)}))\subset p^{m}\mathcal{N}$
for all $m$ (over every open set in $X$). In other words, 
\[
\Phi^{!}(\mathcal{M})\tilde{\to}\lim_{m}\mathcal{H}om_{\mathcal{O}_{\mathfrak{X}_{m}}}(\mathcal{O}_{W_{m}(X)},\mathcal{M}/p^{m}\mathcal{M})
\]

2) Suppose $\mathcal{N}\in\text{mod}-\widehat{\mathcal{D}}_{W(X)}^{(m)}$
is a right-accessible coherent module, which is $p$-torsion-free.
Then 
\[
F^{!}\mathcal{N}\tilde{=}\tilde{\mathcal{H}om}_{\mathcal{O}_{W(X)}}(F_{*}(\mathcal{O}_{W(X)}),\mathcal{N})
\]
where $\tilde{\mathcal{H}om}_{\mathcal{O}_{W(X)}}(F_{*}(\mathcal{O}_{W(X)}),\mathcal{N})\subset\mathcal{H}om_{\mathcal{O}_{W(X)}}(F_{*}(\mathcal{O}_{W(X)}),\mathcal{N})$
consists of morphisms $\epsilon$ satisfying $\epsilon(V^{l}(\mathcal{O}_{W(X)}))\subset p^{l}\mathcal{N}$
for all $m$ (over every open set in $X$). In other words, 
\[
F^{!}\mathcal{N}\tilde{=}\lim_{l}\mathcal{H}om_{\mathcal{O}_{W_{l}(X)}}(F_{*}(\mathcal{O}_{W_{l}(X)}),\mathcal{N}/p^{l})=\lim_{l}F^{!}(\mathcal{N}/p^{l}\mathcal{N})
\]
\end{cor}

\begin{proof}
$1)$ There is a natural map 
\[
\mathcal{M}\otimes_{\widehat{\mathcal{D}}_{\mathfrak{X}}^{(m)}}\tilde{\mathcal{H}om}_{\mathcal{O}_{\mathfrak{X}}}(\mathcal{O}_{W(X)},\widehat{\mathcal{D}}_{\mathfrak{X}}^{(m)})\to\tilde{\mathcal{H}om}_{\mathcal{O}_{\mathfrak{X}}}(\mathcal{O}_{W(X)},\mathcal{M})
\]
taking $(m,\phi)$ to the morphism $f\to m\cdot\phi(f)$. As everything
in sight is $p$-torsion-free and $p$-adically complete, it suffices
to show that the reduction mod $p$
\[
\mathcal{M}/p\otimes_{\mathcal{D}_{X}^{(m)}}\tilde{\mathcal{H}om}_{\mathcal{O}_{\mathfrak{X}}}(\mathcal{O}_{W(X)},\widehat{\mathcal{D}}_{\mathfrak{X}}^{(m)})/p\to\tilde{\mathcal{H}om}_{\mathcal{O}_{\mathfrak{X}}}(\mathcal{O}_{W(X)},\mathcal{M})/p
\]
is an isomorphism. Further, we have 
\[
\tilde{\mathcal{H}om}_{\mathcal{O}_{\mathfrak{X}}}(\mathcal{O}_{W(X)},\widehat{\mathcal{D}}_{\mathfrak{X}}^{(m)})/p=\prod_{(I,r)}p^{r}T^{I}\cdot\mathcal{D}_{X}^{(m)}
\]
is an infinite product of copies of $\mathcal{D}_{X}^{(m)}$, and
similarly $\tilde{\mathcal{H}om}_{\mathcal{O}_{\mathfrak{X}}}(\mathcal{O}_{W(X)},\mathcal{M})/p$
is an infinite product of copies of $\mathcal{M}/p$. Since $\mathcal{M}$
is finitely generated and $\mathcal{D}_{X}^{(m)}$ is noetherian,
the result follows from this.

$2)$ Write $\mathcal{N}=\Phi_{m}^{!}\mathcal{M}$ for a coherent,
$p$-torsion-free $\mathcal{M}$. Then by the previous result we have
\[
F^{!}\Phi_{m}^{!}\mathcal{M}\tilde{\to}\Phi_{m+1}^{!}F^{!}\mathcal{M}
\]
as $\mathcal{M}$ is $p$-torsion-free and complete we have 
\[
\Phi_{m+1}^{!}F^{!}\mathcal{M}=\lim_{l}\mathcal{H}om_{\mathcal{O}_{\mathfrak{X}_{l}}}(\mathcal{O}_{W_{l}(X)},\mathcal{H}om_{\mathcal{O}_{\mathcal{\mathfrak{X}}_{l}}}(F_{*}\mathcal{O}_{\mathfrak{X}_{l}},\mathcal{M}/p^{l}\mathcal{M}))
\]
using part $1)$ to evaluate $\Phi_{m+1}^{!}$. We wish to prove that
\[
F^{!}\mathcal{N}\tilde{=}\lim_{l}\mathcal{H}om_{\mathcal{O}_{W_{l}(X)}}(F_{*}(\mathcal{O}_{W_{l}(X)}),\mathcal{N}/p^{l})
\]
But again by part $1)$ we have
\[
\mathcal{H}om_{\mathcal{O}_{W_{l}(X)}}(F_{*}(\mathcal{O}_{W_{l}(X)}),\mathcal{N}/p^{l})\tilde{=}\mathcal{H}om_{\mathcal{O}_{W_{l}(X)}}(F_{*}(\mathcal{O}_{W_{l}(X)}),\mathcal{H}om_{\mathcal{O}_{\mathfrak{X}_{l}}}(\mathcal{O}_{W_{l}(X)},\mathcal{M}/p^{l}\mathcal{M}))
\]
So the result follows from the fact that $\Phi\circ F=F\circ\Phi$
and the composition of morphisms in Grothendieck duality. 
\end{proof}
Let us continue with $X=\text{Spec}(A)$ as above. Recalling that
$\omega_{\mathfrak{X}}$ has the structure of a right $\mathcal{\widehat{D}}_{\mathfrak{X}}^{(m)}$
module for every $m\geq0$, and, the above isomorphism in Grothendieck
duality theory yields $F^{!}\omega_{\mathfrak{X}}\tilde{=}\omega_{\mathfrak{X}}$
as $\mathcal{\widehat{D}}_{\mathfrak{X}}^{(m+1)}$ modules. Thus $\omega_{\mathfrak{X}}$
acquires the structure of a $\mathcal{\widehat{D}}_{\mathfrak{X}}:=\lim\mathcal{\widehat{D}}_{\mathfrak{X}}^{(m)}$-module.
We have 
\begin{lem}
For $X=\text{Spec}(A)$ with Frobenius lift $\Phi$ we have that there
exists, for each $m\geq0$ a map of $\mathcal{\widehat{D}}_{W(X)}^{(m)}$-modules
$\eta_{m}:\Phi^{!}\omega_{\mathfrak{X}}\tilde{\to}W\omega_{X}$; the
sequence of maps $\eta_{m}$ are compatible with $F^{!}$ in the sense
that $\eta_{m+1}=F^{!}\eta_{m}$. Therefore $W\omega_{X}$ acquires
the structure of a right $\mathcal{\widehat{D}}_{W(X)}$-module and
there is an isomorphism $\eta_{\infty}:\Phi^{!}\omega_{\mathfrak{X}}\tilde{\to}W\omega_{X}$
satisfying $\eta_{\infty}F^{!}=F^{!}\eta_{\infty}$. 
\end{lem}

\begin{proof}
According to Illusie (\cite{key-64}, c.f. also \cite{key-63}, section
1.9) there is, for each $l\geq0$, a canonical isomorphism 
\[
p_{l}^{!}(W_{l}(k))\tilde{\to}W_{l}\omega_{X}
\]
where $p_{l}$ maps $W_{l}(X)$ to a point. And therefore a canonical
isomorphism 
\[
\Phi^{!}p_{l}^{!}(W_{l}(k))\tilde{\to}\Phi^{!}\omega_{\mathfrak{X}_{l}}\tilde{\to}W_{l}\omega_{X}
\]
taking the inverse limit over $l$ we obtain the required isomorphism
at the level of coherent sheaves. Endowing $\omega_{\mathfrak{X}}$
with its right $\mathcal{\widehat{D}}_{\mathfrak{X}}^{(m)}$-module
structure as above gives $W\omega_{X}$ the structure of a right $\mathcal{\widehat{D}}_{W(X)}^{(m)}$-module
and using \prettyref{cor:F-!-Compatibility} shows the compatibility
with $F^{!}$ as required. 
\end{proof}
Now we can proceed to the 
\begin{proof}
(of \prettyref{prop:omega-n-is-D-module}) Without loss of generality
we may suppose $X$ is geometrically connected. Choose some $U\subset X$
which is open affine, such that $U=\text{Spec}(A)$ possesses local
coordinates; choose a coordinatized lift of Frobenius on $\mathfrak{U}$
as above. Then the above yields an isomorphism $\epsilon_{\mathfrak{U}}:\Phi^{!}\omega_{\mathfrak{U}}\tilde{\to}W\omega_{U}$
which respects the Frobenius structure and the structure of a $\mathcal{\widehat{D}}_{W(X)}$-module.
Choose and fix one such isomorphism. 

Now let $V\subset U$ be another such open affine; pick a Frobenius
lift $\Psi$ on $\mathfrak{V}\subset\mathfrak{U}$. Then we have $\epsilon_{\mathfrak{V}}:\Psi^{!}\omega_{\mathfrak{V}}\tilde{\to}W\omega_{V}$,
and so we obtain the isomorphism $\epsilon_{\mathfrak{U}}\circ\epsilon_{\mathfrak{V}}^{-1}$
on $W\omega_{V}$. This isomorphism commutes with $F^{!}$. But it
is easy to see that the set of such maps is simply $\mathbb{Z}_{p}^{\times}$;
in particular the isomorphism $\Phi^{!}\omega_{\mathfrak{U}}\tilde{\to}\Psi^{!}\omega_{\mathfrak{V}}$
must respect the $\widehat{\mathcal{D}}_{W(X)}$-module structure.
Furthermore, we may rescale the map $\epsilon_{\mathfrak{V}}$ and
obtain $\epsilon_{\mathfrak{U}}\circ\epsilon_{\mathfrak{V}}^{-1}=1$
on $\omega_{dRW(V)}^{n}$. Since $X$ is geometrically connected any
open affine $W\subset X$ intersects $U$ nontrivially and the result
follows directly. 
\end{proof}
Now we are ready to discuss the left-right interchange. Given the
right $\mathcal{\widehat{D}}_{W(X)}^{(m)}$-module structure on $W\omega_{X}$,
and the left $\mathcal{\widehat{D}}_{W(X)}^{(m)}$-module structure
on $\mathcal{O}_{W(X)}$, we see that $\mathcal{H}om_{W(k)}(\mathcal{O}_{W(X)},W\omega_{X})$
acquires the structure of a $(\mathcal{\widehat{D}}_{W(X)}^{(m),\text{opp}},\mathcal{\widehat{D}}_{W(X)}^{(m),\text{opp}})$-bimodule,
while $\mathcal{H}om_{W(k)}(W\omega_{X},\mathcal{O}_{W(X)})$ acquires
the structure of a $(\mathcal{\widehat{D}}_{W(X)}^{(m)},\mathcal{\widehat{D}}_{W(X)}^{(m)})$-bimodule.
\begin{prop}
\label{prop:L-R-bimodule}1) Consider the sheaf $W\omega_{X}\widehat{\otimes}_{\mathcal{O}_{W(X)}}\mathcal{\widehat{D}}_{W(X),\text{acc}}^{(m)}$,
which is the completion of $W\omega_{X}\otimes_{\mathcal{O}_{W(X)}}\mathcal{\widehat{D}}_{W(X),\text{acc}}^{(m)}$
along the filtration\footnote{Here, $G^{j}$ denotes the standard filtration on $W\omega_{X}$ coming
from de Rham-Witt theory } \\
$\{V^{i}(\mathcal{\widehat{D}}_{W(X),\text{acc}}^{(m)})\otimes G^{j}(W\omega_{X})\}_{i+j\geq r}$.
There is a natural injective map
\[
W\omega_{X}\widehat{\otimes}_{\mathcal{O}_{W(X)}}\mathcal{\widehat{D}}_{W(X),\text{acc}}^{(m)}\to\mathcal{H}om_{W(k)}(\mathcal{O}_{W(X)},W\omega_{X})
\]
the image of which is a $(\mathcal{\widehat{D}}_{W(X)}^{(m),\text{opp}},\mathcal{\widehat{D}}_{W(X)}^{(m),\text{opp}})$-bi-sub-module;
therefore \\
$W\omega_{X}\widehat{\otimes}_{\mathcal{O}_{W(X)}}\mathcal{\widehat{D}}_{W(X),\text{acc}}^{(m)}$
is itself a $(\mathcal{\widehat{D}}_{W(X)}^{(m),\text{opp}},\mathcal{\widehat{D}}_{W(X)}^{(m),\text{opp}})$-bimodule. 

2) Consider the sheaf $\tilde{\mathcal{H}om}_{\mathcal{O}_{W(X)}}(W\omega_{X},\mathcal{\widehat{D}}_{W(X),\text{acc}}^{(m)})$
where $\tilde{\mathcal{H}om}$ denotes those morphisms which take
$G^{i}(W\omega_{X})$ to $V^{i}(\mathcal{\widehat{D}}_{W(X),\text{acc}}^{(m)})$.
There is a natural injective map
\[
\tilde{\mathcal{H}om}_{\mathcal{O}_{W(X)}}(W\omega_{X},\mathcal{\widehat{D}}_{W(X),\text{acc}}^{(m)})\to\mathcal{H}om_{W(k)}(W\omega_{X},\mathcal{O}_{W(X)})
\]
the image of which is a $(\mathcal{\widehat{D}}_{W(X)}^{(m)},\mathcal{\widehat{D}}_{W(X)}^{(m)})$-bi-sub-module;
therefore \\
$\tilde{\mathcal{H}om}_{\mathcal{O}_{W(X)}}(W\omega_{X},\mathcal{\widehat{D}}_{W(X),\text{acc}}^{(m)})$
is itself a $(\mathcal{\widehat{D}}_{W(X)}^{(m)},\mathcal{\widehat{D}}_{W(X)}^{(m)})$-bimodule. 
\end{prop}

\begin{proof}
$1)$ The map is given as the completion of the map 
\[
W\omega_{X}\otimes_{\mathcal{O}_{W(X)}}\mathcal{\widehat{D}}_{W(X),\text{acc}}^{(m)}\to\mathcal{H}om_{W(k)}(\mathcal{O}_{W(X)},W\omega_{X})
\]
which takes $\delta\otimes P\to\delta\cdot P$ (where $P\in\mathcal{\widehat{D}}_{W(X),\text{acc}}^{(m)}$
is regarded as an endomorphism of $\mathcal{O}_{W(X)}$). To verify
the properties of this map, we work locally. Assume $X=\text{Spec}(A)$,
$\Phi:\mathcal{A}\to W(A)$ and $\pi$ are given. Then as $\mathcal{\widehat{D}}_{W(X),\text{acc}}^{(m)}=\Phi^{*}\mathcal{D}_{\mathcal{A}}^{(m)}\widehat{\otimes}_{\mathcal{D}_{\mathcal{A}}^{(m)}}\Phi^{!}\mathcal{D}_{\mathcal{A}}^{(m)}$
we start by considering $W\omega_{A}\widehat{\otimes}_{W(A)}\Phi^{*}\mathcal{D}_{\mathcal{A}}^{(m)}$,
the completion of $W\omega_{A}\otimes_{W(A)}\Phi^{*}\mathcal{D}_{\mathcal{A}}^{(m)}$
along $\{V^{i}(W\omega_{A})\otimes_{W(A)}V^{j}(\Phi^{*}\mathcal{D}_{\mathcal{A}}^{(m)})\}_{i+j\geq r}$.
. We have
\[
W\omega_{A}\widehat{\otimes}_{W(A)}\Phi^{*}\mathcal{D}_{\mathcal{A}}^{(m)}\tilde{=}W\omega_{A}\widehat{\otimes}_{\mathcal{A}}\mathcal{D}_{\mathcal{A}}^{(m)}
\]
where on the right the completion is along $\{V^{i}(W\omega_{A})\otimes_{\mathcal{A}}p^{j}\mathcal{D}_{\mathcal{A}}^{(m)})\}_{i+j\geq r}$.
We then have 
\[
W\omega_{A}\widehat{\otimes}_{\mathcal{A}}\mathcal{D}_{\mathcal{A}}^{(m)}\to\text{\ensuremath{\tilde{Hom}}}_{W(k)}(\mathcal{A},W\omega_{A})
\]
where the last map is given by the completion of $\delta\otimes Q\to\delta\cdot Q$.
From the fact that $W\omega_{A}$ is an inverse limit of free $\mathcal{A}$-modules,
this is clearly injective, and we have that its image is closed under
the right action of $\mathcal{\widehat{D}}_{W(A)}^{(m)}$ (the proof
of this is extremely similar to \prettyref{prop:Construction-of-bimodule}).
As $\Phi^{!}\mathcal{D}_{\mathcal{A}}^{(m)}\subset\text{\ensuremath{\tilde{Hom}}}_{W(k)}(W(A),\mathcal{A})$
is preserved under the right action of $\mathcal{\widehat{D}}_{W(A)}^{(m)}$,
the map 
\[
W\omega_{A}\widehat{\otimes}_{W(A)}\Phi^{*}\mathcal{D}_{\mathcal{A}}^{(m)}\widehat{\otimes}_{\mathcal{D}_{\mathcal{A}}^{(m)}}\Phi^{!}\mathcal{D}_{\mathcal{A}}^{(m)}\to\mathcal{H}om_{W(k)}(W(A),W\omega_{A})
\]
defined by composition has image closed under the action of $(\mathcal{\widehat{D}}_{W(A)}^{(m),\text{opp}},\mathcal{\widehat{D}}_{W(A)}^{(m),\text{opp}})$. 

$2)$ This is extremely similar to $1)$, using instead that 
\[
\text{\ensuremath{\tilde{Hom}}}_{W(A)}(W\omega_{A},\Phi^{!}(\mathcal{\widehat{D}}_{\mathcal{A}}^{(m)}))\tilde{=}\text{\ensuremath{\tilde{Hom}}}_{\mathcal{A}}(W\omega_{A},\mathcal{\widehat{D}}_{\mathcal{A}}^{(m)})\to\text{\ensuremath{\tilde{Hom}}}_{W(k)}(W\omega_{A},\mathcal{A})
\]
where the last map is given by noting 
\[
\text{\ensuremath{\tilde{Hom}}}_{\mathcal{A}}(W\omega_{A},\mathcal{\widehat{D}}_{\mathcal{A}}^{(m)})\tilde{=}\text{\ensuremath{\tilde{Hom}}}_{\mathcal{A}}(W\omega_{A},\mathcal{A})\widehat{\otimes}_{\mathcal{A}}\mathcal{\widehat{D}}_{\mathcal{A}}^{(m)}
\]
and then embedding that module into $\text{\ensuremath{\tilde{Hom}}}_{W(k)}(W\omega_{A},\mathcal{A})$
by composing maps. As above the image is closed under the left action
of $\mathcal{\widehat{D}}_{W(A)}^{(m)}$, and the result follows.
\end{proof}
From this we deduce
\begin{cor}
\label{cor:Left-right!}The functors 
\[
(W\omega_{X}\widehat{\otimes}_{\mathcal{O}_{W(X)}}\mathcal{\widehat{D}}_{W(X),\text{acc}}^{(m)})\otimes_{\mathcal{\widehat{D}}_{W(X)}^{(m)}}^{L}
\]
and 
\[
\otimes_{\mathcal{\widehat{D}}_{W(X)}^{(m)}}^{L}\tilde{\mathcal{H}om}_{\mathcal{O}_{W(X)}}(W\omega_{X},\mathcal{\widehat{D}}_{W(X),\text{acc}}^{(m)})
\]
give inverse equivalences of categories from left accessible to right
accessible $\mathcal{\widehat{D}}_{W(X)}^{(m)}$-modules.
\end{cor}

\begin{proof}
First let us examine the local situation.

We have
\[
W\omega_{A}\widehat{\otimes}_{W(A)}\Phi^{*}\mathcal{\widehat{D}}_{\mathcal{A}}^{(m)}\tilde{=}W\omega_{A}\widehat{\otimes}_{\mathcal{A}}\mathcal{\widehat{D}}_{\mathcal{A}}^{(m)}
\]
\[
\tilde{=}\tilde{\text{Hom}}_{\mathcal{A}}(W(A),\omega_{\mathcal{A}})\widehat{\otimes}_{\mathcal{A}}\mathcal{\widehat{D}}_{\mathcal{A}}^{(m)}
\]
\[
\tilde{=}\tilde{\text{Hom}}_{\mathcal{A}}(W(A),\omega_{\mathcal{A}}\otimes_{\mathcal{A}}\mathcal{\widehat{D}}_{\mathcal{A}}^{(m)})
\]
\[
\tilde{=}\omega_{\mathcal{A}}\otimes_{\mathcal{A}}\tilde{\text{Hom}}_{A}(W(A),\mathcal{\widehat{D}}_{\mathcal{A}}^{(m)})=\omega_{\mathcal{A}}\otimes_{\mathcal{A}}\Phi^{!}\mathcal{\widehat{D}}_{\mathcal{A}}^{(m)}
\]
where, in the first line, $W\omega_{A}\widehat{\otimes}_{\mathcal{A}}\mathcal{\widehat{D}}_{\mathcal{A}}^{(m)}$
is the completion of $W\omega_{A}\otimes_{\mathcal{A}}\mathcal{\widehat{D}}_{\mathcal{A}}^{(m)}$
along $\{V^{i}(W\omega_{A})\otimes p^{j}\mathcal{\widehat{D}}_{\mathcal{A}}^{(m)}\}$;
the second isomorphism is from $W\omega_{A}\tilde{=}\tilde{\text{Hom}}_{\mathcal{A}}(W(A),\omega_{\mathcal{A}})$,
and the third and fourth are from the fact that $\omega_{\mathcal{A}}$
is locally free over $\mathcal{A}$. 

It follows that 
\[
(W\omega_{A}\widehat{\otimes}_{W(A)}\mathcal{\widehat{D}}_{W(A),\text{acc}}^{(m)})\otimes_{\mathcal{\widehat{D}}_{W(A)}^{(m)}}\Phi^{*}\mathcal{\widehat{D}}_{\mathcal{A}}^{(m)}\tilde{\to}W\omega_{A}\widehat{\otimes}_{W(A)}\Phi^{*}\mathcal{\widehat{D}}_{\mathcal{A}}^{(m)}
\]
and therefore
\[
(W\omega_{A}\widehat{\otimes}_{W(A)}\mathcal{\widehat{D}}_{W(A),\text{acc}}^{(m)})\otimes_{\mathcal{\widehat{D}}_{W(A)}^{(m)}}^{L}\Phi^{*}\mathcal{M}^{\cdot}\tilde{=}(W\omega_{A}\widehat{\otimes}_{W(A)}\Phi^{*}\mathcal{\widehat{D}}_{\mathcal{A}}^{(m)})\widehat{\otimes}_{\mathcal{\widehat{D}}_{\mathcal{A}}^{(m)}}^{L}\mathcal{M}^{\cdot}
\]
\[
(\omega_{\mathcal{A}}\otimes_{\mathcal{A}}\Phi^{!}\mathcal{\widehat{D}}_{\mathcal{A}}^{(m)})\widehat{\otimes}_{\mathcal{\widehat{D}}_{\mathcal{A}}^{(m)}}^{L}\mathcal{M}^{\cdot}=\Phi^{!}(\omega_{\mathcal{A}}\otimes_{\mathcal{A}}\mathcal{M}^{\cdot})
\]
so this functor agrees with the usual left-right interchange over
$\mathcal{A}$. Now, the functor $(W\omega_{X}\widehat{\otimes}_{\mathcal{O}_{W(X)}}\mathcal{\widehat{D}}_{W(X),\text{acc}}^{(m)})\otimes_{\mathcal{\widehat{D}}_{W(X)}^{(m)}}$
admits an adjoint; namely \\
$R\mathcal{H}om_{\mathcal{\widehat{D}}_{W(X)}^{(m),\text{op}}}((W\omega_{X}\widehat{\otimes}_{\mathcal{O}_{W(X)}}\mathcal{\widehat{D}}_{W(X),\text{acc}}^{(m)}))$,
and an extremely similar argument shows that it also preserves the
subcategories of accessible modules; as it must be the inverse locally,
we see that it is so globally as well.

The proof for the functor $\otimes_{\mathcal{\widehat{D}}_{W(A)}^{(m)}}^{L}\tilde{\mathcal{H}om}_{\mathcal{O}_{W(X)}}(W\omega_{X},\mathcal{\widehat{D}}_{W(X),\text{acc}}^{(m)})$
is essentially identical. 
\end{proof}

\subsection{\label{subsec:Operations-Pull}Operations on modules: Pull-Back}

Throughout this section let $m\geq0$. Let $\varphi:\mathfrak{X}\to\mathfrak{Y}$
be a morphism of smooth formal schemes over $W(k)$. Recall that Berthelot
in \cite{key-2}, section 3.2, has shown that $\mathcal{\widehat{D}}_{\mathfrak{X}\to\mathfrak{Y}}^{(m)}:=\varphi^{*}(\mathcal{\widehat{D}}_{\mathfrak{Y}}^{(m)})$
carries the structure of a left $\mathcal{\widehat{D}}_{\mathfrak{X}}^{(m)}$-module
(by $\varphi^{*}$ we mean the $p$-adically completed pullback).
By definition $\varphi^{*}\mathcal{\widehat{D}}_{\mathfrak{Y}}^{(m)}$
carries the structure of a right $\varphi^{-1}(\mathcal{\widehat{D}}_{\mathfrak{Y}}^{(m)})$-module.
This, in turn allows one to define the functor $\varphi^{*}$ via
\[
L\varphi^{*}(\mathcal{M}):=\varphi^{*}\mathcal{\widehat{D}}_{\mathfrak{Y}}^{(j)}\widehat{\otimes}_{\varphi^{-1}(\mathcal{\widehat{D}}_{\mathfrak{Y}})}^{L}\varphi^{-1}(\mathcal{M})
\]
as usual, the completion means the cohomological completion. One sets
$\varphi^{!}:=L\varphi^{*}[d_{X/Y}]$ (where $d_{X/Y}=\text{dim}(X)-\text{dim}(Y)$). 

Now, if $\varphi:X\to Y$, we wish to define the analogous pullback
functor $L\varphi^{*}$ from accessible $\mathcal{\widehat{D}}_{W(Y)}^{(m)}$-modules
to accessible $\mathcal{\widehat{D}}_{W(X)}^{(m)}$-modules. Recall
that, by the functoriality of the Witt vectors, $\varphi$ give rise
to a morphism of ringed spaces $W\varphi:(X,\mathcal{O}_{WX)})\to(Y,\mathcal{O}_{W(Y)})$.
Thus we have a map $W\varphi^{\#}:W\varphi^{-1}(\mathcal{O}_{W(Y)})\to\mathcal{O}_{W(X)}$;
there is also an induced map $W\varphi^{\#}:W\varphi^{-1}(\mathcal{O}_{W(Y)}/p^{r})\to\mathcal{O}_{W(X)}/p^{r}$.
We shall construct a suitable sheaf of $(\mathcal{\widehat{D}}_{W(X)}^{(m)},W\varphi^{-1}(\mathcal{\widehat{D}}_{W(Y)}^{(m)}))$
bimodules; before doing so, let us remark on a slight generalization
of the notion of accessibility. 

Starting in characteristic $p$, we note that there is an exact, conservative
functor 
\[
\mathcal{Q}\to\mathcal{B}_{X}^{(m)}\otimes_{\mathcal{D}_{X}^{(m)}}\mathcal{Q}\otimes_{\varphi^{-1}(\mathcal{D}_{Y}^{(m)})}\varphi^{-1}(\mathcal{B}_{Y}^{(m),r})
\]
from the category of $(\mathcal{D}_{X}^{(m)},\varphi^{-1}(\mathcal{D}_{Y}^{(m)}))$-bimodules
to the category of \\
$(\mathcal{D}_{W(X)}^{(m)}/p,W\varphi^{-1}(\mathcal{D}_{W(Y)}^{(m)})/p)$-bimodules,
with right adjoint given by 
\[
\mathcal{P}\to\mathcal{B}_{X}^{(m),r}\otimes_{\mathcal{D}_{W(X)}^{(m)}/p}\mathcal{P}\otimes_{W\varphi^{-1}(\mathcal{D}_{W(Y)}^{(m)})/p)}W\varphi^{-1}(\mathcal{B}_{Y}^{(m)})
\]
We see (exactly as in \prettyref{prop:Properties-of-B_W}) that there
is a full embedding of categories 
\[
D((\mathcal{D}_{X}^{(m)},\varphi^{-1}(\mathcal{D}_{Y}^{(m)}))-\text{bimod})\to D((\mathcal{D}_{W(X)}^{(m)}/p,W\varphi^{-1}(\mathcal{D}_{W(Y)}^{(m)})/p)-\text{bimod})
\]
whose image we call the accessible $(\mathcal{D}_{W(X)}^{(m)}/p,W\varphi^{-1}(\mathcal{D}_{W(Y)}^{(m)})/p)$-bimodules. 
\begin{defn}
A complex $\mathcal{P}^{\cdot}\in D_{cc}((\mathcal{\widehat{D}}_{W(X)}^{(m)},W\varphi^{-1}(\mathcal{\widehat{D}}_{W(Y)}^{(m)}))-\text{bimod})$
is said to be accessible if $\mathcal{P}^{\cdot}\otimes_{W(k)}^{L}k$
is accessible as a complex of $(\mathcal{D}_{W(X)}^{(m)}/p,W\varphi^{-1}(\mathcal{D}_{W(Y)}^{(m)})/p)$-bimodules. 
\end{defn}

Exactly as in \prettyref{thm:Local-Accessible}, we have that $\mathcal{P}^{\cdot}$
is accessible iff, locally, it can be written as 
\[
(\Phi_{1}^{*}\mathcal{\widehat{D}}_{\mathfrak{X}}^{(m)})\widehat{\otimes}_{\mathcal{\widehat{D}}_{\mathfrak{X}}^{(m)}}^{L}\mathcal{Q}^{\cdot}\widehat{\otimes}_{\varphi^{-1}(\mathcal{\widehat{D}}_{\mathfrak{Y}}^{(m)})}^{L}\varphi^{-1}(\Phi_{2}^{!}\mathcal{\widehat{D}}_{\mathfrak{Y}}^{(m)})
\]
where $\Phi_{1}$ and $\Phi_{2}$ are coordinatized lifts of Frobenius.
Furthermore, we have the natural functor 
\[
\mathcal{P}^{\cdot}\to\mathcal{P}_{\text{acc}}^{\cdot}:=\mathcal{\widehat{D}}_{W(X),\text{acc}}^{(m)}\widehat{\otimes}_{\mathcal{\widehat{D}}_{W(X)}^{(m)}}^{L}\mathcal{P}^{\cdot}\widehat{\otimes}_{W\varphi^{-1}(\mathcal{\widehat{D}}_{W(Y)}^{(m)})}^{L}W\varphi^{-1}(\mathcal{\widehat{D}}_{W(Y)}^{(m)})
\]
which takes any complex of bimodules to an accessible complex, and
which is the identity on accessible complexes.
\begin{defn}
\label{def:Accessible-Transfer}For any morphism $\varphi:X\to Y$,
we have the sheaf of $(\mathcal{\widehat{D}}_{W(X)}^{(m)},W\varphi^{-1}(\mathcal{\widehat{D}}_{W(Y)}^{(m)}))$-bimodules
$\widehat{\mathcal{D}}_{W(X)\to W(Y)}^{(m)}$ constructed in \prettyref{def:D-and-Transfer}
Then we define 
\[
\widehat{\mathcal{D}}_{W(X)\to W(Y),\text{acc}}^{(m)}:=\mathcal{\widehat{D}}_{W(X),\text{acc}}^{(m)}\widehat{\otimes}_{\mathcal{\widehat{D}}_{W(X)}^{(m)}}^{L}\widehat{\mathcal{D}}_{W(X)\to W(Y)}^{(m)}\widehat{\otimes}_{W\varphi^{-1}(\mathcal{\widehat{D}}_{W(Y)}^{(m)})}^{L}W\varphi^{-1}(\mathcal{\widehat{D}}_{W(Y),\text{acc}}^{(m)})
\]
a sheaf of accessible $(\mathcal{\widehat{D}}_{W(X)}^{(m)},W\varphi^{-1}(\mathcal{\widehat{D}}_{W(Y)}^{(m)}))$-bimodules
on $X$. 
\end{defn}

By the local description of $\widehat{\mathcal{D}}_{W(X),\text{acc}}^{(m)}$
(given right above \prettyref{def:D-acc}) we see that when $\varphi$
is the identity map we have $\widehat{\mathcal{D}}_{W(X)\to W(Y),\text{acc}}^{(m)}=\widehat{\mathcal{D}}_{W(X),\text{acc}}^{(m)}$

Then we have the 
\begin{defn}
\label{def:Pull!}Let $\mathcal{M}^{\cdot}\in D_{\text{acc}}(\mathcal{\widehat{D}}_{W(Y)}^{(m)}-\text{mod})$.
We define 
\[
L(W\varphi)^{*}(\mathcal{M}^{\cdot}):=\widehat{\mathcal{D}}_{W(X)\to W(Y),\text{acc}}^{(m)}\widehat{\otimes}_{W\varphi^{-1}(\mathcal{\widehat{D}}_{W(Y)}^{(m)})}^{L}W\varphi^{-1}(\mathcal{M}^{\cdot})\in D(\mathcal{\widehat{D}}_{W(X)}^{(m)}-\text{mod})
\]
Similarly we define, for any $\mathcal{M}^{\cdot}\in D(\mathcal{\widehat{D}}_{W(Y)}^{(m)}/p^{r}-\text{mod})$,
\[
L(W\varphi)^{*}(\mathcal{M}^{\cdot}):=\widehat{\mathcal{D}}_{W(X)\to W(Y),\text{acc}}^{(m)}/p^{r}\widehat{\otimes}_{W\varphi^{-1}(\mathcal{\widehat{D}}_{W(Y)}^{(m)}/p^{r})}^{L}W\varphi^{-1}(\mathcal{M}^{\cdot})\in D(\mathcal{\widehat{D}}_{W(X)}^{(m)}/p^{r}-\text{mod})
\]
\end{defn}

In order to get control over this definition, we need to compare with
the usual pullback functor in the presence of local lifts. Suppose
for a moment that $X=\text{Spec}(A)$ and $Y=\text{Spec}(B)$ are
affine, each possessing local coordinates, as well as coordinatized
Frobenius lifts; let $\Phi_{1}:\mathcal{O}_{\mathfrak{X}}\to\mathcal{O}_{W(X)}$
and $\Phi_{2}:\mathcal{O}_{\mathfrak{Y}}\to\mathcal{O}_{W(Y)}$ be
the associated morphism, with projections $\pi_{1}$ and $\pi_{2}$,
respectively. Suppose we have a lift $\varphi:\mathfrak{X}\to\mathfrak{Y}$. 
\begin{prop}
\label{prop:Pullbback-and-transfer}Let $\widehat{\mathcal{D}}_{\mathfrak{X}\to\mathfrak{Y}}^{(m),\Phi_{1},\Phi_{2}}$
denote the $p$-adic completion of the $(\widehat{\mathcal{D}}_{\mathfrak{X}}^{(m)},\varphi^{-1}(\widehat{\mathcal{D}}_{\mathfrak{Y}}^{(m)}))$
bimodule locally generated by $\pi_{1}\circ W\varphi^{\#}\circ\varphi^{-1}(\Phi_{2}):\varphi^{-1}(\mathcal{O}_{\mathfrak{Y}})\to\mathcal{O}_{\mathfrak{X}}$.
We have 
\[
\widehat{\mathcal{D}}_{W(X)\to W(Y),\text{acc}}^{(m)}\tilde{=}\Phi_{1}^{*}\mathcal{\widehat{D}}_{\mathfrak{X}}^{(m)}\widehat{\otimes}_{\mathcal{\widehat{D}}_{\mathfrak{X}}^{(m)}}^{L}\widehat{\mathcal{D}}_{\mathfrak{X}\to\mathfrak{Y}}^{(m),\Phi_{1},\Phi_{2}}\widehat{\otimes}_{\varphi^{-1}(\mathcal{\widehat{D}}_{\mathfrak{Y}}^{(m)})}^{L}\varphi^{-1}(\Phi_{2}^{!}\widehat{\mathcal{D}}_{\mathfrak{Y}}^{(m)})
\]
i.e., the derived completed tensor product is concentrated in degree
$0$(and is equal to the usual $p$-adic completion of the tensor
product). In particular, if $\varphi\circ\Phi_{1}=\Phi_{2}\circ\varphi$
then we obtain 
\[
\widehat{\mathcal{D}}_{W(X)\to W(Y),\text{acc}}^{(m)}\tilde{=}\Phi_{1}^{*}\mathcal{\widehat{D}}_{\mathfrak{X}}^{(m)}\widehat{\otimes}_{\mathcal{\widehat{D}}_{\mathfrak{X}}^{(m)}}^{L}\widehat{\mathcal{D}}_{\mathfrak{X}\to\mathfrak{Y}}^{(m)}\widehat{\otimes}_{\varphi^{-1}(\mathcal{\widehat{D}}_{\mathfrak{Y}}^{(m)})}^{L}\varphi^{-1}(\Phi_{2}^{!}\widehat{\mathcal{D}}_{\mathfrak{Y}}^{(m)})
\]
\end{prop}

\begin{proof}
The first statement is equivalent to
\[
\Phi_{1}^{!}\mathcal{\widehat{D}}_{\mathfrak{X}}^{(m)}\otimes_{\widehat{\mathcal{D}}_{W(X)}^{(m)}}\widehat{\mathcal{D}}_{W(X)\to W(Y)}^{(m)}\otimes_{W\varphi^{-1}(\mathcal{\widehat{D}}_{W(Y)}^{(m)})}W\varphi^{-1}(\Phi_{2}^{*}\widehat{\mathcal{D}}_{\mathfrak{Y}}^{(m)})\tilde{=}\widehat{\mathcal{D}}_{\mathfrak{X}\to\mathfrak{Y}}^{(m),\Phi_{1},\Phi_{2}}
\]
The composition of morphisms produces a map
\begin{equation}
\Phi_{1}^{!}\mathcal{\widehat{D}}_{\mathfrak{X}}^{(m)}\otimes_{\widehat{\mathcal{D}}_{W(X)}^{(m)}}\widehat{\mathcal{D}}_{W(X)\to W(Y)}^{(m)}\otimes_{W\varphi^{-1}(\mathcal{\widehat{D}}_{W(Y)}^{(m)})}W\varphi^{-1}(\Phi_{2}^{*}\widehat{\mathcal{D}}_{\mathfrak{Y}}^{(m)})\to\mathcal{H}om_{W(k)}(\varphi^{-1}(\mathcal{O}_{\mathfrak{Y}}),\mathcal{O}_{\mathfrak{X}})\label{eq:action-on-X-Y}
\end{equation}
As the left hand side is the summand of $\widehat{\mathcal{D}}_{W(X)\to W(Y)}^{(m)}$
consisting of maps whose image is contained $\pi_{1}(\mathcal{O}_{W(X)})$,
and which vanish on the complement of $\varphi^{-1}(\Phi_{2}(\mathcal{O}_{\mathfrak{Y}}))$
in $\mathcal{O}_{W(Y)}$, this map is clearly injective. We shall
show that the image is equal to $\widehat{\mathcal{D}}_{\mathfrak{X}\to\mathfrak{Y}}^{(m),\Phi_{1},\Phi_{2}}$. 

By \prettyref{cor:Basis-for-bimodule}, $\Phi_{1}^{!}\mathcal{\widehat{D}}_{\mathfrak{X}}^{(m)}\otimes_{\widehat{\mathcal{D}}_{W(X)}^{(m)}}\widehat{\mathcal{D}}_{W(X)\to W(Y)}^{(m)}\otimes_{W\varphi^{-1}(\mathcal{\widehat{D}}_{W(Y)}^{(m)})}W\varphi^{-1}(\Phi_{2}^{*}\widehat{\mathcal{D}}_{\mathfrak{Y}}^{(m)})$
is the sheaf whose sections over $\mathfrak{X}$ are of the form 
\begin{equation}
\sum_{I}(\sum_{r=0}^{\infty}\sum_{J_{I}}\pi_{1}F^{-r}(\alpha_{J_{I}})\cdot W\varphi^{\#}(\{\partial\}_{J_{I}/p^{r}})\{\partial\}^{I}\pi_{2}\label{eq:big-sum-bimodule}
\end{equation}
where the notation is as in \prettyref{cor:Basis-for-bimodule}. 

Now, write 
\[
\{\partial\}_{J_{I}/p^{r}}\{\partial\}^{I}\pi_{2}=\sum_{(K,s)}p^{s}T^{K/p^{s}}\Phi_{2}(b_{K})
\]
for $b_{K}\in\mathcal{\widehat{D}}_{\mathfrak{Y}}^{(m)}(\mathfrak{Y})$,
where $s<r$ implies $b_{K}\in p^{r-s}\mathcal{\widehat{D}}_{\mathfrak{Y}}^{(m)}(\mathfrak{Y})$. 

Then
\[
\pi_{1}F^{-r}(\alpha_{J_{I}})\cdot W\varphi^{\#}(\{\partial\}_{J_{I}/p^{r}})\{\partial\}^{I}\pi_{2}=\sum_{(K,s)}\pi_{1}\circ F^{-r}(\alpha_{J_{I}})\cdot W\varphi^{\#}\circ p^{s}T^{K/p^{s}}\cdot\Phi_{2}(b_{K})
\]
As $p^{r}T^{J/p^{r}}=V^{r}(T^{J})$ and $W\varphi^{\#}$ commutes
with $V$, we see that the above sum is equal to 
\[
\sum_{(K,s)}\pi_{1}F^{-r}(\alpha_{J_{I}})\cdot(V^{r}(W\varphi^{\#}(T^{J})))\circ(W\varphi^{\#}\circ\Phi_{2})(b_{K})
\]
Let $\{X_{1},\dots X_{n}$ be local coordinates on $\mathcal{A}$.
Then we can write 
\[
(V^{r}(W\varphi^{\#}(T^{J})))\circ(W\varphi^{\#}\circ\Phi_{2})(b_{J})=\sum_{(L,m)}p^{m}X^{L/m}\cdot(W\varphi^{\#}\circ\Phi_{2})(b_{L})
\]
and $m<r$ implies $b_{L}\in p^{r-m}\mathcal{\widehat{D}}_{\mathfrak{Y}}^{(m)}(\mathfrak{Y})$.
Now, we have 
\[
\sum_{(L,m)}\pi_{1}F^{-r}(\alpha_{J_{I}})\cdot p^{m}X^{L/m}\cdot(W\varphi^{\#}\circ\Phi_{2})(b_{L})=\sum_{(L,m)}c_{L}\pi_{1}(W\varphi^{\#}\circ\Phi_{2})(b_{L})
\]
where $c_{L}\in p^{M}\mathcal{A}$ with $M=\text{max}\{m,r\}$, we
see that this sum is $p$-adically convergent, and therefore
\[
\pi_{1}F^{-r}(\alpha_{J_{I}})\cdot W\varphi^{\#}(\{\partial\}_{J_{I}/p^{r}})\{\partial\}^{I}\pi_{2}=\sum_{(L,m),(J,r)}c_{L}\cdot\pi_{1}\circ(W\varphi^{\#}\circ\Phi_{2})(b_{J})
\]
which is contained in the $p$-adic completion of the $(\widehat{\mathcal{D}}_{\mathfrak{X}}^{(m)}(\mathfrak{X}),\widehat{\mathcal{D}}_{\mathfrak{Y}}^{(m)}(\mathfrak{Y}))$
bimodule generated by $\pi_{1}\circ W\varphi^{\#}\circ\Phi_{2}$;
from the above conditions on $a_{I},b_{J}$ we also see that $\pi_{1}F^{-r}(\alpha_{J_{I}})\cdot W\varphi^{\#}(\{\partial\}_{J_{I}/p^{r}})\{\partial\}^{I}\pi_{2}$
is contained in $p^{R}\cdot\widehat{\mathcal{D}}_{\mathfrak{X}\to\mathfrak{Y}}^{(m),\Phi_{1},\Phi_{2}}$
where $R=\text{max}\{r,v\}$ where $v$ is the least natural number
such that $\alpha_{J_{I}}\in V^{v}(W(A))$.

Therefore the image of the sum \ref{eq:big-sum-bimodule} is contained
in $\widehat{\mathcal{D}}_{\mathfrak{X}\to\mathfrak{Y}}^{(m),\Phi_{1},\Phi_{2}}$
as well. Thus the image of the map \ref{eq:action-on-X-Y} is contained
in $\widehat{\mathcal{D}}_{\mathfrak{X}\to\mathfrak{Y}}^{(m),\Phi_{1},\Phi_{2}}$.
It is surjective, as the image of \prettyref{eq:action-on-X-Y} is
already a $(\widehat{\mathcal{D}}_{\mathfrak{X}}^{(m)},\varphi^{-1}(\widehat{\mathcal{D}}_{\mathfrak{Y}}^{(m)}))$
bi-submodule containing $\pi_{1}\circ W\varphi^{\#}\circ\varphi^{-1}(\Phi_{2})$;
and so we have the desired result. 

Finally, if $\varphi\circ\Phi_{1}=\Phi_{2}\circ\varphi$ we have $\pi_{1}\circ W\varphi^{\#}\circ\varphi^{-1}(\Phi_{2})=\varphi^{\#}$
and so the last sentence of the proposition follows.
\end{proof}
Now, for some $r\geq1$ (including $r=\infty$) let us suppose we
are given a lift of $\varphi$ to $\varphi:\mathfrak{X}_{r}\to\mathfrak{Y}_{r}$.
We say that this map $\varphi$ is \emph{locally compatible with a
Frobenius lift }if, locally on $\mathfrak{X}_{r}$ and $\mathfrak{Y}_{r}$,
we can find lifts $\mathfrak{X}$ and $\mathfrak{Y}$ of $\mathfrak{X}_{r}$
and $\mathfrak{Y}_{r}$ and a lift $\varphi:\mathfrak{X}\to\mathfrak{Y}$,
which commutes with some coordinatized lifts of Frobenius $\Phi_{1}$
on $\mathfrak{X}$ and $\Phi_{2}$ on $\mathfrak{Y}$. Then
\begin{cor}
\label{cor:pullback-and-transfer-over-X_n}Let $\varphi:\mathfrak{X}_{r}\to\mathfrak{Y}_{r}$
be locally compatible with a Frobenius lift (when $m=0$ and $p=2$,
we suppose $r=1$). Let $\mathcal{N}^{\cdot}\in D(\text{Mod}(\mathcal{D}_{\mathfrak{Y}_{r}}))$.
Then there is an isomorphism 
\[
LW\varphi^{*}(\mathcal{B}_{\mathfrak{Y}_{r}}^{(m)}\otimes_{\mathcal{D}_{\mathfrak{Y}_{r}}^{(m)}}^{L}\mathcal{N}^{\cdot})\tilde{\to}\mathcal{B}_{\mathfrak{X}_{r}}\otimes_{\mathcal{D}_{\mathfrak{X}_{r}}^{(m)}}^{L}L\varphi^{*}(\mathcal{N}^{\cdot})
\]
where on the right hand side the pullback is the pullback in the category
of $\mathcal{D}_{\mathfrak{X}_{r}}^{(m)}$-modules. The analogous
statement holds when $r=\infty$, i.e., for $\varphi:\mathfrak{X}\to\mathfrak{Y}$. 
\end{cor}

\begin{proof}
We have the isomorphism 
\[
\widehat{\mathcal{D}}_{W(X)\to W(Y),\text{acc}}^{(m)}\tilde{=}\Phi_{1}^{*}\mathcal{\widehat{D}}_{\mathfrak{X}}^{(m)}\widehat{\otimes}_{\mathcal{\widehat{D}}_{\mathfrak{X}}^{(m)}}^{L}\widehat{\mathcal{D}}_{\mathfrak{X}\to\mathfrak{Y}}^{(m)}\widehat{\otimes}_{\varphi^{-1}(\mathcal{\widehat{D}}_{\mathfrak{Y}}^{(m)})}^{L}\varphi^{-1}(\Phi_{2}^{!}\widehat{\mathcal{D}}_{\mathfrak{Y}}^{(m)})
\]
for $\Phi\circ\varphi=\Phi\circ\varphi$ (c.f. \prettyref{prop:Pullbback-and-transfer});
taking reduction mod $p^{r}$ this proves the result locally on $\mathfrak{X}_{r}$
and $\mathfrak{Y}_{r}$. If $\Phi_{1}',\Phi'_{2}$ are different local
lifts, by \prettyref{cor:Iso-mod-p^n} the isomorphism $\Phi_{2}^{*}\widehat{\mathcal{D}}_{\mathfrak{Y}}^{(m)}\tilde{\to}(\Phi_{2}')^{*}\widehat{\mathcal{D}}_{\mathfrak{Y}}^{(m)}$
is compatible with reduction mod $p^{r}$ (and similarly for $\Phi_{1}$),
and the result follows; when $p=2$ and $n=1$ we appeal to \prettyref{cor:Global-bimodule-mod-p}
for the required gluing statement. 
\end{proof}
To make better use of this statement, we note the:
\begin{lem}
\label{lem:pull-and-reduce}For $\mathcal{M}^{\cdot}\in D(\mathcal{\widehat{D}}_{W(Y)}^{(m)}-\text{mod})$
we have 
\[
L(W\varphi)^{*}(\mathcal{M}^{\cdot})\otimes_{W(k)}^{L}k\tilde{\to}L(W\varphi)^{*}(\mathcal{M}^{\cdot}\otimes_{W(k)}^{L}k)
\]
where the functor on the right denotes the pullback in the category
of $\mathcal{\widehat{D}}_{W(Y)}^{(m)}/p$ modules. Similarly, if
$\mathcal{M}^{\cdot}\in D(\mathcal{\widehat{D}}_{W(Y)}^{(m)}/p^{n}-\text{mod})$
for some $n\geq1$, then 
\[
L(W\varphi)^{*}(\mathcal{M}^{\cdot})\otimes_{W_{n}(k)}^{L}k\tilde{\to}L(W\varphi)^{*}(\mathcal{M}^{\cdot}\otimes_{W_{n}(k)}^{L}k)
\]
 
\end{lem}

\begin{proof}
As $\varphi^{-1}(\mathcal{\widehat{D}}_{W(Y)}^{(m)})$ and $\widehat{\mathcal{D}}_{W(X)\to W(Y),\text{acc}}^{(m)}$
are flat over $W(k)$ we have 
\[
(\widehat{\mathcal{D}}_{W(X)\to W(Y),\text{acc}}^{(m)}\widehat{\otimes}_{\varphi^{-1}(\mathcal{\widehat{D}}_{W(Y)}^{(m)})}^{L}\varphi^{-1}(\mathcal{M}^{\cdot}))\otimes_{W(k)}^{L}k
\]
\[
\tilde{\to}(\widehat{\mathcal{D}}_{W(X)\to W(Y),\text{acc}}^{(m)}/p)\widehat{\otimes}_{\varphi^{-1}(\mathcal{\widehat{D}}_{W(Y)}^{(m)})/p}^{L}(\varphi^{-1}(\mathcal{M}^{\cdot})\widehat{\otimes}_{W(k)}^{L}k)
\]
which implies the first result. The statement about $\mathcal{M}^{\cdot}\in D(\mathcal{\widehat{D}}_{W(Y)}^{(m)}/p^{n}-\text{mod})$
is similar. 
\end{proof}
Combining these two results yields immediately
\begin{prop}
\label{prop:Pullback-preserves-acc}Let $\varphi:X\to Y$ be locally
compatible with a lift of Frobenius. Then $L(W\varphi)^{*}$ takes
$D_{\text{acc}}(\mathcal{\widehat{D}}_{W(Y)}^{(m)}-\text{mod})$ to
$D_{\text{acc}}(\mathcal{\widehat{D}}_{W(X)}^{(m)}-\text{mod})$.
The same holds mod $p^{r}$, i.e., for $D_{\text{acc}}(\mathcal{\widehat{D}}_{W(Y)}^{(m)}/p^{r}-\text{mod})$. 
\end{prop}

To generalize this to all morphisms, we need some information about
compositions: 
\begin{lem}
\label{lem:Composition-of-bimod}Let $\varphi:X\to Y$ and $\psi:Y\to Z$.
There is a natural map 
\[
\widehat{\mathcal{D}}_{W(X)\to W(Y),\text{acc}}^{(m)}\otimes_{W\varphi^{-1}(\widehat{\mathcal{D}}_{W(Y)}^{(m)})}W\varphi^{-1}(\widehat{\mathcal{D}}_{W(Y)\to W(Z),\text{acc}}^{(m)})\to\widehat{\mathcal{D}}_{W(X)\to W(Z),\text{acc}}^{(m)}
\]
of $(\widehat{\mathcal{D}}_{W(X)}^{(m)},W(\psi\circ\varphi)^{-1}(\widehat{\mathcal{D}}_{W(X)}^{(m)}))$-bimodules. 
\end{lem}

\begin{proof}
The composition of morphisms yields a map 
\[
\widehat{\mathcal{D}}_{W(X)\to W(Y)}^{(m)}\otimes_{\varphi^{-1}(\widehat{\mathcal{D}}_{W(Y)}^{(m)})}\varphi^{-1}(\widehat{\mathcal{D}}_{W(Y)\to W(Z)}^{(m)})\to\mathcal{H}om_{W(k)}((\varphi\circ\psi)^{-1}\mathcal{O}_{W(Z)},\mathcal{O}_{W(X)})
\]
and one sees directly- working locally and use the commutation relations
of \prettyref{lem:big-products}-that the image of this map is contained
in $\widehat{\mathcal{D}}_{W(X)\to W(Z)}^{(m)}$. Now, by definition
we have 
\[
\widehat{\mathcal{D}}_{W(X)\to W(Y),\text{acc}}^{(m)}=\mathcal{\widehat{D}}_{W(X),\text{acc}}^{(m)}\widehat{\otimes}_{\mathcal{\widehat{D}}_{W(X)}^{(m)}}^{L}\widehat{\mathcal{D}}_{W(X)\to W(Y)}^{(m)}\widehat{\otimes}_{W\varphi^{-1}(\mathcal{\widehat{D}}_{W(Y)}^{(m)})}^{L}W\varphi^{-1}(\mathcal{\widehat{D}}_{W(Y),\text{acc}}^{(m)})
\]
and so the natural maps $\mathcal{\widehat{D}}_{W(X),\text{acc}}^{(m)}\to\mathcal{\widehat{D}}_{W(X)}^{(m)}$
(and $\mathcal{\widehat{D}}_{W(Y),\text{acc}}^{(m)}\to\mathcal{\widehat{D}}_{W(Y)}^{(m)}$)
yield a natural map 
\[
\widehat{\mathcal{D}}_{W(X)\to W(Y),\text{acc}}^{(m)}\to\widehat{\mathcal{D}}_{W(X)\to W(Y)}^{(m)}
\]
of $(\widehat{\mathcal{D}}_{W(X)}^{(m)},W\varphi{}^{-1}(\widehat{\mathcal{D}}_{W(Y)}^{(m)}))$-bimodules;
and the analogous fact holds for $\widehat{\mathcal{D}}_{W(Y)\to W(Z),\text{acc}}^{(m)}$.
So, we have 
\[
\widehat{\mathcal{D}}_{W(X)\to W(Y),\text{acc}}^{(m)}\otimes_{W\varphi^{-1}(\widehat{\mathcal{D}}_{W(Y)}^{(m)})}W\varphi^{-1}(\widehat{\mathcal{D}}_{W(Y)\to W(Z),\text{acc}}^{(m)})
\]
\[
\to\widehat{\mathcal{D}}_{W(X)\to W(Y)}^{(m)}\otimes_{W\varphi^{-1}(\widehat{\mathcal{D}}_{W(Y)}^{(m)})}W\varphi^{-1}(\widehat{\mathcal{D}}_{W(Y)\to W(Z)}^{(m)})\to\widehat{\mathcal{D}}_{W(X)\to W(Z)}^{(m)}
\]
And, since $\widehat{\mathcal{D}}_{W(X)\to W(Y),\text{acc}}^{(m)}\otimes_{W\varphi^{-1}(\widehat{\mathcal{D}}_{W(Y)}^{(m)})}W\varphi^{-1}(\widehat{\mathcal{D}}_{W(Y)\to W(Z),\text{acc}}^{(m)})$
is an accessible $(\widehat{\mathcal{D}}_{W(X)}^{(m)},W(\psi\circ\varphi)^{-1}(\widehat{\mathcal{D}}_{W(X)}^{(m)}))$
bimodule, by applying the functor $\mathcal{P}^{\cdot}\to\mathcal{P}_{\text{acc}}^{\cdot}$
for $(\widehat{\mathcal{D}}_{W(X)}^{(m)},W(\psi\circ\varphi)^{-1}(\widehat{\mathcal{D}}_{W(X)}^{(m)}))$
bimodules, we obtain a natural map 
\[
\widehat{\mathcal{D}}_{W(X)\to W(Y),\text{acc}}^{(m)}\otimes_{W\varphi^{-1}(\widehat{\mathcal{D}}_{W(Y)}^{(m)})}W\varphi^{-1}(\widehat{\mathcal{D}}_{W(Y)\to W(Z),\text{acc}}^{(m)})\to\widehat{\mathcal{D}}_{W(X)\to W(Z),\text{acc}}^{(m)}
\]
as required. 
\end{proof}
Now, combining this with the previous result gives: 
\begin{cor}
\label{cor:Composition-of-pullbacks} 1) With notation as in \prettyref{prop:Pullbback-and-transfer},
we have an isomorphism of bimodules $\widehat{\mathcal{D}}_{\mathfrak{X}\to\mathfrak{Y}}^{(m),\Phi_{1},\Phi_{2}}\otimes_{W(k)}^{L}k\tilde{=}\widehat{\mathcal{D}}_{X\to Y}^{(m)}$. 

2) Suppose $\psi:Y\to Z$ is another map of smooth affine varieties
with local coordinates. Then the map of the previous lemma induces
an isomorphism 
\[
\widehat{\mathcal{D}}_{W(X)\to W(Y),\text{acc}}^{(m)}\otimes_{W\varphi^{-1}(\widehat{\mathcal{D}}_{W(Y)}^{(m)})}^{L}W\varphi^{-1}(\widehat{\mathcal{D}}_{W(Y)\to W(Z),\text{acc}}^{(m)})\tilde{\to}\widehat{\mathcal{D}}_{W(X)\to W(Z),\text{acc}}^{(m)}
\]
of $(\widehat{\mathcal{D}}_{W(X)}^{(m)},W(\psi\circ\varphi)^{-1}(\widehat{\mathcal{D}}_{W(X)}^{(m)}))$-bimodules. 

3) There is an isomorphism of functors $L(W\varphi)^{*}\circ L(W\psi)^{*}\tilde{\to}LW(\psi\circ\varphi)^{*}$
from $D_{\text{acc}}(\mathcal{\widehat{D}}_{W(Z)}^{(m)}-\text{mod})$
to $D_{\text{acc}}(\mathcal{\widehat{D}}_{W(X)}^{(m)}-\text{mod})$. 
\end{cor}

\begin{proof}
There is a map $\mathcal{O}_{\mathfrak{X}}\otimes\varphi^{-1}(\widehat{\mathcal{D}}_{\mathfrak{Y}}^{(m)})\to\widehat{\mathcal{D}}_{\mathfrak{X}\to\mathfrak{Y}}^{(m),\Phi_{1},\Phi_{2}}$
given by sending a tensor $a\otimes\varphi^{-1}(P)$ to $a\cdot\pi_{1}\circ W\varphi^{\#}\circ\varphi^{-1}(\Phi_{2}\cdot P)$,
passing to the $p$-adic completion yields a map 
\[
\varphi^{*}\widehat{\mathcal{D}}_{\mathfrak{Y}}^{(m)}\to\widehat{\mathcal{D}}_{\mathfrak{X}\to\mathfrak{Y}}^{(m),\Phi_{1},\Phi_{2}}
\]
and we shall show that this map is an isomorphism. The surjectivity
follows from (the proof of) \prettyref{prop:Pullbback-and-transfer},
where we in fact showed that every term of the form \prettyref{eq:big-sum-bimodule}
is a sum of terms of the form $a\cdot\pi_{1}\circ W\varphi^{\#}\circ\varphi^{-1}(\Phi_{2}\cdot P)$.
As for the injectivity, we may apply the functor of reduction mod
$p$ to obtain a map 
\[
\varphi^{*}\mathcal{D}_{Y}^{(m)}\to\widehat{\mathcal{D}}_{\mathfrak{X}\to\mathfrak{Y}}^{(m),\Phi_{1},\Phi_{2}}/p\to\widehat{\mathcal{D}}_{W(X)\to W(Y),\text{acc}}^{(m)}/p
\]
the latter map is injective as $\widehat{\mathcal{D}}_{\mathfrak{X}\to\mathfrak{Y}}^{(m),\Phi_{1},\Phi_{2}}$
is a summand of $\widehat{\mathcal{D}}_{W(X)\to W(Y),\text{acc}}^{(m)}$.
On the other hand we have maps 
\[
\widehat{\mathcal{D}}_{W(X)\to W(Y),\text{acc}}^{(m)}/p\to\widehat{\mathcal{D}}_{W(X)\to W(Y)}^{(m)}/p\to\varphi^{*}\mathcal{D}_{Y}^{(m)}
\]
where the first map is the canonical one and the second is given by
the quotient by $I^{(1)}$ (c.f. \prettyref{prop:construction-of-transfer}).
It is not hard to see that the composition of these maps is the identity
on $\varphi^{*}\mathcal{D}_{Y}^{(m)}$; therefore $\varphi^{*}\mathcal{D}_{Y}^{(m)}\to\widehat{\mathcal{D}}_{\mathfrak{X}\to\mathfrak{Y}}^{(m),\Phi_{1},\Phi_{2}}/p$
is injective. As $\varphi^{*}\widehat{\mathcal{D}}_{\mathfrak{Y}}^{(m)}\to\widehat{\mathcal{D}}_{\mathfrak{X}\to\mathfrak{Y}}^{(m),\Phi_{1},\Phi_{2}}$
is a surjection of torsion-free $p$-adically complete sheaves, by
applying $\otimes_{W(k)}^{L}k$ we see that this map is an isomorphism.
Furthermore, the composition 
\[
\widehat{\mathcal{D}}_{\mathfrak{X}\to\mathfrak{Y}}^{(m),\Phi_{1},\Phi_{2}}/p\to\widehat{\mathcal{D}}_{W(X)\to W(Y),\text{acc}}^{(m)}/p\to\varphi^{*}\mathcal{D}_{Y}^{(m)}
\]
is easily seen to be a map of $(\mathcal{D}_{X}^{(m)},\varphi^{-1}(\mathcal{D}_{Y}^{(m)}))$
where $\mathcal{D}_{X}^{(m)}$ acts via the identification $\varphi^{*}\mathcal{D}_{Y}^{(m)}\tilde{=}\mathcal{D}_{X\to Y}^{(m)}$
(this follows readily in local coordinates). Thus $1)$ is proved.

As for $2)$, the map in question is obtained from the previous lemma
by noting that 
\[
\mathcal{H}^{0}(\widehat{\mathcal{D}}_{W(X)\to W(Y),\text{acc}}^{(m)}\otimes_{W\varphi^{-1}(\widehat{\mathcal{D}}_{W(Y)}^{(m)})}^{L}W\varphi^{-1}(\widehat{\mathcal{D}}_{W(Y)\to W(Z),\text{acc}}^{(m)}))
\]
\[
=\widehat{\mathcal{D}}_{W(X)\to W(Y),\text{acc}}^{(m)}\otimes_{W\varphi^{-1}(\widehat{\mathcal{D}}_{W(Y)}^{(m)})}W\varphi^{-1}(\widehat{\mathcal{D}}_{W(Y)\to W(Z),\text{acc}}^{(m)})
\]
Using the identifications 
\[
\widehat{\mathcal{D}}_{W(X)\to W(Y),\text{acc}}^{(m)}\tilde{=}\Phi_{1}^{*}\mathcal{\widehat{D}}_{\mathfrak{X}}^{(m)}\widehat{\otimes}_{\mathcal{\widehat{D}}_{\mathfrak{X}}^{(m)}}^{L}\widehat{\mathcal{D}}_{\mathfrak{X}\to\mathfrak{Y}}^{(m),\Phi_{1},\Phi_{2}}\widehat{\otimes}_{\varphi^{-1}(\mathcal{\widehat{D}}_{\mathfrak{Y}}^{(m)})}^{L}\varphi^{-1}(\Phi_{2}^{!}\widehat{\mathcal{D}}_{\mathfrak{Y}}^{(m)})
\]
and 
\[
W\varphi^{-1}(\widehat{\mathcal{D}}_{W(Y)\to W(Z),\text{acc}}^{(m)})\tilde{=}W\varphi^{-1}(\Phi_{2}^{*}\mathcal{\widehat{D}}_{\mathfrak{Y}}^{(m)}\widehat{\otimes}_{\mathcal{\widehat{D}}_{\mathfrak{Y}}^{(m)}}^{L}\widehat{\mathcal{D}}_{\mathfrak{Y}\to\mathfrak{Z}}^{(m),\Phi_{2},\Phi_{3}}\widehat{\otimes}_{\psi^{-1}(\mathcal{\widehat{D}}_{\mathfrak{Z}}^{(m)})}^{L}\psi^{-1}(\Phi_{3}^{!}\widehat{\mathcal{D}}_{\mathfrak{Z}}^{(m)}))
\]
yields an isomorphism 
\[
\widehat{\mathcal{D}}_{W(X)\to W(Y),\text{acc}}^{(m)}\otimes_{W\varphi^{-1}(\widehat{\mathcal{D}}_{W(Y)}^{(m)})}^{L}W\varphi^{-1}(\widehat{\mathcal{D}}_{W(Y)\to W(Z),\text{acc}}^{(m)})
\]
\[
\tilde{\to}\Phi_{1}^{*}\mathcal{\widehat{D}}_{\mathfrak{X}}^{(m)}\widehat{\otimes}_{\mathcal{\widehat{D}}_{\mathfrak{X}}^{(m)}}^{L}\widehat{\mathcal{D}}_{\mathfrak{X}\to\mathfrak{Y}}^{(m),\Phi_{1},\Phi_{2}}\widehat{\otimes}_{\varphi^{-1}(\mathcal{\widehat{D}}_{\mathfrak{Y}}^{(m)})}^{L}\varphi^{-1}(\widehat{\mathcal{D}}_{\mathfrak{Y}\to\mathfrak{Z}}^{(m),\Phi_{2},\Phi_{3}}\widehat{\otimes}_{\psi^{-1}(\mathcal{\widehat{D}}_{\mathfrak{Z}}^{(m)})}^{L}\psi^{-1}(\Phi_{3}^{!}\widehat{\mathcal{D}}_{\mathfrak{Z}}^{(m)}))
\]
\[
\tilde{\to}\Phi_{1}^{*}\mathcal{\widehat{D}}_{\mathfrak{X}}^{(m)}\widehat{\otimes}_{\mathcal{\widehat{D}}_{\mathfrak{X}}^{(m)}}^{L}(\widehat{\mathcal{D}}_{\mathfrak{X}\to\mathfrak{Y}}^{(m),\Phi_{1},\Phi_{2}}\widehat{\otimes}_{\varphi^{-1}(\mathcal{\widehat{D}}_{\mathfrak{Y}}^{(m)})}^{L}\varphi^{-1}(\widehat{\mathcal{D}}_{\mathfrak{Y}\to\mathfrak{Z}}^{(m),\Phi_{2},\Phi_{3}}))\widehat{\otimes}_{(\varphi\circ\psi)^{-1}(\mathcal{\widehat{D}}_{\mathfrak{Z}}^{(m)})}^{L}(\varphi\circ\psi)^{-1}(\Phi_{3}^{!}\widehat{\mathcal{D}}_{\mathfrak{Z}}^{(m)}))
\]
So the map of the previous lemma yields a map 
\[
\widehat{\mathcal{D}}_{\mathfrak{X}\to\mathfrak{Y}}^{(m),\Phi_{1},\Phi_{2}}\widehat{\otimes}_{\varphi^{-1}(\mathcal{\widehat{D}}_{\mathfrak{Y}}^{(m)})}^{L}\varphi^{-1}(\widehat{\mathcal{D}}_{\mathfrak{Y}\to\mathfrak{Z}}^{(m),\Phi_{2},\Phi_{3}})\to\widehat{\mathcal{D}}_{\mathfrak{X}\to\mathfrak{Z}}^{(m),\Phi_{1},\Phi_{3}}
\]
which is seen to be an isomorphism by applying $\otimes_{W(k)}^{L}k$
and part $1)$; so part $2)$ follows as well. 

3) Let $\mathcal{M}^{\cdot}\in D_{\text{acc}}(\mathcal{\widehat{D}}_{W(Y)}^{(m)}-\text{mod})$.
Then by part $2)$ we have 
\[
LW(\psi\circ\varphi)^{*}\mathcal{M}^{\cdot}=\widehat{\mathcal{D}}_{W(X)\to W(Z),\text{acc}}^{(m)}\otimes_{W(\psi\circ\varphi)^{-1}\mathcal{\widehat{D}}_{W(Z)}^{(m)}}^{L}W(\psi\circ\varphi)^{-1}\mathcal{M}^{\cdot}
\]
\[
\tilde{\to}\widehat{\mathcal{D}}_{W(X)\to W(Y),\text{acc}}^{(m)}\otimes_{W\varphi^{-1}(\widehat{\mathcal{D}}_{W(Y)}^{(m)})}^{L}W\varphi^{-1}(\widehat{\mathcal{D}}_{W(Y)\to W(Z),\text{acc}}^{(m)})\otimes_{W\varphi^{-1}\circ W\psi^{-1}(\mathcal{\widehat{D}}_{W(Z)}^{(m)})}^{L}W\varphi^{-1}\circ W\psi{}^{-1}(\mathcal{M}^{\cdot})
\]
\[
\tilde{\to}\widehat{\mathcal{D}}_{W(X)\to W(Y),\text{acc}}^{(m)}\otimes_{W\varphi^{-1}(\widehat{\mathcal{D}}_{W(Y)}^{(m)})}^{L}W\varphi^{-1}(\widehat{\mathcal{D}}_{W(Y)\to W(Z),\text{acc}}^{(m)}\otimes_{W\psi^{-1}(\mathcal{\widehat{D}}_{W(Z)}^{(m)})}^{L}W\psi{}^{-1}(\mathcal{M}^{\cdot}))
\]
\[
\tilde{\to}LW\varphi^{*}(LW\psi^{*}\mathcal{M}^{\cdot})
\]
as required. 
\end{proof}
Thus we may conclude
\begin{cor}
\label{cor:Pull-and-compose} 1) For any $\varphi$, if $\mathcal{M}^{\cdot}\in D_{\text{acc}}(\mathcal{\widehat{D}}_{W(Y)}^{(m)}-\text{mod})$,
we have $L\varphi^{*}(\mathcal{M}^{\cdot})\in D_{\text{acc}}(\mathcal{\widehat{D}}_{W(X)}^{(m)}-\text{mod})$.
The same holds for $\mathcal{M}^{\cdot}\in D_{\text{acc}}(\mathcal{\widehat{D}}_{W(Y)}^{(m)}/p^{r}-\text{mod})$. 

2) The result of \prettyref{cor:pullback-and-transfer-over-X_n} holds
for an arbitrary morphism $\varphi$. 
\end{cor}

\begin{proof}
1) Factor $\varphi:X\to Y$ as $\varphi_{1}:X\to Z$ followed by $\varphi_{2}:Z\to Y$,
where $\varphi$ is a closed immersion and $\varphi_{2}$ is smooth.
By the previous result, it suffices to check that each of $L(W\varphi_{1})^{*}$
and $L(W\varphi_{2})^{*}$ preserve accessibility; as they are each
locally compatible with a lift of Frobenius this follows from \prettyref{prop:Pullback-preserves-acc}. 

2) This is similar to $1)$- we break up $\varphi=\varphi_{2}\circ\varphi_{1}$
and use \prettyref{cor:pullback-and-transfer-over-X_n} . 
\end{proof}
To finish off this section, let's note the compatibility with the
pullback for crystals. Namely,
\begin{prop}
\label{prop:pullback-and-crystals}Let $r\geq1$, and let $\epsilon$
be the functor of \prettyref{thm:Embedding-of-crystals}. Then for
any smooth morphism $\varphi:X\to Y$ smooth, and any $\mathcal{M}^{\cdot}\in D_{\text{qcoh}}(\text{Crys}_{W_{r}(k)}(X))$,
we have 
\[
LW\varphi^{*}\circ\epsilon\tilde{\to}\epsilon\circ L\varphi_{\text{crys}}^{*}
\]
where $L\varphi_{\text{crys}}^{*}$ is the pullback in the derived
category of crystals. 
\end{prop}

\begin{proof}
As $\varphi$ is smooth, $\varphi^{*}$ is exact, and so it suffices
to prove this when $\mathcal{M}$ is concentrated in a single degree.
In that case, we have that, for a local lift $\varphi:\mathfrak{X}_{r}\to\mathfrak{Y}_{r}$
is compatible with a Frobenius lift; and so $LW\varphi^{*}$ corresponds
to the usual pullback under \prettyref{cor:pullback-and-transfer-over-X_n}.
This shows immediately that the pullback preserves the local nilpotence
condition; furthermore, this is exactly the definition of the pullback
in the category of crystals, and the proposition follows. 
\end{proof}

\subsection{\label{subsec:Operations-on-modules:Push!}Operations on modules:
Pushforward}

As discussed in the introduction, the key to defining a pushforward
in (the usual) $\mathcal{D}$-module theory is the definition of a
transfer bimodule $\mathcal{D}_{X\leftarrow Y}$. This is constructed
as follows: we already have a $(\mathcal{D}_{X},\varphi^{-1}(\mathcal{D}_{Y}))$
bimodule $\mathcal{D}_{X\to Y}$; we can apply the left-right interchange
simultaneously to left $\mathcal{D}_{X}$-modules and to right $\varphi^{-1}(\mathcal{D}_{Y})$-modules
to obtain 
\[
\mathcal{D}_{Y\leftarrow X}:=\omega_{X}\otimes_{\mathcal{O}_{X}}\mathcal{D}_{X\to Y}\otimes_{\varphi^{-1}(\mathcal{O}_{Y})}\varphi^{-1}(\omega_{Y}^{-1})
\]
which is now a $(\varphi^{-1}(\mathcal{D}_{Y}),\mathcal{D}_{X})$-bimodule.
In particular, if $Y$ is a point then $\mathcal{D}_{Y\leftarrow X}=\omega_{X}$
with its canonical right $\mathcal{D}$-module structure. 

If we now have $\varphi:\mathfrak{X}\to\mathfrak{Y}$ a morphism of
smooth formal schemes, and we fix some $m\geq0$, an identical procedure
defines $\widehat{\mathcal{D}}_{\mathfrak{Y}\leftarrow\mathfrak{X}}^{(m)}$,
and then the pushforward is defined\footnote{This is not quite Berthelot's definition; as he works with sheaves
over $W_{r}(k)$ and then takes an inverse limit; whereas we bypass
this by using the cohomological completion. The two notions agree
on a large subclass of objects, for instance on coherent $\mathcal{D}$-modules} as 
\[
\int_{\varphi}\mathcal{M}^{\cdot}:=R\varphi_{*}(\widehat{\mathcal{D}}_{\mathfrak{Y}\leftarrow\mathfrak{X}}^{(m)}\widehat{\otimes}_{\widehat{\mathcal{D}}_{\mathfrak{X}}^{(m)}}\mathcal{M}^{\cdot})
\]

We can utilize the same procedure to define $\widehat{\mathcal{D}}_{W(Y)\leftarrow W(X),\text{acc}}^{(m)}$. 
\begin{defn}
\label{def:Right-transfer}Let $\varphi:X\to Y$. The transfer bimodule
$\widehat{\mathcal{D}}_{W(Y)\leftarrow W(X),\text{acc}}^{(m)}$ is
the sheaf of $(W\varphi^{-1}(\mathcal{\widehat{D}}_{W(Y)}^{(m)}),\mathcal{\widehat{D}}_{W(X)}^{(m)})$
bimodules defined as:
\[
(W\omega_{X}\widehat{\otimes}_{\mathcal{O}_{W(X)}}\mathcal{\widehat{D}}_{W(X),\text{acc}}^{(m)})\otimes_{\mathcal{\widehat{D}}_{W(X)}^{(m)}}\widehat{\mathcal{D}}_{W(X)\to W(Y),\text{acc}}^{(m)}\otimes_{W\varphi^{-1}(\mathcal{\widehat{D}}_{W(Y)}^{(m)})}W\varphi^{-1}(\tilde{\mathcal{H}om}_{\mathcal{O}_{W(Y)}}(W\omega_{Y},\mathcal{\widehat{D}}_{W(Y),\text{acc}}^{(m)}))
\]
\end{defn}

With this in hand, one makes the 
\begin{defn}
Let $\mathcal{M}^{\cdot}\in D_{\text{acc}}(\mathcal{\widehat{D}}_{W(X)}^{(m)}/p^{r}-\text{mod})$
(for $r\geq1$) Then for any morphism $\varphi:X\to Y$ we define
\[
\int_{W\varphi}\mathcal{M}^{\cdot}:=R(W\varphi)_{*}((\widehat{\mathcal{D}}_{W(Y)\leftarrow W(X)}^{(m)}/p^{r})\otimes_{\mathcal{\widehat{D}}_{W(X)}^{(m)}/p^{r}}^{L}\mathcal{M}^{\cdot})_{\text{acc}}
\]
Next, let $\mathcal{M}^{\cdot}\in D_{\text{acc}}(\mathcal{\widehat{D}}_{W(X)}^{(m)}-\text{mod})$.
Then we define 
\[
\int_{W\varphi}\mathcal{M}^{\cdot}:=R(W\varphi)_{*}(\widehat{\mathcal{D}}_{W(Y)\leftarrow W(X)}^{(m)}\widehat{\otimes}_{\mathcal{\widehat{D}}_{W(X)}^{(m)}}^{L}\mathcal{M}^{\cdot})_{\text{acc}}
\]
where as above $\widehat{\otimes}$ denotes the derived completion. 
\end{defn}

In fact these two definitions agree on their overlap:
\begin{rem}
\label{rem:Pushforwards-agree}Let $\mathcal{M}^{\cdot}\in D_{\text{acc}}(\mathcal{\widehat{D}}_{W(X)}^{(m)}/p^{r}-\text{mod})$,
and regard $\mathcal{M}^{\cdot}$ as an element of $D_{\text{acc}}(\mathcal{\widehat{D}}_{W(X)}^{(m)}-\text{mod})$
via the obvious inclusion. Then we have an isomorphism 
\[
(\widehat{\mathcal{D}}_{W(Y)\leftarrow W(X)}^{(m)}/p^{r})\otimes_{\mathcal{\widehat{D}}_{W(X)}^{(m)}/p^{r}}^{L}\mathcal{M}^{\cdot}\tilde{\to}\widehat{\mathcal{D}}_{W(Y)\leftarrow W(X)}^{(m)}\widehat{\otimes}_{\mathcal{\widehat{D}}_{W(X)}^{(m)}}^{L}\mathcal{M}^{\cdot}
\]
of complexes of $\varphi^{-1}(\mathcal{\widehat{D}}_{W(Y)}^{(m)})$-modules.
To see this, note that $\widehat{\mathcal{D}}_{W(Y)\leftarrow W(X)}^{(m)}$
is itself $p$-torsion free. So, to show the above, we let $\mathcal{F}^{\cdot}$
be a flat resolution of $\widehat{\mathcal{D}}_{W(Y)\leftarrow W(X)}^{(m)}$
in the category of right-$\mathcal{\widehat{D}}_{W(X)}^{(m)}$-modules.
Then, as $p^{r}$ annihilates each term of $\mathcal{M}^{\cdot}$,
we have 
\[
\widehat{\mathcal{D}}_{W(Y)\leftarrow W(X)}^{(m)}\widehat{\otimes}_{\mathcal{\widehat{D}}_{W(X)}^{(m)}}^{L}\mathcal{M}^{\cdot}=\widehat{\mathcal{D}}_{W(Y)\leftarrow W(X)}^{(m)}\otimes_{\mathcal{\widehat{D}}_{W(X)}^{(m)}}^{L}\mathcal{M}^{\cdot}
\]
and this is computed by 
\[
\mathcal{F}^{\cdot}\otimes_{\mathcal{\widehat{D}}_{W(X)}^{(m)}}\mathcal{M}^{\cdot}\tilde{\to}\mathcal{F}^{\cdot}/p^{r}\otimes_{\mathcal{\widehat{D}}_{W(X)}^{(m)}/p^{r}}\mathcal{M}^{\cdot}
\]
and since $\widehat{\mathcal{D}}_{W(X)\leftarrow W(Y)}^{(m)}$ is
flat over $W(k)$, we have that $\mathcal{F}^{\cdot}/p^{r}$ is a
flat resolution of $\widehat{\mathcal{D}}_{W(X)\leftarrow W(Y)}^{(m)}/p^{r}$
over $\mathcal{\widehat{D}}_{W(X)}^{(m)}/p^{r}$, whence the result. 
\end{rem}

As for pullback, the main tool to study this morphism comes from giving
an explicit expression for it locally. Indeed, combining \prettyref{def:Right-transfer}
with \prettyref{prop:Pullbback-and-transfer} we obtain 
\begin{prop}
\label{prop:local-bimodule-for-push}Let $\widehat{\mathcal{D}}_{\mathfrak{Y}\leftarrow\mathfrak{X}}^{(m),\Phi_{1},\Phi_{2}}$
denote the $(\varphi^{-1}(\mathcal{\widehat{D}}_{\mathfrak{Y}}^{(m)}),\mathcal{\widehat{D}}_{\mathfrak{X}}^{(m)})$-bimodule
obtained from $\widehat{\mathcal{D}}_{\mathfrak{X}\to\mathfrak{Y}}^{(m),\Phi_{1},\Phi_{2}}$
(as defined in \prettyref{prop:Pullbback-and-transfer}) via simultaneously
applying the left right interchange of $\mathcal{\widehat{D}}_{\mathfrak{X}}^{(m)}$-modules
and $\varphi^{-1}(\mathcal{\widehat{D}}_{\mathfrak{Y}}^{(m)})$-modules.
We have 
\[
\widehat{\mathcal{D}}_{W(Y)\leftarrow W(X),\text{acc}}^{(m)}\tilde{=}\varphi^{-1}(\Phi_{2}^{*}\widehat{\mathcal{D}}_{\mathfrak{Y}}^{(m)})\widehat{\otimes}_{\varphi^{-1}(\mathcal{\widehat{D}}_{\mathfrak{Y}}^{(m)})}^{L}\widehat{\mathcal{D}}_{\mathfrak{Y}\leftarrow\mathfrak{X}}^{(m),\Phi_{1},\Phi_{2}}\widehat{\otimes}_{\mathcal{\widehat{D}}_{\mathfrak{X}}^{(m)}}^{L}\Phi_{1}^{!}\mathcal{\widehat{D}}_{\mathfrak{X}}^{(m)}
\]
and this tensor product is concentrated in homological degree $0$.
In particular, if $\varphi\circ\Phi_{1}=\Phi_{2}\circ\varphi$ then
we obtain 
\[
\widehat{\mathcal{D}}_{W(Y)\leftarrow W(X),\text{acc}}^{(m)}\tilde{=}\varphi^{-1}(\Phi_{2}^{*}\widehat{\mathcal{D}}_{\mathfrak{Y}}^{(m)})\widehat{\otimes}_{\varphi^{-1}(\mathcal{\widehat{D}}_{\mathfrak{Y}}^{(m)})}^{L}\widehat{\mathcal{D}}_{\mathfrak{Y}\leftarrow\mathfrak{X}}^{(m)}\widehat{\otimes}_{\mathcal{\widehat{D}}_{\mathfrak{X}}^{(m)}}^{L}\Phi_{1}^{!}\mathcal{\widehat{D}}_{\mathfrak{X}}^{(m)}
\]
\end{prop}

Applying the fact that $\widehat{\mathcal{D}}_{\mathfrak{X}\to\mathfrak{Y}}^{(m),\Phi_{1},\Phi_{2}}/p\tilde{=}\mathcal{D}_{X\to Y}^{(m)}$
(by \prettyref{cor:Composition-of-pullbacks}, part $1$), we obtain
from this the 
\begin{cor}
For any $\varphi:X\to Y$ there is an isomorphism 
\[
\widehat{\mathcal{D}}_{W(Y)\leftarrow W(X),\text{acc}}^{(m)}/p\tilde{\to}\varphi^{-1}(\mathcal{B}_{Y}^{(m)})\otimes_{\varphi^{-1}(\mathcal{D}_{Y}^{(m)})}\mathcal{D}_{Y\leftarrow X}^{(m)}\otimes_{\mathcal{D}_{X}^{(m)}}\mathcal{B}_{X}^{(m),r}
\]
\end{cor}

From this we deduce
\begin{thm}
\label{thm:Push-and-transfer}Let $\varphi:X\to Y$ be an arbitrary
morphism. 

1) Then the natural map 
\[
R(W\varphi)_{*}(W\varphi^{-1}(\mathcal{B}_{Y}^{(m)})\otimes_{\varphi^{-1}(\mathcal{D}_{Y}^{(m)})}^{L}\mathcal{D}_{X\leftarrow Y}^{(m)}\otimes_{\mathcal{D}_{X}^{(m)}}^{L}\mathcal{N}^{\cdot})_{\text{acc}}\to\mathcal{B}_{Y}^{(m)}\otimes_{\mathcal{D}_{Y}^{(m)}}^{L}R\varphi_{*}(\mathcal{D}_{Y\leftarrow X}^{(m)}\otimes_{\mathcal{D}_{X}^{(m)}}^{L}\mathcal{N}^{\cdot})
\]
is an isomorphism for any $\mathcal{N}^{\cdot}\in D(\mathcal{D}_{X}^{(m)}-\text{mod})$. 

2) Let $\mathcal{N}^{\cdot}\in D(\mathcal{D}_{X}^{(m)}-\text{mod})$.
Then there is a natural isomorphism 
\[
\int_{W\varphi}(\mathcal{B}_{X}^{(m)}\otimes_{\mathcal{D}_{X}^{(m)}}^{L}\mathcal{N}^{\cdot})\tilde{\to}\mathcal{B}_{Y}^{(m)}\otimes_{\mathcal{D}_{Y}^{(m)}}^{L}\int_{\varphi}\mathcal{N}^{\cdot}
\]
\end{thm}

\begin{proof}
By general nonsense there is a functorial morphism 
\[
R(W\varphi)_{*}(W\varphi^{-1}(\mathcal{B}_{Y}^{(m)})\otimes_{\varphi^{-1}(\mathcal{D}_{Y}^{(m)})}^{L}\mathcal{M}^{\cdot})\to\mathcal{B}_{Y}^{(m)}\otimes_{\mathcal{D}_{Y}^{(m)}}^{L}R\varphi_{*}(\mathcal{M}^{\cdot})
\]
for any $\mathcal{M}^{\cdot}\in D(\varphi^{-1}(\mathcal{D}_{Y}^{(m)})-\text{mod})$,
c.f. {[}Stacks{]}, tag 0B54. As the right hand side is clearly accessible,
applying the functor $()_{\text{acc}}$ to both sides gives a morphism
\begin{equation}
R(W\varphi)_{*}(W\varphi^{-1}(\mathcal{B}_{Y}^{(m)})\otimes_{\varphi^{-1}(\mathcal{D}_{Y}^{(m)})}^{L}\mathcal{M}^{\cdot})_{\text{acc}}\to\mathcal{B}_{Y}^{(m)}\otimes_{\mathcal{D}_{Y}^{(m)}}^{L}R\varphi_{*}(\mathcal{M}^{\cdot})\label{eq:general-adjunction-map}
\end{equation}
To show that is is an isomorphism it suffices to show that it is so
after applying $R\mathcal{H}om_{\mathcal{\widehat{D}}_{W(Y)}^{(m)}/p}(\mathcal{B}_{Y}^{(m)},)$.
We have 
\[
R\mathcal{H}om_{\mathcal{\widehat{D}}_{W(Y)}^{(m)}/p}(\mathcal{B}_{Y}^{(m)},R(W\varphi)_{*}(W\varphi^{-1}(\mathcal{B}_{Y}^{(m)})\otimes_{\varphi^{-1}(\mathcal{D}_{Y}^{(m)})}^{L}\mathcal{M}^{\cdot})_{\text{acc}})
\]
\[
\tilde{\to}R\mathcal{H}om_{\mathcal{\widehat{D}}_{W(Y)}^{(m)}/p}(\mathcal{B}_{Y}^{(m)},R(W\varphi)_{*}(W\varphi^{-1}(\mathcal{B}_{Y}^{(m)})\otimes_{\varphi^{-1}(\mathcal{D}_{Y}^{(m)})}^{L}\mathcal{M}^{\cdot}))
\]
\[
\tilde{\to}R\varphi_{*}R\mathcal{H}om_{\varphi^{-1}(\mathcal{\widehat{D}}_{W(Y)}^{(m)}/p)}(W\varphi^{-1}(\mathcal{B}_{Y}^{(m)}),\varphi^{-1}(\mathcal{B}_{Y}^{(m)})\otimes_{\varphi^{-1}(\mathcal{D}_{Y}^{(m)})}^{L}\mathcal{M}^{\cdot})
\]
\[
\tilde{\to}R\varphi_{*}(\mathcal{M}^{\cdot})
\]
where the first isomorphism is by the definition of accessibility,
and the second is by the adjunction of $W\varphi^{-1}$ and $R\varphi_{*}$.
It follows that the map \prettyref{eq:general-adjunction-map} is
an isomorphism, and $1)$ follows by setting $\mathcal{M}^{\cdot}=\mathcal{D}_{X\leftarrow Y}^{(m)}\otimes_{\mathcal{D}_{X}^{(m)}}^{L}\mathcal{N}^{\cdot}$.
Now $2)$ follows from the description of the transfer bimodule given
directly above. 
\end{proof}

Now let's discuss the base change functor $\otimes_{W(k)}^{L}k$.
The situation is very nice here:
\begin{lem}
\label{lem:Push-mod-p}There is, for any $\mathcal{M}^{\cdot}\in D_{\text{acc}}(\mathcal{\widehat{D}}_{W(X)}^{(m)}-\text{mod})$,
an isomorphism 
\[
(\int_{W\varphi}\mathcal{M}^{\cdot})\otimes_{W(k)}^{L}k\tilde{=}\int_{W\varphi}(\mathcal{M}^{\cdot}\otimes_{W(k)}^{L}k)
\]
where on the right we take the pushforward in the category $D_{\text{acc}}(\mathcal{\widehat{D}}_{W(X)}^{(m)}/p-\text{mod})$. 
\end{lem}

\begin{proof}
Again by {[}Stacks{]}, tag 0B54 there is a map 
\[
R(W\varphi)_{*}(\widehat{\mathcal{D}}_{W(Y)\leftarrow W(X)}^{(m)}\widehat{\otimes}_{\mathcal{\widehat{D}}_{W(X)}^{(m)}}^{L}\mathcal{M}^{\cdot})\otimes_{W(k)}^{L}k\to R(W\varphi)_{*}((\widehat{\mathcal{D}}_{W(Y)\leftarrow W(X)}^{(m)}\widehat{\otimes}_{\mathcal{\widehat{D}}_{W(X)}^{(m)}}^{L}\mathcal{M}^{\cdot})\otimes_{W(k)}^{L}k)
\]
which is an isomorphism because $k$ is a perfect complex over $W(k)$.
However, we have 
\[
(\widehat{\mathcal{D}}_{W(Y)\leftarrow W(X)}^{(m)}\widehat{\otimes}_{\mathcal{\widehat{D}}_{W(X)}^{(m)}}^{L}\mathcal{M}^{\cdot})\otimes_{W(k)}^{L}k
\]
\[
\tilde{=}(\widehat{\mathcal{D}}_{W(Y)\leftarrow W(X)}^{(m)}/p)\widehat{\otimes}_{\mathcal{\widehat{D}}_{W(X)}^{(m)}/p}^{L}(\mathcal{M}^{\cdot}\otimes_{W(k)}^{L}k)
\]
as both $\widehat{\mathcal{D}}_{W(X)\leftarrow W(Y)}^{(m)}$ and $\widehat{\mathcal{D}}_{W(X)}^{(m)}$
are flat over $W(k)$. Further, as the functor $()_{\text{acc}}$
commutes with $\otimes_{W(k)}^{L}k$, applying $()_{\text{acc}}$
to both sides yields the result. 
\end{proof}
This yields
\begin{cor}
Suppose $\mathcal{M}^{\cdot}\in D_{\text{acc}}(\mathcal{\widehat{D}}_{W(X)}^{(m)}-\text{mod})$,
and let $\varphi:X\to Y$, $\psi:Y\to Z$. There is a functorial morphism
\[
\int_{W\psi}\int_{W\varphi}\mathcal{M^{\cdot}}\to\int_{W(\psi\circ\varphi)}\mathcal{M}^{\cdot}
\]
If $\mathcal{M}^{\cdot}\in D_{\text{qcoh}}(\mathcal{\widehat{D}}_{W(X)}^{(m)}-\text{mod})$,
then this map is an isomorphism. The same holds for $\mathcal{M}^{\cdot}\in D_{\text{qcoh}}(\mathcal{\widehat{D}}_{W(X)}^{(m)}/p^{r}-\text{mod})$,
for the pushforward of $\mathcal{\widehat{D}}_{W(X)}^{(m)}/p^{r}$-modules. 
\end{cor}

\begin{proof}
The existence of the map is a little exercise in base changing. Namely,
from \prettyref{cor:Composition-of-pullbacks} we obtain the isomorphism
\[
W\varphi^{-1}(\widehat{\mathcal{D}}_{W(Z)\leftarrow W(Y),\text{acc}}^{(m)})\otimes_{W\varphi^{-1}(\widehat{\mathcal{D}}_{W(Y)}^{(m)})}^{L}\widehat{\mathcal{D}}_{W(Y)\leftarrow W(X),\text{acc}}^{(m)}\tilde{\to}\widehat{\mathcal{D}}_{W(Z)\leftarrow W(X),\text{acc}}^{(m)}
\]
So we have (by {[}Stacks{]}, tag 0B54)
\[
RW\psi_{*}(\widehat{\mathcal{D}}_{W(Z)\leftarrow W(Y),\text{acc}}^{(m)}\widehat{\otimes}_{\widehat{\mathcal{D}}_{W(Y)}^{(m)}}R(W\varphi)_{*}(\widehat{\mathcal{D}}_{W(Y)\leftarrow W(X)}^{(m)}\widehat{\otimes}_{\mathcal{\widehat{D}}_{W(X)}^{(m)}}^{L}\mathcal{M}^{\cdot}))
\]
\[
\to RW\psi_{*}\circ RW\varphi_{*}(W\varphi^{-1}(\widehat{\mathcal{D}}_{W(Z)\leftarrow W(Y),\text{acc}}^{(m)})\widehat{\otimes}_{W\varphi^{-1}(\widehat{\mathcal{D}}_{W(Y)}^{(m)})}(\widehat{\mathcal{D}}_{W(Y)\leftarrow W(X)}^{(m)}\widehat{\otimes}_{\mathcal{\widehat{D}}_{W(X)}^{(m)}}^{L}\mathcal{M}^{\cdot}))
\]
\[
\tilde{\to}RW(\psi\circ\varphi)_{*}(\widehat{\mathcal{D}}_{W(Z)\leftarrow W(X),\text{acc}}^{(m)}\widehat{\otimes}_{\mathcal{\widehat{D}}_{W(X)}^{(m)}}^{L}\mathcal{M}^{\cdot})
\]
Further, we have

\[
\int_{W\psi}\int_{W\varphi}\mathcal{M}^{\cdot}=RW\psi_{*}(\widehat{\mathcal{D}}_{W(Z)\leftarrow W(Y),\text{acc}}^{(m)}\widehat{\otimes}_{\widehat{\mathcal{D}}_{W(Y)}^{(m)}}\int_{W\varphi}\mathcal{M}^{\cdot})_{\text{acc}}
\]
\[
=RW\psi_{*}(\widehat{\mathcal{D}}_{W(Z)\leftarrow W(Y),\text{acc}}^{(m)}\widehat{\otimes}_{\widehat{\mathcal{D}}_{W(Y)}^{(m)}}R(W\varphi)_{*}(\widehat{\mathcal{D}}_{W(Y)\leftarrow W(X)}^{(m)}\widehat{\otimes}_{\mathcal{\widehat{D}}_{W(X)}^{(m)}}^{L}\mathcal{M}^{\cdot})_{\text{acc}})_{\text{acc}}
\]
\[
\to RW\psi_{*}(\widehat{\mathcal{D}}_{W(Z)\leftarrow W(Y),\text{acc}}^{(m)}\widehat{\otimes}_{\widehat{\mathcal{D}}_{W(Y)}^{(m)}}R(W\varphi)_{*}(\widehat{\mathcal{D}}_{W(Y)\leftarrow W(X)}^{(m)}\widehat{\otimes}_{\mathcal{\widehat{D}}_{W(X)}^{(m)}}^{L}\mathcal{M}^{\cdot}))_{\text{scc}}
\]
vis the natural map $\mathcal{N}_{\text{acc}}^{\cdot}\to\mathcal{N}^{\cdot}$
in $\widehat{\mathcal{D}}_{W(Y)}^{(m)}-\text{mod}$. Then, applying
$()_{\text{acc}}$ in $\widehat{\mathcal{D}}_{W(Z)}^{(m)}$-mod to
the previous map, we obtain a map 
\[
\int_{W\psi}\int_{W\varphi}\mathcal{M}^{\cdot}\to\tilde{\to}RW(\psi\circ\varphi)_{*}(\widehat{\mathcal{D}}_{W(Z)\leftarrow W(X),\text{acc}}^{(m)}\widehat{\otimes}_{\mathcal{\widehat{D}}_{W(X)}^{(m)}}^{L}\mathcal{M}^{\cdot})_{\text{acc}}=\int_{W(\psi\circ\varphi)}\mathcal{M}^{\cdot}
\]

To show its an isomorphism, apply $\otimes_{W(k)}^{L}k$ and the previous
lemma to reduce to the case of $\mathcal{D}_{X}^{(m)}$-modules in
positive characteristic. There, we have 
\[
\int_{\psi}\int_{\varphi}\mathcal{M^{\cdot}\tilde{\to}\int_{\psi\circ\varphi}\mathcal{M}^{\cdot}}
\]
for any $\mathcal{M}^{\cdot}\in D_{\text{qcoh}}(\mathcal{D}_{X}^{(m)}-\text{mod})$
(c.f. \cite{key-16}, lemma 7.7 for this statement), whence the result
for $\mathcal{M}^{\cdot}\in D_{\text{qcoh}}(\mathcal{\widehat{D}}_{W(X)}^{(m)}-\text{mod})$.
To obtain the result for $\mathcal{M}^{\cdot}\in D_{\text{qcoh}}(\mathcal{\widehat{D}}_{W(X)}^{(m)}/p^{r}-\text{mod})$,
apply \prettyref{rem:Pushforwards-agree}. 
\end{proof}
To round things out, we deduce the following compatibility: 
\begin{cor}
1) Let $\varphi:\mathfrak{X}_{r}\to\mathfrak{Y}_{r}$ where, if $m=0$
and $r>1$, we assume $p\neq2$ (we allow $r=\infty$ here). Let $\mathcal{N}^{\cdot}\in D(\mathcal{D}_{\mathfrak{X}_{r}}^{(m)}-\text{mod})$.
Then there is a natural isomorphism 
\[
\int_{\varphi}(\mathcal{B}_{\mathfrak{X}_{r}}^{(m)}\otimes_{\mathcal{D}_{\mathfrak{X_{r}}}^{(m)}}^{L}\mathcal{N}^{\cdot})\tilde{\to}\mathcal{B}_{\mathfrak{Y}_{r}}^{(m)}\otimes_{\mathcal{D}_{\mathfrak{Y}_{r}}^{(m)}}^{L}\int_{\varphi}\mathcal{N}^{\cdot}
\]
\end{cor}

\begin{proof}
This follows by breaking $\varphi$ up into a smooth morphism and
a closed embedding just as in \prettyref{cor:Pull-and-compose}. 
\end{proof}
Finally, we'd like to discuss some further specific properties of
the pushforward when $\varphi$ satisfies extra conditions. First,
we specialize to the case where $\varphi:X\to Y$ is a smooth morphism.
In this case, there is an adjunction between the pullback and the
pushforward; more precisely, we have: 
\begin{cor}
\label{cor:smooth-adjunction}Let $\varphi:X\to Y$ be smooth of relative
dimension $d$; let $\mathcal{M}^{\cdot}\in D_{\text{acc}}(\mathcal{\widehat{D}}_{W(X)}^{(m)}-\text{mod})$
and $\mathcal{N}^{\cdot}\in D_{\text{acc}}(\mathcal{\widehat{D}}_{W(Y)}^{(m)}-\text{mod}))$.
Then there is an isomorphism of functors 
\[
RW\varphi_{*}R\underline{\mathcal{H}om}_{\mathcal{\widehat{D}}_{W(X)}^{(m)}}(L(W\varphi)^{*}\mathcal{N}^{\cdot},\mathcal{M}^{\cdot})[d]\tilde{\to}R\underline{\mathcal{H}om}{}_{\mathcal{\widehat{D}}_{W(Y)}^{(m)}}(\mathcal{N}^{\cdot},\int_{W\varphi}\mathcal{M}^{\cdot})
\]
In particular the functors $LW\varphi^{*}[d]=W\varphi^{!}$ and ${\displaystyle \int_{\varphi}}$
form an adjoint pair on accessible modules. The analogous statement
holds for accessible $\mathcal{\widehat{D}}_{W(X)}^{(m)}/p^{r}$ modules
for any $r\geq1$. 
\end{cor}

We will deduce the this theorem from the analogous fact for $\mathcal{\widehat{D}}^{(m)}$-modules.
As in that case, the key point is to prove the 
\begin{prop}
\label{prop:X-Y-switch-for-smooth}For any smooth morphism $\varphi:X\to Y$
of relative dimension $d$ there is an isomorphism of $(\widehat{\mathcal{D}}_{W(Y)}^{(m)},\widehat{\mathcal{D}}_{W(X)}^{(m)})$
bimodules 
\[
R\mathcal{H}om_{\mathcal{\widehat{D}}_{W(X)}^{(m)}}(\widehat{\mathcal{D}}_{W(X)\to W(Y),\text{acc}}^{(m)},\widehat{\mathcal{D}}_{W(X),\text{acc}}^{(m)})_{Y-\text{acc}}\tilde{=}\widehat{\mathcal{D}}_{W(Y)\leftarrow W(X),\text{acc}}^{(m)}[-d]
\]
where on the left hand side $()_{Y-\text{acc}}$ refers to the functor
of accessibalization applied in the category of left $\varphi^{-1}(\widehat{\mathcal{D}}_{W(Y)}^{(m)})$-modules. 
\end{prop}

Before doing so, we recall that the analogous isomorphism 
\[
R\mathcal{H}om_{\mathcal{\widehat{D}}_{\mathfrak{X}}^{(m)}}(\widehat{\mathcal{D}}_{\mathfrak{X}\to\mathfrak{Y}}^{(m)},\widehat{\mathcal{D}}_{\mathfrak{X}}^{(m)})\tilde{\to}\widehat{\mathcal{D}}_{\mathfrak{Y}\leftarrow\mathfrak{X}}^{(m)}[-d]
\]
is the basic point in proving this adjunction for $\widehat{\mathcal{D}}^{(m)}$-modules;
when $m=0$ this can be proved via the de Rham resolution, and when
$m>0$ it follows via Frobenius descent. 
\begin{proof}
Let us begin in the case where $X$ and $Y$ are affine, with smooth
lifts $\mathfrak{X}$ and $\mathfrak{Y}$. Let $\Phi_{1}$ and $\Phi_{2}$
be lifts of Frobenius on $\mathfrak{X}$ and $\mathfrak{Y}$, respectively,
chosen so that $\Phi_{2}\circ\varphi=\varphi\circ\Phi_{1}$ (this
is possible as $\varphi$ is smooth).This means that, according to
\prettyref{prop:local-bimodule-for-push}
\[
\widehat{\mathcal{D}}_{W(Y)\leftarrow W(X),\text{acc}}^{(m)}\tilde{=}\varphi^{-1}(\Phi_{2}^{*}\widehat{\mathcal{D}}_{\mathfrak{Y}}^{(m)})\otimes_{\varphi^{-1}(\mathcal{\widehat{D}}_{\mathfrak{Y}}^{(m)})}^{L}\widehat{\mathcal{D}}_{\mathfrak{Y}\leftarrow\mathfrak{X}}^{(m)}\widehat{\otimes}_{\mathcal{\widehat{D}}_{\mathfrak{X}}^{(m)}}^{L}\Phi_{1}^{!}\mathcal{\widehat{D}}_{\mathfrak{X}}^{(m)}
\]
As both sides are, by definition, accessible over $\varphi^{-1}(\widehat{\mathcal{D}}_{W(Y)}^{(m)})$,
we may construct the isomorphism by applying the functor $R\mathcal{H}om_{W\varphi^{-1}(\widehat{\mathcal{D}}_{W(Y)}^{(m)})}(W\varphi^{-1}(\Phi_{2}^{*}\widehat{\mathcal{D}}_{\mathfrak{Y}}^{(m)}),?)$
to both sides. On the left hand side, this yields
\[
R\mathcal{H}om_{W\varphi^{-1}(\widehat{\mathcal{D}}_{W(Y)}^{(m)})}(W\varphi^{-1}(\Phi_{2}^{*}\widehat{\mathcal{D}}_{\mathfrak{Y}}^{(m)}),R\mathcal{H}om_{\mathcal{\widehat{D}}_{W(X)}^{(m)}}(\widehat{\mathcal{D}}_{W(X)\to W(Y),\text{acc}}^{(m)},\widehat{\mathcal{D}}_{W(X),\text{acc}}^{(m)}))
\]
\[
\tilde{\to}R\mathcal{H}om_{\mathcal{\widehat{D}}_{W(X)}^{(m)}}(\widehat{\mathcal{D}}_{W(X)\to W(Y),\text{acc}}^{(m)}\widehat{\otimes}_{W\varphi^{-1}(\widehat{\mathcal{D}}_{W(Y)}^{(m)})}^{L}W\varphi^{-1}(\Phi_{2}^{*}\widehat{\mathcal{D}}_{\mathfrak{Y}}^{(m)}),\widehat{\mathcal{D}}_{W(X),\text{acc}}^{(m)}))
\]
\[
\tilde{\to}R\mathcal{H}om_{\mathcal{\widehat{D}}_{W(X)}^{(m)}}(\Phi_{1}^{*}\widehat{\mathcal{D}}_{\mathfrak{X}}^{(m)}\widehat{\otimes}_{\widehat{\mathcal{D}}_{\mathfrak{X}}^{(m)}}^{L}\widehat{\mathcal{D}}_{\mathfrak{X}\to\mathfrak{Y}}^{(m)},\widehat{\mathcal{D}}_{W(X),\text{acc}}^{(m)}))
\]
\[
\tilde{\to}R\mathcal{H}om_{\mathcal{\widehat{D}}_{\mathfrak{X}}^{(m)}}(\widehat{\mathcal{D}}_{\mathfrak{X}\to\mathfrak{Y}}^{(m)},\Phi_{1}^{!}\widehat{\mathcal{D}}_{\mathfrak{X}}^{(m)})
\]
\[
\tilde{\to}R\mathcal{H}om_{\mathcal{\widehat{D}}_{\mathfrak{X}}^{(m)}}(\widehat{\mathcal{D}}_{\mathfrak{X}\to\mathfrak{Y}}^{(m)},\widehat{\mathcal{D}}_{\mathfrak{X}}^{(m)})\otimes_{\widehat{\mathcal{D}}_{\mathfrak{X}}^{(m)}}^{L}\Phi_{1}^{!}\widehat{\mathcal{D}}_{\mathfrak{X}}^{(m)}
\]
\[
\tilde{\to}\widehat{\mathcal{D}}_{\mathfrak{Y}\leftarrow\mathfrak{X}}^{(m)}\otimes_{\widehat{\mathcal{D}}_{\mathfrak{X}}^{(m)}}^{L}\Phi_{1}^{!}\widehat{\mathcal{D}}_{\mathfrak{X}}^{(m)}[-d]
\]
where we used hom tensor adjunction for the first isomorphism, \prettyref{prop:Pullbback-and-transfer}
for the second, the construction of $\widehat{\mathcal{D}}_{W(X),\text{acc}}^{(m)}$
for the third, the fact that $\widehat{\mathcal{D}}_{\mathfrak{X}\to\mathfrak{Y}}^{(m)}$
is isomorphic to as perfect complex of $\widehat{\mathcal{D}}_{\mathfrak{X}}^{(m)}$-modules\footnote{As $\widehat{\mathcal{D}}_{\mathfrak{X}}^{(m)}$ is locally of finite
homological dimension} for the fourth, and isomorphism directly above for the fifth. Thus
we obtain the isomorphism of the proposition when $X$ and $Y$ are
affine by the description of $\widehat{\mathcal{D}}_{W(Y)\leftarrow W(X),\text{acc}}^{(m)}$
given just above. 

For the general case, we note that $\widehat{\mathcal{D}}_{W(Y)\leftarrow W(X),\text{acc}}^{(m)}$
is an accessible $(\varphi^{-1}(\widehat{\mathcal{D}}_{W(Y)}^{(m)}),\widehat{\mathcal{D}}_{W(X)}^{(m)})$
bimodule, and so we have 
\[
\mathcal{E}nd_{(\varphi^{-1}(\widehat{\mathcal{D}}_{W(Y)}^{(m)}),\widehat{\mathcal{D}}_{W(X)}^{(m)})}(\widehat{\mathcal{D}}_{W(Y)\leftarrow W(X),\text{acc}}^{(m)})\tilde{\to}\mathcal{E}nd_{(\varphi^{-1}(\widehat{\mathcal{D}}_{\mathfrak{Y}}^{(m)}),\widehat{\mathcal{D}}_{\mathfrak{X}}^{(m)})}(\widehat{\mathcal{D}}_{\mathfrak{Y}\leftarrow\mathfrak{X}}^{(m)})=W(k)_{\mathfrak{X}}
\]
where $W(k)_{\mathfrak{X}}$ refers to the locally constant sheaf
with sections $W(k)$. As this sheaf is flasque, we see that the the
locally defined isomorphisms constructed above must glue to a globally
defined isomorphism (which is unique up to rescale by an element of
$W(k)$). 
\end{proof}
Given this, let's give the
\begin{proof}
(of \prettyref{cor:smooth-adjunction}) (following \cite{key-4},
Theorem 4.40). We have 
\[
R\mathcal{H}om_{\mathcal{\widehat{D}}_{W(X)}^{(m)}}(LW\varphi^{*}\mathcal{N}^{\cdot},\mathcal{M}^{\cdot})[d]
\]
\[
=R\mathcal{H}om_{\mathcal{\widehat{D}}_{W(X)}^{(m)}}(\widehat{\mathcal{D}}_{W(X)\to W(Y),\text{acc}}^{(m)}\otimes_{W\varphi^{-1}(\widehat{\mathcal{D}}_{W(Y)}^{(m)})}^{L}\varphi^{-1}(\mathcal{N}^{\cdot}),\mathcal{M}^{\cdot})[d]
\]
\[
\tilde{\to}R\mathcal{H}om_{W\varphi^{-1}(\mathcal{\widehat{D}}_{W(Y)}^{(m)})}(W\varphi^{-1}(\mathcal{N}^{\cdot}),R\mathcal{H}om_{\mathcal{\widehat{D}}_{W(X)}^{(m)}}(\widehat{\mathcal{D}}_{W(X)\to W(Y),\text{acc}}^{(m)},\mathcal{M}^{\cdot}))[d]
\]
Now, there is the obvious natural map
\[
R\mathcal{H}om_{\mathcal{\widehat{D}}_{W(X)}^{(m)}}(\widehat{\mathcal{D}}_{W(X)\to W(Y),\text{acc}}^{(m)},\widehat{\mathcal{D}}_{W(X)}^{(m)})\widehat{\otimes}_{\widehat{\mathcal{D}}_{W(X)}^{(m)}}^{L}\mathcal{M}^{\cdot}\to R\mathcal{H}om_{\mathcal{\widehat{D}}_{W(X)}^{(m)}}(\widehat{\mathcal{D}}_{W(X)\to W(Y),\text{acc}}^{(m)},\mathcal{M}^{\cdot})
\]
for any $\mathcal{M}^{\cdot}\in D(\widehat{\mathcal{D}}_{W(X)}^{(m)}-\text{mod})$.
Via the right action of $W\varphi^{-1}(\widehat{\mathcal{D}}_{W(Y)}^{(m)})$
on $\widehat{\mathcal{D}}_{W(X)\to W(Y),\text{acc}}^{(m)}$, this
is a morphism complexes of of left $W\varphi^{-1}(\widehat{\mathcal{D}}_{W(Y)}^{(m)})$-modules.
Thus, applying the functors $R\mathcal{H}om_{W\varphi^{-1}(\mathcal{\widehat{D}}_{W(Y)}^{(m)})}(W\varphi^{-1}(\mathcal{N}^{\cdot}),)$
and $RW\varphi_{*}$ we obtain a morphism 
\[
RW\varphi_{*}R\mathcal{H}om_{W\varphi^{-1}(\mathcal{\widehat{D}}_{W(Y)}^{(m)})}(W\varphi^{-1}(\mathcal{N}^{\cdot}),R\mathcal{H}om_{\mathcal{\widehat{D}}_{W(X)}^{(m)}}(\widehat{\mathcal{D}}_{W(X)\to W(Y),\text{acc}}^{(m)},\widehat{\mathcal{D}}_{W(X)}^{(m)})\widehat{\otimes}_{\widehat{\mathcal{D}}_{W(X)}^{(m)}}^{L}\mathcal{M}^{\cdot})[d]
\]
\[
\to RW\varphi_{*}R\mathcal{H}om_{W\varphi^{-1}(\mathcal{\widehat{D}}_{W(Y)}^{(m)})}(W\varphi^{-1}(\mathcal{N}^{\cdot}),R\mathcal{H}om_{\mathcal{\widehat{D}}_{W(X)}^{(m)}}(\widehat{\mathcal{D}}_{W(X)\to W(Y),\text{acc}}^{(m)},\mathcal{M}^{\cdot}))[d]
\]
\[
\tilde{=}R\mathcal{H}om_{\mathcal{\widehat{D}}_{W(X)}^{(m)}}(LW\varphi^{*}\mathcal{N}^{\cdot},\mathcal{M}^{\cdot})[d]
\]
of complexes of $\widehat{\mathcal{D}}_{W(Y)}^{(m)}$-modules. Now,
as $\mathcal{N}^{\cdot}$ is accessible over $\mathcal{\widehat{D}}_{W(Y)}^{(m)}$,
we have 
\[
R\mathcal{H}om_{W\varphi^{-1}(\mathcal{\widehat{D}}_{W(Y)}^{(m)})}(W\varphi^{-1}(\mathcal{N}^{\cdot}),R\mathcal{H}om_{\mathcal{\widehat{D}}_{W(X)}^{(m)}}(\widehat{\mathcal{D}}_{W(X)\to W(Y),\text{acc}}^{(m)},\widehat{\mathcal{D}}_{W(X)}^{(m)})\widehat{\otimes}_{\widehat{\mathcal{D}}_{W(X)}^{(m)}}^{L}\mathcal{M}^{\cdot})[d]
\]
\[
=R\mathcal{H}om_{W\varphi^{-1}(\mathcal{\widehat{D}}_{W(Y)}^{(m)})}(W\varphi^{-1}(\mathcal{N}^{\cdot}),R\mathcal{H}om_{\mathcal{\widehat{D}}_{W(X)}^{(m)}}(\widehat{\mathcal{D}}_{W(X)\to W(Y),\text{acc}}^{(m)},\widehat{\mathcal{D}}_{W(X)}^{(m)})_{\text{acc}}\widehat{\otimes}_{\widehat{\mathcal{D}}_{W(X)}^{(m)}}^{L}\mathcal{M}^{\cdot})[d]
\]
\[
\tilde{\to}R\mathcal{H}om_{W\varphi^{-1}(\mathcal{\widehat{D}}_{W(Y)}^{(m)})}(W\varphi^{-1}(\mathcal{N}^{\cdot}),\widehat{\mathcal{D}}_{W(Y)\leftarrow W(X),\text{acc}}^{(m)}\widehat{\otimes}_{\widehat{\mathcal{D}}_{W(X)}^{(m)}}^{L}\mathcal{M}^{\cdot})
\]
by the previous proposition. Furthermore 
\[
RW\varphi_{*}R\mathcal{H}om_{W\varphi^{-1}(\mathcal{\widehat{D}}_{W(Y)}^{(m)})}(W\varphi^{-1}(\mathcal{N}^{\cdot}),\widehat{\mathcal{D}}_{W(Y)\leftarrow W(X),\text{acc}}^{(m)}\widehat{\otimes}_{\widehat{\mathcal{D}}_{W(X)}^{(m)}}^{L}\mathcal{M}^{\cdot})
\]
\[
\tilde{\to}R\mathcal{H}om_{\mathcal{\widehat{D}}_{W(Y)}^{(m)}}(\mathcal{N}^{\cdot},RW\varphi_{*}(\widehat{\mathcal{D}}_{W(Y)\leftarrow W(X),\text{acc}}^{(m)}\widehat{\otimes}_{\widehat{\mathcal{D}}_{W(X)}^{(m)}}^{L}\mathcal{M}^{\cdot}))
\]
\[
\tilde{\to}R\mathcal{H}om_{\mathcal{\widehat{D}}_{W(Y)}^{(m)}}(\mathcal{N}^{\cdot},\int_{W\varphi}\mathcal{M}^{\cdot})
\]
where we again use that $\mathcal{N}^{\cdot}$ is accessible in the
last line; this follows from
\[
{\displaystyle \int_{W\varphi}\mathcal{M}^{\cdot}=R\varphi_{*}(\widehat{\mathcal{D}}_{W(Y)\leftarrow W(X),\text{acc}}^{(m)}\widehat{\otimes}_{\widehat{\mathcal{D}}_{W(X)}^{(m)}}^{L}\mathcal{M}^{\cdot})_{\text{acc}}}
\]
 Summing up, we've obtained a functorial map 
\[
R\mathcal{H}om_{\mathcal{\widehat{D}}_{W(Y)}^{(m)}}(\mathcal{N}^{\cdot},\int_{W\varphi}\mathcal{M}^{\cdot})\to R\mathcal{H}om_{\mathcal{\widehat{D}}_{W(X)}^{(m)}}(LW\varphi^{*}\mathcal{N}^{\cdot},\mathcal{M}^{\cdot})[d]
\]
To check it is an isomorphism, we may (by cohomological completeness)
apply $\otimes_{W(k)}^{L}k$, and then use \prettyref{lem:Push-mod-p}
to reduce the statement to the (known) adjunction for $\mathcal{D}^{(m)}$-modules. 
\end{proof}
Combining this with the compatibility of pullback on crystals, we
deduce:
\begin{cor}
\label{cor:Crystalline-Push}Let $\mathcal{M}^{\cdot}\in D^{+}(\text{Qcoh}(\text{Crys}(X,W_{r}(k)))$,
and let $\epsilon(\mathcal{M}^{\cdot})$ denote the associated element
of $D_{\text{qcoh,}\text{acc}}^{+}(\mathcal{\widehat{D}}_{W(X)}^{(0)}/p^{r}-\text{mod})$.

Then there is an isomorphism 
\[
\int_{\varphi}\epsilon(\mathcal{M}^{\cdot})[-d]\tilde{\to}\epsilon(R\varphi_{*,\text{crys}}(\mathcal{M}^{\cdot}))
\]
in $D_{\text{qcoh}}^{+}(\mathcal{\widehat{D}}_{W(Y)}^{(0)}/p^{r}-\text{mod})$. 
\end{cor}

\begin{proof}
By definition, the functors 
\[
L\varphi_{\text{crys}}^{*}:D^{+}(\text{Qcoh}(\text{Crys}(Y,W_{r}(k)))\to D^{+}(\text{Qcoh}(\text{Crys}(X,W_{r}(k)))
\]
and $R\varphi_{*,\text{crys}}:D^{+}(\text{Qcoh}(\text{Crys}(X,W_{r}(k)))\to D^{+}(\text{Qcoh}(\text{Crys}(Y,W_{r}(k)))$
form an adjoint pair. As the functor $\epsilon$ is compatible with
pullback, we simply have to check that ${\displaystyle \int_{\varphi}}$
preserves the image; i.e., that it sends quasicoherent complexes of
nilpotent modules over $\mathcal{\widehat{D}}_{W(X)}^{(0)}/p^{r}$
to quasicoherent complexes of nilpotent modules over $\mathcal{\widehat{D}}_{W(Y)}^{(0)}/p^{r}$.
By the fact that nilpotency (for quasicoherent modules) can be characterized
as being supported on a certain closed subset (of the spectrum of
the center of the sheaf of differential operators), this can be checked
after applying $\otimes_{W(k)}^{L}k$. But then, applying \prettyref{thm:Push-and-transfer},
this follows from the analogous fact for $\mathcal{D}_{X}^{(0)}$-modules,
where it is a well known theorem of Katz (using the language of the
Gauss-Manin connection, it is proved in \cite{key-55}, section $7$,
or, in the language of $\mathcal{D}$-modules, in \cite{key-11},
section 2.5). 
\end{proof}
Finally, we have the following important fact when $\varphi$ is proper: 
\begin{thm}
Let $\varphi:X\to Y$ be a proper morphism. Then ${\displaystyle \int_{\varphi}}$
takes $D_{\text{coh}}^{b}(\mathcal{\widehat{D}}_{W(X)}^{(m)}-\text{mod})$
to $D_{\text{coh}}^{b}(\mathcal{\widehat{D}}_{W(Y)}^{(m)}-\text{mod})$. 
\end{thm}

\begin{proof}
Applying \cite{key-8}, theorem 1.6.3, we see that, a complex $\mathcal{N}^{\cdot}\in D(\mathcal{\widehat{D}}_{W(Y)}^{(m)}-\text{mod})$
is contained in $D_{\text{coh}}^{b}(\mathcal{\widehat{D}}_{W(Y)}^{(m)}-\text{mod})$
iff $\mathcal{N}^{\cdot}\otimes_{W(k)}^{L}k$ is contained in $D_{\text{coh}}^{b}(\mathcal{\widehat{D}}_{W(Y)}^{(m)}/p-\text{mod})$.
So it suffices to show that 
\[
\int_{\varphi}:D(\mathcal{\widehat{D}}_{W(X)}^{(m)}/p-\text{mod})\to D(\mathcal{\widehat{D}}_{W(Y)}^{(m)}/p-\text{mod})
\]
takes $D_{\text{coh }}^{b}$ to $D_{\text{coh}}^{b}$; this, in turn,
follows via \prettyref{thm:Push-and-transfer} from the analogous
fact for $\mathcal{D}^{(m)}$-modules. 
\end{proof}

\subsection{\label{subsec:The-de-Rham-Witt}The de Rham-Witt resolution}

In this section we'll recall the (relative) de Rham-Witt resolution
and explain the connection between the pushforward functor introduced
above and the (relative) de Rham-Witt complex. 

Throughout this section, fix a number $r\geq1$; we shall work mod
$p^{r}$. According to \cite{key-6}, and \cite{key-27} for the relative
case, there is attached to any smooth morphism $\varphi:X\to Y$ the
relative de Rham-Witt complex $W\Omega_{X/Y}^{\cdot}$. It is a dg
algebra inside quasicoherent sheaves on $W(X)$, it is complete with
respect to a canonical filtration which we shall denote $G(W\Omega_{X/Y}^{\cdot})$,
and the quotient $W\Omega_{X/Y}^{\cdot}/G^{1}\tilde{\to}\Omega_{X/Y}^{\cdot}$.
In particular zeroth term of $W\Omega_{X/Y}^{\cdot}$ is $W(X)$,
equipped with its usual $V^{i}$-filtration. 

Suppose $Y=\text{Spec}(k)$, and that $X$ is affine and admits local
coordinates, and we fix a coordinatized lift of $F$. Then there is
an inclusion $\Omega_{\mathfrak{X}_{r}}^{i}\to W\Omega_{X}^{i}/p^{r}$
which is compatible with the differential, and with the action of
$\mathcal{O}_{\mathfrak{X}_{r}}$ (via $\Phi:\mathcal{O}_{\mathfrak{X}_{r}}\to\mathcal{O}_{W(X)}/p^{r}$). 

We begin by reviewing a basic result of Etesse, in a form which is
useful for us: 
\begin{thm}
(\cite{key-56}) Let $\mathcal{M}$ be a quasicoherent crystal on
$X$, over $W_{r}(k)$, and let $\tilde{\mathcal{M}}$ be the associated
sheaf on $W(X)_{p^{r}=0}$. This sheaf admits an integrable de Rham-Witt
connection 
\[
\nabla:\tilde{\mathcal{M}}\to\tilde{\mathcal{M}}\widehat{\otimes}_{\mathcal{O}_{W(X)}/p^{r}}W\Omega_{X}^{1}/p^{r}
\]
which is continuous with respect to the natural topologies on both
sides (the completion on the right is with respect to the tensor product
filtration $V^{i}(\mathcal{O}_{W(X)}/p^{r})\cdot\tilde{\mathcal{M}}\otimes G^{j}(W\Omega_{X}^{1}/p^{r})$). 

Suppose $X$ is affine with local coordinates. Let $\mathfrak{X}$
be a smooth lift of $X$ with coordinatized lift of Frobenius $F$.
Then $\mathcal{M}$ defines a unique sheaf with integrable nilpotent
connection $\mathcal{M}'$ on $\mathfrak{X}_{n}$, and we have 
\[
\tilde{\mathcal{M}}\tilde{=}\widehat{\Phi}^{*}\mathcal{M}'
\]
with the connection given by 
\[
\nabla(p^{m}T^{I/p^{m}}\cdot m)=d(p^{m}T^{I/p^{m}})m+p^{m}T^{I/p^{m}}dm
\]
for $m\in\mathcal{M}'$; here, we are regarding $dm\in\mathcal{M}'\otimes\Omega_{\mathfrak{X}_{r}}^{1}\subset\mathcal{M}'\otimes_{\mathcal{O}_{\mathfrak{X}_{r}}}W\Omega_{X}^{1}/p^{r}\subset\widehat{\Phi}^{*}\mathcal{M}'\widehat{\otimes}_{\mathcal{O}_{W(X)}/p^{r}}W\Omega_{X}^{1}/p^{r}$.
The same results also hold for pro-objects in the category $\text{Qcoh(Crys}(X,W_{r}(k))$. 
\end{thm}

In other words, the theorem tells us that the explicitly defined de
Rham-Witt connection on $\widehat{\Phi}^{*}\mathcal{F}'$ is in fact
independent of the choice of $\Phi$. Let's apply this to $\mathcal{F}=(\widehat{\mathcal{D}}_{W(X),\text{crys}}^{(0)}/p^{r})_{\text{c-acc}}$
the completion (along the left action of $V^{i}(\mathcal{O}_{W(X)}/p^{r})$)
of $\widehat{\mathcal{D}}_{W(X),\text{crys},\text{acc}}/p^{r}$. This
is as inverse limit of quasicoherent crystals on $X$ over $W_{r}(k)$,
and so the theorem applies to it. From this we conclude: 
\begin{cor}
\label{cor:DRW-on-c-acc} Let $(\widehat{\mathcal{D}}_{W(X)}/p^{r})_{c-\text{acc}}$
be the completion (along the left action of $V^{i}(\mathcal{O}_{W(X)}/p^{r})$)
of $\widehat{\mathcal{D}}_{W(X),\text{acc}}/p^{r}$ There is a unique
continuous, integrable de Rham-Witt connection $\nabla$ on $(\widehat{\mathcal{D}}_{W(X)}/p^{r})_{c-\text{acc}}$
satisfying the following: suppose $U\subset X$ is open affine with
local coordinates. Let $\mathfrak{U}$ be a smooth lift of $U$ with
coordinatized lift of Frobenius $\Phi$. Then we have the isomorphism
\[
(\widehat{\mathcal{D}}_{W(U)}^{(0)}/p^{r})_{c-\text{acc}}\tilde{=}\widehat{\Phi}^{*}\Phi^{!}\mathcal{D}_{\mathfrak{U}_{r}}^{(0)}
\]
and therefore a continuous integrable de Rham-Witt connection 
\[
\nabla:\widehat{\Phi}^{*}\Phi^{!}\mathcal{D}_{\mathfrak{U}_{r}}^{(0)}\to\widehat{\Phi}^{*}\Phi^{!}\mathcal{D}_{\mathfrak{U}_{r}}^{(0)}\widehat{\otimes}_{\mathcal{O}_{W(U)/p^{r}}}W\Omega_{U}^{1}/p^{r}
\]
 (where the completion with respect to the tensor product filtration
as above) which is induced from the natural connection on $\Phi^{!}\mathcal{D}_{\mathfrak{U}_{r}}^{(0)}$
(it is a left $\mathcal{D}_{\mathfrak{U}_{r}}^{(0)}$-module by definition).
Then the connection $\nabla$ agrees with this induced connection
under the isomorphism. 
\end{cor}

\begin{proof}
By the previous theorem this is true after completing along the central
ideal defining the nilpotent support condition. As the natural completion
map 
\[
(\widehat{\mathcal{D}}_{W(X)}^{(0)}/p^{r})_{c-\text{acc}}\to(\widehat{\mathcal{D}}_{W(X),\text{crys}}^{(0)}/p^{r})_{\text{c-acc}}
\]
is injective, we see that the de Rham-Witt connection on $(\widehat{\mathcal{D}}_{W(X)}/p^{r})_{c-\text{acc}}$
which is locally defined by the isomorphism $(\widehat{\mathcal{D}}_{W(U)}^{(0)}/p^{r})_{c-\text{acc}}\tilde{=}\widehat{\Phi}^{*}\Phi^{!}\mathcal{D}_{\mathfrak{U}_{r}}^{(0)}$
is in fact independent of the choice of $\Phi$. Therefore this connection
glues to a globally defined connection on $(\widehat{\mathcal{D}}_{W(X)}/p^{r})_{c-\text{acc}}$
as required. 
\end{proof}
From this follows 
\begin{cor}
\label{cor:drW-connection-on-M} Let $\mathcal{M}\in\mathcal{\widehat{D}}_{W(X)}^{(0)}/p^{r}-\text{mod}_{\text{acc}}$.
Then there is an unique, continuous, integrable de Rham-Witt connection
on $\mathcal{\widehat{M}}$ defined via 
\[
\mathcal{\widehat{M}}\tilde{\to}(\widehat{\mathcal{D}}_{W(X)}^{(0)}/p^{r})_{c-\text{acc}}\widehat{\otimes}_{\widehat{\mathcal{D}}_{W(X)}/p^{r}}\mathcal{\widehat{M}}
\]
\[
\xrightarrow{\nabla}((\widehat{\mathcal{D}}_{W(X)}^{(0)}/p^{r})_{c-\text{acc}}\widehat{\otimes}_{\mathcal{O}_{W(X)}}W\Omega_{X}^{1}/p^{r})\widehat{\otimes}_{\widehat{\mathcal{D}}_{W(X)}/p^{r}}\mathcal{\widehat{M}}\tilde{\to}\mathcal{\widehat{M}}\widehat{\otimes}_{\mathcal{O}_{W(X)}}W\Omega_{X}^{1}/p^{r}
\]
where the second arrow is induced by the de Rham-Witt connection of
the previous corollary, and the right action of $\widehat{\mathcal{D}}_{W(X)}^{(0)}/p^{r}$
on $(\widehat{\mathcal{D}}_{W(X)}^{(0)}/p^{r})_{c-\text{acc}}\widehat{\otimes}_{\mathcal{O}_{W(X)}}W\Omega_{X}^{1}/p^{r}$
is via the right action of $\widehat{\mathcal{D}}_{W(X)}/p^{r}$ on
$(\widehat{\mathcal{D}}_{W(X)}/p^{r})_{c-\text{acc}}$. This connection
agrees with the ``obvious'' connection on $\mathcal{\widehat{M}}=\widehat{\Phi}^{*}(\mathcal{N})$
for any affine open with lift of Frobenius $\Phi$. 
\end{cor}

In order to show that the (relative) de Rham-Witt complex computes
the pushforward, we need to define the (relative) de-Rham-Witt resolution.
Let $\varphi:X\to Y$ be a smooth morphism of relative dimension $d$.
Via the natural quotient map $W\Omega_{X}^{1}\to W\Omega_{X/Y}^{1}$,
the de Rham-Witt connection on $(\widehat{\mathcal{D}}_{W(X)}/p^{r})_{c-\text{acc}}$
defines an integrable connection 
\[
\nabla:(\widehat{\mathcal{D}}_{W(X)}^{(0)}/p^{r})_{c-\text{acc}}\to(\widehat{\mathcal{D}}_{W(X)}^{(0)}/p^{r})_{c-\text{acc}}\widehat{\otimes}_{\mathcal{O}_{W(X)}}W\Omega_{X/Y}^{1}/p^{r}
\]
and so we have the associated relative de Rham-Witt complex 
\[
((\widehat{\mathcal{D}}_{W(X)}^{(0)}/p^{r})_{c-\text{acc}}\widehat{\otimes}_{\mathcal{O}_{W(X)}}W\Omega_{X/Y}^{\cdot}/p^{r},\nabla)
\]
concentrated in degrees $\{0,\dots,d\}$; the meaning of $\widehat{\otimes}$
is that, in degree $i$ we take the completion along the tensor product
filtration 
\[
V^{k}(\mathcal{O}_{W(X)}/p^{r})\cdot((\widehat{\mathcal{D}}_{W(X)}^{(0)}/p^{r})_{c-\text{acc}}\otimes G^{j}(W\Omega_{X/Y}^{i}/p^{r})
\]
If $X=\text{Spec}(A)$ admits local coordinates, then from the isomorphism
$(\widehat{\mathcal{D}}_{W(X)}^{(0)}/p^{r})_{c-\text{acc}}\tilde{=}\widehat{\Phi}^{*}\Phi^{!}\mathcal{D}_{\mathfrak{X}}^{(0)}/p^{r}$
one sees that 
\[
(\widehat{\mathcal{D}}_{W(X)}^{(0)}/p^{r})_{c-\text{acc}}\widehat{\otimes}_{\mathcal{O}_{W(X)}}W\Omega_{X/Y}^{i}/p^{r}\tilde{=}\Phi^{!}\mathcal{D}_{\mathfrak{X}}^{(0)}/p^{r}\widehat{\otimes}_{\mathcal{O}_{\mathfrak{X}}/p^{r}}W\Omega_{X/Y}^{i}/p^{r}
\]
where the completion on the right is along $\Phi^{!}\mathcal{D}_{\mathfrak{X}}^{(0)}/p^{r}\otimes_{\mathcal{O}_{\mathfrak{X}}/p^{r}}G^{j}(W\Omega_{X/Y}^{i}/p^{r})$. 

We also define the object $(\widehat{\mathcal{D}}_{W(Y)\leftarrow W(X)}^{(0)}/p^{r})_{\text{c-acc}}$
to be the completion of $(\widehat{\mathcal{D}}_{W(Y)\leftarrow W(X),\text{acc}}^{(0)}/p^{r})$
along the filtration induced by $\{W\varphi^{-1}(V^{i}(\mathcal{O}_{W(Y)}/p^{r})\}$.
We note that since $\varphi$ is smooth, the right $\mathcal{D}_{X}^{(0)}$-module
$\mathcal{D}_{Y\leftarrow X}^{(0)}$ is coherent. This implies (via
the description of \prettyref{prop:local-bimodule-for-push}) that
$(\widehat{\mathcal{D}}_{W(Y)\leftarrow W(X)}^{(0)}/p^{r})_{\text{c-acc}}$
is already complete along the filtration induced from $\{V^{i}(\mathcal{O}_{W(X)}/p^{r})\}$.
When $Y$ is a point, we have 
\[
(\widehat{\mathcal{D}}_{W(Y)\leftarrow W(X)}^{(0)}/p^{r})_{\text{c-acc}}=W\omega_{X}/p^{r}
\]

Then we have 
\begin{thm}
\label{thm:drW-resolution}There is a quasi-isomorphism 
\[
(\widehat{\mathcal{D}}_{W(X)}^{(0)}/p^{r})_{c-\text{acc}}\widehat{\otimes}_{\mathcal{O}_{W(X)/p^{r}}}W\Omega_{X/Y}^{\cdot}/p^{r}[-d]\tilde{\to}(\widehat{\mathcal{D}}_{W(Y)\leftarrow W(X)}^{(0)}/p^{r})_{\text{c-acc}}
\]
\end{thm}

Before proving this, let us recall some notation. Suppose $X=\text{Spec}(A)$
and we have fixed a lift $\mathcal{A}$ as well as a coordinatized
lift of Frobenius $\Phi$. Then in \prettyref{prop:Construction-of-Phi-!}
we constructed an isomorphism 
\[
\tilde{\text{Hom}}_{\mathcal{A}}(W(A),\mathcal{\widehat{D}}_{\mathcal{A}}^{(0)})\tilde{\to}\Phi^{!}(\mathcal{\widehat{D}}_{\mathcal{A}}^{(0)})
\]
where $\tilde{\text{Hom}}_{\mathcal{A}}(W(A),\mathcal{\widehat{D}}_{\mathcal{A}}^{(0)})$
is a suitable subset of $\text{Hom}_{\mathcal{A}}(W(A),\mathcal{\widehat{D}}_{\mathcal{A}}^{(0)})$.
This allows us to put a filtration on $\Phi^{!}(\mathcal{\widehat{D}}_{\mathcal{A}}^{(0)})$
via 
\[
F^{m}(\Phi^{!}(\mathcal{\widehat{D}}_{\mathcal{A}}^{(0)}))=\{\epsilon\in\Phi^{!}(\mathcal{\widehat{D}}_{\mathcal{A}}^{(0)})|\epsilon(W(A))\subset p^{m}\cdot\mathcal{\widehat{D}}_{\mathcal{A}}^{(m)}\}
\]
Note that $\Phi^{!}(\mathcal{\widehat{D}}_{\mathcal{A}}^{(0)})/F^{m}$
is annihilated by $V^{m}(W(A))$. Choosing coordinates on $A$, there
is also the filtration 
\[
\tilde{F}^{m}(\Phi^{!}(\mathcal{\widehat{D}}_{\mathcal{A}}^{(0)}))=\{\epsilon(p^{r}T^{I/p^{r}})=0|r\geq m\}
\]
and, exactly as in \prettyref{lem:two-filtrations}, one has that,
for each $r\geq1$, $\tilde{F}^{m}(\Phi^{!}(\mathcal{\widehat{D}}_{\mathcal{A}}^{(0)})/p^{r})$
and $F^{m}(\Phi^{!}(\mathcal{\widehat{D}}_{\mathcal{A}}^{(0)})/p^{r})$
intertwine one another. 
\begin{proof}
First, we'll construct a map
\[
(\widehat{\mathcal{D}}_{W(X)}^{(0)}/p^{r})_{c-\text{acc}}\widehat{\otimes}_{\mathcal{O}_{W(X)/p^{r}}}W\omega_{X/Y}/p^{r}\to(\widehat{\mathcal{D}}_{W(Y)\leftarrow W(X)}^{(0)}/p^{r})_{\text{c-acc}}
\]
and show that the resulting augmented complex is exact. This map is
constructed as follows: via the composition of morphisms there is
a map 
\[
\mathcal{E}nd_{W(k)}(W\omega_{X})\times\mathcal{H}om_{\mathcal{O}_{W(X)}}(\varphi^{*}(W\omega_{Y}),W\omega_{X})\to\mathcal{H}om_{W(k)}(\varphi^{*}(W\omega_{Y}),W\omega_{X})
\]
which, via the embeddings $W\omega_{X/Y}\to\mathcal{H}om_{W(k)}(\varphi^{-1}(W\omega_{Y}),W\omega_{X})$
and $\widehat{\mathcal{D}}_{W(X)}^{(0)}\to\mathcal{E}nd_{W(k)}(W\omega_{X})$
yields a map 
\[
\widehat{\mathcal{D}}_{W(X)}^{(0)}\otimes_{\mathcal{O}_{W(X)/p^{r}}}W\omega_{X/Y}\to\mathcal{H}om_{W(k)}(\varphi^{*}(W\omega_{Y}),W\omega_{X})
\]
and, after completion and accessibalization, one sees by looking in
local coordinates that the image of this map is contained in $\text{\ensuremath{\lim}}_{r}\widehat{\mathcal{D}}_{W(Y)\leftarrow W(X),\text{c-acc}}^{(0)}/p^{r}$
(when $Y$ is a point this map is just the action map coming from
the right action of $\widehat{\mathcal{D}}_{W(X)}^{(0)}$ on $W\omega_{X}$). 

Now, to show the required exactness it suffices to work after applying
$\otimes_{W_{r}(k)}^{L}k$. Working locally, we have an isomorphism
\[
(\widehat{\mathcal{D}}_{W(X)}^{(0)}/p^{r})_{c-\text{acc}}\widehat{\otimes}_{\mathcal{O}_{W(X)/p^{r}}}W\Omega_{X/Y}^{i}/p^{r}\tilde{=}\Phi^{!}(\widehat{\mathcal{D}}_{\mathfrak{X}}^{(0)}/p^{r})\widehat{\otimes}_{\mathcal{O}_{\mathfrak{X}}/p^{r}}W\Omega_{X/Y}^{i}/p^{r}
\]
where $\widehat{\otimes}$ denotes completion with respect to the
tensor product filtration 
\[
\{F^{s}(\Phi^{!}(\widehat{\mathcal{D}}_{\mathfrak{X}}^{(0)}/p^{r}))\otimes{}_{\mathcal{O}_{\mathfrak{X}}/p^{r}}G^{t}(W\Omega_{X/Y}^{i}/p^{r})\}_{s+t\geq m}
\]
As remarked above, it is equivalent to work with the filtration on
$\tilde{F}$ on $\Phi^{!}(\mathcal{\widehat{D}}_{\mathfrak{X}}^{(0)})/p^{r}$.
Since $(\Phi^{!}(\widehat{\mathcal{D}}_{\mathfrak{X}}^{(0)}/p^{r}))/\tilde{F}^{m}$
is locally free over $\mathcal{O}_{\mathfrak{X}}/p^{r}$, we deduce
that $(\widehat{\mathcal{D}}_{W(X)}^{(0)}/p^{r})_{c-\text{acc}}\widehat{\otimes}_{\mathcal{O}_{W(X)/p^{r}}}W\Omega_{X/Y}^{i}/p^{r}$
is flat over $W_{r}(k)$. Therefore it suffices to analyze the complex
\[
(\widehat{\mathcal{D}}_{W(X)}^{(0)}/p)_{c-\text{acc}}\widehat{\otimes}_{\mathcal{O}_{W(X)/p}}W\Omega_{X/Y}^{\cdot}/p
\]
Again working locally, we have the isomorphism 
\[
(\widehat{\mathcal{D}}_{W(X)}^{(0)}/p)_{c-\text{acc}}\widehat{\otimes}_{\mathcal{O}_{W(X)/p}}W\Omega_{X/Y}^{\cdot}/p
\]
\[
\tilde{\to}\Phi^{!}(\mathcal{D}_{X}^{(0)})\widehat{\otimes}_{\mathcal{O}_{X}}W\Omega_{X/Y}^{\cdot}/p
\]
where the $\widehat{\otimes}$ denotes completion with respect to
the filtration $G^{j}$ on $W\Omega_{X/Y}^{\cdot}/p$. Letting $\Phi_{m}^{!}(\mathcal{D}_{X}^{(0)})=\Phi^{!}(\mathcal{D}_{X}^{(0)})/F^{m}(\mathcal{O}_{W(X)}/p)$
and $W_{j}\Omega_{X/Y}^{\cdot}/p=W(\Omega_{X/Y}^{\cdot}/p)/G^{j}(W\Omega_{X/Y}^{\cdot}/p)$
we shall in fact show that, 
\[
\mathcal{H}^{i}(\Phi_{m}^{!}(\mathcal{D}_{X}^{(0)}))\widehat{\otimes}_{\mathcal{O}_{\mathfrak{X}/p}}W_{j}\Omega_{X/Y}^{\cdot}/p)=0
\]
for $i<d$ and 
\[
\mathcal{H}^{d}(\Phi_{m}^{!}(\mathcal{D}_{X}^{(0)}))\widehat{\otimes}_{\mathcal{O}_{\mathfrak{X}/p}}W_{j}\Omega_{X/Y}^{\cdot}/p)\tilde{\to}\Phi_{m}^{!}(W\varphi^{-1}(\mathcal{O}_{W_{j}(Y)}/p)\widehat{\otimes}_{\varphi^{-1}(\mathcal{O}_{Y})}\mathcal{D}_{Y\leftarrow X}^{(0)})
\]
where, in the right hand side, $\Phi_{m}^{!}$ refers to $\otimes_{\mathcal{D}_{X}^{(0)}}\Phi_{m}^{!}(\mathcal{D}_{X}^{(0)})$
via the right $\mathcal{D}_{X}^{(0)}$-module structure on $\mathcal{D}_{Y\leftarrow X}^{(0)}$,
and the completion is with respect to the $V$-adic completion of
$\mathcal{O}_{W(Y)}$. 

As each tower $\{\Phi_{m}^{!}(\mathcal{D}_{X}^{(0)}))\widehat{\otimes}_{\mathcal{O}_{\mathfrak{X}/p}}W_{j}\Omega_{X/Y}^{\cdot}/p\}_{m}$
is a tower of sheaves satisfying the Mittag-Leffler condition; each
sheaf in the tower is itself an inverse limit of quasicoherent sheaves;
and so we may take the inverse limit over $m$ (c.f. \cite{key-8},
lemma 1.1.6), followed by the inverse limit over $j$, so the result
then follows directly from \prettyref{prop:local-bimodule-for-push}. 

To prove the required isomorphisms, let $(\mathcal{M},\nabla)$ be
any integrable connection on $X$. Then, as proved above, there is
an induced de Rham-Witt connection on $\widehat{\Phi}^{*}\mathcal{M}$,
and therefore a de Rham-Witt complex whose terms are
\[
\widehat{\Phi}^{*}\mathcal{M}\widehat{\otimes}_{\mathcal{O}_{W(X)/p}}W\Omega_{X/Y}^{i}/p\tilde{\to}\mathcal{M}\widehat{\otimes}_{\mathcal{O}_{X}}W\Omega_{X/Y}^{i}/p
\]
(where the completion on the right is with respect to the standard
filtration on $W\Omega_{X/Y}^{\cdot}/p$; this is unnecessary if $\mathcal{M}$
is coherent over $\mathcal{O}_{X}$). We note that the quotient by
$\mathcal{M}\otimes_{\mathcal{O}_{X}}V(W\Omega_{X/Y}^{\cdot}/p)$
induces a map 
\[
\mathcal{M}\otimes_{\mathcal{O}_{X}}W\Omega_{X/Y}^{\cdot}/p\to\mathcal{M}\otimes_{\mathcal{O}_{X}}(\Omega_{X/Y}^{\cdot}\widehat{\otimes}_{\varphi^{-1}(\mathcal{O}_{Y})}W\varphi^{-1}(\mathcal{O}_{W(Y)}/p))
\]
If $\mathcal{M}$ is a vector bundle with nilpotent connection, this
map induces an isomorphism 
\[
\mathcal{H}^{i}(\mathcal{M}\otimes_{\mathcal{O}_{X}}W\Omega_{X/Y}^{\cdot}/p)\tilde{\to}\mathcal{H}^{i}(\mathcal{M}\otimes_{\mathcal{O}_{X}}\Omega_{X/Y}^{\cdot})\widehat{\otimes}_{\varphi^{-1}(\mathcal{O}_{Y})}W\varphi^{-1}(\mathcal{O}_{W(Y)}/p)
\]
(via Cartier descent, this follows formally from the case $\mathcal{O}_{X}$,
where it a basic computation, c.f., \cite{key-27}, section 3.2);
more precisely it is induced from 
\[
\mathcal{H}^{i}(\mathcal{M}\otimes_{\mathcal{O}_{X}}W_{j}\Omega_{X/Y}^{\cdot}/p)\tilde{\to}\mathcal{H}^{i}(\mathcal{M}\otimes_{\mathcal{O}_{X}}\Omega_{X/Y}^{\cdot})\otimes_{\varphi^{-1}(\mathcal{O}_{Y})}W\varphi^{-1}(\mathcal{O}_{W_{j}(Y)}/p)
\]
by taking the inverse limit over $j$. 

Now, the quotient of $\mathcal{D}_{X}^{(0)}$ along any power of the
central ideal $\mathcal{I}_{1}$ defining the nilpotence condition,
is such a vector bundle with nilpotent connection. Thus we obtain
for each $m,j,s$
\[
\mathcal{H}^{i}(\Phi_{m}^{!}(\mathcal{D}_{X}^{(0)}/\mathcal{I}^{s})\otimes_{\mathcal{O}_{X}}W_{j}\Omega_{X/Y}^{\cdot}/p)\tilde{\to}\mathcal{H}^{i}(\Phi_{m}^{!}(\mathcal{D}_{X}^{(0)}/\mathcal{I}^{s})\otimes_{\mathcal{O}_{X}}\Omega_{X/Y}^{\cdot})\otimes_{\varphi^{-1}(\mathcal{O}_{Y})}W\varphi^{-1}(\mathcal{O}_{W_{j}(Y)}/p)
\]
and, as everything in question is a finite right $\mathcal{D}_{X}^{(0)}$-module
we may take the inverse limit over $s$ and obtain 
\[
\mathcal{H}^{i}(\Phi_{m}^{!}(\mathcal{D}_{X,\text{crys}}^{(0)})\otimes_{\mathcal{O}_{X}}W_{j}\Omega_{X/Y}^{\cdot}/p)\tilde{\to}\mathcal{H}^{i}(\Phi_{m}^{!}(\mathcal{D}_{X,\text{crys}}^{(0)})\otimes_{\mathcal{O}_{X}}\Omega_{X/Y}^{\cdot})\otimes_{\varphi^{-1}(\mathcal{O}_{Y})}W\varphi^{-1}(\mathcal{O}_{W_{j}(Y)}/p)
\]
this vanishes for $i<d$ and is isomorphic to $\Phi_{m}^{!}(W\varphi^{-1}(\mathcal{O}_{W_{j}(Y)}/p)\otimes_{\varphi^{-1}(\mathcal{O}_{Y})}\mathcal{D}_{Y\leftarrow X,\text{crys}}^{(0)})$
when $i=d$. This shows that both the kernel and the cokernel of 
\[
\mathcal{H}^{i}(\Phi_{m}^{!}\mathcal{D}_{X}^{(0)}\otimes_{\mathcal{O}_{X}}W_{j}\Omega_{X/Y}^{\cdot}/p)\tilde{\to}W\varphi^{-1}(\mathcal{O}_{W_{j}(Y)}/p)\otimes_{\varphi^{-1}(\mathcal{O}_{Y})}\mathcal{H}^{i}(\Phi_{m}^{!}\mathcal{D}_{X}^{(0)}\otimes_{\mathcal{O}_{X}}\Omega_{X/Y}^{\cdot})
\]
are supported, as modules over $\mathcal{Z}(\mathcal{D}_{X}^{(0)})\tilde{=}\mathcal{O}_{T^{*}X^{(1)}}$,
away from the zero section. 

To show that the kernel and cokernel are actually $0$, we argue as
follows: passing to the algebraic closure of $k$, we consider a closed
point $x\in T^{*}X^{(1)}$. As $X$ is affine and has local coordinates,
we have that $x$ is contained in the subscheme defined by the ideal
$(\partial_{1}-a_{1}^{p},\partial_{2}-a_{2}^{p},\dots,\partial_{n}-a_{n}^{p})$
for some $a_{i}\in k$. The algebra of global sections $\Gamma(\mathcal{D}_{X}^{(0)})$
possesses an automorphism $\chi$ which preserves $\Gamma(\mathcal{O}_{X})$
and sends $\partial_{i}$ to $\partial_{i}-a_{i}$ (this is clear
from the defining relations on $\mathcal{D}_{X}^{(0)}$). This automorphism
preserves $\mathcal{Z}(\mathcal{D}_{X}^{(0)})$, and the associated
action on $T^{*}X^{(1)}$ interchanges the zero section $X^{(1)}$
with the subscheme defined by $(\partial_{1}-a_{1}^{p},\partial_{2}-a_{2}^{p},\dots,\partial_{n}-a_{n}^{p})$. 

Now, we may repeat the above argument for $(\mathcal{D}_{X}^{(0)})^{\chi}$,
the $\mathcal{D}_{X}^{(0)}$-module whose left and right module structure
are twisted by the action of $\chi$. This shows that the kernel and
cokernel of 
\[
\mathcal{H}^{i}(\Phi_{m}^{!}\mathcal{D}_{X}^{(0)}\otimes_{\mathcal{O}_{X}}W_{j}\Omega_{X/Y}^{\cdot}/p)\tilde{\to}W\varphi^{-1}(\mathcal{O}_{W_{j}(Y)}/p)\otimes_{\varphi^{-1}(\mathcal{O}_{Y})}\mathcal{H}^{i}(\Phi_{m}^{!}\mathcal{D}_{X}^{(0)}\otimes_{\mathcal{O}_{X}}\Omega_{X/Y}^{\cdot})
\]
are supported away from the subscheme defined by the ideal $(\partial_{1}-a_{1}^{p},\partial_{2}-a_{2}^{p},\dots,\partial_{n}-a_{n}^{p})$.
As this is true for all $(a_{1},\dots,a_{n})$ in $k$, we see that
in fact the support of the kernel and cokernel are empty, as required.
\end{proof}
Now we want to compare the (relative) de Rham-Witt cohomology with
the relative pushforward constructed above. To do so, we need the
following basic fact:
\begin{thm}
\label{thm:complete-iso}Fix $r\geq1$ and let $\mathcal{M}^{\cdot}$
be a bounded accessible complex over $\widehat{\mathcal{D}}_{W(X)}^{(0)}/p^{r}$;
and let $\widehat{\mathcal{M}}^{\cdot}$ be its derived completion
(as discussed above \prettyref{prop:completion}). Consider the functor
\[
\widehat{\int}_{W\varphi}\mathcal{M}^{\cdot}:=R(W\varphi)_{*}(\widehat{\mathcal{D}}_{W(Y)\leftarrow W(X),\text{\ensuremath{c-}acc}}^{(0)}/p^{r}\widehat{\otimes}_{\widehat{\mathcal{D}}_{W(X)}^{(0)}/p^{r}}^{L}\mathcal{\widehat{M}}^{\cdot})
\]
where on the right we define
\[
\widehat{\mathcal{D}}_{W(Y)\leftarrow W(X),\text{\ensuremath{c-}acc}}^{(0)}/p^{r}\widehat{\otimes}_{\widehat{\mathcal{D}}_{W(X)}^{(0)}/p^{r}}^{L}\mathcal{\widehat{M}}^{\cdot}
\]
\[
:=\text{holim}_{i}(W\varphi^{-1}((\mathcal{O}_{W(Y)}/p^{r})/V^{i})\otimes_{W\varphi^{-1}((\mathcal{O}_{W(Y)}/p^{r})}^{L}\widehat{\mathcal{D}}_{W(Y)\leftarrow W(X),\text{\ensuremath{c-}acc}}^{(0)}/p^{r}\otimes_{\widehat{\mathcal{D}}_{W(X)}^{(0)}/p^{r}}^{L}\mathcal{\widehat{M}}^{\cdot})
\]
Then there is a canonical isomorphism 
\[
\text{holim}_{i}((\mathcal{O}_{W(Y)}/p^{r})/V^{i})\otimes_{\mathcal{O}_{W(Y)}/p^{r}}^{L}\int_{W\varphi}\mathcal{M}^{\cdot}\tilde{\to}\widehat{\int}_{W\varphi}\mathcal{M}^{\cdot}
\]
\end{thm}

\begin{proof}
There is a canonical map 
\[
\widehat{\mathcal{D}}_{W(Y)\leftarrow W(X),\text{acc}}^{(0)}/p^{r}\otimes_{\widehat{\mathcal{D}}_{W(X)}^{(0)}/p^{r}}^{L}\mathcal{M}^{\cdot}\to\widehat{\mathcal{D}}_{W(Y)\leftarrow W(X),\text{\ensuremath{c-}acc}}^{(0)}/p^{r}\widehat{\otimes}_{\widehat{\mathcal{D}}_{W(X)}^{(0)}/p^{r}}^{L}\mathcal{\widehat{M}}^{\cdot}
\]
which induces a map 
\[
R(W\varphi)_{*}(\widehat{\mathcal{D}}_{W(Y)\leftarrow W(X),\text{acc}}^{(0)}/p^{r}\otimes_{\widehat{\mathcal{D}}_{W(X)}^{(0)}/p^{r}}^{L}\mathcal{M}^{\cdot})\to\widehat{\int}_{W\varphi}\mathcal{M}^{\cdot}
\]
Work locally on $Y$; choose a coordinatized lift of Frobenius and
an associated $\Phi$. We will show that the complex ${\displaystyle \widehat{\int}_{W\varphi}\mathcal{M}^{\cdot}}$
is a bounded complex which is quasi-isomorphic to a complex of the
form $\widehat{\Phi}\mathcal{K}^{\cdot}$. This shows that the map
above factors through $\int_{W\varphi}\mathcal{M}^{\cdot}$, we then
show that in fact $\Phi^{*}\mathcal{K}^{\cdot}\tilde{\to}{\displaystyle \int_{W\varphi}\mathcal{M}^{\cdot}}$;
this immediately implies the result. 

Begin by assuming that $X$ is also affine. As $\varphi$ is smooth
it is locally compatible with a lift of Frobenius, which we also call
$\Phi$ and so by \prettyref{prop:local-bimodule-for-push} we have
\[
\widehat{\mathcal{D}}_{W(Y)\leftarrow W(X),\text{acc}}^{(m)}/p^{r}\tilde{=}\varphi^{-1}(\Phi^{*}\mathcal{D}_{\mathfrak{Y_{r}}}^{(0)})\widehat{\otimes}_{\varphi^{-1}(\mathcal{D}_{\mathfrak{Y_{r}}}^{(0)})}^{L}\mathcal{D}_{\mathfrak{Y_{r}}\leftarrow\mathfrak{X}_{r}}^{(0)}\otimes_{\mathcal{D}_{\mathfrak{X}_{r}}^{(0)}}^{L}\Phi^{!}\mathcal{D}_{\mathfrak{X}_{r}}^{(0)}
\]
where the $\widehat{\otimes}$ on the left indicates completion with
respect to $W\varphi^{-1}(V^{i}(\mathcal{O}_{W(X)}/p^{r}))$. Writing
$\mathcal{\widehat{M}}^{\cdot}=\widehat{\Phi}^{*}\mathcal{N}^{\cdot}$
we see that 
\[
\widehat{\mathcal{D}}_{W(Y)\leftarrow W(X),\text{\ensuremath{c-}acc}}^{(0)}/p^{r}\widehat{\otimes}_{\widehat{\mathcal{D}}_{W(X)}^{(0)}/p^{r}}^{L}\mathcal{\widehat{M}}^{\cdot}
\]
\[
\tilde{=}\varphi^{-1}(\Phi^{*}\mathcal{D}_{\mathfrak{Y_{r}}}^{(0)})\widehat{\otimes}_{\varphi^{-1}(\mathcal{D}_{\mathfrak{Y_{r}}}^{(0)})}^{L}\mathcal{D}_{\mathfrak{Y_{r}}\leftarrow\mathfrak{X}_{r}}^{(0)}\otimes_{\mathcal{D}_{\mathfrak{X}_{r}}^{(0)}}^{L}\mathcal{N}^{\cdot}
\]
Now, $\mathcal{D}_{\mathfrak{Y_{r}}\leftarrow\mathfrak{X}_{r}}^{(0)}\otimes_{\mathcal{D}_{\mathfrak{X}_{r}}^{(0)}}^{L}\mathcal{N}^{\cdot}$
is a bounded complex of $\varphi^{-1}(\mathcal{D}_{\mathfrak{Y_{r}}}^{(0)})$-modules;
if we represent it by a bounded complex $\mathcal{F}^{i}$, then the
$i$th term of the above is given by 
\[
\varphi^{-1}(\Phi^{*}\mathcal{D}_{\mathfrak{Y_{r}}}^{(0)})\widehat{\otimes}_{\varphi^{-1}(\mathcal{D}_{\mathfrak{Y_{r}}}^{(0)})}\mathcal{F}^{i}
\]
which is simply an infinite product of copies of $\mathcal{F}^{i}$.
So we see easily that 
\[
R\varphi_{*}(\varphi^{-1}(\Phi^{*}\mathcal{D}_{\mathfrak{Y_{r}}}^{(0)})\widehat{\otimes}_{\varphi^{-1}(\mathcal{D}_{\mathfrak{Y_{r}}}^{(0)})}\mathcal{F}^{i})\tilde{\to}\Phi^{*}\mathcal{D}_{\mathfrak{Y_{r}}}^{(0)}\widehat{\otimes}_{\mathcal{D}_{\mathfrak{Y_{r}}}^{(0)}}R\varphi_{*}(\mathcal{F}^{i})
\]
and the analogous result follows easily for a bounded complex (by
induction on the cohomological length), so that we have
\[
R\varphi_{*}(\varphi^{-1}(\Phi^{*}\mathcal{D}_{\mathfrak{Y_{r}}}^{(0)})\widehat{\otimes}_{\varphi^{-1}(\mathcal{D}_{\mathfrak{Y_{r}}}^{(0)})}^{L}\mathcal{D}_{\mathfrak{Y_{r}}\leftarrow\mathfrak{X}_{r}}^{(0)}\otimes_{\mathcal{D}_{\mathfrak{X}_{r}}^{(0)}}^{L}\mathcal{N}^{\cdot})\tilde{\to}\Phi^{*}\mathcal{D}_{\mathfrak{Y_{r}}}^{(0)}\widehat{\otimes}_{\mathcal{D}_{\mathfrak{Y_{r}}}^{(0)}}^{L}R\varphi_{*}(\mathcal{D}_{\mathfrak{Y_{r}}\leftarrow\mathfrak{X}_{r}}^{(0)}\otimes_{\mathcal{D}_{\mathfrak{X}_{r}}^{(0)}}^{L}\mathcal{N}^{\cdot})
\]
This shows that both ${\displaystyle \widehat{\int}_{W\varphi}\mathcal{M}^{\cdot}}$
is quasi-isomorphic to a complex of the form $\widehat{\Phi}\mathcal{K}^{\cdot}$
and that $\Phi^{*}\mathcal{K}^{\cdot}\tilde{\to}{\displaystyle \int_{W\varphi}\mathcal{M}^{\cdot}}$,
which is what we wanted. 

To get the result for a general $X$, simply cover it with affines;
and use the fact that if $X=U\cup V$ then there is a distinguished
triangle 
\[
\mathcal{M}^{\cdot}\to(j_{U})_{*}\mathcal{M}^{\cdot}\oplus(j_{V})_{*}\mathcal{M}^{\cdot}\to(j_{U\cap V})_{*}\mathcal{M}^{\cdot}
\]
where $j_{U},j_{V},$$j_{U\cap V}$ denote the inclusions from those
open subsets, respectively; this allows one to do induction on the
number of open affines and deduce the result. 
\end{proof}
Putting it all together, we conclude:
\begin{cor}
\label{cor:Completed-push-is-DRW}Let $\varphi:X\to Y$ be smooth,
and let $\mathcal{M}\in\mathcal{D}_{W(X)}^{(0)}/p^{r}-\text{mod}_{\text{acc,qcoh}}$.
Let $\mathcal{\widehat{M}}\widehat{\otimes}_{\mathcal{O}_{W(X)}/p}W\Omega_{X/Y}^{\cdot}$
be the associated de Rham-Witt complex. There is an isomorphism of
sheaves on $W(Y)_{p^{r}=0}$
\[
\widehat{\int}_{W\varphi}\mathcal{M}[d]\tilde{\to}RW\varphi_{*}(\mathcal{\widehat{M}}\widehat{\otimes}_{\mathcal{O}_{W(X)}/p}W\Omega_{X/Y}^{\cdot})
\]
In particular, if $\mathcal{M}$ is nilpotent and quasicoherent, and
$\mathcal{\widehat{M}}=\widehat{\epsilon}(\mathcal{N})$ for a quasicoherent
crystal $\mathcal{N}$, then for each $i\geq0$ we have
\[
\widehat{\epsilon}(R^{i}\varphi_{*,\text{crys}}\mathcal{N})\tilde{\to}R^{i}W\varphi_{*}(\mathcal{M}\widehat{\otimes}_{\mathcal{O}_{W(X)}/p}W\Omega_{X/Y}^{\cdot})
\]
\end{cor}

\begin{proof}
The second sentence follows from the first by \prettyref{cor:Crystalline-Push}.
To prove the first, we shall invoke the previous two results. First,
by \prettyref{thm:drW-resolution} there is a quasi-isomorphism 
\[
\widehat{\mathcal{D}}_{W(Y)\leftarrow W(X),\text{\ensuremath{c-}acc}}^{(0)}/p^{r}\otimes_{\widehat{\mathcal{D}}_{W(X)}^{(0)}/p^{r}}^{L}\mathcal{\widehat{M}}\tilde{\to}((\widehat{\mathcal{D}}_{W(X)}^{(0)}/p^{r})_{c-\text{acc}}\widehat{\otimes}_{\mathcal{O}_{W(X)/p^{r}}}W\Omega_{X/Y}^{\cdot}/p^{r})\otimes_{\widehat{\mathcal{D}}_{W(X)}^{(0)}/p^{r}}^{L}\mathcal{\widehat{M}}
\]
Now, the lemma below ensures that, for each $i$, $(\widehat{\mathcal{D}}_{W(X)}^{(0)}/p^{r})_{c-\text{acc}}\widehat{\otimes}_{\mathcal{O}_{W(X)/p^{r}}}W\Omega_{X/Y}^{i}/p^{r}$
is flat over $\widehat{\mathcal{D}}_{W(X)}^{(0)}/p^{r}$ (in fact,
locally, it is an inverse limit of a surjective system of projective
modules of the form $\Phi^{*}\mathcal{D}_{\mathfrak{X}_{r}}^{(0)}$).
So on the right hand side we actually have the complex whose terms
are 
\[
(\widehat{\mathcal{D}}_{W(X)}^{(0)}/p^{r})_{c-\text{acc}}\widehat{\otimes}_{\mathcal{O}_{W(X)/p^{r}}}W\Omega_{X/Y}^{i}/p^{r}\otimes_{\widehat{\mathcal{D}}_{W(X)}^{(0)}/p^{r}}\mathcal{\widehat{M}}
\]
which we can complete (with respect to the natural filtration on the
tensor product) and thus get a map to the complex whose terms are
\[
((\widehat{\mathcal{D}}_{W(X)}^{(0)}/p^{r})_{c-\text{acc}}\widehat{\otimes}_{\mathcal{O}_{W(X)/p^{r}}}W\Omega_{X/Y}^{i}/p^{r})\widehat{\otimes}_{\widehat{\mathcal{D}}_{W(X)}^{(0)}/p^{r}}\mathcal{\widehat{M}}\tilde{\to}W\Omega_{X/Y}^{i}/p^{r}\widehat{\otimes}_{\mathcal{O}_{W(X)/p^{r}}}\mathcal{\widehat{\mathcal{M}}}
\]
Here, the last isomorphism uses the accessibility of $\mathcal{M}$. 

Summing up, we've obtained a map 
\[
\widehat{\mathcal{D}}_{W(Y)\leftarrow W(X),\text{\ensuremath{c-}acc}}^{(0)}/p^{r}\otimes_{\widehat{\mathcal{D}}_{W(X)}^{(0)}/p^{r}}^{L}\mathcal{\widehat{M}}\to W\Omega_{X/Y}^{\cdot}/p^{r}\widehat{\otimes}_{\mathcal{O}_{W(X)/p^{r}}}\mathcal{\widehat{\mathcal{M}}}
\]
To apply the previous result, we have to take the derived completion
with respect to $\{W\varphi^{-1}V^{i}(\mathcal{O}_{W(Y)}/p^{r})\}$
(i.e., applying $\text{holim}$ to $W\varphi^{-1}((\mathcal{O}_{W(Y)}/p^{r})/V^{i})\otimes_{W\varphi^{-1}((\mathcal{O}_{W(Y)}/p^{r})}^{L}$).
However, it is not difficult to see that $W\Omega_{X/Y}^{\cdot}/p^{r}\widehat{\otimes}_{\mathcal{O}_{W(X)/p^{r}}}\mathcal{\widehat{\mathcal{M}}}$
is already derived complete with respect to $\{W\varphi^{-1}V^{i}(\mathcal{O}_{W(Y)}/p^{r})\}$
(for instance, one may work locally and relate the $\{V^{i}(\mathcal{O}_{W(Y)}/p^{r})\}$
to the filtration $\{F^{i}(\mathcal{O}_{W(Y)}/p^{r})\}$ as in the
proof of \prettyref{prop:completion}). Thus we actually obtain a
map 
\begin{equation}
\widehat{\mathcal{D}}_{W(Y)\leftarrow W(X),\text{\ensuremath{c-}acc}}^{(0)}/p^{r}\widehat{\otimes}_{\widehat{\mathcal{D}}_{W(X)}^{(0)}/p^{r}}^{L}\mathcal{\widehat{M}}\to W\Omega_{X/Y}^{\cdot}/p^{r}\widehat{\otimes}_{\mathcal{O}_{W(X)/p^{r}}}\mathcal{\widehat{\mathcal{M}}}\label{eq:key-iso}
\end{equation}
The result follows if we can show that it is a quasi-isomorphism.
Working locally, we can assume $\mathcal{M}=\Phi^{*}\mathcal{N}$.
We begin in the case $r=1$. In that case, we remark that if $\mathcal{M}^{\cdot}$
is any complex of accessible quasicoherent modules, by applying the
above construction to each $\mathcal{M}^{i}$ and taking the total
complex, we obtain a morphism 
\begin{equation}
\widehat{\mathcal{D}}_{W(Y)\leftarrow W(X),\text{\ensuremath{c-}acc}}^{(0)}/p\widehat{\otimes}_{\widehat{\mathcal{D}}_{W(X)}^{(0)}/p}^{L}\mathcal{\widehat{M}}^{\cdot}\to W\Omega_{X/Y}^{\cdot}/p\widehat{\otimes}_{\mathcal{O}_{W(X)/p}}^{L}\mathcal{\widehat{\mathcal{M}}}^{\cdot}\label{eq:key-iso-complexes}
\end{equation}
(this works as $\mathcal{M}$ is accessible and $W\Omega_{X/Y}^{\cdot}/p$
is flat over $\mathcal{O}_{X}$ by the lemma below) and so for each
$j$ we obtain an induced map 
\[
\widehat{\mathcal{D}}_{W(Y)\leftarrow W(X),\text{\ensuremath{c-}acc}}^{(0)}/p\widehat{\otimes}_{\widehat{\mathcal{D}}_{W(X)}^{(0)}/p}^{L}\mathcal{\widehat{M}}^{\cdot}\to W_{j}\Omega_{X/Y}^{\cdot}/p\widehat{\otimes}_{\mathcal{O}_{W_{j}(X)/p}}^{L}\mathcal{M}_{j}^{\cdot}
\]
(where $\mathcal{M}_{j}^{\cdot}=\mathcal{M}^{\cdot}\otimes_{\mathcal{O}_{W(X)}/p}^{L}(\mathcal{O}_{W(X)}/p/V^{j})$.
Consider the set $\mathcal{S}$ of complexes $\mathcal{M}^{\cdot}$
for which this map is a quasi-isomorphism for all $j$. By the previous
theorem we have 
\[
\mathcal{H}^{i}(\widehat{\mathcal{D}}_{W(Y)\leftarrow W(X),\text{\ensuremath{c-}acc}}^{(0)}/p\widehat{\otimes}_{\widehat{\mathcal{D}}_{W(X)}^{(0)}/p}^{L}\mathcal{\widehat{M}}^{\cdot})\tilde{\to}W\varphi^{-1}(\mathcal{O}_{W(Y)}/p)\widehat{\otimes}_{\varphi^{-1}(\mathcal{O}_{Y})}\mathcal{H}^{i}(\Omega_{X/Y}^{\cdot}\otimes\mathcal{N},\nabla)
\]
where $\mathcal{O}_{Y}\to\mathcal{O}_{W(Y)}/p$ is the map coming
from a given lift of Frobenius, and the completion is with respect
to $V^{i}(\mathcal{O}_{W(Y)}/p)$. 

So, by \prettyref{thm:drW-resolution} $\mathcal{M}=\Phi^{*}\mathcal{D}_{X}^{(0)}\in\mathcal{S}$.
It is also clear that $\mathcal{S}$ is closed under cones, arbitrary
sums, and summands. But since it contains $\Phi^{*}\mathcal{D}_{X}^{(0)}$,
it must therefore\footnote{By the fact that $\Phi^{*}\mathcal{D}_{X}^{(0)}$, along with all
of its shifts, are set of compact generators for $D_{\text{acc,qcoh}}(\widehat{\mathcal{D}}_{W(Y)}^{(0)}/p^{r}-\text{mod})$} contain any quasi-coherent accessible complex; taking the complex
to be in a single degree proves the result when $r=1$. It then follows,
by induction on the cohomological length, that \prettyref{eq:key-iso-complexes}
is an isomorphism for bounded complexes. 

Now let $r\geq1$. To prove that \prettyref{eq:key-iso} is an isomorphism,
we regard both sides as complexes of $W(k)$-modules and apply $\otimes_{W(k)}^{L}k$
to both sides. By the argument of \prettyref{rem:Pushforwards-agree}
(where the role of $\widehat{\mathcal{D}}_{W(Y)\leftarrow W(X)}^{(m)}$
is played by ${\displaystyle \lim_{r}\widehat{\mathcal{D}}_{W(Y)\leftarrow W(X),\text{\ensuremath{c-}acc}}^{(0)}/p^{r}}$),
we have 
\[
(\widehat{\mathcal{D}}_{W(Y)\leftarrow W(X),\text{\ensuremath{c-}acc}}^{(0)}/p^{r}\widehat{\otimes}_{\widehat{\mathcal{D}}_{W(X)}^{(0)}/p^{r}}^{L}\mathcal{\widehat{M}})\otimes_{W(k)}^{L}k\tilde{\to}\widehat{\mathcal{D}}_{W(Y)\leftarrow W(X),\text{\ensuremath{c-}acc}}^{(0)}/p\widehat{\otimes}_{\widehat{\mathcal{D}}_{W(X)}^{(0)}/p}^{L}(\mathcal{\widehat{M}}\otimes_{W(k)}^{L}k)
\]
(here, we are using that the functor $\otimes_{W(k)}^{L}k$ commutes
with the derived completion functor; this follows directly from the
fact that $k\tilde{\to}W(k)\xrightarrow{p}W(k)$) and similarly for
$W\Omega_{X/Y}^{\cdot}/p^{r}\widehat{\otimes}_{\mathcal{O}_{W(X)/p^{r}}}\mathcal{\widehat{\mathcal{M}}}$.
As the complex $(\mathcal{\widehat{M}}\otimes_{W(k)}^{L}k)$ is bounded,
the result follows from \prettyref{eq:key-iso-complexes} for bounded
complexes.
\end{proof}
In the proof, we needed the
\begin{lem}
\label{lem:Flatness-of-Omega-i}Suppose $A$ and $B$ are smooth $k$-algebras
so that $B\to A$ is smooth. Let $A$ posses local coordinates, and
let $\Phi:\mathcal{A}\to W(A)$. For each $i\geq0$, the module $W\Omega_{A/B}^{i}$
is (faithfully) flat over $\mathcal{A}$. Further, for each $j\geq0$
the module $W_{j}\Omega_{A/B}^{i}/p$ is finite flat over $A$. 
\end{lem}

\begin{proof}
It is enough to show that $W\Omega_{A/B}^{i}/p$ is faithfully flat
over $A$, where $A$ acts via the embedding $\Phi:A\to W(A)/p$;
and therefore it is enough to show that second statement of the lemma.
As $W\Omega_{A/B}^{i}/p$ is local for the etale topology, we can
assume that $A=B[T_{1},\dots T_{n}]$; and, by extending the ground
field, that $k$ is algebraically closed. 

The module $W\Omega_{A/B}^{i}/p$ is, by definition, the inverse limit
of $W_{j}\Omega_{A/B}^{i}/p$, the quotient of $W\Omega_{A/B}^{i}/p$
by the module $V^{j}(W\Omega_{A/B}^{i}/p):=V^{j}(W\Omega_{A/B}^{i})/p$.
Each $\text{gr}_{j}(W\Omega_{A}^{i}/p)=V^{j}(W\Omega_{A/B}^{i}/p)/V^{j+1}(W\Omega_{A/B}^{i}/p)$
is a module over 
\[
(W(A)/p)/V^{j+1}(W(A)/p)=W_{j+1}(A)/p
\]
 In general this action does not factor through the quotient $W_{j+1}(A)/p/V(W_{j+1}(A)/p)=A$.
We therefore refine this filtration by defining 
\[
\text{gr}_{j,t}(W\Omega_{A}^{i}/p)=V^{t}(W(A)/p)\cdot\text{gr}_{j}(W\Omega_{A}^{i}/p)
\]
for $0\leq t\leq j+1$. Then the action of $W_{j+1}(A)/p$ on each
$\text{gr}_{j,t}(W\Omega_{A}^{i}/p)/\text{gr}_{j,t+1}(W\Omega_{A}^{i}/p)$
factors through $A$; and this action of $A$ agrees with the action
of $A$ on \\
$\text{gr}_{j,t}(W\Omega_{A}^{i}/p)/\text{gr}_{j,t+1}(W\Omega_{A}^{i}/p)$
defined via $\Phi_{1}$. Thus $(W\Omega_{A/B}^{i}/p)/V^{j}(W\Omega_{A/B}^{i}/p)$
is filtered by finitely many $A$-modules $\text{gr}_{j,t}(W\Omega_{A/B}^{i}/p)/\text{gr}_{j,t+1}(W\Omega_{A/B}^{i}/p)$
and so it suffices to show each one is projective over $A$. 

To do that, note that any automorphism $\sigma$ of $A$ which preserves
$B$ yields an isomorphism of $A$-modules
\[
\sigma^{*}\text{gr}_{j,t}(W\Omega_{A/B}^{i}/p)/\text{gr}_{j,t+1}(W\Omega_{A/B}^{i}/p)\tilde{\to}\text{gr}_{j,t}(W\Omega_{A/B}^{i}/p)/\text{gr}_{j,t+1}(W\Omega_{A/B}^{i}/p)
\]
 Indeed, let $\sigma:W_{n+1}(A)/p\to W_{n+1}(A)/p$ denote the induced
isomorphism. Then we have $\sigma^{*}\text{gr}_{j}(W\Omega_{A/B}^{i}/p)\tilde{\to}\text{gr}_{j}(W\Omega_{A/B}^{i}/p)$
by the functoriality of the de Rham Witt construction. The result
then follows for $\text{gr}_{j,t}(W\Omega_{A}^{i}/p)$ from the fact
that the automorphism $\sigma$ preserves $V^{t}(W(A)/p)$. As the
action of automorphisms of $A=B[T_{1},\dots T_{n}]$ is transitive
on closed points in each fibre over $\text{Spec}(B)$, we see that
the finite $A$-module $\text{gr}_{j,t}(W\Omega_{A/B}^{i}/p)/\text{gr}_{j,t+1}(W\Omega_{A/B}^{i}/p)$
is necessarily projective (the geometric fibres all have the same
rank) and so the result follows.  
\end{proof}

\section{The algebra $\widehat{\mathcal{U}}$, and applications}

In this chapter we give an alternate construction of accessibility
on the scheme $W(X)/p$ via an algebra $\widehat{\mathcal{U}}$, which
is a kind of enveloping algebra for derivations of $\mathcal{O}_{W(X)}/p$.
This allows us to prove \prettyref{prop:Properties-of-B_W}in all
positive characteristics, in particular avoiding \prettyref{thm:Uniqueness-of-bimodule}.
As mentioned above, the key to the proof is the fact that $W(A)/p\to A$
is a square zero extension; the argument is a really as rephrasing
of Grothendieck's fundamental result which says that a flat connection
extends uniquely over a square zero extension (c.f. \cite{key-10},
the introduction to chapter 2 for a nice discussion of this). Once
this is done, we turn to a deeper study of the relationship between
de Rham-Witt connections and accessible $\mathcal{D}_{W(X)}^{(0)}$-modules.
Working mod $p$ we will use $\widehat{\mathcal{U}}$ to show that
there is (essentially) an equivalence between the two. Then we will
lift this mod $p^{r}$ and in particular prove \prettyref{thm:Accessible-to-DRW}.

\subsection{The algebra $\widehat{\mathcal{U}}$. }

Until further notice $A$ is a smooth $k$-algebra which possesses
local coordinates. We begin by constructing an algebra of differential
operators over $W(A)/p$, which is closely related to $\mathcal{\widehat{D}}_{W(A)}^{(0)}/p$,
and which we will use to analyze the pull-back from $\mathcal{D}_{A}^{(0)}$-modules.
Let us set some notation:
\begin{defn}
For each $j\geq1$ let $\mathcal{I}_{j}=\text{ker}(W(A)/p\to W_{j}(A)/p)=V^{j}(W(A)/p)$;
and set $\mathcal{I}_{0}=W(A)/p$. Define $\mathcal{T}'_{W(X)/p}:=\{\partial\in\text{Der}_{\text{cont}}(W(A)/p)|\partial(\mathcal{I}_{j})\subset\mathcal{I}_{j}\phantom{i}\text{for all}\phantom{i}j\}$. 
\end{defn}

Clearly $\mathcal{T}'_{W(A)/p}$ is a Lie algebras under the natural
bracket of derivations. However, we have more:
\begin{lem}
\label{lem:composition-is-a-derivation}Let $\mathcal{T}'_{W(A)/p}(1)$
be the sub-sheaf of Lie-algebras of $\mathcal{T}'_{W(A)/p}$ consisting
of derivations whose image lies in $\mathcal{I}_{1}$. 

1) There is a natural map $\mathcal{T}'_{W(A)/p}\to\mathcal{T}_{A}$.
The kernel of this map is $\mathcal{T}'_{W(A)/p}(1)$. 

2) For $\partial_{1},\partial_{2}\in\mathcal{T}'_{W(A)/p}(1)$, the
composition $\partial_{1}\circ\partial_{2}\in\mathcal{T}'_{W(A)/p}(1)$
as well. Thus $\mathcal{T}'_{W(A)/p}(1)$ carries the structure of
a sheaf of (non-unital) algebras. 
\end{lem}

\begin{proof}
1) Since any $\delta\in\mathcal{T}'_{W(A)/p}$ takes $\mathcal{I}_{1}$
to itself, it induces a derivation $\bar{\delta}:(W(A)/p)/\mathcal{I}_{1}\to(W(A)/p)/\mathcal{I}_{1}$.
But since $(W(A)/p)/\mathcal{I}_{1}=A$ we obtain the result. 

2) Since $\mathcal{I}_{1}$ is a square zero ideal, if $\partial_{i}\in\mathcal{T}'_{W(A)/p}(1)$
for $i=1,2$, then $\partial_{1}(a)\partial_{2}(b)=0$ for all $a$
and $b$. Therefore, we have 
\[
\partial_{1}\partial_{2}(ab)=\partial_{1}(b\partial_{2}(a)+a\partial_{2}(b))=b\partial_{1}\partial_{2}(a)+a\partial_{1}\partial_{2}(b)
\]
as claimed. 
\end{proof}
Now we recall the definition of a Lie-Rinehart enveloping algebra.
Let $S$ be an algebra over a commutative ring $R$. Denote by $\mathcal{T}_{S}$
the module of $R$-derivations on $S$, and suppose that $\mathfrak{a}$is
a Lie algebra which is also an $S$-module, equipped with a morphism
$\rho:\mathfrak{a}\to\mathcal{T}_{Y}$ which is both $S$-linear and
a morphism of Lie algebras, and such that 
\[
[a,fb]=f[a,b]+\rho(a)(f)\cdot b
\]
for all $a,b$ in $\mathfrak{a}$ and $f\in S$. In this situation
we have the sheaf of universal enveloping algebras $\mathcal{U}(\mathfrak{a},\rho)$
(or simply $\mathcal{U}(\mathfrak{a})$ if $\rho$ is understood)
as constructed in \cite{key-46}, c.f. also {[}BB{]} section 1.4.
Namely, $\mathcal{U}(\mathfrak{a},\rho)$ is the $S$-algebra generated
by $\mathfrak{a}$ and satisfying the relations:
\[
f\cdot a=fa
\]
for $f\in S$ and $a\in\mathfrak{a}$; here, the left hand side denotes
the multiplication in $\mathcal{U}(\mathfrak{a},\rho)$ and the right
hand side denotes the action of $S$ on $\mathfrak{a}$; next we demand
\[
a_{1}a_{2}-a_{2}a_{1}=[a_{1},a_{2}]
\]
for $a_{1},a_{2}\in\mathfrak{a}$, where the bracket on the right
denotes the Lie-algebra bracket, and finally 
\[
a\cdot f-f\cdot a=\rho(a)(f)
\]
for $f\in S$ and $a\in\mathfrak{a}$. 

To get control of this object, one typically makes the assumption
that $\mathfrak{a}$ is a projective $S$-module (c.f., e.g. \cite{key-45}
for examples of this), which allows one to give an explicit description
of $\mathcal{U}(\mathfrak{a},\rho)$. However, this assumption does
not hold in the situations of interest in this paper; instead, we
make the following:
\begin{defn}
The $W(A)/p$-module $\mathcal{T}'_{W(A)/p}$ is a sub-Lie-algebra
of the Lie algebra of all derivations $\mathcal{T}_{W(A)/p}$; it
is therefore a Lie algebroid via the inclusion map $\rho:\mathcal{T}'_{W(A)/p}\to\mathcal{T}{}_{W(A)/p}$.
Define the algebra $\mathcal{U}'_{W(A)/p}$ to be quotient of $\mathcal{U}(\mathcal{T}'_{W(A)/p})$
by the two-sided ideal generated by 
\[
\partial_{1}\cdot\partial_{2}-\partial_{1}\circ\partial_{2}
\]
for all $\partial_{1},\partial_{2}\in\mathcal{T}'_{W(X)/p}(1)$ (here
we are using that $\partial_{1}\circ\partial_{2}$ is a derivation
by \prettyref{lem:composition-is-a-derivation}). This algebra is
filtered by the two sided ideals $\{\mathcal{K}^{(i)}\}$, where $\mathcal{K}^{(i)}$
is generated by $V^{i}(W(A)/p)$ and $\{\partial\in\mathcal{T}'_{W(X)/p}|\partial(W(A)/p)\subset\mathcal{I}_{i}\}$. 
\end{defn}

In order to discuss local coordinates on this algebra, we choose a
lift of Frobenius $F:\mathcal{A}\to\mathcal{A}$, along with the corresponding
$\Phi:\mathcal{A}\to W(A)$. The reduction mod $p$ of $\Phi$ is
a splitting of the natural reduction map $W(A)/p\to A$. Until further
notice, we fix such a map and we regard $A\subset W(A)/p$; in fact
this puts $A\subset W_{r+1}(A)/p$ for each $r$, and we have 
\begin{lem}
The inclusion $A\subset W_{r+1}(A)/p$ makes $W_{r+1}(A)/p$ a free
$A$-module. Furthermore, there is an isomorphism of $A$-modules
\[
W_{r+1}(A)/p\tilde{=}A\oplus\bigoplus_{i=1}^{r}K_{i}/K_{i+1}
\]
\end{lem}

\begin{proof}
In fact, one may check easily that a basis of $W_{r+1}(A)/p$ over
$A$ is given by 
\[
1\cup\{V^{i}(T^{I})\}_{1\leq i\leq r}
\]
where, for a fixed $i$, the multi-index $I$ ranges over $\{(i_{1},\dots,i_{n})\}$
such that each $i_{j}$ is contained in $\{1,\dots,p^{i}\}$ and at
least one $i_{j}$ is not divisible by $p$. This directly implies
both statements.
\end{proof}
Using the isomorphism of the previous lemma we have
\begin{lem}
\label{lem:local-descrip-of-T'}There is an isomorphism of Lie algebras
\[
\mathcal{T}'_{W(A)/p}\tilde{=}\text{Der}_{A}(A,W(A)/p)\oplus\tilde{\text{Hom}}_{A}(\mathcal{I}_{1})
\]
where on the right we have filtration-preserving $A$-linear maps
on $\mathcal{I}_{1}$. Equivalently, 
\[
\mathcal{T}'_{W(A)/p}\tilde{=}\text{Der}_{A}(A,W_{1}(A)/p)\oplus\mathcal{T}'_{1}
\]
where $\mathcal{T}'_{1}$ consists of derivations in $\mathcal{T}'_{W(A)/p}$
which vanish on $A$. 
\end{lem}

\begin{proof}
Let $\partial\in\mathcal{T}'_{W(A)/p}$. Then the restriction of $\partial$
to $A$ is a derivation from $A$ into $W(A)/p$. Further, if $\delta:A\to W(A)/p$
is any derivation, we may extend it to an element of $\mathcal{T}'_{W(A)/p}$
by setting $\delta$ to be zero on each term of the form $V^{i}(T^{I})$
where, as above, each index $i_{j}$ in $I$ is contained in $\{1,\dots,p^{i}\}$
and at least one $i_{j}$ is not divisible by $p$. Thus we have 
\[
\mathcal{T}'_{W(A)/p}\tilde{=}\text{Der}_{A}(A,W_{1}(A)/p)\oplus\mathcal{T}'_{1}
\]
where $\mathcal{T}'_{1}$ consists of derivations in $\mathcal{T}'_{W(A)/p}$
which vanish on $A$ which proves the second claim of the lemma.

To show that this implies the first claim, let $\phi\in\mathcal{T}_{1}$.
Then for any $a\in A$ and $b\in\mathcal{I}_{1}$ we have 
\[
\phi(ab)=a\phi(b)+\phi(a)b=a\phi(b)
\]
so that $\phi:\mathcal{I}_{1}\to\mathcal{I}_{1}$ is an $A$-linear
map. Conversely, since $\mathcal{I}_{1}$ is a square-zero ideal,
any $A$-linear map $\phi:\mathcal{I}_{1}\to\mathcal{I}_{1}$, extended
to all of $W(A)/p$ by setting $\phi(A)=0$, is a derivation of $W(A)/p$.
Since derivations in $\mathcal{T}'_{W(A)/p}$ are required to preserve
the filtration given by the $\{\mathcal{I}_{j}\}$, one deduces the
first claim of the lemma. 
\end{proof}
From this we deduce 
\begin{cor}
\label{cor:right-ideal-genned-by-T}Let $\{\partial_{1},\dots,\partial_{n}\}$
be a set of coordinate derivations on $A$. Then every element of
$\mathcal{U}'_{W(A)/p}$ may be written as 
\[
\sum_{J}b_{J}\partial^{J}+\phi
\]
where $b_{J}\in W(A)/p$, $\phi$ is contained in the left ideal generated
by $\mathcal{T}_{1}$, and the sum is finite. The analogous statement
holds in $\mathcal{\widehat{U}}'_{W(A)/p}$; namely, any element thereof
may be written as 
\[
\sum_{J}b_{J}\partial^{J}+\phi
\]
where $b_{J}\in W(A)/p$, $\phi$ is contained in the completed left
ideal generated by $\mathcal{T}_{1}$, and $b_{J}\to0$ in $W(A)/p$
as $|J|\to\infty$. 
\end{cor}

\begin{proof}
Let $x\in\mathcal{U}'_{W(A)/p}$. By definition $x$ can be written
as a sum of terms of the form $\alpha_{1}\cdots\alpha_{m}$ where
each $\alpha_{i}$ is either an element of $W(A)/p$, one of the $\partial_{i}$,
or $\phi\in\mathcal{T}_{1}$ (here we use the fact that any element
of $\text{Der}_{A}(A,W_{1}(A)/p)$ is a sum of elements of the form
$a_{i}\partial_{i}$ with $a_{i}\in W(A)/p$). Commuting the $\alpha_{i}$
past one another and using \prettyref{lem:local-descrip-of-T'}, and
then noting that $\mathcal{T}_{1}$ is closed under products, we deduce
the existence of a form as above. The case of $\mathcal{\widehat{U}}'_{W(A)/p}$
follows by taking completion. 
\end{proof}
We now construct the fundamental bimodule $W(A)/p\otimes_{A}\mathcal{D}_{A}^{(0)}$: 
\begin{prop}
\label{prop:Action-of-U-on-bimodule}There is a left action of $\mathcal{U}'_{W(A)/p}$
on $W(A)/p\otimes_{A}\mathcal{D}_{A}^{(0)}$ which makes $W(A)/p\otimes_{A}\mathcal{D}_{A}^{(0)}$
into a $(\mathcal{U}'_{W(A)/p},\mathcal{D}_{A}^{(0)})$ bimodule. 

Let $\Phi^{*}\mathcal{D}_{A}^{(0)}$ be the completion of $W(A)/p\otimes_{A}\mathcal{D}_{A}^{(0)}$
along the filtration $\mathcal{I}_{j}\otimes_{A}\mathcal{D}_{A}^{(0)}$.
There is a left action of $\widehat{\mathcal{U}}'_{W(A)/p}$ on $\Phi^{*}\mathcal{D}_{A}^{(0)}$
making it into a $(\mathcal{\widehat{U}}'_{W(A)/p},\mathcal{D}_{A}^{(0)})$
bimodule. 
\end{prop}

\begin{proof}
In \prettyref{lem:composition-is-a-derivation} above, we constructed
a surjection $\mathcal{T}_{W(A)/p}'\to\mathcal{T}_{A}$. By the functoriality
of the Lie-Rinehart construction, we obtain a map $\mathcal{U}(\mathcal{T}'_{W(A)/p})\to\mathcal{U}(\mathcal{T}_{A})=\mathcal{D}_{A}^{(0)}$.
As the ideal $\mathcal{T}'_{W(A)/p}(1)$ is contained in the kernel
of this map, it necessarily factors through $\mathcal{U}'_{W(A)/p}$;
so we obtain $Q:\mathcal{U}'_{W(A)/p}\to\mathcal{D}_{A}^{(0)}$. 

Now consider a derivation $\partial\in\mathcal{T}_{W(A)/p}'$. If
we restrict $\partial$ to $A$ we obtain an element of $\text{Der}_{A}(A,W(A)/p)$;
so we may write 
\[
\partial|_{A}=Q(\partial)+\sum_{i=1}^{n}\epsilon_{i}\partial_{i}
\]
where $\epsilon_{i}\in\mathcal{I}_{1}$ and $Q(\partial)\in\text{Der}_{A}(A,A)$
is the application of the algebra map $Q$ to $\partial$. (Here $\partial_{i}$
stands for the extension of the canonical coordinate derivation to
$W(A)/p$ which vanishes on terms of the form $p^{r}T^{I/p^{r}}$
as above). 

Then for $a\in W(A)/p$ and $P\in\mathcal{D}_{A}^{(0)}$ we set 
\[
\partial\cdot(a\otimes P)=\partial(a)\otimes P+a\otimes Q(\partial)P+\sum_{i=1}^{n}a\epsilon_{i}\otimes\partial_{i}P
\]
We claim that this defines an action of $\mathcal{U}'_{W(A)/p}$ on
$W(A)/p\otimes_{A}\mathcal{D}_{A}^{(0)}$. To see this, let $b\in A$,
and $a$ and $P$ as above. Then we have 
\[
\partial(ab\otimes P)=\partial(ab)\otimes P+ab\otimes Q(\partial)P+\sum_{i=1}^{n}ab\epsilon_{i}\otimes\partial_{i}P
\]
\[
=\partial(a)b\otimes P+a\partial(b)\otimes P+ab\otimes Q(\partial)P+\sum_{i=1}^{n}ab\epsilon_{i}\otimes\partial_{i}P
\]
\[
=\partial(a)\otimes bP+a\otimes Q(\partial)(b)P+\sum_{i=1}^{n}a\epsilon_{i}\otimes\partial_{i}(b)P+a\otimes bQ(\partial)P+\sum_{i=1}^{n}a\epsilon_{i}\otimes b\partial_{i}P
\]
where we used ${\displaystyle \partial(b)=Q(\partial)(b)+\sum_{i=1}^{n}\epsilon_{i}\partial_{i}(b)}$.
Further, using the equalities $Q(\partial)(b)P+bQ(\partial)P=Q(\partial)bP$
and $\partial_{i}(b)P+b\partial_{i}P=\partial_{i}bP$, we see that
the previous line is equal to 
\[
\partial(a)\otimes bP+a\otimes Q(\partial)bP+\sum_{i=1}^{n}a\epsilon_{i}\otimes\partial_{i}bP=\partial(a\otimes bP)
\]
Therefore each derivation defines an endomorphism of $W(A)/p\otimes_{A}\mathcal{D}_{A}^{(0)}$;
one sees directly that this defines an action of a left action of
$\mathcal{U}'_{W(A)/p}$ on $W(A)/p\otimes_{A}\mathcal{D}_{A}^{(0)}$.
Furthermore, let $b\in A$ and $\partial_{i}$ a coordinate derivation
on $A$. Then 
\[
\partial\cdot(a\otimes Pb)=\partial(a)\otimes Pb+a\otimes Q(\partial)Pb+\sum_{i=1}^{n}a\epsilon_{i}\otimes\partial_{i}Pb
\]
\[
=\partial\cdot(a\otimes P)b
\]
and similarly $\partial\cdot(a\otimes P\partial_{i})=\partial\cdot(a\otimes P)\partial_{i}$.
So the left action of $\mathcal{U}'_{W(A)/p}$ commutes with the natural
right action of $\mathcal{D}_{A}^{(0)}$, as required. The second
sentence follows from the first by simply completing both objects
along their natural filtrations.
\end{proof}
Then we have
\begin{cor}
There is an isomorphism 
\[
\Phi^{*}\mathcal{D}_{A}^{(0)}\tilde{\leftarrow}\widehat{\mathcal{U}}'_{W(A)/p}/\widehat{\mathcal{U}}'_{W(A)/p}\cdot\mathcal{T}'_{1}
\]
where $\widehat{\mathcal{U}}'_{W(A)/p}\cdot\mathcal{T}'_{1}$ denotes
the completion of the left ideal generated by $\mathcal{T}'_{1}$. 
\end{cor}

\begin{proof}
The map is given by sending an element $x\in\widehat{\mathcal{U}}'_{W(A)/p}$
to $x\cdot(1\otimes1)$. Writing 
\[
x=\sum_{J}b_{J}\partial^{J}+\phi
\]
as in \prettyref{cor:right-ideal-genned-by-T}, we see 
\[
x\cdot(1\otimes1)=(\sum_{J}b_{J}\partial^{J}+\phi)\cdot(1\otimes1)
\]
\[
=\sum_{J}b_{j}\otimes\partial^{J}
\]
(the last equality follows from the definition of the action in \prettyref{prop:Action-of-U-on-bimodule}).
As the last sum is zero iff all $b_{J}$ are $0$, the result follows
directly. 
\end{proof}
Now suppose we are in the presence of two distinct coordinatized lifts
of Frobenius $\Phi_{1},\Phi_{2}$ on $\mathcal{A}$, with two associated
embeddings which we now label $\iota_{1},\iota_{2}:A\to W(A)/p$.
Then we have
\begin{cor}
\label{cor:Bimodule-for-affine-in-positive-char}There is a canonical
isomorphism $\Phi_{1}^{*}\mathcal{D}_{A}^{(0)}\tilde{\to}\Phi_{2}^{*}\mathcal{D}_{A}^{(0)}$
of bimodules. If given a third lift $\Phi_{3}$, then the cocycle
condition on isomorphisms is satisfied. 
\end{cor}

\begin{proof}
Slightly modifying the above notation, we let $\mathcal{T}'_{1,\iota_{1}}$
denote the set of derivations in $\mathcal{T}'$ which vanish on $\iota_{1}(A)$,
and similarly for $\mathcal{T}'_{1,\iota_{2}}$. We need to provide
a canonical isomorphism 
\begin{equation}
\widehat{\mathcal{U}}'_{W(A)/p}/\widehat{\mathcal{U}}'_{W(A)/p}\cdot\mathcal{T}'_{1,\iota_{1}}\tilde{\to}\widehat{\mathcal{U}}'_{W(A)/p}/\widehat{\mathcal{U}}'_{W(A)/p}\cdot\mathcal{T}'_{1,\iota_{2}}\label{eq:can-iso}
\end{equation}
Now, the fact that the kernel of $W(A)/p\to A$ is a square zero ideal
implies that $\iota_{12}=i_{1}-i_{2}:A\to W(A)/p$ is a derivation.
We can therefore consider $\iota_{12}$ as a derivation $\iota_{2}(A)\to W(A)/p$,
and then we extend this to a derivation $\tilde{\iota}_{12}:W(A)/p\to W(A)/p$
as in the proof of \prettyref{prop:Action-of-U-on-bimodule}, so that
$\tilde{\iota}_{12}(p^{r}T^{I/p^{r}})=0$, where $T$ refers to the
coordinates for $\Phi_{2}$. Then we can consider $\tilde{\iota}_{12}$
as an element of $\mathcal{T}'$ and therefore as an element of $\widehat{\mathcal{U}}'_{W(A)/p}$.
The element $1+\tilde{\iota}_{12}$ is a unit in $\widehat{\mathcal{U}}'_{W(A)/p}$,
with inverse $1-\tilde{\iota}_{12}$ (this follows from the relation
$\partial_{1}\cdot\partial_{2}-\partial_{1}\circ\partial_{2}$ for
all $\partial_{1},\partial_{2}\in\mathcal{T}'_{W(A)/p}(1)$). Therefore
there is an isomorphism
\[
\widehat{\mathcal{U}}'_{W(A)/p}\xrightarrow{\cdot(1+\tilde{\iota}_{12})}\widehat{\mathcal{U}}'_{W(A)/p}
\]
of left $\widehat{\mathcal{U}}'_{W(A)/p}$-modules. It clearly takes
$\widehat{\mathcal{U}}'_{W(A)/p}\cdot\mathcal{T}'_{1,\iota_{1}}$
to $\widehat{\mathcal{U}}'_{W(A)/p}\cdot\mathcal{T}'_{1,\iota_{2}}$
(and the inverse map $\cdot(1-\tilde{\iota}_{12}$ takes $\widehat{\mathcal{U}}'_{W(A)/p}\cdot\mathcal{T}'_{1,\iota_{2}}$
to $\widehat{\mathcal{U}}'_{W(A)/p}\cdot\mathcal{T}'_{1,\iota_{1}}$).
Therefore we deduce an isomorphism \prettyref{eq:can-iso}, and we
need to check that this map intertwines the associated right $\mathcal{D}_{A}^{((0)}$-module
structures; the cocycle condition is clear from the definition of
the isomorphism.

To see this, we shall show 
\[
x\cdot\iota_{1}(a)(1+\tilde{\iota}_{12})=x\cdot(1+\tilde{\iota}_{12})\iota_{2}(a)
\]
 for all $x\in\widehat{\mathcal{U}}'_{W(A)/p}$ and $a\in A$ and
\[
x\cdot\delta_{1}(1+\tilde{\iota}_{12})=x\cdot(1+\tilde{\iota}_{12})\delta_{2}
\]
for all $x\in\widehat{\mathcal{U}}'_{W(A)/p}$ ; here the notation
is as follows let $\delta$ be a derivation of $A$. Then $\delta_{1}$
and $\delta_{2}$ refer to elements of $\mathcal{T}'_{W(A)/p}$ ,
so that $\delta_{1}$ preserves $\iota_{1}(A)$, and $\delta_{1}|_{\iota_{1}(A)}=\delta$,
and $\delta_{2}$ preserves $\iota_{2}(A)$ and $\delta_{2}|_{\iota_{2}(A)}=\delta$
and so that $\delta_{1}|_{\mathcal{I}_{1}}=\delta_{2}|_{\mathcal{I}_{1}}$
(it is easy to see that it is always possible to choose such $\delta_{1},\delta_{2}$,
for a given $\delta$). As $\mathcal{D}_{A}^{((0)}$ is generated
by $A$ and derivations, this implies the result by the previous corollary. 

Now let us verify these equations; for the first, we have 
\[
x\cdot(1+\tilde{\iota}_{12})\iota_{2}(a)=x\iota_{2}(a)+x\tilde{\iota}_{12}\cdot\iota_{2}(a)
\]
\[
=x\iota_{2}(a)+x\tilde{\iota}_{12}\cdot\iota_{2}(a)
\]
and, as $\tilde{\iota}_{12}\cdot\iota_{2}(a)=\iota_{2}(a)\cdot\tilde{\iota}_{12}+\tilde{\iota}_{12}(\iota_{2}(a))=\iota_{2}(a)\cdot\tilde{\iota}_{12}+i_{1}(a)-\iota_{2}(a)$
we see that the above term equals 
\[
x\iota_{2}(a)\cdot\tilde{\iota}_{12}+x\cdot i_{1}(a)
\]
Further, $\iota_{2}(a)\cdot\tilde{\iota}_{12}=\iota_{1}(a)\cdot\tilde{\iota}_{12}$,
as the difference $\iota_{1}(a)-\iota_{2}(a)\in\mathcal{I}_{1}$ and
therefore annihilates $\tilde{\iota}_{12}$ (as $\tilde{\iota}_{12}$
takes values in $\mathcal{I}_{1}$ and this ideal has square zero).
Thus this term is equal to 
\[
x\iota_{1}(a)\cdot\tilde{\iota}_{12}+x\cdot i_{1}(a)
\]
as required. 

For the second equation, we set $\epsilon=[\tilde{\iota}_{12},\delta_{2}]\in\mathcal{T}_{A}'$
and note that this derivation takes values in $\mathcal{I}_{1}$ and
vanishes on $\mathcal{I}_{1}$; indeed, one easily sees that in fact
$\epsilon=\delta_{1}-\delta_{2}$. Therefore 
\[
x\cdot(1+\tilde{\iota}_{12})\delta_{2}=x\delta_{1}+x\delta_{2}\tilde{\iota}_{12}
\]
and since $(\delta_{1}-\delta_{2})\cdot\tilde{\iota}_{12}=0$ (as
both of these derivations take values in $\mathcal{I}_{1}$) we see
that this equals $x\delta_{1}+x\delta_{1}\tilde{\iota}_{12}$ as required. 
\end{proof}
Now we wish to globalize the whole situation. In order to do that,
we need one more basic fact about the structure of $\widehat{\mathcal{U}}'_{W(A)/p}$.
Let us assume again that we are in the presence of local coordinates
on $A$. Choose a set of elements $\{\phi_{i}\}_{i\in\mathbb{N}}$
which form a topological basis for $\mathcal{T}'_{1}$ as an $A$-module,
and such that $\phi_{i}\to0$ as $i\to\infty$. Then we have
\begin{cor}
\label{cor:basis-of-U}Let $\{\partial_{1},\dots,\partial_{n}\}$
be a set of coordinate derivations on $A$. Then every element of
$\mathcal{U}'_{W(A)/p}$ may be written uniquely as 
\[
\sum_{J}(b_{J}\partial^{J}+\sum_{i_{J}=1}^{\infty}a_{i_{J}}\phi_{i_{J}}\partial^{J})
\]
where $J$ ranges over multi-indices in $\mathbb{N}^{n}$, $i_{J}\in\mathbb{N}$,
$a_{i_{J}},b_{J}\in W(A)/p$, and all but finitely many $b_{J}$ are
$0$. Similarly, every element of $\widehat{\mathcal{U}}'_{W(A)/p}$
can be written uniquely as 
\[
\sum_{J}(b_{J}\partial^{J}+\sum_{i_{J}=1}^{\infty}a_{i_{J}}\phi_{i_{J}}\partial^{J})
\]
with the notation as above, but we allow infinitely many $b_{J}$,
where $b_{J}\to0$ as $|J|\to\infty$. 
\end{cor}

\begin{proof}
The existence is proved exactly as in \prettyref{cor:right-ideal-genned-by-T}.
For the uniqueness, we use the action of $\mathcal{U}'_{W(A)/p}$
on $W(A)/p\otimes_{A}\mathcal{D}_{A}^{(0)}$ (respectively, the action
of $\widehat{\mathcal{U}}'_{W(A)/p}$ on $\Phi^{*}\mathcal{D}_{A}^{(0)}$);
we'll just do the case of $\mathcal{U}'_{W(A)/p}$, the case of the
completion being entirely analogous. Suppose that 
\[
\sum_{J}(b_{J}\partial^{J}+\sum_{i_{J}=1}^{\infty}a_{i_{J}}\phi_{i_{J}}\partial^{J})=0
\]
for some $\{a_{i_{J}},b_{J}\}$ as above. Then 
\[
0=\sum_{J}(b_{J}\partial^{J}+\sum_{i_{J}=1}^{\infty}a_{i_{J}}\phi_{i_{J}}\partial^{J})(1\otimes1)=\sum_{J}b_{J}\otimes\partial^{J}
\]
so that each $b_{J}=0$. Now, let $p^{r}T^{I/p^{r}}\in\mathcal{I}_{1}$.
then each $\partial_{i}(p^{r}T^{I/p^{r}})=0$, so that 
\[
0=(\sum_{i_{J}=1}^{\infty}a_{i_{J}}\phi_{i_{J}}\partial^{J})(p^{r}T^{I/p^{r}}\otimes1)=\sum_{i_{J}=1}^{\infty}a_{i_{J}}\phi_{i_{J}}(p^{r}T^{I/p^{r}})\otimes\partial^{J}
\]
as this is true for all basis elements $p^{r}T^{I/p^{r}}$, we see
that ${\displaystyle a_{i_{J}}\phi_{i_{J}}=0}$ inside $\mathcal{T}_{1}$.
So $a_{i_{J}}=0=b_{J}$ for all $J$ as required. 
\end{proof}
This implies 
\begin{cor}
Let $A$ be a smooth $k$-algebra admitting local coordinates. The
assignment $B\to\widehat{\mathcal{U}}'_{W(B)/p}$ is a sheaf on the
etale site of $A$. 
\end{cor}

\begin{proof}
The fact that this assignment is a presheaf follows immediately from
the functoriality of the Lie-Rinehart construction. To see that it
is actually a sheaf, we note that that previous corollary yields an
isomorphism 
\[
(B\otimes_{A}\widehat{\mathcal{U}}'_{W(A)/p})^{\widehat{}}\tilde{\to}\mathcal{U}'_{W(B)/p}
\]
where the completion on the left is with respect to $B\otimes_{A}\mathcal{I}_{j}$
(where $\mathcal{I}_{j}$ is the two sided ideal in $\widehat{\mathcal{U}}'_{W(A)/p}$). 
\end{proof}
So we may now define for any smooth $X$ over $k$, $\widehat{\mathcal{U}}'_{W(X)_{p=0}}$
as the Zariski sheaf associated to the functor $U\to\widehat{\mathcal{U}}'_{W(A)/p}$
where $U=\text{Spec}(A)$ is an open subset of $X$ admitting local
coordinates (these form a base for the topology). From this definition
and \prettyref{cor:Bimodule-for-affine-in-positive-char}, we obtain 
\begin{cor}
\label{cor:Global-bimodule-for-U}There is a well-defined $(\widehat{\mathcal{U}}'_{W(X)_{p=0}},\mathcal{D}_{A}^{(0)})$-bimodule
$\mathcal{B}_{X}^{(0)}$, which for every open affine $U=\text{Spec}(A)$
(admitting local coordinates), and every lift of Frobenius $\Phi$
to $\mathcal{A}$, is isomorphic to $\Phi^{*}\mathcal{D}_{A}^{(0)}$. 
\end{cor}

Finally, we need to relate all this to Witt-differential operators:
\begin{thm}
\label{thm:onto-mod-p}There is a surjective morphism of sheaves of
algebras 
\[
\widehat{\mathcal{D}}_{W(X)}^{(0)}/p\to\mathcal{\widehat{U}}'_{W(X)_{p=0}}
\]
This map is continuous with respect to the inverse limit topologies
on these algebras. 
\end{thm}

To prove this, we first need the following fact about the action of
the Witt-differential operators on $W(X)_{1}$; we let $X=\text{Spec}(A)$
as above.
\begin{prop}
\label{prop:TW-mod-p-is-derivations}Let $\varphi\in\mathcal{E}W_{A}^{(0)}$.
Then the induced map $\overline{\varphi}:W(A)/p\to W(A)/p$ is a derivation
of the algebra $W(A)/p$; furthermore, $\overline{\varphi}$ preserves
the filtration $\{\mathcal{I}_{i}\}$ of $W(A)/p$. 
\end{prop}

\begin{proof}
For each $r\geq0$, we have the natural surjection $W(A)/p\to W_{r+1}(A)/p$,
and there is an isomorphism ${\displaystyle W(A)/p\tilde{=}\lim_{r}W_{r+1}(A)/p}$.
By construction $\overline{\varphi}$ is an inverse limit of operators
on $W_{r+1}(A)/p$, and hence preserves each $K_{r+1}$, and it suffices
to show that the restriction of $\overline{\varphi}$ to $W_{r+1}(A)/p$
is a derivation for each $r\geq0$. 

Let $D=(D_{0},\dots,D_{i})$ be a Hasse-Schmidt derivation on $A$,
with $i\leq p^{r}$, and $\tilde{D}_{i}$ is the canonical lift to
$W_{r+1}(A)$, then we must show that for any $a\in W_{r+1}(A)$,
$F^{\text{val}_{p}(i)-r}(a)\cdot\tilde{D}_{i}$ acts as a derivation
on $W_{r+1}(A)/p$. By definition, we have 
\[
F^{\text{val}_{p}(i)-r}(a)\tilde{D}_{i}(\alpha\beta)=F^{\text{val}_{p}(i)-r}(a)\sum_{j+l=i}\tilde{D}_{j}(\alpha)\tilde{D}_{l}(\beta)
\]
and we must show that $F^{\text{val}_{p}(i)-r}(a)D_{j}(\alpha)D_{l}(\beta)\in pW_{r+1}(A)$
whenever $0<j,l<i$. To that end, we may, by \prettyref{prop:Move-Down},
write $\tilde{D}_{j}(\alpha)=V^{r-\text{val}_{p}(j)}(x)$ and $\tilde{D}_{l}(\beta)=V^{r-\text{val}_{p}(l)}(y)$.
Suppose $\text{val}_{p}(j)\geq\text{val}_{p}(l)$ (the other case
having an identical proof). Then 
\[
V^{r-\text{val}_{p}(j)}(x)V^{r-\text{val}_{p}(l)}(y)=p^{r-\text{val}_{p}(j)}V^{r-\text{val}_{p}(l)}(F^{\text{val}_{p}(j)-\text{val}_{p}(l)}(x)\cdot y)
\]
so that 
\[
F^{\text{val}_{p}(i)-r}(a)\cdot V^{r-\text{val}_{p}(j)}(x)V^{r-\text{val}_{p}(l)}(y)
\]
\[
=p^{r-\text{val}_{p}(j)}F^{\text{val}_{p}(l)-r}(F^{\text{val}_{p}(i)-\text{val}_{p}(l)}(a))V^{r-\text{val}_{p}(l)}(F^{\text{val}_{p}(j)-\text{val}_{p}(l)}(x)\cdot y)
\]
where we have used $\text{val}_{p}(i)\geq\text{min}\{\text{val}_{p}(j),\text{val}_{p}(l)\}=\text{val}_{p}(l)$.
Thus we conclude that 
\[
F^{\text{val}_{p}(i)-r}(a)\tilde{D}_{j}(\alpha)\tilde{D}_{l}(\beta)=p^{r-\text{val}_{p}(j)}F^{\text{val}_{p}(l)-r}(A)V^{r-\text{val}_{p}(l)}(B)=p^{r-\text{val}_{p}(j)}V^{r-\text{val}_{p}(l)}(AB)
\]
where $A=F^{\text{val}_{p}(i)-\text{val}_{p}(l)}(a)$ and $B=F^{\text{val}_{p}(j)-\text{val}_{p}(l)}(x)\cdot y$;
so, since $\text{val}_{p}(j)<r$ we deduce that $F^{\text{val}_{p}(i)-r}(a)D_{j}(\alpha)D_{l}(\beta)\in pW_{r+1}(A)$
as required. 
\end{proof}
Now we can give the 
\begin{proof}
(of \prettyref{thm:onto-mod-p}) Since the image of $\mathcal{E}W_{X}^{(0)}$
in $\widehat{\mathcal{D}}_{W(X)}^{(0)}/p$ topologically generates
this sheaf of algebras over $\mathcal{O}_{W(X)_{p=0}}$, we see that
there \emph{at most one }morphism $\widehat{\mathcal{D}}_{W(X)}^{(0)}/p\to\mathcal{\widehat{U}}'_{W(X)_{p=0}}$
which is the identity on $\mathcal{O}_{W(X)_{p=0}}$ and which sends
the image of a section $\varphi\in\mathcal{E}W_{X}^{(0)}/p$ to the
derivation $\overline{\varphi}$. It therefore suffices to prove that
such a morphism exists locally, i.e. for $X=\text{Spec}(A)$ possessing
local coordinates.

In that case, invoking \prettyref{thm:Basis}, we see that it suffices
to compute the commutator of any pair of operators of the form $F^{\text{min}\{0,-r\}}(\alpha_{J_{I}})\{\partial\}_{J_{I}/p^{r}}$
and show that it maps to the commutator in $\mathcal{\widehat{U}}'_{W(X)_{p=0}}$,
but this is straightforward.
\end{proof}
Combining this result with \prettyref{cor:Global-bimodule-for-U},
we obtain 
\begin{cor}
\label{cor:Global-bimodule-mod-p}There is a well-defined $(\widehat{\mathcal{D}}_{W(A)}^{(0)}/p,\mathcal{D}_{A}^{(0)})$-bimodule
$\mathcal{B}_{X}^{(0)}$, which for every open affine $U=\text{Spec}(A)$
(admitting local coordinates), and every lift of Frobenius $\Phi$
to $\mathcal{A}$, is isomorphic to $\Phi^{*}\mathcal{D}_{\mathcal{A}}^{(0)}/p$. 
\end{cor}

\begin{proof}
The only thing to do is to check that the action of $\widehat{\mathcal{D}}_{W(A)}^{(0)}/p$
on $\Phi^{*}\mathcal{D}_{A}^{(0)}$ via the surjection $\widehat{\mathcal{D}}_{W(A)}^{(0)}/p\to\widehat{\mathcal{U}}'_{W(A)/p}$
agrees with the action of $\widehat{\mathcal{D}}_{W(A)}^{(0)}/p$
on $\Phi^{*}\mathcal{D}_{\mathcal{A}}^{(0)}/p$ obtained from \prettyref{prop:Construction-of-bimodule}.
But this is a straightforward check in local coordinates. 
\end{proof}

\subsection{De Rham-Witt connections mod $p^{r}$}

In this subsection we study the de Rham-Witt connections over $W(X)_{p^{r}=0}$.
. To state the main result, we need some notation. We define the category
$\text{MIC}_{W(X)_{p^{r}=0}}$ to be the category of quasicoherent
sheaves $\mathcal{M}$ on $W(X)_{p^{r}=0}$, equipped with an integrable
de Rham-Witt connection 
\[
\nabla:\mathcal{M}\to\mathcal{M}\widehat{\otimes}_{\mathcal{O}_{W(X)/p^{r}}}W\Omega_{X}^{1}/p^{r}
\]
where the $\widehat{?}$ stands for completion with respect to the
filtration 
\[
\{V^{i}(\mathcal{O}_{W(X)}/p^{r})\cdot\mathcal{M}\otimes G^{j}(W\Omega_{X}^{1}/p^{r})\}_{i+j\geq l}
\]
where $G^{j}$ is the standard filtration on the de Rham-Witt complex,
defined by $V^{i}+dV^{i-1}$. We also demand that $\nabla$ is continuous,
in the sense that it is the inverse limit of the connections
\[
\nabla_{j}:\mathcal{M}/V^{j}(\mathcal{O}_{W(X)}/p^{r})\to(\mathcal{M}\otimes_{\mathcal{O}_{W(X)/p^{r}}}W\Omega_{X}^{1}/p^{r})/V^{j}(\mathcal{O}_{W(X)}/p^{r})
\]
The morphisms in this category are the continuous maps of sheaves
which respect the connections. 

In \prettyref{cor:drW-connection-on-M} we constructed a functor $\mathcal{C}$
from $\widehat{\mathcal{D}}_{W(X)}^{(0)}/p^{r}-\text{mod}_{\text{acc,qcoh}}$
to integrable de Rham-Witt connection on $W(X)_{p^{r}=0}$. Passing
to the completions, we obtain a functor $\widehat{\mathcal{C}}$ from
$\widehat{\mathcal{D}}_{W(X)}^{(0)}/p^{r}-\text{mod}_{\text{c-qcoh}}$
to integrable de Rham-Witt connection on $W(X)_{p^{r}=0}$; this functor
clearly takes values in $\text{MIC}_{W(X)_{p^{r}=0}}$. Then the main
result is:
\begin{thm}
\label{thm:Equivalence-mod-p^r}The functor $\widehat{\mathcal{C}}$
from $\widehat{\mathcal{D}}_{W(X)}^{(0)}/p^{r}-\text{mod}_{\text{c-qcoh}}$
to $\text{MIC}_{W(X)_{p^{r}=0}}$ is an equivalence of categories. 
\end{thm}

The result, and the proof, are inspired by \cite{key-24}, where the
case of a vector bundle with connection is handled. We should mention
that he also takes the inverse limit over $r$ to obtain the following
equivalence for locally free crystals on $X$: 
\begin{thm}
(Bloch) There is an equivalence of categories between locally free
crystals over $W(k)$ (on $X$) and vector bundles on $W(X)$ equipped
with integrable de Rham-Witt connections, which are locally nilpotent. 
\end{thm}

One can take the inverse limit over $r$ in \prettyref{thm:Equivalence-mod-p^r},
but I don't know any reasonable characterization of the image.

As in Bloch's paper, we start in the case $r=1$. Let us assume this
until further notice. We shall also suppose, to start with, that $X=\text{Spec}(A)$
for a smooth $A$ which admits local coordinates; as usual we fix
also a lift $\Phi:\mathcal{A}\to\mathcal{A}$ of Frobenius. Therefore
any element of $\widehat{\mathcal{D}}_{W(X)}^{(0)}/p^{r}-\text{mod}_{\text{qcoh}}$
may be expressed uniquely as $\widehat{\Phi}^{*}\mathcal{M}$ for
a quasicoherent $\mathcal{D}_{X}^{(0)}$-module $\mathcal{M}$. As
above, we let $\mathcal{I}=\mathcal{I}_{1}:=\text{ker}(W(A)/p\to A)$. 

Then the full faithfulness of the functor is straightforward to check,
it follows directly from the fact that $(\widehat{\Phi}^{*}\mathcal{M})/\mathcal{I}=\mathcal{M}$
(as flat connections). To show the essential surjectivity, we start
by recalling the following lemma of Bloch (\cite{key-24}, lemma 3.2): 
\begin{lem}
\label{lem:Basic-Bloch-lemma}Suppose $A$ is a smooth $k$-algebra
which admits local coordinates. Let $d:W(A)/p\to(W\Omega_{A}^{1}/p)/\mathcal{I}$
denote the map induced by the differential. Then there is an isomorphism
of $A$-modules
\[
(W\Omega_{A}^{1}/p)/\mathcal{I}\tilde{=}d(\mathcal{I})\oplus\Omega_{A}^{1}
\]
and $d:\mathcal{I}\to d(\mathcal{I})$ is an isomorphism of $A$-modules.
Further, the induced filtration on $d(\mathcal{I})$ (coming from
$V^{i}(W\Omega_{A}^{1}/p)$) coincides with the $V$-adic filtration
on $\mathcal{I}$. The analogue
\[
(W\Omega_{X}^{1}/p)/\mathcal{I}\tilde{=}d(\mathcal{I})\oplus\Omega_{X}^{1}
\]
for the sheaf $(W\Omega_{X}^{1}/p)/\mathcal{I}$ also holds.
\end{lem}

This implies 
\begin{prop}
\label{prop:T-Omega-Dual}Let $\tilde{\text{Hom}}{}_{W(A)/p}((W\Omega_{A}^{1}/p)/\mathcal{I},W(A)/p)$
denote the set of $W(A)/p$-module homomorphisms which preserve the
$V$-adic filtrations on both sides. Then 
\[
\tilde{\text{Hom}}{}_{W(A)/p}((W\Omega_{A}^{1}/p)/\mathcal{I},W(A)/p)\tilde{=}\mathcal{T}'_{W(A)/p}(1)
\]
Similarly let $\tilde{\mathcal{H}om}_{\mathcal{O}_{W(X)/p}}((W\Omega_{X}^{1}/p)/\mathcal{I},\mathcal{O}_{W(X)/p})$
denote the sheaf of $\mathcal{O}_{W(X)/p}$-module homomorphisms which
preserve the $V$-adic filtrations on both sides. Then 
\[
\tilde{\mathcal{H}om}_{\mathcal{O}_{W(X)/p}}((W\Omega_{X}^{1}/p)/\mathcal{I},\mathcal{O}_{W(X)/p})\tilde{=}\mathcal{T}'_{W(X)/p}(1)
\]
\end{prop}

\begin{proof}
By precomposing with $d:W(A)/p\to(W\Omega_{A}^{1}/p)/\mathcal{I}$
we obtain a map
\[
\tilde{\text{Hom }}((W\Omega_{A}^{1}/p)/\mathcal{I},W(A)/p)\to\text{Der}_{W(A)/p}
\]
and the fact that each element of the left hand side preserves the
relevant filtrations implies that the image lives in $\mathcal{T}'_{W(A)/p}$.
As $\mathcal{I}$ annihilates $(W\Omega_{A}^{1}/p)/\mathcal{I}$,
the image of any map $\phi\in\text{Hom}_{W(A)/p}((W\Omega_{A}^{1}/p)/\mathcal{I},W(A)/p)$
must be contained in $\mathcal{I}$. Using the above lemma, we see
that
\[
\tilde{\text{Hom }}((W\Omega_{A}^{1}/p)/\mathcal{I},W(A)/p)\tilde{=}
\]
\[
\text{Hom}_{A}(\Omega_{A}^{1},\mathcal{I})\oplus\tilde{\text{Hom}}_{A}(\mathcal{I},\mathcal{I})
\]
where $\tilde{\text{Hom}}_{A}(\mathcal{I},\mathcal{I})$ consists
of $A$-linear endomorphisms of $\mathcal{I}$ which preserve the
$V$-adic filtration. Thus the result follows from \prettyref{lem:local-descrip-of-T'}.
The sheaf variant follows directly.
\end{proof}
If $\mathcal{M}$ is any sheaf of $\mathcal{O}_{W(X)}/p$ modules
with an integrable connection, denote by $\tilde{\nabla}$ the induced
map 
\[
\mathcal{M}\xrightarrow{\nabla\otimes1}\mathcal{M}\widehat{\otimes}_{\mathcal{O}_{W(X)/p}}W\Omega_{X}^{1}/p\to\mathcal{M}\widehat{\otimes}_{\mathcal{O}_{W(X)/p}}(W\Omega_{X}^{1}/p)/\mathcal{I}
\]
Therefore 
\begin{cor}
Let $\mathcal{M}\in\text{MIC}_{W(X)_{p=0}}$. Then $\mathcal{M}$
inherits the structure of a module over $\mathcal{\widehat{U}}'_{W(X)_{p=0}}(1)$,
the sub-sheaf of algebras of $\mathcal{\widehat{U}}'_{W(X)_{p=0}}$
topologically generated by $\mathcal{T}'_{W(X)/p}(1)$. 
\end{cor}

\begin{proof}
By the previous result there is an evaluation map 
\[
(W\Omega_{X}^{1}/p)/\mathcal{I}\otimes_{\mathcal{O}_{W(X)/p}}\mathcal{T}'_{W(X)/p}(1)\xrightarrow{\text{ev}}\mathcal{O}_{W(X)/p}
\]
So, we obtain a map 
\[
\mathcal{M}\otimes_{\mathcal{O}_{W(X)/p}}\mathcal{T}'_{W(X)/p}(1)\xrightarrow{\tilde{\nabla}\otimes1}\mathcal{M}\otimes_{\mathcal{O}_{W(X)/p}}(W\Omega_{X}^{1}/p)/\mathcal{I}\otimes_{\mathcal{O}_{W(X)/p}}\mathcal{T}'_{W(X)/p}(1)
\]
\[
\xrightarrow{1\otimes\text{ev}}\mathcal{M}
\]
This defines a continuous action of $\mathcal{T}'_{W(X)/p}(1)\subset\mathcal{\widehat{U}}'_{W(X)_{p=0}}(1)$
on $\mathcal{M}$. From the flatness of the connection, one sees that
this extends to an associated action of $\mathcal{\widehat{U}}'_{W(X)_{p=0}}(1)$
on $\mathcal{M}$ as required. 
\end{proof}
This result already yields some powerful restrictions on the structure
of $\mathcal{M}$. Namely, we have 
\begin{prop}
\label{prop:Splitting-over-nabla-tilde}Regarding $\mathcal{M}$ as
a sheaf of $\mathcal{O}_{X}$ modules via $\Phi:\mathcal{O}_{X}\to\mathcal{O}_{W(X)}/p$,
there is a splitting of $\mathcal{O}_{X}$ modules 
\[
\mathcal{M}=\mathcal{M}^{\pi}\oplus\widehat{\mathcal{I}\mathcal{M}^{\pi}}
\]
where $\mathcal{M}^{\pi}$ is a quasicoherent $\mathcal{O}_{X}$-submodule
of $\mathcal{M}$ satisfying $\mathcal{M}^{\pi}\tilde{\to}\mathcal{M}/\mathcal{I}$;
and $\widehat{\mathcal{I}\mathcal{M}^{\pi}}$ is the completion of
$\mathcal{I}\mathcal{M}^{\pi}$ along the filtration $V^{i}(\mathcal{I})\cdot\mathcal{M}^{\pi}$.
In fact there is an isomorphism of $\mathcal{O}_{W(X)}$-modules $\mathcal{M}\tilde{\to}\widehat{\Phi}^{*}\mathcal{M}^{\pi}$.
Furthermore, under the isomorphism 
\[
(W\Omega_{X}^{1}/p)/\mathcal{I}\tilde{=}d(\mathcal{I})\oplus\Omega_{X}^{1}
\]
 the connection $\tilde{\nabla}$ takes $\mathcal{M}^{\pi}$ to $\mathcal{M}^{\pi}\otimes_{\mathcal{O}_{X}}\Omega_{X}^{1}$;
the action of $\tilde{\nabla}$ on $\mathcal{M}^{\pi}$ is simply
the quotient connection on $\mathcal{M}/\widehat{\mathcal{I}\mathcal{M}}$. 
\end{prop}

\begin{proof}
Inside $\mathcal{\widehat{U}}'_{W(X)_{p=0}}(1)$ we have the derivation
$\delta$ which is zero on $A$ and the identity on $\mathcal{I}$;
and therefore also the element $\pi=1-\delta$. This is an idempotent
in the global sections of $\mathcal{\widehat{U}}'_{W(X)_{p=0}}(1)$,
and, as $\pi$ annihilates $\mathcal{I}$ it follows that there is
a splitting
\[
\mathcal{M}=\mathcal{M}^{\pi}\oplus\mathcal{M}^{\delta}
\]
Let us check that $\widehat{\mathcal{I}\mathcal{M}}=\widehat{\mathcal{I}\mathcal{M}^{\pi}}=\mathcal{M}^{\delta}$.
Consider an element ${\displaystyle \sum_{i}a_{i}m_{i}}\in\widehat{\mathcal{I}\mathcal{M}}$
where $a_{i}\in\mathcal{I}$ converge to zero as $i\to\infty$. Then
\[
\tilde{\nabla}(\sum_{i}a_{i}m_{i})=\sum_{i}\tilde{\nabla}(a_{i})\otimes m_{i}\in(W\Omega_{X}^{1}/p)/\mathcal{I}\widehat{\otimes}_{\mathcal{O}_{W(X)/p}}\mathcal{M}
\]
Abusing notation slightly, let $\delta:(W\Omega_{X}^{1}/p)/\mathcal{I}\to\mathcal{O}_{W(X)}/p$
denote the map associated to $\delta$ via \prettyref{prop:T-Omega-Dual}.
Then 
\[
\delta(\sum_{i}a_{i}m_{i})=\sum_{i}\delta(\tilde{\nabla}(a_{i}))\cdot m_{i}=\sum_{i}a_{i}m_{i}
\]
since $\delta$ is the identity on $\mathcal{I}$ by definition. Therefore
$\widehat{\mathcal{I}\mathcal{M}}\subset\mathcal{M}^{\delta}$. For
the other direction, we have the isomorphism 
\[
(W\Omega_{X}^{1}/p)/\mathcal{I}\tilde{=}d(\mathcal{I})\oplus\Omega_{X}^{1}
\]
as $\mathcal{O}_{X}$-modules. Therefore we can write
\[
\tilde{\nabla}(m)=\sum_{i}da_{i}\otimes m_{i}+\sum_{j=1}^{n}dT_{j}\otimes m_{j}
\]
and so $m=\delta m$ implies ${\displaystyle m=\sum_{i}a_{i}m_{i}}$,
i.e., $\widehat{\mathcal{I}\mathcal{M}}=\mathcal{M}^{\delta}$, and
since $\mathcal{I}$ is square $0$ we see that $\widehat{\mathcal{I}\mathcal{M}}=\widehat{\mathcal{I}\mathcal{M}^{\pi}}$.
By the completeness of $\mathcal{M}$ with respect to $V^{i}(\mathcal{I})\cdot\mathcal{M}$,
this yields a morphism of $\mathcal{O}_{W(X)}$-modules 
\begin{equation}
\widehat{\Phi}^{*}\mathcal{M}^{\pi}\to\mathcal{M}\label{eq:structure}
\end{equation}
which is clearly surjective. To see that it is injective, consider
a term of the form 
\[
\sum_{r,I}p^{r}T^{I/p^{r}}\otimes m_{r,I}\in\widehat{\Phi}^{*}\mathcal{M}^{\pi}
\]
where $r\in\mathbb{N}$ and $I$ is a multi-index with at least one
term not divisible by $p$. Then
\[
\tilde{\nabla}(\sum_{r,I}p^{r}T^{I/p^{r}}\cdot m_{r,I})=\sum_{r,I}d(p^{r}T^{I/p^{r}})\otimes m_{r,I}\in d(\mathcal{I})\widehat{\otimes}_{\mathcal{O}_{X}}\mathcal{M}^{\pi}
\]
(here the $\widehat{\otimes}$ refers to completion with respect to
the natural filtration on $d(\mathcal{I})$). However, this term does
not vanish unless all of the $m_{r,I}$ do by \prettyref{lem:Basic-Bloch-lemma}.
So the map \prettyref{eq:structure} is an isomorphism. Finally, if
we have $m\in\mathcal{M}^{\pi}$, then $\delta\cdot m=0$ and so,
by writing out the connection as above, we see that $\tilde{\nabla}(m)\in\mathcal{M}^{\pi}\otimes_{\mathcal{O}_{X}}\Omega_{X}^{1}$
and that this connection agrees with the quotient connection. 
\end{proof}
This proposition is very close to proving the case $r=1$ of \prettyref{thm:Equivalence-mod-p^r}.
To finish the argument, we need one additional fact: 
\begin{lem}
\label{lem:required-flatness}Consider the map 
\[
\overline{d}:\widehat{\mathcal{I}\cdot(W\Omega_{X}^{1}/p)}\to W\Omega_{X}^{2}/p\to(W\Omega_{X}^{2}/p)/\widehat{(\mathcal{I}\cdot W\Omega_{X}^{2}/p)}
\]
then $\overline{d}$ is an injective $\mathcal{O}_{X}$-linear (via
$\Phi:\mathcal{O}_{X}\to\mathcal{O}_{W(X)}/p$) morphism, whose cokernel
is flat over $\mathcal{O}_{X}$ (in fact, the cokernel is an inverse
limit of a surjective system of locally free sheaves). 
\end{lem}

\begin{proof}
As everything in sight is etale local, we can assume that $X=\mathbb{A}_{k}^{n}$.
The injectivity of the map $\overline{d}$ then follows from \cite{key-24},
lemma 2.3. For the flatness statement we argue as in the proof of
\prettyref{lem:Flatness-of-Omega-i}. Indeed, by the Mittag-Leffler
condition we see that the cokernel of $\overline{d}$ is the inverse
limit of the cokernels of 
\[
\overline{d}_{r}:\mathcal{I}\cdot(W_{r}\Omega_{X}^{1}/p)\to(W_{r}\Omega_{X}^{2}/p)/\mathcal{I}\cdot W_{r}\Omega_{X}^{2}/p
\]
But by the universality of the constructions involved, one sees that
$\text{coker}(\overline{d}_{r})$ is invariant under any automorphism
of $\mathbb{A}_{k}^{n}$. Therefore $\text{coker}(\overline{d})$
is an inverse limit of a surjective system of free modules, and is
therefore flat as claimed. 
\end{proof}
Now we can proceed to the 
\begin{proof}
(of \prettyref{thm:Equivalence-mod-p^r}, when $r=1$). Assume for
now that $X$ is affine and admits local coordinates. To prove the
essential surjectivity of $\widehat{\Phi}^{*}$, we shall check that,
for an object $\mathcal{M}\in\text{MIC}_{W(X)_{p^{r}=0}}$, the isomorphism
\[
\widehat{\Phi}^{*}\mathcal{M}^{\pi}\tilde{\to}\mathcal{M}
\]
constructed above in \prettyref{prop:Splitting-over-nabla-tilde}
is actually an isomorphism of de Rham-Witt connections (and not merely
for the maps $\tilde{\nabla}$ as proved above). This is equivalent
to checking that, under the isomorphism 
\[
W\Omega_{X}^{1}=\Omega_{X}^{1}\oplus G^{1}(W\Omega_{X}^{1}/p)
\]
induced by the lift $\Phi$, we have $\nabla(m)\in\Omega_{X}^{1}\otimes_{\mathcal{O}_{X}}\mathcal{M}^{\pi}$
for any $m\in\mathcal{M}^{\pi}$. By \prettyref{prop:Splitting-over-nabla-tilde},
we have for such an $m$
\[
\nabla(m)=\nabla_{0}(m)+\nabla_{1}(m)
\]
where $\nabla_{0}(m)\in\Omega_{X}^{1}\otimes_{\mathcal{O}_{X}}\mathcal{M}^{\pi}$
and $\nabla_{1}(m)\in\widehat{\mathcal{I}\cdot(W\Omega_{X}^{1}/p)}\otimes\mathcal{M}$.
Write $\nabla_{0}(m)=\sum_{i=1}^{n}dT_{i}\otimes m_{i}$. Then 
\[
0=\nabla(\nabla(m))=\sum_{i=1}^{n}dT_{i}\otimes\nabla_{0}(m_{i})+\sum_{i=1}^{n}dT_{i}\otimes\nabla_{1}(m_{i})+\nabla(\nabla_{1}(m))
\]
Now, the first term is zero because $\nabla_{0}$ is flat, and the
second term goes to zero in $(\mathcal{M}\widehat{\otimes}_{\mathcal{O}_{W(X)/p}}W\Omega_{X}^{2}/p)/\mathcal{M}\widehat{\otimes}_{\mathcal{O}_{W(X)/p}}\mathcal{I}\cdot W\Omega_{X}^{2}/p)$.
Thus we are left with the third term. But note that, by the Leibniz
rule, the map 
\[
\overline{\nabla}:\mathcal{M}\widehat{\otimes}_{\mathcal{O}_{W(X)/p}}\mathcal{I}\cdot(W\Omega_{X}^{1}/p)\to(\mathcal{M}\widehat{\otimes}_{\mathcal{O}_{W(X)/p}}W\Omega_{X}^{2}/p)/\mathcal{M}\widehat{\otimes}_{\mathcal{O}_{W(X)/p}}\mathcal{I}\cdot W\Omega_{X}^{2}/p)
\]
is in fact simply (the completion of) $\overline{d}\otimes1$, where
$\overline{d}$ is the map of the previous lemma. So, by that lemma,
this map is injective. Thus if $\nabla_{1}(m)\neq0$ then $\nabla(\nabla_{1}(m))$
has nonzero image in $(\mathcal{M}\widehat{\otimes}_{\mathcal{O}_{W(X)/p}}W\Omega_{X}^{2}/p)/\mathcal{M}\widehat{\otimes}_{\mathcal{O}_{W(X)/p}}\mathcal{I}\cdot W\Omega_{X}^{2}/p)$;
however, its image must be zero by the flatness of the connection.
Therefore $\nabla_{1}(m)=0$ so that $\widehat{\Phi}^{*}\mathcal{M}^{\pi}\tilde{\to}\mathcal{M}$
is an isomorphism of de Rham-Witt connections. Thus $\mathcal{M}$
is in the essential image of the functor $\widehat{\mathcal{C}}$. 

This proves the theorem (for $r=1$) in the local case. To prove it
for an arbitrary $X$, we note that the full faithfulness in the local
case implies that 
\[
\mathcal{H}om_{\widehat{\mathcal{D}}_{W(X)}^{(0)}/p}(\mathcal{N}_{1},\mathcal{N}_{2})\to\mathcal{H}om_{\nabla}(\widehat{\mathcal{C}}\mathcal{N}_{1},\widehat{\mathcal{C}}\mathcal{N}_{2})
\]
is an isomorphism of sheaves. Applying global sections shows that
$\widehat{\mathcal{C}}$ is fully faithful. But this implies, using
the cocycle definition of sheaves, that the essential surjectivity
can be checked locally, which is what we have just done. 
\end{proof}
Now we want to prove the result for $r>1$. Again we start by supposing
that $X=\text{Spec}(A)$ and that a coordinatized lift of Frobenius
is fixed. In this case we have the morphism $\mathcal{A}_{r}\to W(A)/p^{r}$
and the splitting 
\[
W(A)/p^{r}=\mathcal{A}_{r}\oplus F^{1}(W(A)/p^{r})
\]
as $\mathcal{A}_{r}$-modules (c.f. \prettyref{lem:two-filtrations});
the same holds at the level of sheaves. This means that for sheaf
of $\mathcal{O}_{W(X)}/p^{r}$-modules $\mathcal{M}$ we can consider
the $\mathcal{O}_{\mathfrak{X}_{r}}$-submodule $F^{1}(\mathcal{O}_{W(X)}/p^{r})\cdot\mathcal{M}$.
To save space, we denote the resulting quotient $\mathcal{O}_{\mathfrak{X}_{r}}$-module
by $\mathcal{M}/F^{1}$. If $\mathcal{M}=\widehat{\Phi}^{*}(\mathcal{N})$
then we have 
\[
\mathcal{M}/F^{1}\tilde{\to}\mathcal{N}
\]
Further, if $\nabla:\mathcal{M}\to\mathcal{M}\otimes_{\mathcal{O}_{W(X)/p^{r}}}W\Omega_{X}^{1}/p^{r}$
is the induced flat connection, we can consider the induced map 
\[
\tilde{\nabla}:\mathcal{M}\to(\mathcal{M}\widehat{\otimes}_{\mathcal{O}_{W(X)/p^{r}}}W\Omega_{X}^{1}/p^{r})/F^{1}
\]
Now, using Nakayama's lemma, it follows from \prettyref{lem:Basic-Bloch-lemma}
that we have an isomorphism 
\[
W\Omega_{X}^{1}/p^{r}/F^{1}\tilde{=}d(F^{1}(\mathcal{O}_{W(X)}/p^{r}))\oplus\Omega_{\mathfrak{X}_{r}}^{1}
\]
of $\mathcal{O}_{\mathfrak{X}_{r}}$-modules. We have 
\begin{lem}
For any $\mathcal{M}\in\text{MIC}_{W(X)_{p^{r}=0}}$ , we have an
isomorphism 
\[
(\mathcal{M}\widehat{\otimes}_{\mathcal{O}_{W(X)/p^{r}}}W\Omega_{X}^{1}/p^{r})/F^{1}
\]
\[
\tilde{\to}(\mathcal{M}/F^{1})\widehat{\otimes}_{\mathcal{O}_{\mathfrak{X}_{r}}}(W\Omega_{X}^{1}/p^{r}/F^{1})
\]
\[
\tilde{\to}\mathcal{M}/F^{1}\widehat{\otimes}_{\mathcal{O}_{\mathfrak{X}_{r}}}d(F^{1}(\mathcal{O}_{W(X)}/p^{r}))\oplus(\mathcal{M}/F^{1}\otimes_{\mathcal{O}_{\mathfrak{X}_{r}}}\Omega_{\mathfrak{X}_{r}}^{1})
\]
Here, the symbol $\widehat{\otimes}_{\mathcal{O}_{\mathfrak{X}_{r}}}$
denotes taking the tensor product followed by the completion with
respect to the $V^{i}$- filtration on $W\Omega_{X}^{1}/p^{r}/F^{1}$,
respectively $d(F^{1}(\mathcal{O}_{W(X)}/p^{r}))$. 
\end{lem}

\begin{proof}
The last isomorphism follows from the above discussion. For the first,
we start by noting that there is a natural map 
\[
\mathcal{M}\times W\Omega_{X}^{1}/p^{r}\to(\mathcal{M}/F^{1})\widehat{\otimes}_{\mathcal{O}_{\mathfrak{X}_{r}}}(W\Omega_{X}^{1}/p^{r}/F^{1})
\]
sending a pair of local sections $(m,\phi)$ to $\overline{m}\otimes\overline{\phi}$.
By using the decomposition $\mathcal{O}_{W(X)}/p^{r}=\mathcal{O}_{\mathfrak{X}_{r}}\oplus F^{1}(\mathcal{O}_{W(X)}/p^{r})$
we see that this map is bilinear over $\mathcal{O}_{W(X)}/p^{r}$.
Thus we obtain a map 
\[
(\mathcal{M}\widehat{\otimes}_{\mathcal{O}_{W(X)/p^{r}}}W\Omega_{X}^{1}/p^{r})\to(\mathcal{M}/F^{1})\widehat{\otimes}_{\mathcal{O}_{\mathfrak{X}_{r}}}(W\Omega_{X}^{1}/p^{r}/F^{1})
\]
which clearly factors through $(\mathcal{M}\widehat{\otimes}_{\mathcal{O}_{W(X)/p^{r}}}W\Omega_{X}^{1}/p^{r})/F^{1}$.
For the map in the other direction, note that there is a canonical
map 
\[
\mathcal{M}\widehat{\otimes}_{\mathcal{O}_{\mathfrak{X}_{r}}}W\Omega_{X}^{1}/p^{r}\to\mathcal{M}\widehat{\otimes}_{\mathcal{O}_{W(X)/p^{r}}}W\Omega_{X}^{1}/p^{r}
\]
\[
\to(\mathcal{M}\widehat{\otimes}_{\mathcal{O}_{W(X)/p^{r}}}W\Omega_{X}^{1}/p^{r})/F^{1}
\]
which factors through $(\mathcal{M}/F^{1})\widehat{\otimes}_{\mathcal{O}_{\mathfrak{X}_{r}}}(W\Omega_{X}^{1}/p^{r}/F^{1})$.
It is straightforward to see that these maps are inverse to one another. 
\end{proof}
This implies
\begin{cor}
Let $\mathcal{M}\in\text{MIC}_{W(X)_{p^{r}=0}}$. Then the surjection
$\mathcal{M}\to\mathcal{M}/F^{1}$ admits a unique $W_{r}(k)$-linear
splitting, $\iota$, satisfying the following: for $\overline{m}\in\mathcal{M}/F^{1}$,
$\iota(\overline{m})$ satisfies 
\[
\tilde{\nabla}(\iota(\overline{m}))\in\mathcal{M}/F^{1}\otimes_{\mathcal{O}_{\mathfrak{X}_{r}}}\Omega_{\mathfrak{X}_{r}}^{1}
\]
where the latter is regarded as a summand of $(\mathcal{M}\widehat{\otimes}_{\mathcal{O}_{W(X)/p^{r}}}W\Omega_{X}^{1}/p^{r})/F^{1}$
by the previous lemma. 
\end{cor}

\begin{proof}
Let $U\subset X$ be open affine. Then $\Gamma(U,\mathcal{M}/F^{1})=\Gamma(U,\mathcal{M})/\Gamma(F^{1})$.
This follows from the fact that $\mathcal{M}$ is quasicoherent on
$W(X)_{p^{r}=0}$. 

By the previous lemma, we have, for any $m\in\Gamma(U,\mathcal{M})$
we have 
\[
\tilde{\nabla}(m)=\sum_{(I,r)}m_{I}d(p^{r}T^{I/p^{r}})+\sum_{i=1}^{n}m_{i}dT_{i}
\]
So replacing $m$ with ${\displaystyle m-\sum_{(I,r)}m_{I}(p^{r}T^{I/p^{r}})}$
gives the existence of a lift of $\overline{m}$ with the required
property. If there are two such lifts, called $m_{1}$ and $m_{2}$,
then $m_{1}-m_{2}\in F^{1}(\mathcal{O}_{W(U)}/p^{r})\cdot\mathcal{M}$.
Therefore 
\[
\tilde{\nabla}(m_{1}-m_{2})\in\mathcal{M}/F^{1}\widehat{\otimes}_{\mathcal{O}_{\mathfrak{U}_{r}}}d(F^{1}(\mathcal{O}_{W(U)}/p^{r}))
\]
As both $\tilde{\nabla}(m_{1})$ and $\tilde{\nabla}(m_{2})$ take
values in $\mathcal{M}/F^{1}\otimes_{\mathcal{O}_{\mathfrak{U}_{r}}}\Omega_{\mathfrak{U}_{r}}^{1}$,
we see that $\tilde{\nabla}(m_{1}-m_{2})=0$. But then, writing 
\[
m_{1}-m_{2}=\sum_{(I,r)}p^{r}T^{I/p^{r}}m_{I}
\]
we see that 
\[
\tilde{\nabla}(m_{1}-m_{2})=\sum_{(I,r)}d(p^{r}T^{I/p^{r}})m_{I}=0
\]
which forces $m_{I}=0$ for all $I$, so that $m_{1}=m_{2}$ as required. 
\end{proof}
Note that the value of $\tilde{\nabla}$ on $\iota(\overline{m})$
is completely determined by the flat connection on $\mathcal{M}/F^{1}$.
Now we can give the
\begin{proof}
(of \prettyref{thm:Equivalence-mod-p^r} in the general case) As above
it suffices to consider the case $X=\text{Spec}(A)$. Then we may
consider
\[
\mathcal{H}om_{\tilde{\nabla}}(\widehat{\Phi}^{*}\mathcal{D}_{\mathfrak{X}_{r}}^{(0)},\mathcal{M})=\mathcal{N}
\]
where $\mathcal{H}om_{\tilde{\nabla}}$ refers to those continuous
morphisms which respect the induced map $\tilde{\nabla}$ on both
sides. As $\widehat{\Phi}^{*}\mathcal{D}_{\mathfrak{X}_{r}}^{(0)}/F^{1}=\mathcal{D}_{\mathfrak{X}_{r}}^{(0)}$
with its canonical connection, we see that a local section of $\mathcal{N}$
is determined by the value of the map at $1$, and by the remark directly
above we see that that we may send $1$ to any local section in the
image of the map $\iota$ constructed in the previous corollary. Thus
$\mathcal{N}$ is a summand of $\mathcal{M}$, and in fact we see
\[
\mathcal{N}\tilde{\to}\mathcal{M}/F^{1}
\]
as flat connections (the connection on $\mathcal{N}$ coming from
the right action of $\mathcal{D}_{\mathfrak{X}_{r}}^{(0)}$ on $\widehat{\Phi}^{*}\mathcal{D}_{\mathfrak{X}_{r}}^{(0)}$).
Therefore there is an induced morphism of $\tilde{\nabla}$-modules
\[
\widehat{\Phi}^{*}\mathcal{N}\to\mathcal{M}
\]
and as $\mathcal{N}$ is a summand of $\mathcal{M}$ when we apply
$\otimes_{W_{r}(k)}^{L}k$ we obtain the analogous maps 
\[
\widehat{\Phi}^{*}\mathcal{H}^{i}(\mathcal{N}\otimes_{W_{r}(k)}^{L}k)\to\mathcal{H}^{i}(\mathcal{M}\otimes_{W_{r}(k)}^{L}k)
\]
for $i\in\{0,-1,-2,\dots\}$. But these maps are isomorphisms by the
$r=1$ case of the theorem. Thus the original map $\widehat{\Phi}^{*}\mathcal{N}\to\mathcal{M}$
is an isomorphism as well. 

To finish the proof, we need to show that this map is an isomorphism
of de Rham-Witt connections (not just modules over $\tilde{\nabla}$).
This follows exactly as in the $r=1$ case above if we have the analogue
of \prettyref{lem:required-flatness}; namely, the statement that
the map 
\[
\overline{d}:\widehat{F^{1}\cdot(W\Omega_{X}^{1}/p^{r})}\to W\Omega_{X}^{2}/p^{r}\to(W\Omega_{X}^{2}/p^{r})/\widehat{(F^{1}\cdot W\Omega_{X}^{2}/p^{r})}
\]
is injective, with cokernel which is flat over $\mathcal{O}_{\mathfrak{X}_{r}}$.
But this follows from \prettyref{lem:required-flatness} and the fact
that $(W\Omega_{X}^{2}/p^{r})/\widehat{(F^{1}\cdot W\Omega_{X}^{2}/p^{r})}$
is flat over $W_{r}(k)$. 
\end{proof}

The University of Illinois at Urbana-Champaign, csdodd2@illinois.edu
\end{document}